\documentclass[12pt,letter]{amsart}
\usepackage{amscd}
\usepackage{amssymb}
\usepackage[centering,text={14.5cm,23cm}]{geometry}
\usepackage{graphicx}
\usepackage{color}
\usepackage[all]{xy}
\usepackage{mathrsfs}
\usepackage{marvosym}
\usepackage{stmaryrd}
\usepackage{srcltx}
\usepackage{hyperref}

\definecolor{shadecolor}{rgb}{1,0.9,0.7}

\newtheorem{theorem}{Theorem}[section]
\newtheorem{lemma}[theorem]{Lemma}
\newtheorem{proposition}[theorem]{Proposition}

\theoremstyle{definition}
\newtheorem{definition}[theorem]{Definition}
\newtheorem{construction}[theorem]{Construction}

\newtheorem{example}[theorem]{Example}
\newtheorem{examples}[theorem]{Examples}

\theoremstyle{remark}
\newtheorem{remark}[theorem]{Remark}

\numberwithin{equation}{section}
\numberwithin{figure}{section}

\newcommand{\doublebar}[1]{\shortstack{$\overline{\hphantom{#1}}
	$\\[-4pt] $\overline{#1}$}}

\newcommand {\lfor} {\llbracket}
\newcommand {\rfor} {\rrbracket}

\newcommand{\BB} {\mathbb{B}}

\newcommand{\NN} {\mathbb{N}}
\newcommand{\ZZ} {\mathbb{Z}}
\newcommand{\QQ} {\mathbb{Q}}
\newcommand{\RR} {\mathbb{R}}
\newcommand{\CC} {\mathbb{C}}

\newcommand{\PP} {\mathbb{P}}
\renewcommand{\AA} {\mathbb{A}}
\newcommand{\GG} {\mathbb{G}}
\newcommand{\TT} {\mathbb{T}}
\newcommand{\VV} {\mathbb{V}}

\newcommand {\shAff} {\mathcal{A}\text{\textit{ff}}}

\newcommand {\shD}  {\mathcal{D}}

\newcommand {\shExt} {\mathcal{E} \!\text{\textit{xt}}}

\newcommand {\shF}  {\mathcal{F}}

\newcommand {\shH}  {\mathcal{H}}
\newcommand {\shHom} {\mathcal{H}\!\text{\textit{om}}}
\newcommand {\shI}  {\mathcal{I}}

\newcommand {\shL}  {\mathcal{L}}
\newcommand {\shM}  {\mathcal{M}}
\newcommand {\shMPA} {\mathcal{MPA}}

\newcommand {\shPA} {\mathcal{PA}}
\newcommand {\shPL} {\mathcal{PL}}
\newcommand {\shQ}  {\mathcal{Q}}
\newcommand {\shR}  {\mathcal{R}}

\newcommand {\shT}  {\mathcal{T}}

\newcommand {\shP}  {\mathcal{P}}

\newcommand {\shX}  {\mathcal{X}}
\newcommand {\shY}  {\mathcal{Y}}

\newcommand {\foB}  {\mathfrak{B}}

\newcommand {\foD}  {\mathfrak{D}}

\newcommand {\foL}  {\mathfrak{L}}

\newcommand {\foP}  {\mathfrak{P}}

\newcommand {\foW}  {\mathfrak{W}}
\newcommand {\foX}  {\mathfrak{X}}
\newcommand {\foY}  {\mathfrak{Y}}

\newcommand {\foa}  {\mathfrak{a}}
\newcommand {\fob}  {\mathfrak{b}}

\newcommand {\fod}  {\mathfrak{d}}

\newcommand {\foj}  {\mathfrak{j}}

\newcommand {\fom}  {\mathfrak{m}}

\newcommand {\fop}  {\mathfrak{p}}

\newcommand {\fou}  {\mathfrak{u}}

\newcommand {\Aff}  {\operatorname{Aff}}

\newcommand {\Aut}  {\operatorname{Aut}}

\newcommand {\can} {{\mathrm{can}}}

\newcommand {\cl}  {\operatorname{cl}}
\newcommand {\codim} {\operatorname{codim}}
\newcommand {\cone} {{\mathbf{C}}}

\newcommand {\conv} {\operatorname{conv}}

\newcommand {\ev}  {\operatorname{ev}}

\newcommand {\Ext}  {\operatorname{Ext}}

\newcommand {\GL}  {\operatorname{GL}}

\newcommand {\Gm} {\GG_m}
\newcommand {\gp}  {{\operatorname{gp}}}

\newcommand {\Hom}  {\operatorname{Hom}}
\newcommand {\height}  {\operatorname{ht}}

\newcommand {\id}  {\operatorname{id}}
\newcommand {\im}  {\operatorname{im}}
\newcommand {\inc}  {\mathrm{in}}

\newcommand {\Int}  {\operatorname{Int}}
\newcommand {\inte}  {\mathrm{int}}

\renewcommand {\ker } {\operatorname{ker}}
\newcommand {\kk} {\Bbbk}

\newcommand {\liminv} {\varprojlim}

\newcommand {\lra}  {\longrightarrow}

\newcommand {\M} {\mathcal{M}}

\newcommand {\Map} {{\operatorname{Map}}}
\renewcommand {\max} {{\operatorname{max}}}

\newcommand {\MPA} {\operatorname{MPA}}

\newcommand {\NE}  {\operatorname{NE}}

\renewcommand{\O}  {\mathcal{O}}
\newcommand {\ob}  {\operatorname{ob}}

\newcommand {\ol} {\overline}

\renewcommand{\P}  {\mathscr{P}}
\newcommand {\PA} {\operatorname{PA}}

\newcommand {\PL} {\operatorname{PL}}

\newcommand {\Proj} {\operatorname{Proj}}

\newcommand {\rk} {\operatorname{rk}}

\newcommand {\scrS}  {\mathscr{S}}
\newcommand {\sat}  {{\operatorname{sat}}}

\newcommand {\Sing} {\operatorname{Sing}}

\newcommand {\Spec} {\operatorname{Spec}}
\newcommand {\Spf}  {\operatorname{Spf}}

\newcommand {\tor} {\mathrm{tor}}

\newcommand {\ul} {\underline}

\newcommand {\vect}  {\operatorname{\mathbf{vect}}}

\newcommand {\X} {\mathfrak X}

\newcommand {\W} {\mathscr{W}}
\newcommand {\bT} {\mathbf{T}}

\newcommand{\hooklongrightarrow}{\lhook\joinrel\longrightarrow}

\def\mapright#1{\smash{
  \mathop{\longrightarrow}\limits^{#1}}}

\def\mydate{\ifcase\month \or January\or February\or March\or
April\or May\or June\or July\or August\or September\or October\or 
November\or December\fi \space\number\day,\space\number\year}

\begin{document}

%===========================================================
\title[Theta functions]{Theta functions on varieties with effective
anti-canonical class}

\author{Mark Gross} \address{DPMMS, Centre for Mathematical Sciences,
Wilberforce Road, Cambridge, CB3 0WB, UK}
%\curraddr{}
\email{mgross@dpmms.cam.ac.uk}

% author two information
\author{Paul Hacking} \address{Department of Mathematics and
Statistics, University of Massachusetts, Amherst, MA01003-9305}
%\curraddr{}
\email{hacking@math.umass.edu}

% author three information
\author{Bernd Siebert} \address{FB Mathematik,
Universit\"at Hamburg, Bundesstra\ss e~55, 20146 Hamburg,
Germany}
%\curraddr{}
\email{siebert@math.uni-hamburg.de}

\thanks{This work was partially supported by NSF grants DMS-1262531
(M.G.), DMS-1201439 and DMS-1601065 (P.H.) and by a Royal Society 
Wolfson Merit Award (M.G.).}

\date{\today}
\maketitle

\begin{abstract}
We show that a large class of maximally degenerating families of $n$-dimensional
polarized varieties comes with a canonical basis of sections of powers of the
ample line bundle. The families considered are obtained by smoothing a reducible
union of toric varieties governed by a wall structure on a real
$n$-(pseudo-)manifold. Wall structures have previously been constructed
inductively for cases with locally rigid singularities \cite{affinecomplex} and
by Gromov-Witten theory for mirrors of log Calabi-Yau surfaces and $K3$ surfaces
\cite{GHK1},\cite{GHKS}. For trivial wall structures on the $n$-torus we
retrieve the classical theta functions.

We anticipate that wall structures can be constructed quite generally from
maximal degenerations. The construction given here then provides the homogeneous
coordinate ring of the mirror degeneration along with a canonical basis. The
appearance of a canonical basis of sections for certain degenerations points
towards a good compactification of moduli of certain polarized varieties via
stable pairs, generalizing the picture for K3 surfaces \cite{GHKS}.
Another possible application apart from mirror symmetry may be to geometric
quantization of varieties with effective anti-canonical class.
\end{abstract}

\tableofcontents
%\bigskip

%===========================================================
%===========================================================
\section*{Introduction}

It is anticipated that one can construct a mirror to any maximally unipotent
degeneration of Calabi-Yau varieties. Precisely, given a Calabi-Yau variety
$\shY_{\eta}$ over $\eta:=\Spec\kk((t))$ with maximally unipotent monodromy, one
should be able to construct a mirror variety $\foX$ defined over something like
the field of fractions of a completion of $\kk[\NE(\shY_\eta)]$, where
$\NE(\shY_{\eta})$ denotes the monoid of effective curve classes on
$\shY_{\eta}$.

While this general goal has not yet been achieved, various combinations of the
authors of this paper have obtained partial results in this direction. For these
results a crucial input is a suitably chosen extension $\shY\rightarrow\Spec
\kk\lfor t\rfor$ of $\shY_{\eta}$. 

Starting in \cite{logmirror1} and culminating in \cite{affinecomplex}, the first
and last authors of this paper showed how to construct the mirror if
$\shY\rightarrow\Spec\kk\lfor t \rfor$ was a sufficiently nice polarized
\emph{toric degeneration}. This is a degeneration whose central fibre is toric
and is described torically near the deepest points of the central fibre. The
mirror was then constructed as a toric degeneration
$\foX\rightarrow\Spec\kk\lfor t\rfor$ (and more generally a family of such). The
class of toric degenerations is a natural one from the point of view of mirror
symmetry, as the mirror of a toric degeneration is a toric degeneration. This
point of view incorporates, for example, all Batyrev-Borisov mirrors
\cite{GrBB}. However, it is not clear how generally toric degenerations can be
constructed given $\shY_{\eta}$.

On the other hand, in \cite{GHK1}, the first two authors jointly with
Sean Keel generalized certain aspects of the construction of
\cite{affinecomplex} to construct the mirror to an arbitrary log Calabi-Yau
surface $(Y,D)$ with $D$ an anti-canonical cycle of $n$ rational curves. These
mirrors were constructed as smoothings of a union of $n$ copies of $\AA^2$,
called the $n$-vertex, and the natural base space for the smoothing is
$\Spf( \kk\lfor \NE(Y)\rfor)$, with $\NE(Y)$ the cone of effective
curves in $Y$. Using similar techniques, \cite{GHKS} will provide mirrors to K3
surfaces $\shY_{\eta}$. This construction in particular will provide canonical
families over certain toroidal compactifications of $F_g$, the moduli space of
K3 surfaces of genus $g$. Both these papers used \emph{theta functions}, certain
canonically defined functions, as a key part of the construction. In particular,
in \cite{GHK1}, while it was easy to describe deformations of the $n$-vertex
with origin deleted, theta functions were necessary to provide an extension of
such deformations across the origin.

Nevertheless, the key point in common to these constructions is an explicit
description of the family $\foX$ whose starting point is combinatorial data
recorded in a cell complex of integral convex polyhedra that form a topological
manifold $B$, along with some additional data. The family is then constructed by
patching standard toric pieces extracted from the discrete data with corrections
carried by a \emph{wall structure}, a collection of real codimension one
rational polyhedra along with certain polynomial data. The main difficulty is
then determining a suitable wall structure. In the case of \cite{affinecomplex},
the wall structure was determined by a small amount of additional polynomial
starting data, and then an inductive process for the $k$-th step determined the
family $\foX\to \Spec \kk\lfor t\rfor$ modulo $t^{k+1}$. In \cite{GHK1},
however, the wall structure was written down all at once in terms of enumerative
data on $(Y,D)$. The wall structure in this case records the Gromov-Witten
theory of so-called $\AA^1$-curves in $Y\setminus D$, rational curves meeting
$D$ in exactly one point. In some sense it is a limiting case of
\cite{affinecomplex} in that the one-parameter families of \cite{affinecomplex}
arise after certain localizations of the base. In fact, the core argument for
why the wall structures provide a well-defined deformation relies on the
reduction to this situation via the enumerative interpretation of the inductive
process of wall insertion for toric surfaces in \cite{GPS}. A combination of the
two methods will be used in \cite{GHKS} in the general K3 case.

The purpose of this paper is two-fold. First, we want to provide a unified
framework for both kinds of construction. The focus is not on the specific
construction of wall structures that, depending on context, can come from an
inductive insertion process as in \cite{affinecomplex} or be related to certain
enumerative invariants as in \cite{GHK1}. Rather we are striving for maximal
generality in the treatment of singularities allowed on $B$ and in the treatment
of parameters leading to higher dimensional base spaces of families. The general
construction of wall-crossing structures will be done elsewhere using the
framework developed in this paper: see \cite{UtahSurvey} for
announcements of results in this direction.

Second, and more importantly, we treat in this framework the occurrence of a
canonical basis of global functions or, in the projective setting, of sections
of powers of an ample line bundle that the construction comes with. Thus we
provide here, in the projective setting, the homogeneous coordinate ring of the
family via an explicit basis as a module over the base space. In the case of
degenerations of abelian varieties the canonical sections agree with classical
theta functions. In fact, in Section~\ref{Sect: Abelian varieties} we show that
we obtain all classical theta functions naturally within our framework. We thus
also call our canonical sections \emph{theta functions}.

One of the main results of this paper is therefore the existence of theta
functions in the canonical degenerations constructed in \cite{affinecomplex}.

\begin{theorem}
\label{Thm: GS situation} Let $\pi: \foX\to S$ be one of the
canonical degenerations of varieties with effective anticanonical bundle over a
complete local ring $S$ constructed in \cite{affinecomplex}. Assume that there
is an ample line bundle $L$ on the central fibre $X_0\subseteq \foX$ that
restricts to the natural ample line bundles on the irreducible components
provided by the construction.

Then there is a distinguished extension of $L$ to an ample line bundle $\foL$ on
$\foX$, and $\foL^d$ for $d\ge 1$ has a canonical basis of sections indexed by
the $1/d$-integral points of $B$, the integral affine manifold underlying the
construction.
\end{theorem}

In the appendix we clarify the natural parameter space $S$ for
\cite{affinecomplex} in the Calabi-Yau situation, under the natural
local indecomposability assumption of the discrete data (``simple
singularities''). Theorem~\ref{Thm: GS situation} then follows from
the principal technical result Theorem~\ref{Thm: Main}.

In somewhat more detail, we discuss the broad picture presented in this paper.
The fundamental combinatorial object of the construction is an integral affine
manifold with singularities $B$ with a polyhedral decomposition $\P$. The
singular locus is taken to be as large as is possible for our approach: it is
(modulo some issues along the boundary of $B$) the union of codimension two
cells of the barycentric subdivision of $\P$ not intersecting the interiors of
maximal cells. Thus the singular locus is considerably bigger than is taken in
\cite{affinecomplex}. Further, unlike the previously cited work, we don't
actually insist that $B$ is a manifold: it can fail to be a manifold in
codimension $\ge 3$. While we do not give the details here, a typical situation
in which such a $B$ arises is as the dual intersection complex of a dlt minimal
model of a maximally unipotent degeneration of Calabi-Yau varieties. Koll\'ar
and Xu in \cite{KollarXu}, \S33 showed that $B$ will indeed be a manifold off of
a codimension three subset (see also \cite{NX} for somewhat weaker results).

The parameterizing family $S$ for our construction then arises by choosing a
ring $A$ and a toric monoid $Q$, so that we take $S=\Spec A[Q]/I$ for various
choices of ideal $I$ with radical a fixed ideal $I_0$. In \cite{affinecomplex},
$Q$ was taken to be $\NN$, while in \cite{GHK1}, typically $Q$ was closely
related to $\NE(Y)$. An additional combinatorial piece of data is a multi-valued
piecewise linear function $\varphi$ defined on $B_0:=B\setminus \Delta$ with
values in $Q^{\gp}\otimes_{\ZZ}\RR$. In \cite{affinecomplex}, this is viewed as
specified data, while in \cite{GHK1}, this function is canonically given by the
mirror construction presented there. This combinatorial data is all described in
\S\ref{Section: affine geometry}.

The goal then is to specify additional information which determines an
appropriate family $\foX\rightarrow S$. This family should have the property
that $\foX\times_S \Spec A[Q]/I_0$ is a union of polarized toric varieties
defined over $\Spec A[Q]/I_0$; these polarized toric varieties are determined by
their Newton polyhedra, which run over the maximal cells of $\P$, and are glued
together as dictated by the combinatorics of $\P$. The local structure of this
family over $S$ in neighbourhoods of codimension one strata should roughly be
determined by the function $\varphi$.

The necessary additional information is a \emph{wall structure} $\scrS$,
consisting of a collection of walls with attached functions. These walls
instruct us how to specify gluings between various standard charts. However,
unlike in \cite{affinecomplex}, we only have models for charts in codimensions 0
and $1$, and thus a wall structure is only able to produce a thickening
$\foX^{\circ} \rightarrow S$ of $X_0^{\circ}\rightarrow \Spec A[Q]/I_0$, the
reduced scheme obtained from $X_0$ by deleting codimension $\ge 2$ strata. This
construction is explained in \S\ref{Sect: Wall structures}. 

Roughly, in a mirror symmetry context, a wall structure can be viewed as a way
of encoding information about Maslov index zero disks with boundary in the fibre
of an SYZ fibration (where $B$ plays the role of the base of the fibration). We
expect, based on our experiences in \cite{GHK1} and \cite{GHKS}, that it will be
possible to define suitable wall structures in great generality using a version
of logarithmic Gromov-Witten invariants which shall be presented in forthcoming
work of Abramovich, Chen, Gross and Siebert \cite{ACGS}.

This leaves the question of (partially) compactifying the family
$\foX^{\circ}\rightarrow S$ to $\foX\rightarrow S$. This is where we make
contact with the innovation of \cite{GHK1}, where theta functions were used
precisely to achieve this compactification. If $X_0$ is affine, then a flat
infinitesimal deformation will also be affine, and hence we can hope to
construct $\foX$ by taking the spectrum of the $A[Q]/I$-algebra
$\Gamma(\foX^{\circ}, \O_{\foX^{\circ}})$. Thus we need the latter algebra to be
sufficiently large. This is achieved via the general construction of theta
functions, given in \S \ref{Sect: Global functions}, using broken lines. These
were introduced in \cite{PP2mirror} and first used to construct regular
functions in the context of \cite{affinecomplex} in \cite{CPS}. The definition
of broken line depends on the structure $\scrS$, and we say a structure is
\emph{consistent} if suitable counts of broken lines yield regular functions on
$\foX^{\circ}$. In the consistent affine case, theta functions can then be
viewed as canonically given lifts of monomial functions on $X_0$, and they are
labelled by asymptotic directions on $B$.

So far, this only allows the partial compactification in the affine
case. However, in \S\ref{Sect: Theta functions}, we turn to the
general case, most importantly including the projective case. The
key point is that in general we can reduce to the affine case as
follows. The cone over $B$ itself carries a natural affine
structure, and corresponds (after suitably truncating the cone) to
the total space of $\foL^{-1}$, where $\foL$ is an ample line bundle
on $\foX^{\circ}$ specified by the data of $B$. Then regular
functions on the total space of $\foL^{-1}$ homogeneous of weight
$d\ge 0$ with respect to the fibrewise $\Gm$-action correspond to
sections of  $\foL^{\otimes d}$. As a result, one is able to
construct a homogeneous coordinate ring for $\foX$.

In this projective context, theta functions are then viewed as
sections of $\foL^{\otimes d}$ for $d\ge 0$, and if $d>0$, these
functions are parameterized by the set $B({1\over d}\ZZ)$, the set
of points of $B$ with coordinates in ${1\over d}\ZZ$. Broken lines
can then be viewed via projection from the truncated cone over $B$
to $B$, to obtain objects we call jagged paths, see \S\ref{Subsect:
Jagged paths}. In fact, historically, jagged paths were discovered
before broken lines, in discussions between the first and third
authors of this paper and Mohammed Abouzaid.

The construction is then summarized as follows in the case that $B$
is compact. For a given base ring $R=A[Q]/I$, we define a
homogeneous graded $R$-algebra 
\[
A:=R\oplus \bigoplus_{d>0}\bigoplus_{p\in B({1\over d}\ZZ)} R \vartheta_p.
\]
We give a tropical rule for the multiplication law in terms of
counting trees with three leaves where the edges are jagged paths,
or the corresponding count in terms of broken lines on the truncated
cone over $B$. Note that here associativity follows from the fact
that the functions $\vartheta_p$ are actually functions on
$\foX^{\circ}$ constructed by gluing. We then define $\foX=\Proj
A$. 

Thus we emphasize there are three levels of tropical constructions: the wall
structure $\scrS$ on $B$ which governs the construction can be viewed as a
tropicalization of Maslov index zero disks. In the affine case the broken lines
which describe theta functions can be viewed as a tropicalization of Maslov
index two disks, while the trees which yield the multiplication law can be
viewed as a tropicalization of Maslov index four disks. In the projective case,
jagged paths contributing to the description of a theta function can be viewed
as tropicalizations of holomorphic disks contributing to Floer multiplication
for two Lagrangian sections and a fibre of the SYZ fibration, while the trees
which yield multiplication can be viewed as a tropicalization of holomorphic
disks contributing to Floer multiplication involving three Lagrangian
sections.

 This latter point of view has
been explained in \cite{DBr}, Chapter 8 and \cite{thetasurvey}.

Once a suitable theory for counting such disks in an algebro-geometric setting
is developed, it should be possible to write down the algebra $A$ directly from
enumerative geometry of a log Calabi-Yau variety, generalizing the construction
of \cite{GHK1}. This in turn should lead to a general mirror construction for a
maximally unipotent family of Calabi-Yau varieties. This chain of ideas will be
pursued elsewhere: see \cite{UtahSurvey} for a
more recent account. However, one should view the wall structure $\scrS$ as
giving the richest description of the construction.

The correspondence between points of $B({1\over d}\ZZ)$ and theta functions is
particularly illuminating in the case of abelian varieties. Classically, the
existence of a canonical basis of sections of powers of the ample line bundle
relies on explicit formulas. In the case of abelian varieties our formal family
is the completion of an analytic family $\shX\to \tilde S$ over an analytic open
subset $\tilde S$ of an affine toric variety, with $\foL$ the completion of a
holomorphic line bundle $\shL$. The affine manifold is a real $n$-torus
$B=\RR^n/\Gamma$ with $\Gamma\subseteq \RR^n$ a lattice of rank $n$, and
$B({1\over d}\ZZ)$ can be viewed as one-half of the kernel of the polarization
induced by $\shL^{\otimes d}$. In \S\ref{Sect: Abelian varieties}, we then show
that our theta functions coincide with classical theta functions. This was the
original motivation for using the term ``theta function'' for our canonical
functions.

In the appendix, we make the connection between the general framework we
consider here and that of \cite{affinecomplex}. We leave this discussion to the
appendix as the presentation in the rest of the paper is self-contained, but the
appendix relies on greater details from earlier work of the Gross-Siebert
program. The discussion of the appendix leads to the proof of Theorem \ref{Thm:
GS situation}.

In \cite{GHKK} theta functions are used to construct canonical bases
of cluster algebras. A cluster algebra can be understood as the ring
of global functions on the interior $U=Y \setminus D$ of a log
Calabi--Yau variety $(Y,D)$. The variety $U$ admits a flat
degeneration to an algebraic torus which is used to give a
perturbative construction of the theta functions. This is a special
case of the general construction described in this paper. In the
dimension $2$ case the theta functions can be described explicitly,
see \cite{CGMMRSW}. For cluster varieties describing the open double
Bruhat cell in a semi-simple algebraic group it is an open question
if our theta functions coincide with Lusztig's canonical basis, see
\cite{GHKK}, Corollary~0.20 and the discussion following it.

Other cases with an alternative characterization of theta functions include
mirrors to certain log Calabi-Yau surfaces \cite{GHK2} and possibly also
higher-dimensional Fano varieties with anticanonical polarization.

While in general we do not currently have a characterization
of theta functions other than via our construction, the case of
abelian varieties does lead to some speculation. Indeed, there is an
interpretation of classical theta functions in terms of geometric
quantization that also generalizes to moduli spaces of flat bundles
over a Riemann surface \cite{axelrodetal}, \cite{hitchin},
\cite{tyurin}, \cite{baieretal}. From this point of view, the
degeneration of abelian varieties is viewed as a degenerating family
of complex structures on a fixed Lagrangian fibration $A\to B$.
Similarly, $\shL$ can be viewed as a degenerating family of
compatible complex structures on a complex line bundle $L$ over $A$.
Viewing the complex structure as a distribution on the tangent
spaces, the limit $s\to 0$ is given by the tangent spaces to fibres
of the Lagrangian fibration. In this picture, a $1/d$-integral point
$x\in B$ labels a distributional section of $L$ with support the
fibre of the Lagrangian fibration over $x$. These distributional
sections provide the initial data for the heat equation fulfilled
by classical theta functions due to the functional equation.

A similar picture is expected to hold in much greater generality
\cite{andersen}. In the context of geometric quantization of
Calabi-Yau varieties with a Lagrangian fibration provided by the SYZ
conjecture, the existence of generalized theta functions  was indeed
conjectured by the late Andrei Tyurin \cite{tyurin}. We believe that
our theta functions should also fulfill some heat equation with
distributional limit over the limiting Lagrangian fibration, but the
nature of this equation is unknown to date.

The present theta functions were conjectured to exist by the first
and third authors of this article in the context of homological
mirror symmetry applied to the degenerations of
\cite{affinecomplex}. In the affine case the first proof of
existence in dimension two has appeared in \cite{GHK1}, while
\cite{CPS} established the existence of canonical functions in any
dimension in the framework of \cite{affinecomplex}. See
\cite{thetasurvey} for more details on the history.
\bigskip

\emph{Acknowledgements}: We would like to thank M.~Abouzaid, J.~Andersen,
D.~Pomerleano, S.~Keel and C.~Xu for discussions on various aspects of
this paper. We also thank the anonymous referee for his very attentive
reading.

%===========================================================
%===========================================================
\section{The affine geometry of the construction}
\label{Section: affine geometry}

%===========================================================
\subsection{Polyhedral affine pseudomanifolds}
\label{Subsect: Polyhedral affine manifolds}

We give a common setup for \cite{logmirror1}, \cite{affinecomplex}
and \cite{GHK1}. An \emph{affine manifold} $B_0$ is a differentiable
manifold with an equivalence class of charts with transition
functions in $\Aff(\RR^n)= \RR^n \rtimes \GL(\RR^n)$. It is
\emph{integral} if the transition functions lie in $\Aff(\ZZ^n)=
\ZZ^n\rtimes\GL(\ZZ^n)$. A map between (integral) affine manifolds
preserving this structure is called an \emph{(integral) affine map}.
In the integral case it makes sense to talk about
\emph{$1/d$-integral points} $B_0(\frac{1}{d}\ZZ)\subseteq B_0$,
locally defined as the preimage of $\frac{1}{d}\ZZ^n\subseteq \RR^n$
in a chart. An integral affine manifold $B_0$ comes with a sheaf of
integral (co-) tangent vectors $\Lambda=\Lambda_{B_0}$ (dually
$\check\Lambda=\check\Lambda_{B_0}$) and of integral affine
functions $\shAff(B_0,\ZZ)$. These sheaves are locally constant with
stalks isomorphic to $\ZZ^n$ and to $\Aff(\ZZ^n,\ZZ) \simeq
\ZZ^n\oplus\ZZ$, respectively. The corresponding real versions are
denoted $\Lambda_\RR$, $\check\Lambda_\RR$ and $\shAff(B_0,\RR)$.
Generally, if $A$ is an abelian group then $A_\RR:=
A\otimes_\ZZ\RR$. We have an exact sequence
\begin{equation}
\label{Eqn: Aff}
0\lra \ul\ZZ \lra \shAff(B_0,\ZZ)\lra \check\Lambda\lra 0,
\end{equation}
dividing out the constant functions. Taking $\Hom_{B_0}
(\check\Lambda,\,.\,)$ provides a connecting homomorphism
\[
\Hom_{B_0}(\check\Lambda,\check\Lambda)\lra
\Ext^1_{B_0}(\check\Lambda,\ZZ)=
H^1(B_0,\Lambda).
\]
The image of the identity defines the \emph{radiance obstruction} of
$B_0$, which is an obstruction class to the existence of a set of
charts with linear rather than affine transition functions
(see \cite{GH} or \cite{logmirror1}, pp.179ff).

A (convex) \emph{polyhedron} in $\RR^n$ is the solution set of
finitely many affine inequalities. A polyhedron is \emph{integral}
if each face contains an integral point and the affine inequalities can be
taken with rational coefficients. In particular, any vertex of an
integral polyhedron is integral. We use lower case Greek letters
for integral polyhedra, where we reserve $\sigma,\sigma',\ldots$ for
maximal cells and $\rho$ for codimension-one cells. For a
polyhedron $\tau$ we write $\partial\tau$ for the union of proper
faces of $\tau$ and $\Int\tau:=\tau\setminus \partial\tau$ for the
complement. Note that for $\tau\subseteq\RR^n$ and $\dim\tau<n$ this
does not agree with the topological boundary. Another notation is
$\Lambda_\tau$ for the sheaf of integral tangent vectors on $\tau$,
viewed as an integral affine manifold with boundary. We will not be
too picky and sometimes also use the notation $\Lambda_\tau$ for the
stalk of $\Lambda_\tau$ at any $y\in\Int\tau$ or the abelian group
of global sections $\Gamma(\tau, \Lambda_\tau)$. The precise meaning
should always be obvious from the context. Also, if
$\tau\subseteq\tau'$ we consider $\Lambda_\tau$ naturally as a
subgroup of $\Lambda_{\tau'}$.

The arena for all that follows is a topological space $B$ of
dimension $n$, possibly with boundary, with an integral affine
structure on $B_0:=B\setminus\Delta$ with $\Delta\subseteq B$ of
codimension two, and a compatible decomposition $\P$ into integral
polyhedra. Unlike in much previous work, we will not assume that $B$
is a manifold, but rather will have some weaker properties.
The details are contained in the following construction.

\begin{construction}
\label{Construction: B}
\emph{(Polyhedral affine manifolds.)}
Let $\P$ be a set of integral polyhedra along with a set of integral
affine maps $\omega\to\tau$ identifying $\omega$ with a face of
$\tau$, making $\P$ into a category. We require that any proper face
of any $\tau\in\P$ occurs as the domain of an element of $\hom(\P)$
with target $\tau$. We assume that the direct limit in the category
of topological spaces
\[
B:=\varinjlim_{\tau\in\P}\tau
\]
satisfies the following conditions:
\begin{enumerate}
\item 
For each $\tau\in\P$ the map $\tau\rightarrow B$ is injective, that is, no 
cells self-intersect
(unlike in \cite{logmirror1}).
\item
By abuse of notation we view the elements of $\P$ as subsets
of $B$, also referred to as \emph{cells} of $\P$. We assume that the
intersection of any two cells of $\P$ is a cell of $\P$.
\item $B$
is pure dimension $n$, in the sense that every cell  of $\P$ is
contained in at least one $n$-dimensional cell.
\item
Every $(n-1)$-dimensional cell of $\P$ is contained in one or two
$n$-dimensional cells, so that $B$ is a manifold with boundary away from
codimension $\ge 2$ cells. 
\item
\emph{The $S_2$ condition.} If $\tau\in\P$ satisfies $\dim\tau\le n-2$, then
any $x\in\Int \tau$ has a neighbourhood basis in $B$
consisting of open sets $V$ with $V\setminus\tau$ connected.
\end{enumerate}
If $\P$ consisted only of simplices, then the above conditions are
somewhat stronger than the usual notion of pseudomanifold (with
boundary). For lack of better terminology, and to remind the reader
that $B$ need not be a manifold, we call $B$ a pseudomanifold, but
the reader should also remember the precise conditions stated above.

Cells of dimensions $0$, $1$ and $n$ are also called
\emph{vertices}, \emph{edges} and \emph{maximal cells}. The notation
for the set of $k$-cells is $\P^{[k]}$ and we often write
$\P_\max:=\P^{[n]}$ for the set of maximal cells. A cell
$\rho\in\P^{[n-1]}$ only contained in one maximal cell is said to
lie on the \emph{boundary} of $B$, and we let $\partial B$ be the
union of all $(n-1)$-cells lying on the boundary of $B$. Any cell of
$\P$ contained in $\partial B$ is called a \emph{boundary cell}.
Cells not contained in $\partial B$ are called \emph{interior},
defining $\P_\inte\subseteq \P$. Thus $\P_\partial:=\P\setminus\P_\inte$
is the induced polyhedral decomposition of $\partial B$. 

Next we want to endow $B$ with an affine structure outside a subset
$\Delta\subseteq B$ of codimension two, sometimes referred to as the
\emph{discriminant locus}. For $\Delta$ we take the union of the $(n-2)$-cells
of a barycentric subdivision $\tilde\P$ of $\P$ that neither intersect the
interiors of maximal cells nor the interiors of maximal cells of the boundary
$\partial B$. Two remarks are in order here. First, while the barycenter of a
bounded polyhedron can be defined invariantly in affine geometry, the precise
location of $\Delta$ is not important as long as it respects the cell structure.
So the construction of $\Delta$ is purely topological. Second, for an unbounded
cell $\tau$ we take the barycenter at infinity, that is, replace the barycenter
by an unbounded direction $u_\tau\in(\Lambda_\tau)_\RR$. A piecewise linear
choice of $\Delta$ is explained in \cite{affinecomplex}, p.1310 and runs as
follows. For each bounded cell choose a point $a_\tau\in \Int \tau$, which is to
become its barycenter. For unbounded cells the direction vectors $u_\tau$ need
to be parallel for faces with the same asymptotic cone.\footnote{The notion of
asymptotic cone is discussed at the beginning of \S\ref{Sect: Wall structures}.}
Then a $k$-cell of $\tilde\P$ labelled by a sequence $\tau_0\subsetneq
\tau_1\subsetneq\ldots \subsetneq \tau_k$ in $\P$ with $\tau_0,\ldots,\tau_l$,
$l\ge 0$, bounded and $\tau_{l+1},\ldots,\tau_k$ unbounded is taken as $\conv\{
a_{\tau_0},\ldots, a_{\tau_l}\}+ \sum_{i=l+1}^k \RR_{\ge 0}
u_{\tau_i}$.\footnote{Note that in the unbounded case the $\tau_i$ need to have
strictly ascending asymptotic cones for the dimension of this cell of $\tilde
\P$ to be $k$.} For unbounded $\tau\in\P$ with bounded faces of
dimension $n-1$ there is then a deformation retraction of $\tau\setminus\Delta$
to the union of bounded faces of $\tau\setminus\Delta$. Note that if
$\sigma,\sigma' \in\P_\max$ intersect in $\rho\in\P^{[n-1]}$ then
$\rho\not\subseteq\partial B$ and $\rho\setminus \Delta$ has a number of
connected components, one for each $(n-1)$-cell of the barycentric subdivision
of $\rho$. Thus each connected component of $\rho\setminus \Delta$ is labelled
uniquely by a sequence $\tau_0\subseteq \tau_1\subseteq \ldots\subseteq
\tau_{n-1}=\rho$ with $\tau_k\in\P^{[k]}_\inte$. We denote such an $(n-1)$-cell
of the barycentric subdivision of $\rho$ by $\ul\rho\in \tilde\P^{[n-1]}_\inte$.
With this notation it is understood that $\ul\rho$ is contained in
$\rho\in\P^{[n-1]}_\inte$. In particular, we take only those $(n-1)$-cells of
$\tilde\P$ that do not intersect the interiors of maximal cells of $\P$.

To define an affine structure on $B_0:=B\setminus \Delta$ compatible with the
given affine structure on the cells it suffices to provide, for each $\ul\rho\in
\tilde\P^{[n-1]}_\inte$, an identification of tangent spaces of the adjacent
maximal cells $\sigma,\sigma'$ inducing the identity on $\Lambda_\rho$.
Equivalently, if $\xi\in\Lambda_\sigma$ is such that $\Lambda_\rho+\ZZ\xi=
\Lambda_\sigma$ then for each $\ul\rho\in\tilde\P^{[n-1]}_\inte$ with
$\ul\rho\subseteq\rho$ we have to provide $\xi'\in\Lambda_{\sigma'}$ with
$\Lambda_\rho+\ZZ\xi' =\Lambda_{\sigma'}$. Each such data defines an integral
affine structure on $B_0$ via the local identification of tangent
spaces taking $\xi$ to $\xi'$.

This ends the construction of the pseudomanifold $B$, a
codimension two subset $\Delta$, a decomposition $\P$ of $B$ into
integral affine polyhedra and a compatible integral affine structure
on $B_0$. For brevity we refer to all these data as a
\emph{polyhedral affine pseudomanifold} or just \emph{polyhedral
pseudomanifold}, denoted $(B,\P)$.
\end{construction}

\begin{remark}
\label{Rem: topology of B minus Delta}
The complement $B_0$ of $\Delta$ retracts onto a simplicial complex of dimension
one. In fact, by the very definition of $\Delta$, $B_0$ is covered by the
interiors as subsets of $B$ of the maximal cells $\sigma\in \P$ and by
$\ul\rho\setminus\Delta$, $\ul\rho\in\tilde\P^{[n-1]}_\inte$. By assumption on
$B$, each interior $(n-1)$-cell is contained in precisely two maximal cells.
Thus $B_0$ deformation retracts to a one-dimensional simplicial subspace having
one vertex $a_\sigma\in\Int \sigma$ for each $\sigma\in \P_\max$ and an edge
connecting $a_\sigma$, $a_{\sigma'}$ for each $\ul\rho\in
\tilde\P^{[n-1]}_\inte$ with $\ul\rho\subseteq\sigma\cap \sigma'$.
\end{remark}

There are two major series of examples.

\begin{examples}
\label{Expl: standard examples}
1)\ \ In \cite{logmirror1}, \cite{affinecomplex} the affine structure extends
over a neighbourhood of the vertices. In fact, in this case we can replace
$\Delta$ by the union $\breve\Delta$ of $(n-2)$-cells of $\Delta$ not containing
any vertex. This example also requires a compatibility condition between the
charts (\cite{affinecomplex}, Definition~1.2), which only arises if $B_0$
intersects cells of codimension at least two. Additional aspects of the main
body of this paper particular to this case are discussed in the appendix. Note
in this case $B$ is actually a topological manifold, possibly with
boundary.\\[1ex]
2)\ \ In \cite{GHK1}, \cite{GHKS}, the polyhedral affine pseudomanifolds
used (while still actually manifolds)
are quite different from those of~(1). These two papers give
two-dimensional examples where all singularities occur at vertices
of a polyhedral decomposition. In the case of \cite{GHK1}, one
starts with a so-called Looijenga pair $(Y,D)$, that is, $Y$ is a rational
surface and $D\in |-K_Y|$ is a cycle of rational curves. Write
$D=D_1+\cdots+D_n$ in cyclic order. One associates to the pair its
dual intersection complex $(B,\Sigma)$. Topologically $B=\RR^2$ and
$\Sigma$, the polyhedral decomposition, is a complete fan with a
two-dimensional cone $\sigma_{i,i+1}$ associated to each double
point $D_i\cap D_{i+1}$ and ray $\rho_i= \sigma_{i-1,i}\cap
\sigma_{i,i+1}$ associated to each irreducible component $D_i$.
Abstractly, $\sigma_{i,i+1}$ is integral affine isomorphic to the
first quadrant of $\RR^2$. The discriminant locus $\Delta$ coincides
with the zero-dimensional cell in $\Sigma$, which we denote by $0$. 
The affine structure on $B_0$ is given by charts 
\[
\psi_i:U_i=\Int(\sigma_{i-1,i}\cup\sigma_{i,i+1})
\rightarrow \RR^2
\]
where $\psi_i$ is defined on the closure of $U_i$ by
\[
\psi_i(v_{i-1})=(1,0), \quad \psi_i(v_i)=(0,1), \quad \psi_i(v_{i+1})=
(-1,-D_i^2)
\]
with $v_i$ denoting a primitive generator of $\rho_i$ and $\psi_i$
is defined linearly on the two two-dimensional cones. 

If one wishes a compact example with boundary, one can choose a
compact two-dimensional subset $\bar B\subseteq B$ with polyhedral
boundary and $0\in\Int \bar B$. In certain cases one may find such
a $\bar B$ with locally convex boundary. Indeed, one can show
that such a $\bar B$ exists if $D$ supports a nef and big divisor.

In \cite{GHKS}, we will need a version of this applied to
degenerations of K3 surfaces. Let $\shY\rightarrow T$ be a
one-parameter degeneration of K3 surfaces which is simple normal
crossings, relatively minimal, and maximally unipotent. Let $(B,\P)$
be the dual intersection complex of the degenerate fibre: $\P$ has a
vertex $v$ for every irreducible component $Y_v$ of the central
fibre $\shY_0$, and $\P$ contains a simplex with vertices
$v_0,\ldots,v_n$ if $Y_{v_0}\cap\cdots\cap Y_{v_n}\not=\emptyset$.
We take $\Delta$ to be the set of vertices. The affine structure is
defined as follows. Each two-dimensional simplex of $\P$ carries the
affine structure of the standard simplex. Given simplices of $\P$
with a common edge,
\[
\sigma_1=\langle v_0,v_1,v_2\rangle, \quad \sigma_2=\langle
v_0,v_1,v_3\rangle,
\]
we define a chart $\psi:\Int(\sigma_1\cup\sigma_2)\rightarrow \RR^2$
via
\begin{equation}
\label{Eq: GHKS chart}
\psi(v_0)=(0,0), \quad \psi(v_1)=(0,1), \quad \psi(v_2)=(1,0), \quad
\psi(v_3)=(-1,-(Y_{v_0}\cap Y_{v_1})^2),
\end{equation}
where the latter self-intersection is computed in $Y_{v_0}$. Again,
$\psi$ is affine linear on each two-cell.

These constructions generalize to higher dimensions, producing many
examples of polyhedral affine pseudomanifolds with singular locus the union
of codimension two cells. This can be applied, for example, to log Calabi-Yau
manifolds with suitably well-behaved compactifications, or to log smooth
relatively minimal maximally unipotent degenerations of Calabi-Yau manifolds.
More generally, one can consider relatively minimal dlt models of
such degenerations.
The general construction will be taken up elsewhere.
\end{examples}

Continuing with the general case, an important piece of data that
comes with a polyhedral pseudomanifold are certain tangent vectors along
any codimension one cell $\rho$ that encode the monodromy of the
affine structure in a neighbourhood of $\rho$. Let
$\ul\rho,\ul\rho'\subseteq \rho$ be two $(n-1)$-cells of the
barycentric subdivision, and let $\sigma,\sigma'\in\P^{[n]}$ be the
maximal cells adjacent to $\rho$. Consider the affine parallel
transport $T$ along a path starting from $x\in\Int\ul\rho$ via
$\Int\sigma$ to $\Int\ul\rho'$ and back to $x$ through
$\Int\sigma'$. By the definition of the affine structure on $B_0$
this transformation leaves $\Lambda_\rho\subseteq \Lambda_x$
invariant. Thus $T$ takes the form
\begin{equation}
\label{Eqn: monodromy vectors}
T(m)= m+ \check d_\rho(m)\cdot m_{\ul\rho\,\ul\rho'},
\quad m\in \Lambda_x,
\end{equation}
where $\check d_\rho\in\check\Lambda_x$ is a generator of
$\Lambda_\rho^\perp\subseteq \check\Lambda_x$ and
$m_{\ul\rho\,\ul\rho'}\in\Lambda_\rho$. To fix signs we require
$\check d_\rho$ to take non-negative values on $\sigma$. Since changing
the roles of $\sigma$ and $\sigma'$ reverses both the sign of
$\check d_\rho$ and the orientation of the path, the \emph{monodromy
vector} $m_{\ul\rho\,\ul\rho'}$ is well-defined. Note also that
$m_{\ul\rho'\ul\rho}= -m_{\ul\rho\,\ul\rho'}$.

In the first series of examples (Example~\ref{Expl: standard
examples},1) the connected components of $\rho\setminus\Delta$ are
in bijection with vertices $v\in \rho$, and the notation was
$m_\rho^{vv'}$. In the second series of examples (Example~\ref{Expl:
standard examples},2) the affine structure extends to a
neighbourhood of $\Int\rho$ and hence $m_{\ul \rho\,\ul\rho'}=0$.
\smallskip

The last topic in this subsection concerns the case $\partial
B\neq\emptyset$. First note that the boundary $\partial B$ of $B$
does not generally carry a natural structure of connected polyhedral
pseudomanifold. In fact, $\partial B\setminus\Delta$ is merely the
disjoint union of the interiors of the cells $\rho\in\P^{[n-1]}
\setminus\P^{[n-1]}_\inte$. An exception is if for any pair of
adjacent $(n-1)$-cells $\rho,\rho'\subseteq\partial B$ the tangent
spaces $\Lambda_\rho$, $\Lambda_{\rho'}$ are parallel, measured in a
chart at some point close to $\rho\cap\rho'\in \P^{[n-2]}$. Then
$\partial B$ with the induced polyhedral decomposition is naturally
a polyhedral sub-pseudomanifold of $(B,\P)$. While this case has some
special importance (see e.g.\ \cite{CPS}), it is irrelevant in this
paper. We therefore always assume $\Delta$ contains the
$(n-2)$-skeleton of $\partial B$.

Unlike in \cite{affinecomplex} we also make no assumption on local
convexity of $B$ along its boundary.
\medskip

%===========================================================
\subsection{Convex, piecewise affine functions}

The next ingredient is a multi-valued convex PL-function on $B_0$.
Here ``PL'' stands for ``piecewise linear''. Let $Q$ be a toric
monoid and $Q_\RR\subseteq Q_\RR^\gp$ the corresponding cone, that is,
$Q=Q^\gp\cap Q_\RR$. Recall that a monoid $Q$ is called \emph{toric}
if it is finitely generated, integral, saturated and if in addition
$Q^\gp$ is torsion-free. Thus toric monoids are precisely the
monoids that are isomorphic to a finitely generated saturated
submonoid of a
free abelian group.\footnote{We do not require toric monoids to be sharp,
that is, $Q$ may have non-trivial invertible elements.}

\begin{definition}
\label{Def: PL-functions}
A \emph{$Q^\gp$-valued piecewise affine (PA-) function} on an open set
$U\subseteq B_0=B\setminus \Delta$ is a continuous map
\[
U\lra Q_\RR^\gp
\]
which restricts to a $Q_\RR^\gp$-valued integral affine function on
each maximal cell of $\P$. The sheaf of $Q^\gp$-valued integral piecewise affine
functions on $B_0$ is denoted $\shPA(B,Q^\gp)$. The sheaf of
\emph{$Q^\gp$-valued piecewise linear (PL-) functions} is the quotient
$\shPL(B,Q^\gp):= \shPA(B, Q^\gp)/\ul Q^\gp$ by the locally constant
functions. The respective spaces of global sections are denoted
$\PA(B,Q^\gp)$ and $\PL(B,Q^\gp)$.
\end{definition}

\begin{remark}
\label{Rem: PL functions}
This definition is less restrictive than the one given in
\cite{logmirror1}, Definition~1.43. In particular, we do not require
that locally around the interior of $\rho\in\P^{[n-1]}$ a
PA-function $\varphi$ is the sum of an affine function and a
PA-function on the quotient fan along $\rho$. If
$\rho\not\subseteq\partial B$ then this quotient fan is just the fan
of $\PP^1$ in $\RR$. The condition says that the change of slope
(cf.\ Definition~\ref{Def: kink} below) of $\varphi$ along a connected
component of $\rho\setminus\Delta$ is independent of the choice of
connected component. See Example~\ref{Expl: Kink not constant} for
an illustration.
\end{remark}

The change of a PA- (or PL-) function $\varphi$ along
$\ul\rho\in\tilde\P^{[n-1]}_\inte$ is given by an element $\kappa\in
Q^\gp$ as follows. Let $\sigma,\sigma'$ be the two maximal cells
containing $\ul\rho$. Then $V:=\Int\sigma\cup\Int\sigma'\cup
\Int\ul\rho$ is a contractible  open neighbourhood of $\Int\ul\rho$
in $B_0=B\setminus\Delta$. An affine chart at $x\in\Int\ul\rho$ thus
provides an identification $\Lambda_\sigma=\Lambda_x=
\Lambda_{\sigma'}$. Let $\delta:\Lambda_x\to\ZZ$ be the quotient by
$\Lambda_\rho \subseteq\Lambda_x$. Fix signs by requiring that $\delta$
is non-negative on tangent vectors pointing from $\rho$ into
$\sigma'$. Let $n,n'\in \check\Lambda_x\otimes Q^\gp$ be the slopes of
$\varphi|_\sigma$, $\varphi|_{\sigma'}$, respectively. Then
$(n'-n)(\Lambda_\rho)=0$ and hence there exists $\kappa\in Q^\gp$
with
\begin{equation}
\label{Eqn: kink}
n' -n =\delta \cdot\kappa.
\end{equation}

\begin{definition}
\label{Def: kink}
The element $\kappa_{\ul\rho}(\varphi):=\kappa$ defined in
\eqref{Eqn: kink} is called the \emph{kink} of the $Q^\gp$-valued
PA-function $\varphi$ along $\ul\rho\in\tilde\P^{[n-1]}_\inte$.
\end{definition}

Clearly, a PA-function is integral affine on an open set $U\subseteq
B_0$ if and only if $\kappa_{\ul\rho}(\varphi)=0$ whenever
$U\cap\ul\rho\neq\emptyset$. Moreover, if $U$ is connected, then a
PA-function $\varphi$ on $U$ is determined uniquely by the
restriction to $\Int\sigma$ for one $\sigma\in\P_\max$ intersecting
$U$ and the kinks $\kappa_{\ul\rho}(\varphi)$. Conversely, if
$U\subseteq B_0$ is simply-connected then there exists a
PA-function $\varphi$ with any prescribed set of kinks
$\kappa_{\ul\rho}(\varphi)\in Q^\gp$.

\begin{example}
\label{Expl: Kink not constant}
To illustrate how the kink can depend on the choice of
$\ul\rho\subseteq\rho$ let us look at the simplest example of an
affine manifold with singularities, see Example~1.16 in
\cite{logmirror1}, or \S3.2 in \cite{invitation}. There are only two
maximal cells, the $2$-simplices $\sigma_1:=
\conv\{(-1,0),(0,0),(0,1)\}$,  $\sigma_2:= \conv\{(0,0),(1,0),
(0,1)\}$. Take $B=\sigma_1\cup \sigma_2\subseteq \RR^2$ as a
topological manifold and $\Delta= \{(0,1/2)\}$ the midpoint of the
interior edge $\rho$. Then $\rho= \ul\rho_1\cup\ul\rho_2$ with
$\ul\rho_\mu\in \tilde\P^{[n-1]}_\inte$ two intervals of integral
affine length $1/2$, say $\ul\rho_1$ the lower one containing
$(0,0)$ and $\ul\rho_2$ the upper one containing $(0,1)$. The given
embedding into $\RR^2$ defines the affine chart on $B\setminus
\ul\rho_2$ (Chart~I). The other chart (Chart~II) is given by
applying $\scriptsize\begin{pmatrix} 1& 0\\ 1&1\end{pmatrix}$ to
$\sigma_2$. Thus the image of this chart is
$\conv\{(-1,0),(0,0),(1,1),(0,1)\}$ minus the image of $\ul\rho_1$.

Writing $x,y$ for the standard coordinates on $\RR^2$, consider the
function $\varphi$ that in Chart~I is given by $y$. In this chart it
is an affine function and hence has kink $\kappa=0$. However, in
Chart~II the restriction of $\varphi$ to the image of $\sigma_1$
equals $y$ and the restriction to the image of $\sigma_2$ equals
$y-x$. Thus $\kappa_{\ul\rho_1}(\varphi)=0$ while
$\kappa_{\ul\rho_2}(\varphi)=-1$. In particular, $\varphi$ is not a
piecewise affine function in the sense of \cite{logmirror1},
Definition~1.43, but it is in the sense of this paper.

Note also that the described phenomenon can only occur under the
presence of non-trivial monodromy $m_{\ul\rho\,\ul\rho'}\neq 0$ along
$\rho$ \eqref{Eqn: monodromy vectors}.
\end{example}

\begin{definition}
The sheaf of \emph{$Q^\gp$-valued multivalued piecewise affine
(MPA-) functions} on $B_0=B\setminus\Delta$ is
\[
\shMPA(B, Q^\gp):= \shPA(B, Q^\gp)/ \shAff(B,Q^\gp).
\]
A section of $\shMPA(B, Q^\gp)$ over an open set $U\subseteq B_0$ is
called a \emph{($Q^\gp$-valued) MPA-function}, and we write
$\MPA(B,Q^\gp):= \Gamma(B_0, \shMPA(B,Q^\gp))$.
\end{definition}

Note that dividing out locally constant functions gives the
alternative definition
\[
\shMPA(B, Q^\gp)= \shPL(B, Q^\gp)/ \shHom(\Lambda,\ul Q^\gp).
\]

Since $\shHom(\Lambda,\ul Q^\gp) = \check\Lambda\otimes\ul Q^\gp$
there is an exact sequence of abelian sheaves on $B\setminus\Delta$,
\begin{equation}
\label{Eqn: MPA}
0 \lra \check\Lambda \otimes\ul Q^\gp \lra \shPL(B,Q^\gp)\lra
\shMPA(B, Q^\gp)\lra 0.
\end{equation}
The connecting homomorphism of the restriction to $U\subseteq
B\setminus \Delta$,
\begin{equation}
\label{Eqn: c_1(varphi)}
c_1: \MPA(U,Q^\gp)\lra H^1(U,\check\Lambda\otimes Q^\gp),
\end{equation}
measures the obstruction to lifting an MPA-function $\varphi$ on $U$
to a PL-function. The notation $c_1(\varphi)$ comes from the
interpretation on the Legendre-dual side as being the affine
representative of the first Chern class of a line bundle defined by
$\varphi$, see \cite{logmirror1}, \cite{logmirror2}. We are working
in what is called the \emph{cone picture} here (\cite{logmirror1},
\S2.1), while the bulk of the discussion in \cite{logmirror1},
\cite{logmirror2} takes place in the \emph{fan picture}
(\cite{logmirror1},\S2.2). The two pictures are related by a
\emph{discrete Legendre transform} (\cite{logmirror1}, \S1.4), which
swaps the roles of $c_1(\varphi)$ and the radiance obstruction $c_B$
of $B$ (\cite{logmirror1}, Proposition~1.50,3). Thus in the current
paper, $c_1(\varphi)$ takes the role of the radiance obstruction in
\cite{logmirror2}, which represents the residue of the Gauss-Manin
connection (\cite{logmirror2}, Theorem~5.1,(4)). Hence in the
present setup $c_1(\varphi)$ is related to the complex structure
moduli. Indeed, the MPA-function $\varphi$ has a prominent role in
the construction of our deformation, see \S\ref{Subsect: foX^o}.

An MPA-function is uniquely determined by its kinks:

\begin{proposition}
\label{Prop: MPA is defined by kinks}
There is a canonical decomposition
\[
\shMPA (B, Q^\gp)= \bigoplus_{\ul\rho\in \tilde\P^{[n-1]}_\inte}
Q^\gp_{\Int\ul\rho},
\]
where $Q^\gp_{\Int\ul\rho}$ is the push-forward to $B_0$ of the
locally constant sheaf on $\Int\ul\rho$ with stalks $Q^\gp$.
The induced canonical isomorphism
\[
\MPA(B, Q^\gp)= \textstyle{\Gamma\Big(B_0,
\bigoplus_{\ul\rho\in \tilde\P^{[n-1]}_\inte} Q^\gp_{\Int\ul\rho}\Big) }=
\Map (\tilde\P^{[n-1]}_\inte, Q^\gp)
\]
identifies $\varphi\in \MPA(B, Q^\gp)$ with the map associating to
$\ul\rho\in\tilde\P^{[n-1]}_\inte$ the kink $\kappa_{\ul\rho} (\varphi)
\in Q^\gp$ along $\ul\rho$ of a local PA-representative of
$\varphi$.
\end{proposition}

\begin{proof}
This is immediate from the local description~\eqref{Eqn: kink} of
piecewise affine functions.
\end{proof}
\smallskip

To obtain local toric models for the deformation construction our
MPA-function needs to be convex in the following sense.

\begin{definition}
\label{Def: convex MPA-function}
A \emph{convex ($Q$-valued) MPA-function} on $B$ is a $Q^\gp$-valued
MPA-function $\varphi$ with $\kappa_{\ul\rho}(\varphi)\in Q$ for all
$\ul\rho\in\tilde\P^{[n-1]}_\inte$. The monoid of convex $Q$-valued
MPA-functions on $B$ is denoted $\MPA(B,Q)$.
\end{definition}

\begin{example}
\label{Expl: MPA-functions}
1)\ \ In \cite{logmirror1}, \cite{affinecomplex} we took $Q=\NN$ and
considered only those functions fulfilling certain additional linear
conditions. This defines a subspace of our $\MPA(B,Q^\gp)$ that can
be characterized as follows. The first requirement is
$\kappa_{\ul\rho}(\varphi)= \kappa_{\ul\rho'}(\varphi)$ for any
$\ul\rho,\ul\rho'\in\tilde\P^{[n-1]}_\inte$ contained in the same
$(n-1)$-cell $\rho$ of $\P$, see Remark~\ref{Rem: PL functions}. We can
then write $\kappa_\rho(\varphi)$. The second requirement comes from
the behaviour in codimension two. Let $\tau\in \P_\inte$ be a cell of
codimension two and $\rho_1,\ldots,\rho_k$ be the adjacent cells of
codimension one. Working in a chart at a vertex $v\in\tau$ let
$n_1,\ldots,n_k\in \check\Lambda_v$ be the primitive normal vectors
to $\Lambda_{\rho_i}$, with signs chosen following a simple loop
about the origin in $(\Lambda_v)_\RR/(\Lambda_\tau)_\RR \simeq
\RR^2$. Then the following \emph{balancing condition} must hold in
$Q^\gp\otimes\check\Lambda_v$:
\begin{equation}
\label{Eqn: Balancing condition}
\sum_{i=1}^k \kappa_{\rho_i}(\varphi)\otimes n_i=0.
\end{equation}
The balancing condition assures that locally $\varphi$ has a single-valued
piecewise linear representative, even in higher codimension. In this
way the MPA-functions of \cite{logmirror1}, \cite{affinecomplex} can be
interpreted as \emph{tropical divisors} on $B$.
\\[1ex]
2)\ \ In \cite{GHK1} the monoid $Q$ comes with a monoid
homomorphism $\operatorname{NE}(Y)\to Q$ from the cone of classes
of effective curves of the rational surface $Y$. The convex
MPA-function is obtained by defining $\kappa_\rho$ for an edge
$\rho\in \P$ to be the class of the component
$D_\rho\subseteq D$.

Analogous statements hold in \cite{GHKS} with $\operatorname{NE}(Y)$
replaced by $\operatorname{NE}(\shY)$, the cone of effective curve
classes in the total space of the degeneration $\shY
\to T$, and with $D_\rho\subseteq \shY_0$ the double curve
corresponding to $\rho\in\P^{[1]}$.
\end{example}

There is also a universal MPA-function. It takes values in a free
monoid and even happens to be convex. Denote by $\ul{\mathrm{MPA}}$
the category of convex MPA-functions on $B$ taking values in arbitrary
commutative monoids $Q$ and with morphisms from $\varphi_1\in
\MPA(B,Q_1)$ to $\varphi_2\in \MPA(B,Q_2)$ the homomorphisms $h:
Q_1\to Q_2$ with $\varphi_2=h\circ \varphi_1$.

\begin{proposition}
\label{Prop: Universal MPA function}
a)\ The monoid $Q_0:=\Hom\big( \MPA(B,\NN),\NN\big)$ is canonically
isomorphic to $\NN^{\tilde\P^{[n-1]}_\inte}$.\\[1ex]
b)\ The $Q_0$-valued MPA-function $\varphi_0$ taking value the
generator $e_{\ul \rho}\in \NN^{\tilde\P^{[n-1]}_\inte}=Q_0$ at $\ul\rho\in
\tilde\P^{[n-1]}_\inte$ is an initial object in the category
$\ul{\mathrm{MPA}}$. In other words, for
any monoid $Q$ and any $Q$-valued MPA-function $\varphi$ on $B$, there
exists a unique monoid homomorphism $h: Q_0 \to Q$
such that
\[
\varphi=h\circ\varphi_0.
\]
\end{proposition}

\begin{proof}
The $\NN^{\tilde\P^{[n-1]}_\inte}$-valued MPA-function $\varphi_0$ fulfills the
universal property in (b). In fact, if $\varphi$ is a $Q$-valued MPA-function
the equation $\varphi=h\circ\varphi_0$ holds if and only if $h$ is defined by
$h(e_{\ul \rho}):= \kappa_{\ul\rho}(\varphi)$ for
$\ul\rho\in\tilde\P^{[n-1]}_\inte$. In particular, $\MPA(B,\NN)= \Hom\big(
\NN^{\tilde\P^{[n-1]}_\inte}, \NN\big)$, which shows (a).
\end{proof}

\begin{example}
\label{Expl: torus}
Let $M=\ZZ^n$ be a lattice, $M_{\RR}=M\otimes_{\ZZ}\RR$,
$\Gamma\subseteq M$ a rank $n$ sublattice. Consider the real $n$-torus
$B=M_{\RR}/\Gamma$, with affine structure induced by the natural
affine structure on $M_{\RR}$. A polyhedral decomposition $\P$ of
$B$ is induced by a $\Gamma$-periodic polyhedral decomposition
$\bar\P$ of $M_{\RR}$. Because there are no singularities one
imposes \eqref{Eqn: Balancing condition} and thus restricts to
multi-valued piecewise linear functions which are locally
single-valued, even around codimension $\ge 2$ cells of $\P$. Going
to the universal cover of $B=B_0$ implies that such a section
$\varphi$ of  ${\mathcal{MPA}}(B,\ZZ)$ is given up to an affine
linear function by a piecewise affine function $\bar\varphi:M_{\RR}
\rightarrow\RR$ affine linear with integral slope on each cell of
$\bar\P$ and satisfying a periodicity condition
\begin{equation}
\label{periodicity}
\bar\varphi(x+\gamma)=\bar\varphi(x)+\alpha_{\gamma}(x)\quad \forall
x\in M_{\RR},\gamma\in\Gamma,
\end{equation}
where $\alpha_{\gamma}$ is an integral affine linear  function
depending on $\gamma$. Let $P$ be the monoid of all such functions
which are in addition (not necessarily strictly) convex: these are
those functions whose kink at each codimension one cell of $\P$ is
non-negative. (Note that in this case, the kink only depends on the
codimension one cell of $\P$, and not on a cell of the barycentric
subdivision of $\P$). Then $P^\times=0$, as the zero multi-valued
piecewise linear function is the only convex function all of whose kinks
are invertible, that is, $0$.

Let $Q=\Hom(P,\NN)$. We can then assemble all the piecewise linear
functions in $P$ into a single function in $\MPA(B,Q^{\gp})$,
defined as a function $\varphi_0:M_{\RR} \rightarrow
Q^{\gp}\otimes_{\ZZ}\RR=\Hom(P,\RR)$ given by the formula
\[
\varphi_0(x)=(\varphi\mapsto \varphi(x)).
\]
The kink of $\varphi_0$ along $\rho\in\P^{[n-1]}$ is 
\[
\kappa_{\rho}(\varphi_0)=(\varphi\mapsto \kappa_{\rho}(\varphi))\in 
\Hom(P,\NN).
\]
Note that $\kappa_{\rho}(\varphi_0)\in Q$, so $\varphi_0$ is a
convex function.

Here we have fixed a single polyhedral decomposition $\P$. It is
possible to consider all polyhedral decompositions arising as the
domains of linearity of some convex multi-valued piecewise linear
function, producing an analogue of the secondary fan for periodic
decompositions: this was explored  by Alexeev in \cite{alexeev}.
\end{example}
\medskip

Given an MPA-function $\varphi\in\MPA(B,Q^\gp)$ we can construct a
new polyhedral pseudomanifold $(\BB_\varphi,\P_\varphi)$ of dimension
$\dim B +\rk Q^\gp$, along with a $Q_\RR^\gp$-action and an integral
affine map $\pi:\BB_\varphi\to B$ making $\BB_\varphi$ into a
$Q_\RR^\gp$-torsor over $B$. In fact, $\BB_\varphi= B\times
Q_\RR^\gp$ as a set, but the affine structure of $\BB_\varphi$ is
twisted by $\varphi$ as we will explain shortly. The zero section
$B\to B\times\{0\}\subseteq \BB_\varphi$ defines a piecewise affine
right-inverse to $\pi$. The image of this section can be viewed as the
graph of $\varphi$.

\begin{construction}
\label{Construction: B_varphi}
\emph{(The $Q_\RR^\gp$-torsor $\BB_\varphi\to B$)}\; Let $(B,\P)$ be
a polyhedral pseudomanifold, $Q$ a toric monoid and $\varphi\in
\MPA(B,Q^\gp)$. Take $\BB_\varphi:=B\times Q^\gp_\RR$ with
polyhedral decomposition
\[
\P_\varphi:=\{\tau\times Q_\RR^\gp\,|\,\tau\in\P\}.
\]
To define the affine structure along a codimension one cell
$\ul\rho\times Q_\RR^\gp$, $\ul\rho\in\tilde\P_\varphi^{[n-1]}$ an
interior cell, let $\sigma,\sigma'\in\P_\max$ be the cells adjacent
to $\ul\rho$. Let $\delta: \Int\sigma\cup \Int\sigma'
\cup\Int\ul\rho\to \RR$ be the integral affine map with
$\delta(\ul\rho)=\{0\}$, $\delta(\sigma')\subseteq \RR_{\ge 0}$ and
surjective differential $D\delta:\Lambda_\sigma\to \ZZ$. In other
words, $\delta$ is the signed integral distance from $\rho$ that is
positive on $\sigma'$. Then for a chart $f: U\to \RR^n$ for $B_0$
with $U\subseteq \Int\sigma\cup \Int\sigma'\cup\Int\ul\rho$, define a
chart for $\BB_\varphi$ by
\begin{equation}
\label{charts for BB}
U\times Q_\RR^\gp\lra \RR^n\times Q_\RR^\gp,\quad
(x,q)\longmapsto \begin{cases} (f(x),q),&x\in \sigma\\
(f(x),q+\delta(x)\cdot\kappa_{\ul\rho}(\varphi)),&x\in\sigma'.\end{cases}
\end{equation}
Here $\kappa_{\ul\rho}(\varphi)\in Q^\gp$ is the kink of $\varphi$
along $\ul\rho$ defined in Definition~\ref{Def: kink}. The
projection $\pi:\BB_\varphi\to B$ is integral affine and the
translation action of $Q_\RR^\gp$ on the second factor of
$\BB_\varphi= B\times Q_\RR^\gp$ endows $\BB_\varphi$ with the
structure of a $Q_\RR^\gp$-torsor over $B$. We may now interpret
$\varphi$ as the zero section $B\to B\times Q_\RR^\gp=\BB_\varphi$
since in an affine chart the composition with the projection
$\BB_\varphi\to Q_\RR^\gp$ indeed represents $\varphi$. Note that by
\eqref{charts for BB} the zero section of $\BB_\varphi$ is only a
piecewise integral affine map.

If $\varphi$ is convex (Definition~\ref{Def: convex MPA-function})
we can also define the \emph{upper convex hull of $\varphi$} as the
subset $\BB_\varphi^+:= B\times Q_\RR\subseteq \BB_\varphi$. Here
$Q_\RR= \RR_{\ge0}\cdot Q\subseteq Q_\RR^\gp$ is
the convex cone generated by $Q$. In this case $\partial
\BB_\varphi^+ \subseteq \BB_\varphi$ is the image of $\varphi$, viewed as
a map $B\to \BB_\varphi$, plus the preimage of $\partial B$ under
the projection $\BB_\varphi^+\to B$.
\end{construction}

In the situation of Construction~\ref{Construction: B_varphi} there
are also two sheaves of monoids on $B$. Later these will carry the
exponents of certain rings of Laurent polynomials that provide the
local models of the total space of our degeneration. 

\begin{definition}
\label{Def: shP}
Let $(B,\P)$ be a polyhedral pseudomanifold, $Q$ a toric monoid,
$\varphi\in \MPA(B,Q)$ a $Q$-valued convex MPA-function and $\pi:
\BB_\varphi\to B$ the $Q_\RR^\gp$-torsor defined in
Construction~\ref{Construction: B_varphi} with canonical section
$\varphi:B\to \BB_\varphi$. Define the locally constant sheaf of
abelian groups
\[
\shP:= \varphi^*\Lambda_{\BB_\varphi}
\]
on $B_0$ with fibres $\ZZ^n\oplus Q^\gp$.
\end{definition}

For the second sheaf observe that on the interior of a maximal cell
$\sigma$, the product decomposition $\BB_\varphi= B\times Q_\RR^\gp$
is a local isomorphism of affine manifolds independent of any
choices. Hence for any $\sigma\in\P_\max$ we have a canonical
identification
\begin{equation}
\label{shP on Int(sigma)}
\Gamma(\Int\sigma, \shP)= \Lambda_\sigma\times Q^\gp.
\end{equation}
Furthermore, if $\sigma\cap\partial B\neq\emptyset$ and
$\rho\in\P^{[n-1]}$, $\rho\subseteq\sigma \cap\partial B$, define
\[
\Lambda_{\sigma,\rho}\subseteq\Lambda_\sigma
\]
as the submonoid of
tangent vector fields on $\sigma$ pointing from $\rho$ into
$\sigma$. In other words, $\Lambda_{\sigma,\rho}\simeq
\Lambda_\rho\times\NN$ is the preimage of $\NN$ under the
homomorphism $\Lambda_\sigma\to \Lambda_\sigma/\Lambda_\rho\simeq
\ZZ$ for an appropriate choice of sign for the isomorphism.

\begin{definition}
\label{Def: shP+}
Denote by $\shP^+\subseteq \shP$ the subsheaf with sections over an
open set $U\subseteq B_0$ given by $m\in\shP(U)$ with $m|_{\Int\sigma}
\in \Lambda_\sigma\times Q$ under the identification \eqref{shP on
Int(sigma)}, for any $\sigma\in\P_\max$. Moreover, if $\rho\cap
U\neq\emptyset$ for $\rho\in \P^{[n-1]}$, $\rho\subseteq\partial B$,
we require $m|_{\Int \sigma}\in \Lambda_{\sigma,\rho} \times
Q$, for $\sigma\in\P_\max$ the unique maximal cell containing
$\rho$. 
\end{definition}

The affine projection $\pi:\BB_\varphi\to B$ induces a homomorphism
$\pi_*: \shP\to \Lambda$ and hence an exact sequence
\begin{equation}
\label{shP versus Lambda}
0\lra \ul Q^\gp\lra \shP\stackrel{\pi_*}{\lra} \Lambda\lra 0
\end{equation}
of sheaves on $B_0$. Note that the action of $Q^\gp$ on the stalks
of $\shP$ defined by this sequence is induced by the
$Q_\RR^\gp$-action on $\BB_\varphi$.

%===========================================================
%===========================================================

\section{Wall structures}
\label{Sect: Wall structures}
Throughout this section we fix a polyhedral pseudomanifold $(B,\P)$ as
introduced in Construction~\ref{Construction: B} and a Noetherian base ring
$A$. The base ring is completely arbitrary subject to the Noetherian condition
unless otherwise stated.
Let moreover be given a toric monoid $Q$ and a convex MPA function
$\varphi$ on $B$ with values in $Q$ (Definition~\ref{Def: convex
MPA-function}). Let $I\subseteq A[Q]$ be an ideal and write
$I_0:=\sqrt{I}$ for the radical ideal of $I$. For any
$\ul\rho\in\tilde\P^{[n-1]}_\inte$ we assume
$z^{\kappa_{\ul\rho}(\varphi) }\in I_0$. As a matter of notation,
the monomial in $A[Q]$ associated to $m\in Q$ is denoted $z^m$. So
throughout the paper $z$ has a special meaning as a dummy variable
in our monoid rings.

An important special case is that $I\subseteq I_0$ is generated by a
monoid ideal in $Q$, but we do not want to restrict to this
case.\footnote{The meaning of writing the base ring as $A[Q]/I$ is
that in this form it comes with a chart $Q\to A[Q]/I$ for a log
structure that is implicit in the construction.} Note also that we
do not assume $Q^\times=\{0\}$. 

 From this data we are first going to construct a non-normal but
reduced scheme $X_0$ over $\Spec (A[Q]/I_0)$ by gluing together toric
varieties along toric divisors. Write $X_0^\circ\subseteq X_0$ for the
complement of the toric strata of codimension at least two.

In a second step we assume given a \emph{wall structure}. We
then produce a deformation $\foX^\circ\to \Spec (A[Q]/I)$ of $X_0^\circ$.

In the affine and projective cases, a third step, treated only in
Sections~\ref{Sect: Global functions} and \ref{Sect: Theta
functions}, extends the deformation over the deleted
codimension two locus by constructing enough global sections of an
ample line bundle.
\begin{examples}
\label{Expls: Q}
1)\ \ In the setup of \cite{logmirror1}, \cite{affinecomplex} we
considered the case $Q=\NN$ and $A$ some base ring of
characteristic~$0$, usually a field or $\ZZ$.\\[1ex]
2)\ \ In the case of \cite{GHK1}, as outlined in Example~\ref{Expl:
standard examples},2, the monoid $Q$ is taken to be a submonoid of
$H_2(Y,\ZZ)$ such that $Q^{\gp}=H_2(Y,\ZZ)$ and $\NE(Y)\subseteq Q$,
where $\NE(Y)$ denotes the cone of effective curves. The latter
monoid frequently is not finitely generated, so it is usually
convenient to choose $Q$ to be a larger but finitely generated
monoid. The ideal $I_0$  is often taken to be the ideal of a closed toric
stratum of $\Spec\kk[Q]$, for example the ideal of the smallest toric
stratum.

In the case of \cite{GHKS}, also outlined in Example~\ref{Expl:
standard examples},2, we typically work with a finitely generated
monoid $Q$ with  $\NE(\shY/T)\subseteq Q \subseteq N_1(\shY/T)$, and
again $I_0$ will be the ideal of a closed toric stratum of
$\Spec\kk[Q]$. Here $N_1(\shY/T)$ is the group of numerical
equivalence classes of algebraic $1$-cycles with integral
coefficients.
\end{examples}

%===========================================================
\subsection{Construction of $X_0$}
\label{Subsect: X_0}

Given a polyhedral pseudomanifold $(B,\P)$ and convex MPA function
$\varphi$ with values in the toric monoid $Q$, we construct here the
scheme $X_0$ along with a projective morphism to an affine scheme
$W_0$. Both $X_0$ and $W_0$ are reduced but reducible schemes over
$A[Q]/I_0$ whose irreducible components are toric varieties.

The construction is easiest by writing down the respective
(homogeneous) coordinate rings. Recall that if $\sigma\subseteq \RR^n$
is an integral polyhedron, the \emph{cone over $\sigma$} is  
\begin{equation}
\label{Eqn: cone over unbounded cell}
\cone{\sigma}:= \cl\big( \RR_{\ge0}\cdot (\sigma\times\{1\})\big)
\subseteq \RR^n\times\RR.
\end{equation}
The closure is necessary to deal with unbounded polyhedra. In fact,
if $\sigma=\sigma_0+ \sigma_\infty$ with $\sigma_0$ bounded and
$\sigma_\infty$ a cone, then $\cone{\sigma}$ is the Minkowski sum of
$\RR_{\ge0}\cdot (\sigma_0\times\{1\})$ with $\sigma_\infty
\times\{0\}$, and the two subcones only intersect in the origin, the
tip of $\cone{\sigma}$. The proof of this statement is
straightfoward by writing down the inequalities defining $\sigma$.
Note also that the cone $\sigma_\infty$ is uniquely determined by
$\sigma$; it is called the \emph{asymptotic cone} (or
\emph{recession cone}) of $\sigma$.

For $d>0$ an integer, rescaling by $1/d$ defines a bijection
\[
\cone{\sigma}\cap \big(\ZZ^n\times\{d\}\big) \lra
\sigma\cap \frac{1}{d} \ZZ^n
\]
between the integral points of $\cone{\sigma}$ of height $d$ and
$1/d$-integral points of $\sigma$. As a matter of notation define for
$d\ge 0$ (now including $0$)
\[
B\big({\textstyle\frac{1}{d}\ZZ}\big) := \bigcup_{\sigma\in\P_\max}
\cone{\sigma}\cap \big(\Lambda_\sigma\times\{d\}\big).
\]
Here we identify elements in common faces. Thus for $d>0$ the set
$B\big(\frac{1}{d}\ZZ\big)$ can be identified with the subset of $B$
of points that in some integral affine chart of a cell can be
written with coordinates of denominator $d$.

For any ring $S$ consider the free $S$-module
\[
S[B]:= \bigoplus_{d\in \NN} S^{B(\frac{1}{d}\ZZ)}
\]
with basis elements $z^m$, $m\in B\big( \frac{1}{d}\ZZ\big)$ for
some $d\in\NN$. We turn $S[B]$ into an $S$-algebra by defining the
multiplication of basis elements $z^m\cdot z^{m'}=0$ unless there
exists $\sigma\in\P$ with $m,m'\in\sigma$, and in this case
$z^m\cdot z^{m'}:=z^{m+m'}$, the sum taken in the monoid
$\cone{\sigma}$. The index $d$ defines a $\ZZ$-grading on $S[B]$,
and the homogeneous part of degree $d$ is denoted $S[B]_d$. Note
also that because our polytopes have integral vertices, $S[B]$ is
generated in degree one.

Now take $S=A[Q]/I_0$ and define
\begin{equation}
\label{Eqn: X_0}
W_0:=\Spec (S[B]_0),\quad
X_0:= \Proj\big( S[B]\big).
\end{equation}
By construction $X_0$ is naturally a projective scheme over $W_0$.
To characterize the irreducible components of $W_0$ and $X_0$
consider the primary decomposition $I_0=\bigcap_{i=1}^r \fop_i$ of
$I_0$ and let $S_i:= S/\fop_i$. Since $I_0$ is reduced the
$\fop_i$ are prime ideals, and hence $\Spec S=\bigcup_i \Spec S_i$
is a decomposition into integral subschemes.

For an integral polyhedron $\sigma$ we have the $S_i$-algebra
$S_i[\sigma_\infty\cap\Lambda_\sigma]$
defined by the asymptotic cone $\sigma_\infty$
of $\sigma$ and
\[
\PP_{S_i}(\sigma):=
\Proj \big(S_i[\cone{\sigma}\cap (\Lambda_\sigma\oplus\ZZ)]\big),
\] 
the projective toric variety over $S_i[\sigma_\infty
\cap\Lambda_\sigma]$ defined by $\cone{\sigma}$. For a bounded cell
$\sigma$ the asymptotic cone is trivial and hence
$S_i[\sigma_\infty\cap\Lambda_\sigma] = S_i$.

\begin{proposition}
\label{Prop: X_0}
The schemes $W_0$ and $X_0$ are reduced. The irreducible components
of $X_0$ are $\PP_{S_i}(\sigma)$ with $\sigma$ running over the
maximal cells of $\P$. The irreducible components of
$W_0$ are $\Spec(S_i[\sigma_\infty \cap\Lambda_\sigma])$ with
$\sigma$ running over a subset of the unbounded maximal cells of $\P$.
\end{proposition}

\begin{proof}
We first give the proof for $X_0$. For each $\sigma\in\P_\max$
denote by $J_\sigma\subseteq S[B]$ the monomial ideal generated by
$z^m$ with $m\not\in\cone{\sigma}$. We have canonical isomorphisms
\[
S[B]/(J_\sigma + \fop_i) \simeq S_i[B]/J_\sigma \simeq
S_i[\cone{\sigma}\cap (\Lambda_\sigma\oplus\ZZ)].
\]
The ring on the right-hand side is the homogeneous coordinate ring
of $\PP_{S_i}(\sigma)$, an integral domain as $S_i$ is one. Hence
the $J_\sigma+\fop_i$ are prime ideals. Since
$\bigcap_{\sigma\in\P_\max} J_\sigma= 0$ and the $J_\sigma$ are not
contained in one another we see that $J_\sigma+\fop_i$ with
$\sigma\in\P_\max$ and $i=1,\ldots,r$ are the minimal prime ideals
in $S[B]$.

For $W_0$ the analogs of $J_\sigma$ defined from the asymptotic
cones $\sigma_\infty$ may not be distinct and one has
to pick a minimal subset. Otherwise the proof is completely
analogous to the case of $X_0$.
\end{proof}

\begin{remark}
We note that there is not a precise correspondence between unbounded
maximal cells of $\P$ and irreducible components of $W_0$. For example,
if $B=\RR\times \RR/3\ZZ$ with a subdivision $\P$ induced by the subdivision
of $\RR$ into two rays with endpoint the origin and the subdivision
of $\RR/3\ZZ$ into three unit intervals, $W_0$ consists of two copies
of $\AA^1$ glued at the origin. This occurs because all unbounded
cells $\RR_{\ge 0} \times [i,i+1]$ have the same asymptotic cone, and
hence are responsible for the same irreducible component of $W_0$.

Furthermore, $W_0$ need not be equidimensional if $B$ has several
ends. For example, the above $B$ can easily be modifed by cutting
along $\RR_{\ge 0}\times \{0\}$ and then gluing in the two
edges of $(\RR_{\ge 0})^2$
along the cut. Then $W_0$ has a one-dimensional and two-dimensional
irreducible component.
\end{remark}

\begin{remark}
\label{Rem: Further properties of X_0}
Note also that $X_0$ and $W_0$ can be written as base change to
$A[Q]/I_0$ of the analogous schemes over $\ZZ$ defined with $A=\ZZ$,
$Q=\NN$ and $I_0=\NN\setminus\{0\}$. The scheme $X_0$ over $\ZZ$ has one
irreducible component for each maximal cell of $\P$.
\end{remark}

\begin{example}
\label{Expl: the n-vertex}
Following up on Examples~\ref{Expl: standard examples},2 and
\ref{Expl: MPA-functions}, if we choose $Q$ so that $Q^\times=0$ and
$I_0=Q\setminus Q^\times$, then in the case of \cite{GHK1}, the
corresponding scheme $X_0$ is the \emph{$n$-vertex}, a union of
coordinate planes in affine $n$-space as labelled:
\[
\VV_n=\AA^2_{x_1,x_2}\cup \cdots \cup \AA^2_{x_{n-1},x_n}
\cup\AA^2_{x_n,x_1}\subseteq
\AA^n_{x_1,\ldots,x_n}.
\]
Here $n$ is the number of irreducible components of $D\subseteq Y$. 
In the case of \cite{GHKS}, $X_0$ is a union of copies of $\PP^2$.
\end{example}

A coarser way to state Proposition~\ref{Prop: X_0} is that
$X_0$ is a union of toric varieties over
$S= A[Q]/I_0$ labelled by elements $\sigma\in\P_\max$,
\[
\PP_{S}(\sigma)= \Proj \big(S[\cone\sigma
\cap(\Lambda_\sigma\oplus\ZZ)]\big).
\]
This viewpoint motivates the definition of toric strata of higher
codimension.

\begin{definition}
\label{Def: Toric strata}
A closed subset $T$ of $X_0$ is called a \emph{toric stratum (of
dimension $k$)} if there exists $\tau\in\P$ of dimension $k$ such
that $T$ is the intersection of $\PP_{S}(\sigma)\subseteq X_0$, the
intersection taken over all $\sigma\in\P_\max$ containing $\tau$.
\end{definition}

Note that if $I_0$ is not a prime ideal then toric
varieties over $S$ are not irreducible and neither are our
toric strata.

It is not hard to see that $X_0$ is seminormal 
if the base ring is seminormal. Similarly, we have the essential:

\begin{proposition}
\label{prop: S2 condition}
Every fibre of $X_0\rightarrow \Spec S$ satisfies Serre's condition
$S_2$. If $S$ satisfies Serre's condition $S_2$, so does $X_0$.
\end{proposition}

\begin{proof}
The second statement follows from the first by \cite{BH}, Prop.\ 2.1.16,(b).
Thus we can consider the case when $S$ is a field. Further, $X_0$
satisfies Serre's condition $S_1$ since $X_0$ is reduced.

Thus given $x\in X_0$ a point of height $\ge 2$, we need to show $\O_{X_0,x}$
has depth $\ge 2$. There is a minimal toric stratum of $X_0$ containing $x$,
indexed by $\tau\in\P$. If $v\in \Int (\cone\tau)
\cap(\Lambda_{\tau}\oplus\ZZ)$, let $z^v\in S[B]$ be the corresponding monomial.
Denoting by $S[B]_{(z^v)}$ the homogeneous degree $0$ part of the localization
$S[B]_{z^v}$, we obtain $U_v=\Spec (S[B]_{(z^v)})$ is an affine open
neighbourhood of $x$.

This neighbourhood can be described combinatorially as follows. Let $\bar v\in
\tau$ be the image of $v$ under the projection $\cone\tau \setminus
(\tau_{\infty}\times \{0\})\rightarrow \tau$ given by $(v,r)\mapsto v/r$.
Necessarily $\bar v\in\Int \tau$. We can construct a polyhedral cone complex
$B_v$ by gluing together the tangent wedges at $\bar v$ to maximal cells of $\P$
containing $\tau$, and then $U_v=\Spec(S[B_v]_0)$. If $\tau\subseteq\sigma
\in\P$, then we write $\sigma_v$ for the tangent wedge to $\sigma$ at $\bar v$,
and $\sigma_v$ is a cell in $B_v$. Note that $B_v$ retains the same $S_2$
condition of Construction~\ref{Construction: B},(5) as $B$, and in
particular if $\dim \tau\le n-2$, then $B_v\setminus\tau_v$ is connected.
Further, $\tau_v$ is a vector space and $B_v=B'_v\times \tau_v$, corresponding
to a decomposition $U_v=\Spec(S[B'_v]_0)\times \Gm^{\dim\tau}$.

If the image of $x$ under the projection $U_v\rightarrow \Gm^{\dim\tau}$
is height $\ge 1$, then a regular sequence of length two in $\O_{U_v,x}$
is easily constructed. Thus we can assume $x$ projects to the generic
point of $\Gm^{\dim\tau}$, and so after replacing $S$ with a field extension,
we can assume $x$ is the unique zero-dimensional stratum of $\Spec(S[B'_v]_0)$.
If $\dim\tau=n$ or $n-1$, then $x$ is a height zero or one point, and there is 
nothing to show. Thus we may assume that $\dim B'_v\ge 2$, and if $\tau$
now denotes the unique zero-dimensional cell of $B'_v$, then $B'_v\setminus
\tau$ is connected.

The result now follows from \cite{BBR}, Theorem 1.1. Indeed, write $\P'_v$
for the polyhedral cone complex on $B'_v$. This is a poset ordered by inclusion,
and carries the order topology. Let $\shF$ denote the sheaf of $S$-algebras
on $\P'_v$ whose stalk at $\sigma\in\P'_v$ is the ring 
$S[\sigma\cap\Lambda_{\sigma}]$. It follows
from the criterion of \cite{Yu}, Cor.\ 1.12 that $\shF$ is flasque. Also,
$\Gamma(\P'_v,\shF)=S[B'_v]$. If $I$ is the ideal of the point $x$, then
the hypotheses of \cite{BBR}, Theorem 1.1 are satisfied and 
the local cohomology $H^1_I(S[B'_v])$ is calculated using the formula of
that theorem. This is seen to be zero from the connectedness
of $B'_v\setminus \tau$, which implies the desired depth statement.
\end{proof}

%===========================================================
\subsection{Monomials, rings and gluing morphisms}
\label{Subsect: rings}

Recall from Construction~\ref{Construction: B_varphi} that we
interpreted $\varphi$ as a piecewise affine section of the
$Q_\RR^\gp$-torsor $\pi:\BB_\varphi\to B$, and recall the sheaves
$\shP=\varphi^* \Lambda_{\BB_\varphi}$ and $\shP^+\subseteq\shP$ on
$B_0$ from Definitions~\ref{Def: shP} and \ref{Def: shP+}. 
Denote by $A[\shP^+]$ the sheaf of $A[Q]$-algebras on $B_0$
with stalk at $x$ the monoid ring $A[\shP^+_x]$. The sheaf of ideals
generated by $I\subseteq A[Q]$ is denoted $\shI$.

\begin{definition}
\label{Def: monomials}
A \emph{monomial} at $x\in B_0$ is a formal expression $a z^m$ with
$a\in A$ and $m\in\shP^+_x$. A monomial $az^m$ at $x$ has
\emph{tangent vector} $\ol m:=\pi_*(m)\in \Lambda_x$ with $\pi_*$
defined in \eqref{shP versus Lambda}. Moreover, for
$\sigma\in\P_\max$ containing $x$, the \emph{$\sigma$-height}
$\height_\sigma(m) \in Q$ of $m\in\shP_x^+$ is the projection of $m$
to the second component under the identification~\eqref{shP on
Int(sigma)}.
\end{definition}

By abuse of notation we also refer to elements $m\in\shP_x^+$ as
monomials.

\begin{example}
Let $B=\RR^n$, $\P$ be the fan defining $\PP^n$, with rays generated
by the standard basis vectors $e_1,\ldots,e_n$ and
$e_0:=-e_1-\cdots-e_n$. Let $Q=\NN$, and take $\varphi:B\rightarrow
Q^{\gp}_{\RR}=\RR$ to be the piecewise linear function taking the
value $0$ at $0,e_1,\ldots,e_n$ and the value $1$ at $e_0$. Then
$\shP$ is the constant sheaf with stalks $\ZZ^n\times \ZZ$. The
stalk $\shP^+_0$ of $\shP^+$ at $0$ is the monoid $\{(m,r)\,|\, m\in
\ZZ^n, r\ge \varphi(m)\}\subseteq \ZZ^{n+1}$. Note this monoid is
isomorphic to $\NN^{n+1}$, generated by $(e_1,0),\ldots,
(e_n,0),(e_0,1)$. For general $x\in \RR^n$, if $x$ lies in the
interior of the cone generated by $\{e_i\,|\, i\in I\}$, then
$\shP^+_x$ is the localization of $\shP^+_0$ at the elements
$\{(e_i,\varphi(e_i))\,|\,i\in I\}$. This localization is abstractly
isomorphic to $\ZZ^{\# I}\times\NN^{n+1-\# I}$. Note that
$\Spec\kk[\shP^+_0] \rightarrow \Spec\kk[Q]$ induced by the obvious
inclusion $Q\rightarrow \shP^+_0$ is a reduced normal crossings
degeneration of an algebraic torus to a union of affine spaces.
\end{example}

The aim of this section is to construct a flat $A[Q]/I$-scheme
$\foX^\circ$ by gluing spectra of $A[Q]/I$-algebras that are
quotients of $A[\shP^+_x]$ for $x\in B_0$.  Note first that for
$\tau\in\P$ parallel transport inside $\tau\setminus\Delta$ induces
canonical identifications $\shP_x=\shP_y$ for $x,y$ in the same
connected component of $\tau\setminus\Delta$. This identification
maps $\shP^+_x$ to $\shP^+_y$ and $\shI_x$ to $\shI_y$ and hence
induces an identification of rings
\begin{equation}
\label{Eq: ring iso via parallel trp}
A[\shP^+_x] \lra A[\shP^+_y]
\end{equation}
mapping $\shI_x$ to $\shI_y$. There are thus only finitely many
rings to be considered, one for  each $\sigma\in\P_\max$ and one for
each $\ul\rho\in\tilde\P^{[n-1]}_\inte$ an $(n-1)$-cell of the
barycentric subdivision of some $\rho\in\P^{[n-1]}_\inte$. In the
case of a maximal cell $\sigma$ with $\dim (\sigma\cap\partial
B)=n-1$ there is in addition one more ring for each
$\rho\in\P^{[n-1]}$ with $\rho\subseteq\sigma\cap\partial B$.

For $\sigma\in\P_\max$ choose $x\in\Int\sigma$ and define
\begin{equation}
\label{Eq: R_sigma}
R_\sigma:= A[\shP_x^+]/\shI_x.
\end{equation}
In view of \eqref{Eq: ring iso via parallel trp} the
$A[Q]/I$-algebra $R_\sigma$ is defined uniquely up to unique
isomorphism. Moreover, by \eqref{shP on Int(sigma)} there is a
canonical isomorphism
\begin{equation}
\label{Eq: R_sigma can iso}
R_\sigma= (A[Q]/I)[\Lambda_\sigma].
\end{equation}
In particular, $\Spec(R_\sigma)$ is an algebraic
torus over $\Spec(A[Q]/I)$ of dimension $n=\rk\Lambda_\sigma$.

Similarly, if $\rho\in\P^{[n-1]}$ is a non-interior
codimension one cell with adjacent maximal cell $\sigma$
choose $x\in \Int\rho$ and define
\begin{equation}
\label{Eqn: R_sigmarho}
R_{\sigma,\rho}:= A[\shP^+_x]/\shI_x = (A[Q]/I)[\Lambda_{\sigma,\rho}].
\end{equation}
The canonical inclusion
\begin{equation}
\label{Eq: R_rho -> R_sigma non-interior}
R_{\sigma,\rho}\lra R_\sigma
\end{equation}
exhibits $R_\sigma$ as the localization of  $R_{\sigma,\rho}$ by the
monomial associated to the (unique) toric divisor of $\Spec
R_{\sigma,\rho}$.

For an interior codimension one cell $\ul\rho\in\tilde\P^{[n-1]}$
the situation is a little more subtle. If $x\in\ul\rho$,
$y\in\ul\rho'$ are contained in the same $\rho\in\P^{[n-1]}$ then
parallel transport inside an adjacent maximal cell $\sigma$ still
induces an isomorphism $A[\shP^+_x]/\shI_x \to A[\shP^+_y]/\shI_y$;
but if the affine structure does not extend over $\Int\rho$ the
isomorphism depends on the choice of $\sigma$ and hence is not
canonical. The naive gluing would thus not fulfill the cocycle
condition even locally. To cure this problem we now adjust
$A[\shP^+_x]/\shI_x$ to arrive at the correct rings $R_{\ul\rho}$.

For $x\in \rho\setminus\Delta$, $\rho\in\P^{[n-1]}$, there is a
submonoid $\Lambda_\rho\times Q^\gp\subseteq \shP_x$. Under the
identification $\shP_x= \Lambda_{\BB_\varphi,\varphi(x)}$
(Definition~\ref{Def: shP}) this submonoid equals
$\Lambda_{\rho\times Q_\RR^\gp}$, the integral tangent space of the
cell $\rho\times Q_\RR^\gp\subseteq \BB_\varphi$. This submonoid
is invariant under parallel transport in a neighbourhood of
$\Int\rho$. In particular, for any $x\in\rho\setminus\Delta$ we
obtain a subring $A[\Lambda_\rho\times Q]\subseteq A[\shP^+_x]$ and
similarly modulo $\shI_x$. To generate $A[\shP^+_x]$ as an
$A[\Lambda_\rho\times Q]$-algebra let $\xi\in\Lambda_x$ generate
$\Lambda_x/\Lambda_\rho\simeq \ZZ$. Then there are unique lifts
$Z_+,Z_-\in \shP_x$ of $\pm\xi$ with
\begin{equation}
\label{description of A[shP_x]}
A[\shP^+_x]/\shI_x\simeq (A[Q]/I)[\Lambda_\rho][Z_+,Z_-]/
(Z_+Z_- - z^{\kappa_{\ul\rho}}).
\end{equation}
Here $\kappa_{\ul\rho}= \kappa_{\ul\rho}(\varphi)\in Q$ is the kink
of $\varphi$ along the $(n-1)$-cell $\ul\rho\in\tilde\P^{[n-1]}_\inte$ of
the barycentric subdivision containing $x$. Indeed, if $\varphi_x:
\Lambda_x\to Q^\gp$ is a local representative of $\varphi$ at $x$
with $\varphi_x(x)=0$ then
\[
\shP_x^+= \big\{ (m,q)\in\Lambda_x\times Q^\gp\,\big|\,
q\in\varphi(m)+Q  \big\},
\]
and $Z_+= z^{(\xi,\varphi(\xi))}$, $Z_-= z^{(-\xi, \varphi(-\xi))}$,
$Z_+Z_-= z^{(0,\kappa_{\ul\rho})}$. Note that changing $\xi$ to
$\xi+m$ with $m\in\Lambda_\rho$ changes $Z_+$ to $z^{(m,\varphi(m))}
\cdot Z_+$ and $Z_-$ to $z^{(-m,-\varphi(m))}\cdot Z_-$. In
particular, the isomorphism of~\eqref{description of A[shP_x]}
implicitly depends on the choice of $\xi$. For each $\rho$ we
therefore choose an adjacent maximal cell $\sigma=\sigma(\rho)$ and
a tangent vector $\xi=\xi(\rho)\in\Lambda_\sigma$ with
\begin{equation}
\label{xi(rho)}
\Lambda_\rho+\ZZ\cdot\xi(\rho)= \Lambda_\sigma.
\end{equation}
To fix signs we require that $\xi(\rho)$ points from $\rho$ into
$\sigma(\rho)$.

We now assume that for each $\ul\rho\in\tilde\P^{[n-1]}_\inte$ we have a
polynomial $f_{\ul\rho}\in (A[Q]/I)[\Lambda_\rho]$ with the
compatibility property that if $\ul\rho,\ul\rho'\subseteq\rho$ then
\begin{equation}
\label{f_rho' versus f_rho}
z^{\kappa_{\ul\rho'}}f_{\ul\rho'}= 
z^{m_{\ul\rho'\ul\rho}} \cdot z^{\kappa_{\ul\rho}}f_{\ul\rho}.
\end{equation}
Now define
\begin{equation}
\label{Eqn: R_ul rho}
R_{\ul\rho}:= (A[Q]/I)[\Lambda_\rho][Z_+,Z_-]/(Z_+Z_- - f_{\ul\rho}
\cdot z^{\kappa_{\ul\rho}}).
\end{equation}
As an abstract ring $R_{\ul\rho}$ depends only on $f_{\ul\rho}$, but the
interpretation of $Z_\pm$ as a monomial defined by the affine geometry of $B$
also depends on the above choice of a maximal cell $\sigma(\rho)\supset\rho$ and
$\xi=\xi(\rho)\in \Lambda_\sigma$. This choice will become important in the
gluing to the rings $R_\sigma$, $\sigma\in\P_\max$, see~\eqref{Eq: Localization
morphism rho -> sigma} below.

The rings $R_{\ul\rho}$ are now compatible with local parallel
transport. Specifically, let $\rho\in \P^{[n-1]}$ contain
$\ul\rho,\ul\rho'\in\tilde\P^{[n-1]}_\inte$. Let $Z_\pm\in R_{\ul\rho}$ be
the lifts of $\pm\xi(\rho)$ as defined above, and $Z'_\pm\in
R_{\ul\rho'}$ the lifts for $\ul\rho'$. Since $\xi(\rho)$ is a
vector field on $\sigma=\sigma(\rho)$, parallel transport of
monomials inside $\sigma$ maps $Z_+$ to $Z'_+$. Use parallel
transport inside the other maximal cell $\sigma'\supset\rho$ to
define the image of $Z_-$ in $R_{\ul\rho'}$. Let $y\in\Int\ul\rho'$.
The result $\xi'$ of parallel transport of $-\xi(\rho)\in\Lambda_x$ through
$\sigma'$ differs from $-\xi(\rho)\in
\Lambda_y$ by monodromy around a loop passing from
$y\in\Int\ul\rho'$ via $\sigma$ to $x\in\Int \ul\rho$ and back to $y$ via
$\sigma'$. By \eqref{Eqn: monodromy vectors} we obtain
\begin{equation}
\label{xiprimexieq}
\xi'= -\xi(\rho)+\check d_{\rho}(-\xi(\rho))\cdot m_{\ul\rho'\ul\rho}
= -\xi(\rho)- m_{\ul\rho'\ul\rho}= -\xi(\rho)+m_{\ul\rho\,\ul\rho'}.
\end{equation}
This computation suggests that we identify $R_{\ul\rho}$ and $R_{\ul\rho'}$ by
mapping $Z_+$ to $Z'_+$ and $Z_-$ to $z^{m_{\ul\rho\,\ul\rho'}} Z'_-$.

\begin{lemma}
\label{Lem: R_rho well-defined}
Let $\rho\in\P^{[n-1]}$, $\rho\not\subseteq\partial B$, contain
$\ul\rho$, $\ul\rho'\in\tilde\P^{[n-1]}_\inte$ and let $Z_\pm\in R_{\ul
\rho}$, $Z'_\pm\in R_{\ul\rho'}$ be lifts of $\pm \xi(\rho)$ as
defined above. Then there is a canonical isomorphism of
$(A[Q]/I)[\Lambda_\rho]$-algebras
\begin{equation}
\label{Eq: gluing of R_{ul rho}}
R_{\ul\rho}\lra R_{\ul\rho'}
\end{equation}
mapping $Z_+$ to $Z'_+$ and $Z_-$ to $z^{m_{\ul\rho\,\ul\rho'}} Z'_-$.
\end{lemma}

\begin{proof}
By \eqref{f_rho' versus f_rho} we have the equality
$z^{\kappa_{\ul\rho}}f_{\ul\rho}= z^{m_{\ul\rho\,\ul\rho'}}
\cdot z^{\kappa_{\ul\rho'}}f_{\ul\rho'}$ in
$(A[Q]/I)[\Lambda_\rho]$. Thus under the stated map the relation
$Z_+ Z_- - f_{\ul\rho} z^{\kappa_{\ul\rho}}$ in $R_{\ul\rho}$
maps to
\[
Z'_+ z^{m_{\ul\rho\,\ul\rho'}} Z'_- - z^{\kappa_{\ul\rho}} f_{\ul\rho}
=z^{m_{\ul\rho\,\ul\rho'}}\big(Z'_+Z'_-
- z^{\kappa_{\ul\rho'}} f_{\ul\rho'}\big).
\]
 From this computation the statement is immediate.
\end{proof}
\medskip

If $\sigma\in\P_\max$ contains $\ul\rho\in\tilde\P^{[n-1]}_\inte$ there is also a
canonical localization homomorphism
\begin{equation}
\label{Eq: Localization morphism rho -> sigma}
R_{\ul\rho}\lra \begin{cases} 
(R_{\ul\rho})_{Z_+}= R_\sigma&,\ \sigma=\sigma(\rho)\\
(R_{\ul\rho})_{Z_-}= R_\sigma&,\ \sigma\neq \sigma(\rho).
\end{cases}
\end{equation}
The isomorphism $(R_{\ul\rho})_{Z_+}= R_\sigma$ is defined by
eliminating $Z_-$ via the equation $Z_+Z_-=f_{\ul\rho} \cdot
z^{\kappa_{\ul\rho}}$ and mapping $Z_+$ to $z^{\xi(\rho)}$. The
other monomials are identified via $\Lambda_\rho\subseteq
\Lambda_\sigma$. Note that this map is not injective since $Z_-^l$
maps to zero for $l\gg 0$, due to the fact that $\kappa_{\ul\rho}\in
I_0$ and $I_0=\sqrt{I}$. A similar reasoning holds for
$(R_{\ul\rho})_{Z_-}$, using parallel transport through $\ul\rho$ to
view $\xi(\rho)\in\Lambda_{\sigma(\rho)}$ as an element of
$\Lambda_\sigma$.

%===========================================================
\subsection{Walls and consistency}

The rings $R_\sigma$, $R_{\sigma,\rho}$ and $R_{\ul\rho}$ together
with the isomorphisms \eqref{Eq: gluing of R_{ul rho}} and the
localization homomorphisms~\eqref{Eq: Localization morphism rho ->
sigma}, \eqref{Eq: R_rho -> R_sigma non-interior} form a category
(or inverse system) of $A[Q]/I$-algebras. Choosing for each
$\rho\in\P^{[n-1]}$ one $\ul\rho\in\tilde\P^{[n-1]}_\inte$ with
$\ul\rho\subseteq\rho$ defines an equivalent subcategory. Taking
$\Spec$ of this subcategory then defines a direct system of affine
schemes and open embeddings with the property that the only
non-trivial triple fibre products come from maximal cells $\sigma$
with $\sigma\cap\partial B\neq\emptyset$ and the codimension one
cells $\rho\subseteq\partial B\cap\sigma$. Fixing $\sigma$, this
latter system of schemes has a limit, the open subscheme
$\bigcup_{\rho\subseteq\sigma\cap\partial B} \Spec\big(
R_{\sigma,\rho}\big)$ of the toric variety $\PP_\sigma$ with
momentum polytope $\sigma$. Let $D_\inte\subseteq \PP_\sigma$ be the
union of toric divisors corresponding to facets $\rho\subseteq\sigma$
with $\rho\not\subseteq\partial B$. Then the complement of this open
subscheme in $\PP_\sigma\setminus D_\inte$ is the union of toric
strata of codimension larger than~$1$. Hence there exists a colimit
of our category of schemes as a separated scheme over $\Spec
(A[Q]/I)$. It has an open cover by the affine schemes $\Spec
R_\sigma$, $\Spec R_{\sigma,\rho}$ and $\Spec R_{\ul\rho}$, for the
chosen subset of $\ul\rho$'s in $\tilde\P_\inte^{[n-1]}$. 

This scheme is not quite what we want, since it is both a bit too
simple and it may not possess enough regular functions
semi-locally.\footnote{On a technical level, consistency in
codimension two (Definition~\ref{Def: Consistency in codim two})
may fail for this uncorrected scheme.} Rather we will introduce higher
order corrections to the functions $f_{\ul\rho}$ and to the gluing
morphisms. The latter are carried by locally polyhedral
subsets of $B$ of codimension one, called walls. For a polyhedral
subset $\fop$ of some $\sigma\in\P_\max$ with
$\Delta\cap \Int\fop=\emptyset$ we write $\Lambda_\fop
\subseteq\Lambda_\sigma$ for the vectors tangent to $\fop$.

\begin{definition}
\label{Def: Wall structure}
1)\ \ A \emph{wall} on our polyhedral pseudomanifold $(B,\P)$ is a
codimension one rational polyhedral subset $\fop\not\subseteq \partial
B$ of some $\sigma\in\P_\max$ with $\Int\fop\cap\Delta=\emptyset$,
along with an element
\[
f_\fop=\sum_{m\in\shP_x^+,\, \ol m\in\Lambda_\fop}
c_m z^m\in A[\shP_x^+],
\]
for $x\in \Int\fop$. Identifying $\shP_y$ with $\shP_x$ by
parallel transport inside $\sigma$ we require $m\in\shP_y^+$ for all
$y\in \fop\setminus\Delta$ when $c_m\neq0$. Moreover, the
following holds:
\begin{enumerate}
\item[] \underline{$\codim=0$:}\quad
If $\fop\cap\Int\sigma\neq\emptyset$
then $f_\fop\equiv1$ modulo $I_0$. 
\item[] \underline{$\codim=1$:}\quad
If $\fop\subseteq\ul\rho$  for some
$\ul\rho\in\tilde\P^{[n-1]}_\inte$ then $f_\fop\equiv f_{\ul\rho}$ modulo
$I_0$.
\end{enumerate}
Here $\codim$ refers to the \emph{codimension of $\fop$}, defined as
the codimension of the minimal cell of $\P$ containing $\fop$.
Codimension one walls are also called \emph{slabs}, denoted
$\fob$.\\[1ex]
2)\ \ A \emph{wall
structure} on $(B,\P)$ is a set $\scrS$ of walls such that the underlying
polyhedral sets of $\scrS$ are the codimension one cells of a rational
polyhedral decomposition $\P_\scrS$ of $B$ refining $\P$. In particular, the
interior of a wall does not intersect any other wall. We also require
that if $\fou\in\P_\scrS$ is a maximal cell, then there is at most one
$\rho\in\P^{n-1}$ with $\rho\subseteq\partial B$ and $\dim(\fou\cap\rho)=n-1$.
\\[1ex]
A maximal cell $\fou$ of $\P_\scrS$ is called a \emph{chamber} of the
wall structure. Two chambers $\fou,\fou'$ are \emph{adjacent} if
$\dim\fou\cap\fou'=n-1$. A chamber $\fou$ with $\dim(\fou\cap\partial B)=n-1$ is
called a \emph{boundary chamber}, otherwise an \emph{interior chamber}. Elements
$\foj\in\P_\scrS$ of codimension two are called \emph{joints}. A joint $\foj$
with $\foj\subseteq\partial B$ is called a \emph{boundary joint}, otherwise an
\emph{interior joint}. The \emph{codimension} $k\in\{0,1,2\}$ of a joint is the
codimension of the smallest cell of $\P$ containing $\foj$.
\end{definition}

\begin{remark}
\label{Rem: wall structures}
1)\ In the definition of wall structure we do not require that the walls and
chambers form a polyhedral decomposition of $B$. A typical phenomenon is that a
wall $\fop\subseteq\sigma$, $\sigma\in\P_\max$, intersects the interior of
$\rho\in\P^{[n-1]}_\inte$, but $\rho\cap\fop$ is not contained in a union of
walls on the other adjacent maximal cell $\sigma'\neq\sigma$,
$\rho=\sigma\cap\sigma'$.\\[1ex]
2)\ By the definition of walls and the condition that $\P_\scrS$
refines $\P$, the discriminant locus $\Delta$ is covered by joints. In
particular, if $\fou$, $\fou'$ are adjacent chambers then
$\Int(\fou\cap\fou')\cap\Delta=\emptyset$. This is different from the convention
in \cite{affinecomplex} where $\Delta$ was chosen trans\-cend\-ental and
transverse to all joints. In particular, a slab $\fob$ in \cite{affinecomplex}
could be disconnected by $\Delta$. See Appendix~\ref{Subsect: GS one-parameter}
for how the construction of \cite{affinecomplex} also produces wall structures
with the present conventions.

Furthermore, in \cite{affinecomplex} walls could intersect
in a subset of dimension~$n-1$, which in the present definition is
excluded. This is, however, no restriction for one can always
first subdivide walls to make this only happen if the underlying
polyhedral subsets of two walls $\fop',\fop''$ agree. Then in a
second step, replace all walls $\fop_i$ with the same underlying
polyhedral set by the wall $\fop:=\fop_i$ for any $i$ and define $f_\fop:=
\prod_i f_{\fop_i}$. This process is compatible with taking the
composition of the automorphisms associated to walls to be defined
in \eqref{Eqn: theta_fop}.\\[1ex]
3)\ By definition any chamber $\fou$ is contained in a unique
maximal cell $\sigma=\sigma_\fou$. Thus chambers $\fou,\fou'$ can be
adjacent in two ways. (I)~If $\sigma_\fou=\sigma_{\fou'}$ then
$\fou\cap\fou'$ must be an $(n-1)$-cell of $\P_\scrS$ intersecting
the interior of a maximal cell, hence the underlying set of a
codimension zero wall. Otherwise, (II), $\sigma_{\fou} \cap
\sigma_{\fou'}$ has dimension smaller than $n$, but contains the
$(n-1)$-dimensional subset $\fou\cap\fou'$. Hence
$\rho=\sigma_{\fou}\cap\sigma_{\fou'}\in\P^{[n-1]}_\inte$,
and $\fou\cap\fou'$ is covered by the underlying sets of slabs.\\[1ex]
4)\ The condition on boundary chambers is purely technical and can
always be achieved by introducing some walls $\fop$ with
$f_\fop=1$.\\[1ex]
5)\ In this paper the $f_{\ul\rho}$ are redundant information once we
assume a wall structure $\scrS$ to be given. Moreover,
condition~\eqref{f_rho' versus f_rho} follows from consistency of
$\scrS$ in codimension one introduced in Definition~\ref{Def:
Consistency in codim one} below. On the other hand, the reduction of
$f_{\ul\rho}$ modulo $I_0$ determines (and is indeed equivalent to)
the log structure induced by the degeneration on the central fibre,
see \cite{logmirror1}. Thus these functions already contain crucial
information. In the case with locally rigid singularities treated in
\cite{affinecomplex} the whole wall structure can even be
constructed inductively just from this knowledge. 
\end{remark}

Let now be given a wall structure $\scrS$ on $(B,\P)$. There are
three kinds of rings associated to $\scrS$. First, for any chamber
$\fou$ define
\begin{equation}
\label{Eqn: R_fou}
\fbox{$R_\fou:= R_\sigma= (A[Q]/I)[\Lambda_\sigma].$}
\end{equation}
for the unique $\sigma\in\P_\max$ containing $\fou$. Second, if
$\fou$ is a boundary chamber, then according to Definition~\ref{Def:
Wall structure},2(c) there is a unique $\rho\in\P^{[n-1]}$ with
$\dim(\rho\cap\fou\cap \partial B)=n-1$. We then
have the subring
\begin{equation}
\label{R_fou^partial}
\fbox{$
R_\fou^\partial:= R_{\sigma,\rho}\subseteq R_\fou.$}
\end{equation}
Thus $R_\fou$ and $R_\fou^\partial$ are just different notations for
the rings already introduced in \eqref{Eq: R_sigma} and
\eqref{Eqn: R_sigmarho}. The third kind of ring is
a deformation of the ring $R_{\ul\rho}$ from \eqref{Eqn: R_ul
rho} given by a slab $\fob\subseteq\ul\rho$:
\begin{equation}
\label{Def: R_fob}
\fbox{
$R_\fob:= (A[Q]/I)[\Lambda_\rho][Z_+,Z_-]/(Z_+Z_- - f_\fob
\cdot z^{\kappa_{\ul\rho}}).$}
\end{equation}
We indeed have $R_\fob/I_0= R_{\ul\rho}$ since $f_\fob\equiv
f_{\ul\rho}$ modulo $I_0$ according to Definition~\ref{Def: Wall
structure},1.

Between the rings $R_\fou$, $R_\fou^\partial$ and $R_\fob$ there are
two types of localization homomorphisms, namely
\begin{equation}
\label{Eqn: Localization for slabs}
\fbox{$\chi_{\fob,\fou}: R_\fob\lra R_\fou,
\quad \chi_\fou^\partial: R_\fou^\partial\lra R_\fou$}
\end{equation}
defined as in \eqref{Eq: Localization morphism rho -> sigma}
for $\fob\subseteq\fou$ and in \eqref{Eq: R_rho -> R_sigma non-interior}
for $\fou$ a boundary chamber, respectively.

Furthermore, to a codimension zero wall $\fop$ separating interior
chambers $\fou$, $\fou'$ (contained in $\sigma\in\P_\max$, say) we
associate an isomorphism $\theta_\fop: R_\fou\to R_{\fou'}$ as
follows. Let $n_\fop$ be a generator of $\Lambda_\fop^\perp
\subseteq\check\Lambda_x$ for some $x\in\Int\fop$. Denote by $\fou$,
$\fou'$ the two chambers separated by $\fop$ with $n_\fop\ge 0$ as a
function on $\fou$ in an affine chart mapping $x$ to the origin.
Then define
\begin{equation}
\label{Eqn: theta_fop}
\fbox{$\theta_\fop: R_\fou \lra R_{\fou'},\quad
z^m\longmapsto f_\fop^{\langle n_\fop,\ol m\rangle} z^m.$}
\end{equation}
Here we view $f_\fop$ as an element of $R_\sigma^\times=
R_{\fou'}^\times$ by reduction modulo $I$. We refer to $\theta_\fop$
as the automorphism associated to \emph{crossing the wall $\fop$} or
to \emph{passing from $\fou$ to the adjacent chamber $\fou'$}.

If $\dim \fop\cap\partial B=n-2$ then $\fop$ separates two boundary
chambers $\fou,\fou'$. Assuming $\fop$ intersects also the interior
of some $\rho\in\P_\partial^{[n-1]}$ the rings $R_\fou^\partial$ and
$R_{\fou'}^\partial$ are the same localization $R_{\sigma,\rho}$ of
$R_\sigma$. In this case there is an induced isomorphism
\begin{equation}
\label{Eqn: theta_fop^partial}
\fbox{$\theta_\fop^\partial: R_\fou^\partial \to
R_{\fou'}^\partial.$}
\end{equation}
In fact, the requirement of the monomials occurring in $f_\fop$ to
lie in $\shP_y^+$ for all $y\in \fop\setminus\Delta$ implies that
they do not point outward from $\partial B$. Thus $f_\fop$ makes
sense as an element of $R_{\fou'}^\partial$. This shows
$\theta_\fop( R_\fou^\partial)\subseteq R_{\fou'}^\partial$. The
converse inclusion follows from considering $\theta_\fop^{-1}$.
\medskip

Next we would like to glue the affine schemes $\Spec R_\fob$, $\Spec
R_\fou$ and $\Spec R_\fou^\partial$ via the natural localization
morphisms \eqref{Eqn: Localization for slabs}, analogous to the
discussion in the introductory paragraph of this subsection, but
observing the wall crossing isomorphisms \eqref{Eqn:
theta_fop},\eqref{Eqn: theta_fop^partial} between the $R_\fou$. The
scheme $\foX^\circ$ is thus constructed as the colimit of a category
with morphisms generated by all possible wall crossings and the
localization homomorphisms. Since now we have many triple
intersections we need a compatibility condition for this colimit to
be meaningful. Eventually there will be three consistency
conditions: (1)~Around codimension zero joints. (2)~Around
codimension one joints. (3)~Local consistency in higher codimension
tested by broken lines (see Section~\ref{Sect: Global functions}).
The last point is only necessary for the construction of enough
functions and does not concern us for the moment.

As for consistency around a codimension zero joint $\foj$ let
$\fop_1,\ldots,\fop_r$ be the walls containing $\foj$. Working in
the quotient space $\Lambda_{\sigma,\RR}/
\Lambda_{\foj,\RR}\simeq\RR^2$, any quotient
$\fop_i/\Lambda_{\foj,\RR}$ is a line segment emanating from the
origin. Note that since the $\fop_i$ are maximal cells of the
polyhedral decomposition $\P_\scrS$ the line segments intersect
pairwise only at the origin. We may assume the $\fop_i$ are labelled
in such a way that these line segments are ordered cyclically.
Define $\theta_{\fop_i}$ by \eqref{Eqn: theta_fop} with signs fixed
by crossing the walls in a cyclic order.

\begin{definition}
\label{Def: Consistency in codim zero}
The set of walls $\fop_1,\ldots,\fop_r$ containing the codimension
zero joint $\foj$ is called \emph{consistent} if
\[
\theta_{\fop_r}\circ\ldots\circ\theta_{\fop_1}=\id,
\]
as an automorphism of $R_\sigma$, for $\sigma\in\P_\max$ the unique
maximal cell containing $\foj$.

A wall structure $\scrS$ on the polyhedral pseudomanifold $(B,\P)$ is
\emph{consistent in codimension zero} if for any codimension zero
joint $\foj$ the set $\big\{\fop\in\scrS\,\big|\,
\fop\subseteq\foj\big\}$ of walls containing $\foj$ is consistent. 
\end{definition}

Consistency around a codimension one joint $\foj$ is a little more
subtle. There is no condition for a codimension one joint contained
in $\partial B$. Otherwise, let $\rho\in\P$ be the codimension one cell 
containing
$\foj$ and let $\sigma,\sigma'$ be the unique maximal cells
containing $\rho$.  By the polyhedral decomposition property of
$\scrS$ and since $\foj\not\subseteq\partial B$ there are unique slabs
$\fob_1,\fob_2\subseteq\rho$ with $\foj=\fob_1\cap\fob_2$.  Denote by
$\fop_1,\ldots,\fop_r\subseteq\sigma$ and $\fop'_1,\ldots,\fop'_s
\subseteq\sigma'$ the codimension zero walls containing $\foj$. We
assume that the sequence $\fob_1,\fop_1,\ldots,\fop_r,
\fob_2,\fop'_1,\ldots,\fop'_s$ is a cyclic ordering around $\foj$
similarly to the case of codimension one walls.\footnote{If
$\foj\subseteq\Delta$ the quotient space is not well-defined as an
affine plane, but only as a union of two affine half-planes, the
tangent wedges of $\sigma$ and $\sigma'$ along
$\sigma\cap\sigma'\in\P^{[n-1]}$. This is enough for our purposes.}
There are then (non-injective) localization homomorphisms
\begin{equation}
\label{Eqn: Localization hom}
\chi_{\fob_i,\sigma}: R_{\fob_i}\lra R_\sigma,\quad
\chi_{\fob_i,\sigma'}: R_{\fob_i}\lra R_{\sigma'},\quad i=1,2,
\end{equation}
and a composition of wall crossings on either side of $\rho$:
\begin{eqnarray*}
\theta:= \theta_r\circ\theta_{r-1}\circ\ldots\circ\theta_1:&&
R_\sigma\to R_\sigma\\
\theta':= \theta'_1\circ\theta'_2\circ\ldots\circ\theta'_s:&&
R_{\sigma'}\to R_{\sigma'}.
\end{eqnarray*}
Now observe that
\[
(\chi_{\fob_i,\sigma},\chi_{\fob_i,\sigma'}): R_{\fob_i}\lra
R_\sigma\times R_{\sigma'}
\]
is injective. In fact, assuming without restriction
$\sigma=\sigma(\rho)$ for $\rho=\sigma\cap\sigma'$, we have
$\ker(\chi_{\fob_i,\sigma}) \subseteq (Z_-)$, and
$\chi_{\fob_i,\sigma'}(Z_-)$ is invertible in $R_{\sigma'}$. The
consistency condition is the requirement that $\theta\times\theta'$
induces a well-defined map $R_{\fob_1}\to R_{\fob_2}$.

\begin{definition}
\label{Def: Consistency in codim one}
The set
$\{\fop_1,\ldots,\fop_r,\fop'_1,\ldots,\fop'_s,\fob_1,\fob_2\}$ of
walls and slabs containing the codimension one joint $\foj$ is
\emph{consistent} if
\[
(\theta\times\theta')\big((\chi_{\fob_1,\sigma},\chi_{\fob_1,\sigma'})
(R_{\fob_1})\big)= (\chi_{\fob_2,\sigma},\chi_{\fob_2,\sigma'})
(R_{\fob_2}).
\]
In this case we define
\begin{equation}
\label{Eqn: theta_foj}
\theta_\foj: R_{\fob_1}\lra R_{\fob_2}
\end{equation}
as the isomorphism induced by $\theta\times\theta'$.

A wall structure $\scrS$ on the polyhedral pseudomanifold $(B,\P)$ is
\emph{consistent in codimension one} if for any codimension one interior
joint $\foj$ the set $\big\{\fop\in\scrS\,\big|\,
\foj\subseteq\fop\big\}$ of walls and slabs containing $\foj$ is
consistent. 
\end{definition}

\begin{example}
\label{Expls: wall structures}
1)\ \ Wall structures were introduced in \cite{affinecomplex}, with
a slight difference in the treatment of slabs. In
\cite{affinecomplex}, a slab $\fob$ (a codimension 1 wall) could
have $(\Int\fob)\cap\Delta\not=\emptyset$. There was not a single
function attached to a slab, but rather, one choice of function for
each connected component of $\fob\setminus\Delta$, with relations
between these functions determined by the local monodromy analogous
to \eqref{f_rho' versus f_rho}. Indeed, in loc.cit.\ the
discriminant locus was taken with irrational position in such a way
that no codimension zero wall could ever contain an open part of
$\Delta$. In such a situation consistency in codimension one is
equivalent to an equation of the form~\eqref{f_rho' versus f_rho}
relating the functions on the various connected components of
$\fob\setminus\Delta$. In \cite{affinecomplex}, a wall structure
consistent in all codimensions was constructed; in the current
setup, we cannot define consistency in codimension two directly but
only after the construction of local functions, see \S\ref{Subsect:
Consistency in codim two}. This wall structure was used to construct
a deformation $\foX$ of $X_0$, rather than just a deformation
$\foX^\circ$ of the complement of codimension two strata of $X_0$ as
given in Proposition~\ref{Prop: foX^o exists} below. The
construction of \cite{affinecomplex} makes use of local models for
the smoothings of $X_0$ in neighbourhoods of higher codimension
strata; here we just use codimension one strata, where the local
model is given by \eqref{Def: R_fob}. This makes the construction
technically much easier than in \cite{affinecomplex}.\\[1ex]
2)\ \ \cite{GHK1} defined the notion of \emph{scattering diagram} on
the pair  $(B,\Sigma)$ arising from a pair $(Y,D)$ as in
Example~\ref{Expl: standard examples},2. This is a special case of a
wall structure on $(B,\Sigma)$, in which every wall has support a
ray with endpoint $0\in B$. In this case consistency in codimensions
zero and one are automatic.

If one is interested in a compact example, with $\bar B\subseteq B$ a
compact two-dimensional subset as described in Example~\ref{Expl:
standard examples},2,  a scattering diagram $\foD$ on $(B,\Sigma)$
gives rise to a wall structure $\scrS$ on $(\bar B, \P=\{\tau\cap
\bar B\,|\,\tau\in\Sigma\})$. One takes
\begin{align*}
\scrS= {} & \{(\fod\cap \bar B,f_{\fod})\,|\, (\fod,f_{\fod})\in \foD,
\codim\fod=0\}\\
&\cup\{(\underline \rho,f_{\fod})\,|\, \hbox{$(\fod, f_{\fod})\in\foD$,
$\codim\fod=1$, $\ul\rho\in\tilde\P^{[1]}$, $\ul\rho\subseteq\fod$}\}.
\end{align*}
Note that the only singularity of the affine structure on $\bar B$
is at the origin. Thus the barycentric subdivision for the slabs only
appears here to conform to the conventions of the present paper. Again,
consistency in codimension zero and one is automatic.
\end{example}

%===========================================================
\subsection{Construction of $\foX^\circ$}
\label{Subsect: foX^o}

With the notion of consistency of a wall structure in codimension zero
and one at hand (Definitions~\ref{Def: Consistency in codim zero}
and \ref{Def: Consistency in codim one}) we are now in position to
construct our family $\foX\to \Spec(A[Q]/I)$ outside codimension
two.

\begin{proposition}
\label{Prop: foX^o exists}
Let $\scrS$ be a wall structure on the polyhedral pseudomanifold $(B,\P)$.
If $\scrS$ is consistent in codimensions zero and one there exists a
unique scheme $\foX^\circ$ flat over $\Spec(A[Q]/I)$ together with
open embeddings $\Spec R_\fou\to \foX^\circ$, $\Spec
R_\fou^\partial\to \foX^\circ$ and $\Spec R_\fob\to \foX^\circ$ for
(boundary) chambers $\fou$ and slabs $\fob$ of $\scrS$ that are
compatible with the morphisms $\theta_\fop$ and
$\theta_\fop^\partial$ for codimension zero walls \eqref{Eqn:
theta_fop},\eqref{Eqn: theta_fop^partial}, $\theta_\foj$ for
codimension one joints \eqref{Eqn: theta_foj} and with the open
embeddings from \eqref{Eqn: Localization for slabs}, that is, $\Spec
R_\fou\to \Spec R_\fob$ for $\fob\subseteq\fou$ and with $\Spec
R_\fou\to\Spec R_\fou^\partial$ for $\fou$ a boundary chamber.
\end{proposition}

\begin{proof}
Define a category $\ul C$ whose objects are the following 
schemes over $A[Q]/I$. We have
$U_\fou:=\Spec R_\fou$ for chambers $\fou$ of
$\scrS$, $U_\fou^\partial:=\Spec R_\fou^\partial$ for boundary
chambers $\fou$ and $U_\fob:=\Spec R_\fob$ for slabs $\fob\in\scrS$.
The morphisms in $\ul C$ are defined by compositions of the
following three types of morphisms on the ring level:
(1)~Localization homomorphisms $R_\fob\to R_\fou$ and
$R_\fou^\partial\to R_\fou$ \eqref{Eqn: Localization for slabs}; (2)
Automorphisms $\theta_\fop$ and $\theta_\fop^\partial$ associated to
crossing a codimension zero wall \eqref{Eqn: theta_fop},\eqref{Eqn:
theta_fop^partial}; (3)~Isomorphisms $\theta_\foj: R_\fob\to
R_{\fob'}$ associated to crossing a codimension one joint
\eqref{Eqn: theta_foj}.

Consistency implies that for chambers $\fou$ and $\fou'$ in the same
$\sigma\in\P_\max$ any two morphisms $U_\fou\to U_{\fou'}$ coincide.
This follows from a simple topological argument presented in detail
in Step~3 of the proof of \cite{affinecomplex}, Lemma~2.30. The same
argument holds for showing uniqueness of the morphism
$U_\fou^\partial\to U_{\fou'}^\partial$ for boundary chambers
$\fou,\fou'$ intersecting the same $\rho\in\P^{[n-1]}$
full-dimensionally. Similarly, for slabs $\fob,\fob'$ contained in
the same $\rho\in\P^{[n-1]}$ any two morphisms $U_\fob\to U_{\fob'}$
coincide. Finally, for a slab $\fob\subseteq\rho\in \P^{[n-1]}$ and a
chamber $\fou\subseteq\sigma\in\P_\max$ with $\rho\subseteq\sigma$
all morphisms $U_\fou\to U_\fob$ agree.

We have thus shown that the full subcategory with exactly one object
$U_\rho:=U_\fob$ for each $\rho\in\P^{[n-1]}$ with
$\fob\subseteq\rho$ any slab, one object
$U_\sigma:=U_\fou$ for each $\sigma\in\P_\max$ with
$\fou\subseteq\sigma$ and one object $U_\rho:=\Spec R_\fou^\partial$
for each $\rho\in \P^{[n-1]}$, $\rho\subseteq\partial B$, with $\dim
\fou\cap\rho=n-1$, defines a skeleton for $\ul C$. 
In particular, whenever $\rho\subseteq\sigma$, we obtain an open embedding
$U_{\sigma}\rightarrow U_{\rho}$. This gives gluing data for the set of
schemes $\{U_{\rho}\,|\,\rho\in\P^{[n-1]}\}$ in the sense of 
\cite{hartshorneAG}, Ex.\ II.2.12, gluing $U_{\rho}$ and $U_{\rho'}$
along the open subsets $U_{\sigma}\subseteq U_{\rho}, U_{\rho'}$
whenever $\rho,\rho'\subseteq \sigma$, using the identity map on $U_{\sigma}$.
The conditions of \cite{hartshorneAG}, Ex.\ II.2.12 are trivially satisfied.
Hence one obtains a colimit $\foX^{\circ}$
of the category $\ul{C}$ in the category of schemes covered by
the open sets $U_{\rho}$. The remaining properties are then obvious
by construction.
\end{proof}

\begin{example}
\label{Expl: Conical case with boundary}
Consider as $B$ the cone in $\RR^2$ generated by $(-1,0)$ and
$(1,1)$ with the standard affine structure:
\[
B=\RR_{\ge 0}\cdot (-1,0)+\RR_{\ge0}\cdot (1,1).
\]
Take the polyhedral decomposition with two maximal cells
\[
\sigma_1=\RR_{\ge 0}\cdot (-1,0)+\RR_{\ge0}\cdot (0,1),\quad
\sigma_2=\RR_{\ge 0}\cdot (0,1)+\RR_{\ge0}\cdot (1,1).
\]
We then have one vertex $v=(0,0)$, and three codimension one cells
$\rho_1=\sigma_1\cap\partial B$, $\rho_2=\sigma_1\cap\sigma_2$,
$\rho_3=\sigma_2\cap\partial B$. Taking $\Delta=\{v\}$ each
$\rho_i\setminus\Delta$ is connected. The universal choice of MPA
function $\varphi$ takes values in $\NN$ with kink one along the
interior edge $\rho_2$. We can fix a representative of $\varphi$
which takes the value $0$ on $\sigma_1$. This gives a splitting of
the sheaf $\shP$ as the constant sheaf $\ZZ^2\oplus \ZZ$, with the
first $\ZZ^2$ factor being the integral tangent vectors to $B$. We
thus write exponents as elements of $\ZZ^3$. As final ingredient we
take one slab $\fob$ with underlying set $\rho_2$ and $f_\fob=
1+z^{(0,-1,0)}$. We thus have two chambers $\fou_1=\sigma_1$ and
$\fou_2=\sigma_2$. We work over $A=\kk$ some field, write
$A[Q]=\kk[t]$ and take $I=(t^{k+1})$.

Now $\foX^\circ$ is covered by the spectra of the following three rings,
written with $t= z^{(0,0,1)}$, $x= z^{(-1,0,0)}$, $y=z^{(1,1,1)}$,
$w=z^{(0,1,0)}$ for readability:
\begin{eqnarray*}
R_{\fou_1}^\partial&=& \kk[x^{\pm 1}, w,t]/(t^{k+1})\\
R_{\fob}&=& \kk[x,y,w^{\pm1},t]/(t^{k+1}, xy-(1+w^{-1})wt)\\
R_{\fou_2}^\partial&= &\kk[y^{\pm 1},w,t]/(t^{k+1}).
\end{eqnarray*}
To glue we also need the localizations $R_{\fou_1}=
(R_{\fou_1}^\partial)_w$ and $R_{\fou_2}= (R_{\fou_2}^\partial)_w$.
In any case, it is not hard to see that $\foX^\circ$ is isomorphic
to the complement of the single point $V(X,Y,W,t)$ in the affine
scheme
\[
\foX=\Spec \big( \kk[X,Y,W,t]/(t^{k+1}, XY-(1+W)t)\big).
\]
If we represent $f\in \Gamma(\foX^\circ,\O_{\foX^\circ})$ by the
tuple $(f_1,f_2,f_3)$ of restrictions $f_i$ to
$R_{\fou_1}^\partial$, $R_\fob$, $R_{\fou_2}^\partial$, the three
generators $X,Y,W$ are given by
\[
X|_{\foX^\circ}= (x,x,(1+w)y^{-1}t),\quad
Y|_{\foX^\circ}= ((1+w)x^{-1}t,y,y),\quad
W|_{\foX^\circ}= (w,w,w).
\] 
These triples of functions are clearly compatible with the gluing
morphisms, and they exhibit the relation $XY=(1+W)t$ on each of
the three covering affine open sets, hence on $\foX^\circ$.

We will see in \S\ref{Subsect: The conical case} how $X, Y$ and $W$ are
instances of global canonical functions that always exist and that
generate the ring $R$ of global functions. Moreover, in the present
case of a conical $B$ these functions provide an embedding of
$\foX^\circ$ as the complement of a codimension two subset in $\Spec R$.
See Example~\ref{Expl: Conical case with boundary revisited} for
details.
\end{example}

\begin{remark}
\label{Rem: Boundary divisor}
If $\partial B\neq\emptyset$ our degeneration $\foX^\circ\to
\Spec(A[Q]/I)$ comes with a divisor $\foD^\circ\subseteq \foX^\circ$ as
follows. For each boundary chamber $\fou$ we have a monomial ideal
in $R_\fou^\partial =(A[Q]/I)[\Lambda_{\sigma,\rho}]$ generated by
$\Lambda_{\sigma, \rho}\setminus \Lambda_{\sigma,\rho}^\times$.  As
observed in the discussion leading to \eqref{Eqn: theta_fop^partial}
these monomial ideals are compatible with the gluing of rings
$R_\fou$ for chambers $\fou\subseteq\sigma$. Since these are the only
gluings involving boundary chambers intersecting $\rho$
we thus obtain a reduced divisor $\foD^\circ_\rho\subseteq \foX^\circ$.

Moreover, $\foD^\circ$ can be described in the same way as $\foX^\circ$ by a
wall structure $\scrS_\rho$ on $\rho$, albeit with the codimension
one locus removed. To this end define a wall structure $\scrS_\rho$
by considering those walls $\fop\in\scrS$ with $\fop\cap
\Int\rho\neq\emptyset$. Then each such wall $\fop$ defines the wall
with underlying polyhedral set $\fop_\rho:=\fop\cap\rho$ and
function $f_{\fop_\rho}$ the image of $f_\fop$ under $R_\sigma\to
R_\rho$. Here $R_{\rho}$ is the ring analogous to $R_{\sigma}$ associated
to the cell $\rho$ of the decomposition $\P_\partial$ of $\partial B$.
Then $\scrS_\rho$ has no slabs and it is clear that the
construction of $\foD_\rho^\circ$ as a closed subscheme of $\foX^\circ$ agrees
with the construction by applying the gluing construction to $\rho$
and $\scrS_\rho$.

Note that the $\foD^\circ_\rho\subseteq \foX^\circ$ are pairwise disjoint and
hence $\foD^{\circ}:=\bigcup_{\rho\subseteq\partial B} \foD^\circ_\rho$ defines
a closed subscheme of $\foX^\circ$ of codimension one with reduced
fibres over $\Spec(A[Q]/I)$.
\end{remark}

Reduction of $\foX^\circ$ modulo $I_0$ yields an open dense subscheme
of the scheme $X_0$ considered in \S\ref{Subsect: X_0}.

\begin{proposition}
\label{Prop: foX^o modulo I_0}
The reduction of $\foX^\circ$ modulo $I_0$ is canonically isomorphic to
the complement of the union of codimension two strata in $X_0$. In
particular, $\foX^\circ$ is separated as a scheme over $A[Q]/I$.
\end{proposition}

\begin{proof}
This follows immediately from the construction.
\end{proof}

%===========================================================
%===========================================================
\section{Broken lines and canonical global functions}
\label{Sect: Global functions}

The main objective in this paper is the construction of a canonical set
of globally defined functions on $\foX^\circ$. There is one such
function $\vartheta_m$ for each asymptotic monomial $m$ on an
unbounded cell (Definition~\ref{Def: asymptotic monomial}). If $X_0$
is affine the reduction of the $\vartheta_m$ modulo $I_0$ form a basis
of the coordinate ring of $X_0$ as an $A[Q]/I_0$-module. Hence they
can be used to construct a flat affine scheme $\foX$ over $A[Q]/I$
containing $\foX^\circ$ as an open subscheme. In the projective case we
apply the procedure to the total space $\foL$ of the inverse of
the polarizing line bundle $\shL$ to construct a canonical basis of
sections of $\shL^d$ for any $d\ge 0$. These are the theta functions
in the title of the paper.

Throughout the section $\scrS$ is a wall structure on a polyhedral
pseudomanifold $(B,\P)$ that is consistent in codimensions zero and one and
$\varphi$ is a convex MPA-function with values in a toric
monoid $Q$ with $z^{\kappa_{\ul\rho}(\varphi)}\in I_0$ for any
$\ul\rho\in\tilde\P^{[n-1]}$.

Here is the definition of the set of monomials labelling the
functions $\vartheta_m$.

\begin{definition}
\label{Def: asymptotic monomial}
For a polyhedron $\tau\in\P$, an \emph{asymptotic monomial on $\tau$} is
a section $m$ of the restriction of $\shP$ to any connected component
$V\subseteq\tau\setminus\Delta$ with $\height_\sigma(m)=0$ (Definition~\ref{Def:
monomials}) for each $\sigma\in\P_\max$ containing $\tau$, and such that for any
$x\in V\cap\Int\tau$,
\[
\tau+\RR_{\ge0} \ol m_x\subseteq\tau.
\]
An asymptotic monomial on a polyhedral pseudomanifold $(B,\P)$ is an asymptotic
monomial on any $\tau\in\P$. Here we identify asymptotic monomials via inclusion
of faces and extension by parallel transport.

If $m$ is an asymptotic monomial, $\ol m$ is called its
\emph{tangent vector}.
\end{definition}

Note that any $\ol m\in\Lambda_x$ has at most one lift to a
monomial $m$ with $\height_\sigma(m)=0$ for a given
$\sigma\in\P_\max$ containing $x$. Hence an asymptotic monomial is
uniquely determined by its tangent vector.

For a bounded polyhedron only the zero vector is an asymptotic
monomial. In general, writing $\tau=\tau_0+\tau_\infty$ with
$\tau_0$ a bounded polyhedron and $\tau_\infty$ the asymptotic cone,
the asymptotic monomials on $\tau$ are precisely the integral points
of $\tau_\infty$.

%===========================================================
\subsection{Broken lines}

The construction of the canonical function $\vartheta_m$ is based on the
propagation of monomials along piecewise straight paths. A path can
bend when crossing a wall, and the possible new directions of propagation
depends on the result of applying the wall crossing isomorphism.

Let $\fou,\fou'$ be adjacent chambers of $\scrS$. If $\fou$ and $\fou'$
are separated by a codimension zero wall $\fop$ let $\theta_\fop$ be
the automorphism of $R_\sigma$ associated to passing from $\fou$ to
$\fou'$ by crossing the wall $\fop$. As a matter of notation we
now write $\theta_{\fou'\fou}$ instead of $\theta_{\fop}$:
\begin{equation}
\label{Def: Change of chamber I}
\theta_{\fou'\fou}:=\theta_\fop: R_\fou= R_\sigma\lra
R_\sigma= R_{\fou'}.
\end{equation}
If $\fou$ and $\fou'$ are separated by a slab $\fob\subseteq\ul\rho$
let $\sigma,\sigma'$ be the maximal cells containing $\fou,\fou'$,
respectively. Denote by $R_\fou^\fob\subseteq R_\fou= R_\sigma$ the
$A[Q]/I$-subalgebra generated by $\Lambda_\rho$ and by the image of
$Z_+$ under the localization homomorphism $\chi_{\fob,\sigma}$, see
\eqref{Eqn: Localization for slabs} and \eqref{Eqn: Localization
hom}. The conventions are such that $Z_+$ has tangent vector
$\xi(\rho)$ which points from $\rho$ into $\sigma=\sigma(\rho)$. Now
define
\begin{equation}
\label{Def: Change of chamber II}
\theta_{\fou'\fou}:  R_\fou^\fob\lra R_{\fou'}
\end{equation}
as follows.
Note that $R_\fou^\fob$ is generated
as an $A[Q]/I$-algebra by $\Lambda_\rho$ and by
$\chi_{\fob,\sigma}(Z_+)$, while $R_{\fou'}$ is generated by
$\Lambda_\rho$ and by $\chi_{\fob,\sigma'} (Z_-)^{\pm1}$. 
We then define $\theta_{\fou'\fou}$ to be the identity on $\Lambda_\rho$ and
\begin{equation}
\label{eq: thetaslab}
\theta_{\fou'\fou}\big(\chi_{\fob,\sigma}(Z_+)\big)=
\chi_{\fob,\sigma'}(Z_-)^{-1} \cdot f_\fob \cdot
z^{\kappa_{\ul\rho}}.
\end{equation}
From this one sees easily that if $h\in R^{\fob}_{\fou}$, then there exists a
unique 
\[
\tilde h\in (A[Q]/I)[\Lambda_{\rho}][Z_+]\subseteq R_{\fob}
\]
with
$\chi_{\fob,\sigma}(\tilde h)=h$ and $\chi_{\fob,\sigma'}(\tilde
h)=\theta_{\fou'\fou}(h)$.

With $\theta_{\fou'\fou}$ defined for adjacent chambers $\fou,\fou'$
we are now able to propagate certain monomials from $\fou$ to
$\fou'$.

\begin{definition}
\label{Def: transport of monomials}
Let $\fou,\fou'$ be adjacent chambers of $\scrS$ and
$\sigma=\sigma_\fou$, $\sigma'=\sigma_{\fou'}$ the maximal cells
containing $\fou$ and $\fou'$, respectively. Let $az^m$, $a\in
A[Q]/I$, $m\in \Lambda_x$ for some $x\in \Int \sigma$, be an
expression defined at a point of $\Int(\fou\cap\fou')$, using the
canonical identification \eqref{shP on Int(sigma)} on $\sigma$, and
assume that $m$ points from $\fou'$ to $\fou$.\footnote{This means
precisely that $m\in \Lambda_\sigma$, $\sigma=\sigma_\fou$, lies in
the half-space generated by tangent vectors pointing from
$\fou\cap\fou'$ to $\fou$.} Then in the expansion in $R_{\fou'}=
R_{\sigma'}$,
\begin{equation}
\label{Eqn: transport}
\theta_{\fou'\fou}(az^m)=\sum_i a_iz^{m_i}
\end{equation}
with $m_i\in \Lambda_{\sigma_{\fou'}}$ mutually distinct and 
$a_i\in A[Q]/I$, we call any summand
$a_i z^{m_i}$ a \emph{result of transport of $az^m$} from $\fou$ to
$\fou'$.
\end{definition}

Note that in the case that $\fou,\fou'$ are separated by a slab
$\fob$ the assumption on $m$ implies that $az^m\in
R_\fou^\fob$. It is also important to note that by the definition of
$\theta_{\fou'\fou}$ any of the exponents $m_i$ with $a_i\neq 0$
also have the property that $m_i$ points from $\fou'$ to
$\fou$.

Next we define the piecewise straight paths carrying propagations
of monomials.

\begin{definition}
\label{Def: broken lines}
(Cf.\ \cite{PP2mirror}, Definition~4.9.)
A \emph{broken line} for a wall structure $\scrS$ on $(B,\P)$ is a
proper continuous map
\[
\beta: (-\infty,0]\to B_0
\]
with image disjoint from any joints of $\scrS$, along with a
sequence $-\infty=t_0<t_1<\cdots< t_r=0$ for some $r\ge 1$ with
$\beta(t_i)\in |\scrS|$ for $i\le r-1$, and for each $i=1,\ldots,r$
an expression $a_iz^{m_i}$ with $a_i\in A[Q]/I$, $m_i\in \Lambda_{\beta(t)}$
for any $t\in (t_{i-1},t_i)$, defined at all points of
$\beta([t_{i-1},t_i])$ (for $i=1$: $\beta((-\infty,t_1])$),
and subject to the following conditions.
\begin{enumerate}
\item
$\beta|_{(t_{i-1},t_i)}$ is a non-constant affine map with image
disjoint from $|\scrS|$, hence contained in the interior of a unique
chamber $\fou_i$ of $\scrS$, and $\beta'(t)=-m_i$ for all $t\in
(t_{i-1},t_i)$.
\item
For each $i=1,\ldots,r-1$ the expression $a_{i+1} z^{m_{i+1}}$ is a
result of transport of $a_i z^{m_i}$ from $\fou_i$ to $\fou_{i+1}$
(Definition~\ref{Def: transport of monomials}).\footnote{Note that
$\beta(t_i)\in\Int(\fou\cap\fou')$ since $\im(\beta)$ is disjoint
from joints, so the transport of monomials makes sense.}
\end{enumerate}
Denote by $\sigma\in\P_\max$ the cell containing
$\beta((-\infty,t_1])$. The broken line is called \emph{normalized}
if $a_1=1$.

A broken line with $\beta(0)$ contained in a wall is said to
\emph{end on a wall}. The \emph{type} of $\beta$ is the tuple of all
$\fou_i$ and $m_i$. By abuse of notation we suppress the data
$t_i,a_i, m_i$ when talking about broken lines, but introduce the
notation
\[
a_\beta:= a_r,\quad m_\beta:=m_r.
\]
%For $p\in B$ the set of broken lines $\beta$ with $\beta(0)=p$ is
%denoted $\foB(p)$.
\end{definition}

\begin{remark}
\label{Rem: broken lines}
1)\ \ Since $\beta\big((-\infty, t_1]\big)$ is an affine half-line
in $B\setminus\Delta$ in direction $m_1$ it follows that
$m_1$ is an asymptotic monomial (Definition~\ref{Def:
asymptotic monomial}). We therefore call $m_1$ 
the \emph{asymptotic monomial of $\beta$}.
\\[1ex]
2)\ \ A normalized broken line $\beta$ is determined uniquely by its endpoint
$\beta(0)$ and its type. In fact, the coefficients $a_i$ are
determined inductively from $a_1=1$ by Equation~\eqref{Eqn:
transport}. \\[1ex]
3)\ \ If $\partial B\neq\emptyset$ it may happen that a broken line
has its endpoint $\beta(0)$ on $\partial B$. By condition~(1) in
Definition~\ref{Def: broken lines} the broken line is then maximal,
that is, it is not the restriction of another broken line to a
proper subset of its domain. Moreover, if $\fou=\fou_r$ is the last
chamber visited by $\beta$, the monomial $z^{m_\beta}$ is an element
of $R_\fou^\partial$. This follows by the same condition~(1) and the
definition of $R_\fou^\partial= R_{\sigma,\rho}$, see \eqref{Eqn:
R_sigmarho}.
\end{remark}

According to Remark~\ref{Rem: broken lines},2 the map
$\beta\mapsto \beta(0)$ identifies the space of broken lines of a
fixed type with a subset of $\fou_r$, the last chamber visited by
$\beta$. This subset is the interior of a polyhedron:

\begin{proposition}
\label{prop: broken line moduli polyhedral}
For each type $(\fou_i,m_i)$, $i=1,\ldots,r$, of broken lines there
is a rational closed convex polyhedron $\Xi$, of dimension $n$ if
non-empty, and an affine immersion
\[
\Phi: \Xi\lra \fou_r,
\]
so that $\Phi\big (\Int\Xi\big)$ is the set of endpoints $\beta(0)$ of
broken lines $\beta$ of the given type not ending on a wall.
\end{proposition}

\begin{proof}
This is an exercise in polyhedral geometry left to the reader. For
the statement on dimensions it is important that broken lines are
disjoint from joints.
\end{proof}

\begin{remark}
\label{degenerate broken lines}
A point $p\in \Phi(\partial \Xi)$ still has a meaning as an endpoint
of a piecewise affine map $\beta:(-\infty,0]\to B$ together with
data $t_i$ and $a_i z^{m_i}$, defining a \emph{degenerate broken
line}. The point $p$ may correspond to a broken line which ends on a
wall, that is,   $\beta(0)\in\partial\fou_r$ while
$\beta^{-1}(\fou_r)$ contains an open set. Otherwise, this data does
not define a broken line, and $\im(\beta)$ has to intersect a joint.
Note that by convexity of the chambers, the non-empty intersection
with joints comprises the cases that $\beta$ maps a whole interval
to $|\scrS|$ or that $t_{i-1}=t_i$. All other conditions in the
definition of broken lines are open. 

By definition, the set of endpoints $\beta(0)$ of degenerate broken
lines of a given type is the $(n-1)$-dimensional polyhedral subset
$\Phi(\partial\Xi)\subseteq \fou$. The set of endpoints of degenerate
broken lines \emph{not transverse} to some joint of $\scrS$, that
is, with an interval mapping to a joint or intersecting the boundary
of a joint, is polyhedral of codimension at least two. Thus there is
a dense open subset of $\Phi(\partial\Xi)$ of endpoints of
degenerate broken lines that are transverse to all joints, but
intersecting at least one joint or with endpoint on a wall.
\end{remark}

For any fixed asymptotic monomial we have the following finiteness
result for types of broken lines.

\begin{lemma}
\label{Lem: Finiteness of broken lines}
For each asymptotic monomial $m$ the set of types of broken lines
with asymptotic monomial $m$ is finite.
\end{lemma}

\begin{proof}
There is a $k$ such that $I_0^{k}\subseteq I$, since $A$ is assumed
to be Noetherian. Let $\beta$ be a broken line.
 From \eqref{eq: thetaslab} it follows that 
if $\beta(t_i)$ lies in a codimension one cell $\rho$ and $a_iz^{m_i}
\in I_0^{k'} R_{\sigma_{\fou_i}}$,  then $a_{i+1}z^{m_{i+1}}\in
z^{\kappa_{\ul \rho}} I_0^{k'} R_{\sigma_{\fou_{i+1}}}
\subseteq I_0^{k'+1} R_{\sigma_{\fou_{i+1}}}$. Similarly, if $\beta(t_i)
\in\Int \sigma$ for some maximal $\sigma$, then
$\beta(t_i)\in\fop$ for some wall $\fop$, and $f_{\fop}\equiv 1\mod I_0$.
Thus if $m_i\not=m_{i+1}$ and $a_iz^{m_i}\in I_0^{k'}R_{\sigma}$, we must have
$a_{i+1}z^{m_{i+1}}\in I_0^{k'+1}R_{\sigma}$. 
Thus any broken line crosses less than $k$ codimension one walls and bends
less than $k$ times. Furthermore, the expansion \eqref{Eqn: transport} 
is finite,
and hence there are at most a finite number of choices for $m_{i+1}$ given
$m_i$. Since every maximal cell in $\P$ contains only a finite
number of chambers and walls, it is then clear that the number of types
of broken lines for a given asymptotic monomial is finite.
\end{proof}
\medskip

By Lemma~\ref{Lem: Finiteness of broken lines} and
Proposition~\ref{prop: broken line moduli polyhedral} the following
definition is meaningful.

\begin{definition}
\label{Def: general points}
A point $p\in B$ is called \emph{general} (for the given structure
$\scrS$) if it is not contained in $\Phi(\partial \Xi)$, for any
$\Phi$ as in Proposition~\ref{prop: broken line moduli polyhedral}
for any type of broken line.
\end{definition}

%===========================================================
\subsection{Consistency and rings in codimension two}
\label{Subsect: Consistency in codim two}

The canonical global functions will be defined on $U_\fou= \Spec
R_\fou$ as a sum of expressions $a_\beta z^{m_\beta}$ over broken
lines ending at a point $x\in\Int\fou$. For this definition to lead
to a globally well-defined function we need an additional
consistency condition, localized at joints of codimension two. We
continue to assume that $\scrS$ is a wall structure on the
polyhedral pseudomanifold $(B,\P)$ and $\varphi$ is a convex MPA-function
with values in the toric monoid $Q$ with
$z^{\kappa_{\ul\rho}(\varphi)}\in I_0$ for any
$\ul\rho\in\tilde\P^{[n-1]}$. For the moment we do not impose any
consistency assumption in codimensions zero and one.

Let $\foj$ be a joint of codimension two and let $\omega\in
\P^{[n-2]}$ be the smallest cell containing $\foj$. Build a new
polyhedral pseudomanifold $(B_\foj,\P_\foj)$ by replacing any $\tau\in\P$
with $\tau\supseteq\foj$ by the tangent wedge of $\omega$ in $\tau$.
Note that the inclusion $\tau\subseteq\tau'$ of faces induces an
inclusion of the respective tangent wedges. So $B_\foj$ is a local
model for $(B,\P)$ near $\foj$ all of whose cells are cones. By the $S_2$
condition on $B$, in fact $B_{\foj}$ is a manifold (with boundary if $\foj$ 
is a boundary joint).
Moreover, each such cell contains the codimension two linear space
$\Lambda_{\foj,\RR}$. Thus $(B_\foj,\P_\foj)$ is topologically the
preimage of a fan in $\RR^2$ by a piecewise integral affine submersion
$\RR^n\to \RR^2$. Similarly, the wall structure $\scrS$ induces a
wall structure $\scrS_\foj$ by considering only the walls containing
$\foj$ and going over to tangent wedges based at $\omega$ for the
underlying polyhedral subsets of codimension one. Since the only
joint of $\scrS_\foj$ is the codimension two cell
$\Lambda_{\foj,\RR}$ this wall structure is trivially consistent in
codimensions zero and one. Denote by $\foX_\foj^\circ$ the scheme
over $A[Q]/I$ constructed in \S\ref{Subsect: foX^o} for
$(B_\foj,\P_\foj)$.

Now let $m$ be an asymptotic monomial on
$(B_\foj,\P_\foj)$. For a general point $p\in B_\foj$, say contained
in the chamber $\fou$ for $\scrS_\foj$, define
\begin{equation}
\label{Def: vartheta^foj_m(p)}
\vartheta^\foj_m(p):=\sum_\beta a_\beta z^{m_\beta}\in R_\fou.
\end{equation}
The sum runs over all normalized broken lines on $(B_\foj,\P_\foj)$ with
asymptotic monomial $m$ and endpoint $p$. 

\begin{definition}
\label{Def: Consistency in codim two}
The wall structure $\scrS$ is \emph{consistent along the codimension
two joint $\foj$} if the $\vartheta^\foj_m(p)$ (a)~do not depend on
the choice of general point $p$ in the same chamber $\fou$ and (b)~are
compatible with the change of chambers morphisms
$\theta_{\fou'\fou}$ for $\scrS_\foj$ defined in \eqref{Def: Change
of chamber I} and \eqref{Def: Change of chamber II}.

A wall structure $\scrS$ is \emph{consistent} if it is consistent in
codimensions zero, one (Definitions~\ref{Def: Consistency in codim
zero} and \ref{Def: Consistency in codim one}) and along each
codimension two joint. \end{definition}

\begin{remark}
\label{Rem: 2d nature of consistency}
Consistency at a joint $\foj$ can be reduced to the two-dimensional
case as follows. Denote by $\doublebar B_\foj$ the image of
$B_\foj$ under the piecewise integral affine submersion $B_\foj\to
\RR^2$ that contracts $\foj$ to the origin. If $\fop\in\scrS_\foj$
is a wall denote its image in $\doublebar B_\foj$ by
$\doublebar\fop$. By extending the base ring from $A[Q]$ to
$A[Q\oplus \Lambda_\foj]$ the function $f_\fop$ attached to $\fop$
can be interpreted as a function attached to $\doublebar\fop$, thus
endowing $\doublebar B_\foj$ with a wall structure
$\doublebar\scrS_\foj$. Then there is a one-to-one correspondence
between broken lines $\beta$ for $ \scrS_\foj$ with asymptotic
direction not contained in $\Lambda_\foj$ with fixed endpoint $p$
and broken lines for $\doublebar\scrS_\foj$ with fixed endpoint the image
of $p$ in $\doublebar B_{\foj}$.
\end{remark}

\begin{example}
\label{Expl: Consistency at vertices}
Continuing with Example~\ref{Expls: wall structures},2, we noted
that an arbitrary scattering diagram on the pair $(B,\Sigma)$
arising from a pair $(Y,D)$ provides a wall structure consistent in
codimensions zero and  one. The only joint in codimension two is
$\foj=\{0\}$, and $(B_{\foj},\Sigma_{\foj})=(B,\Sigma)$. It is
highly non-trivial to construct a wall structure which is consistent
in codimension two; in fact, the construction of such a wall
structure can be viewed as the main result of \cite{GHK1}. In
particular, Definition 3.3 of \cite{GHK1} defines the
\emph{canonical scattering diagram}  which gives a wall structure of
the current paper as in Example  \ref{Expls: wall structures},2.
This data, motivated by \cite{GPS}, is determined by certain
relative Gromov-Witten invariants of $(Y,D)$. The definition of this
diagram requires the choice of the monoid $Q$ and multi-valued
function $\varphi$. As in Example~\ref{Expls: Q}, one chooses a
monoid $Q$ containing $\NE(Y)$. The function $\varphi$ is chosen to
have kink $\kappa_{\rho_i}(\varphi)$ the class $[D_i]\in
H_2(Y,\ZZ)$ of the irreducible component of $D$ corresponding to
$\rho_i$. If $Q$ is chosen so that $Q^\times=0$, it follows from
\cite{GHK1}, Theorem~3.8 that the canonical scattering diagram  is
consistent in codimension two.

Again if we choose a set $\bar B\subseteq B$ a compact two-dimensional
subset containing the origin, the canonical scattering diagram gives
a wall structure on $(\bar B, \P)$ as in  Example~\ref{Expls: wall
structures},2. It has only one interior codimension two joint
$\foj=\{0\}$, and $(\bar B_{\foj},\P_{\foj})= (B,\Sigma)$. From
the previous paragraph, it follows that the wall structure on $(\bar
B,\P)$ induced by the canonical scattering diagram on $(B,\Sigma)$
is consistent along this interior codimension two joint.

Consistency at the boundary joints is more subtle. In the present
case with all monomials outgoing, that is, extending to the
compactifying divisor, consistency is equivalent to local convexity
of $\bar B$ along the boundary, see Proposition~\ref{Prop:
Boundary consistency}. One can show \cite{GHK2} that $\bar B$ with
locally convex boundary exists if and only if the divisor $D$
supports a big and nef divisor for $Y$.
\end{example}

At an interior joint consistency poses a condition on the behaviour
of broken lines when crossing the joint. In contrast, at a boundary
joint the question is about sufficient local convexity of the
boundary to balance the incoming monomials on the walls containing
$\foj$, see Proposition~\ref{Prop: Boundary consistency} below. To
formulate this convexity condition recall that by Remark~\ref{Rem:
2d nature of consistency} we may restrict to $\dim B_\foj=2$.
Because the only singular point is the origin, $B_\foj$ can be
embedded into $\RR^2$ as a not necessarily convex cone containing
$\RR_{>0} \cdot(0,1)$ in its interior and with boundary
$\RR_{\ge0}\cdot (-1,0)\cup \RR_{\ge0}\cdot (a,b)$, $a>0$. Denote
the walls (of codimensions $0$ and $1$) not contained in $\partial
B$ by $\fop_j= \RR_{\ge0} \ol m_j$, $j=1,\ldots, r$, ordered
clockwise and with $\ol m_j= (a_j,b_j)$ primitive. Any monomial in
the function $f_{\fop_j}$ has tangent vector $-\delta \ol m_j$ for
$\delta\in\ZZ$. Let $\delta_j$ be the maximum of the $\delta$ occurring
in $f_{\fop_j}$. As we will see in the proof of
Proposition~\ref{Prop: Boundary consistency} a broken line
approaching $\fop_j$ in direction $(1,\lambda_j)$ and maximally
bent away from the boundary leaves $\fop_j$ in direction
\begin{equation}
\label{Eqn: Change of slope}
\big(1+ (b_j-a_j\lambda_j)\delta_ja_j, \lambda_j+(b_j-a_j\lambda_j)
\delta_j b_j\big).
\end{equation}
This computation motivates us to define $\lambda_j\in\QQ$ for $j\ge 0$
inductively by $\lambda_0:=0$ and
\begin{equation}
\label{Eqn: Recursive slopes}
\lambda_{j+1}:= \frac{\lambda_j+(b_j-a_j\lambda_j)
\delta_j b_j}{1+(b_j-a_j\lambda_j) \delta_j a_j}.
\end{equation}

\begin{definition}
\label{Def: boundary convexity}
The wall structure $\scrS$ is called \emph{convex at a boundary joint
$\foj\subseteq\partial B$} if $\RR_{\ge0}\cdot (1,\lambda_r)\not\subseteq 
\Int(B_\foj)$.
\end{definition}

This notion of convexity at a boundary joint $\foj$ a priori depends
on the choice of orientation of the normal space to $\foj$. The
recursive equation~\eqref{Eqn: Recursive slopes} for $\lambda_j$,
however, is equivalent to
\[
\lambda_j= \frac{\lambda_{j+1}+(-b_j+a_j\lambda_{j+1})
\delta_j b_j}{1+(-b_j+a_j\lambda_{j+1}) \delta_j a_j},
\]
which agrees with the change of slope when approaching the wall
$\fop_j$ from the other side.

\begin{proposition}
\label{Prop: Boundary consistency}
The wall structure $\scrS$ is consistent at a joint
$\foj\subseteq\partial B$ if it is convex at $\foj$.
\end{proposition}

\begin{proof}
Let us first verify the claim above that \eqref{Eqn: Recursive
slopes} describes the maximal change of slope away from $\partial B$
of a broken line when passing through the wall $\fop_j$. Let
$(c,d)\in\ZZ^2$ be the tangent vector of the monomial $z^m$ of the
broken line before hitting $\fop_j$. Then the result of transport
through $\fop_j$ selects a monomial of $f_{\fop_j}^{-b_j c +a_j
d}\cdot z^m$. The tangent vector of such a monomial is of the
form
\[
\mu\ol m_j +(c,d) =c\cdot\Big( 1+\frac{\mu}{c} a_j, \frac{d}{c}
+\frac{\mu}{c} b_j\Big),
\]
with 
\begin{equation}
\label{eq:mumj}
\mu\ge -\delta_j (-b_j c+a_j d)
\end{equation}
an integer. Putting $\lambda_j =d/c$ and $\mu/c=
\delta_j(b_j-a_j\lambda_j)$ gives \eqref{Eqn: Change of slope}.

To prove the proposition observe that each type of broken line
$\beta$ on $B_\foj$ with asymptotic monomial $m$ has its endpoint in
a chamber $\fou_j$, one of the cones $\RR_{\ge0}\cdot \ol m_j+
\RR_{\ge0}\cdot \ol m_{j+1}$, $j=0,\ldots,r$. To cover the cases
$j=0$ and $j=r$ we define $\ol m_0:= (-1,0)$ and $\ol m_{r+1}:=
(a,b)$. 

We now consider two cases for broken lines. The first case
is that the monomial $z^{m_\beta}$ at the endpoint of $\beta$
has  $-\ol m_\beta\in \Int\fou_j$. In this case, the sum defining
$\vartheta^\foj_m(p)$ loses one term when $p$ is moved across the
ray $-\RR_{\ge 0}\cdot \ol m_{\beta}$ in $\fou_j$, with $p$ moving
from the side of the ray containing the asymptotic direction $m$.
Indeed, if, say, the asymptotic direction $m$ lies in $\RR_{\ge 0}
\ol m_0-\RR_{\ge 0} \ol m_{\beta}$, the last segment of $\beta$ must
begin on the ray $\RR_{\ge 0} \ol m_j$ and hence lie in the cone
$\RR_{\ge 0} \ol m_j-\RR_{\ge 0} \ol m_{\beta}$.
The second case is that the monomial $z^{m_{\beta}}$ at the endpoint of
$\beta$ does not satisfy $-\ol m_{\beta}\in B$. Then for $j<r$ any broken line
of this type can be extended until it hits $\fop_{j+1}$, and then
$\vartheta^\foj_m(p)$ is compatible with the change of chambers
morphism $\theta_{\fop_{j+1}}$. For $j=r$ there is no further wall
to be considered.

The upshot of this discussion is that consistency fails if
there is a type of broken line where the closure of the cone of
endpoints does not fill the last chamber $\fou_r$. By monotonicity
of $\lambda_{j+1}$ as a function of $\lambda_j$  (noting that
$\partial \lambda_{j+1}/\partial \lambda_j$ is easily seen to be
non-negative) this is the case if it holds for the
extreme cases of broken lines with asymptotic monomial $(-1,0)$ or
$(a,b)$ and maximal possible bend, that is, where the inequality
\eqref{eq:mumj} is an equality for each $j$. By symmetry it
suffices to consider the first case. Thus we consider a type of
broken lines $\beta$ with $\beta'(t)$ positively proportional to
$(1,\lambda_j)$  inside the chamber $\fou_j$. The convexity
condition $\RR_{\ge 0} \cdot (1,\lambda_r)\not\subseteq \Int(B_\foj)$
implies that the endpoints of broken
lines of this type fill the chamber containing $\RR_{\ge 0}\cdot
(a,b)$. Thus consistency at $\foj$ is implied by convexity of
$\scrS$ at $\foj$.
\end{proof}
\medskip

\begin{remark}
Consistency at a joint $\foj\subseteq \partial B$ only fails to be
equivalent to convexity at $\foj$ because there may be no broken
line for which the inequality \eqref{eq:mumj} is in fact an
equality. This is because a product of coefficients in $f_{\fop_j}$
making up a term  in $f_{\fop_j}^{-b_jc+a_jd}$ may in fact be $0$ as
the coefficients lie in $A[Q]/I$, which has nilpotents. However, if
one is interested in working over the formal completion
$\widehat{A[Q]}$ of $A[Q]$ with respect to the ideal $I_0$, then
this problem disappears. This is not particularly satisfactory as
$\delta_j$ may no longer exist as $f_{\fop_j}$ is now a formal power
series, but in many cases, it is reasonable to check consistency by
hand using the proof of  Proposition~\ref{Prop: Boundary
consistency}
\end{remark}

\begin{example}
The convexity notion of Definition \ref{Def: boundary convexity}
does not imply that $B$ is convex in the usual sense: the wall structure
$\scrS$ can ``repair'' a non-convex boundary point. For example, 
let $B$ be the union of the two cones 
\[
\sigma_1=\RR_{\ge 0}\cdot (-1,0)+\RR_{\ge0}\cdot (0,1),\quad
\sigma_2=\RR_{\ge 0}\cdot (0,1)+\RR_{\ge0}\cdot (1,-1).
\]
Let $\rho=\RR_{\ge 0}\cdot (0,1)$.
We can take $\varphi$ to take the values
$0$ at $(-1,0)$ and $(0,1)$ and the value $1$ at $(1,-1)$ Finally,
we take a structure $\scrS=\{(\rho_2, z^{(0,1,0)})\}$. It is easy to
see this satisfies our modified definition of convexity.

Applying our construction to this data in fact gives the same result
as applying it to $B=\sigma_1\cup\sigma_2$ with $\sigma_1$ as before
and $\sigma_2$ the first quadrant, with $\scrS$ empty.
\end{example}

The point of the definition of consistency in codimension two is
that the $\vartheta_m^\foj(p)$ now patch to regular functions on
$\foX_\foj^\circ$, as we will now show. The analogous global statement is
the content of Theorem~\ref{Thm: vartheta_m exist} below. 

\begin{proposition}
\label{Prop: vartheta^foj_m exist}
Assume that $\scrS$ is consistent along the codimension two joint
$\foj$ (Definition~\ref{Def: Consistency in codim two}). Then for an
asymptotic monomial $m$ on $(B_\foj,\P_\foj)$ there
is a function $\vartheta^\foj_m$ on $\foX_\foj^\circ$ that restricts to
$\vartheta^\foj_m (p)\in R_\fou$ at any general point $p$ of a
chamber $\fou$.
\end{proposition}

\begin{proof}
Condition~(a) in Definition~\ref{Def: Consistency in codim two}
implies that for any chamber $\fou$ of $\scrS_\foj$ there is a
well-defined element $\vartheta^\foj_m(\fou)\in R_\fou$. Then (b)
means that for chambers $\fou,\fou'$ of $\scrS_\foj$ separated by a
codimension zero wall $\fop$ it holds $\vartheta^\foj_m(\fou')=
\theta_\fop \big( \vartheta^\foj_m(\fou)\big)$.

If $\fou,\fou'$ are separated by a slab $\fob$ we claim the
existence of an element $\vartheta^\foj_m(\fob)\in R_\fob$ with
$\vartheta^\foj_m(\fou)= \chi_{\fob,\sigma}\big(
\vartheta^\foj_m(\fob) \big)$, $\vartheta^\foj_m(\fou')=
\chi_{\fob,\sigma'}\big( \vartheta^\foj_m(\fob) \big)$ for
$\sigma,\sigma'\in \P_\max$ the maximal cells containing
$\fou,\fou'$, respectively. By injectivity of the diagonal map
$R_\fob\to R_\sigma\times R_{\sigma'}$ the element
$\vartheta^\foj_m(\fob)$ is unique if it exists.

To show existence consider the set of possibly degenerate
broken lines which end on the wall $\fob$, (that is, $\beta(0)\in
\fob$) but with $\beta((-\epsilon,0])\not \subseteq \fob$ for any
$\epsilon>0$. Similarly to Proposition~\ref{prop: broken line moduli
polyhedral} there is a finite union $A\subseteq \fob$ of rational
polyhedral subsets of dimension at most $n-2$ such that any such
degenerate broken line  (see Remark \ref{degenerate broken lines})
with asymptotic monomial $m$ ending at $\fob\setminus A$ is in fact
a broken line. Fix $p\in \fob\setminus A$. The set of broken lines
$\beta$ with final segment not contained in $\fob$ and with
asymptotic monomial $m$ and $\beta(0)=p$ decomposes as a disjoint
union $\foB_I\sqcup \foB_{II}$ depending if $\beta$ maps the last
interval $(t_{r-1},0)$ into (I) $\Int\fou$ or (II) $\Int\fou'$. Let
$\vartheta^{\|}_m$ be the sum of terms in either
$\vartheta_m^{\foj}(\fou)$ or $\vartheta_m^{\foj}(\fou')$ lying in
$(A[Q]/I)[\Lambda_{\fob}]$. By the definition of consistency and
of $\theta_{\fou'\fou}$, $\vartheta^{\|}_m$ is well-defined. We now
define $\vartheta^\foj_m(\fob)$ as an element of $R_\fob$ by
\[
\vartheta^\foj_m(\fob):=\sum_{\beta\in\foB_I} a_\beta z^{m_\beta}
+\sum_{\beta\in\foB_{II}} a_\beta z^{m_\beta}
+\vartheta^{\|}_m,
\]
with $a_\beta\in A[Q]/I$ and the individual monomials interpreted as
follows. Let $\xi=\xi(\rho)\in \Lambda_p$ be the chosen generator of
$\Lambda_p/\Lambda_\rho$, assumed without loss of generality to
point into $\fou$. Then for $\beta\in \foB_I$ the exponent $m_\beta$
can be written as $a\xi+m'_\beta$ with $a>0$ and $m'_\beta\in 
\Lambda_\rho$. Now interpret $z^{m_\beta}$ as the monomial
$Z_+^a\cdot z^{m'_\beta}$ in $R_\fob$, which is the unique lift of
$z^{m_\beta}\in R_\fou$ of the stated form under the localization
homomorphism $\chi_{\fob,\sigma}$. Similarly, for
$\beta\in\foB_{II}$ there is a unique lift of $z^{m_\beta}\in
R_{\fou'}$ of the form $Z_-^a \cdot z^{m'_\beta}$. Finally, 
$\vartheta^{\|}_m$ is interpreted as an element of $R_{\fob}$ via
the inclusion
$(A[Q]/I)[\Lambda_\rho]\subseteq R_\fob$. This maps to $\theta^{\|}_m\in
R_{\fou}, R_{\fou'}$ under the respective localizations
$\chi_{\fob,\sigma}$, $\chi_{\fob,\sigma'}$.

Moving $p$ into $\fou$, a broken line in $\foB_I$ 
deforms uniquely without changing $a_\beta z^{m_\beta}$. If
$\beta\in\foB_{II}$ it follows from the definition of the transport
of monomials through slabs that a broken line splits into several
broken lines according to the expansion of $\chi_{\fob,\sigma}(Z_-^a
\cdot z^{m'_\beta})$ into monomials. This shows $\chi_{\fob,\sigma}
(\vartheta^\foj_m(\fob)) =\vartheta^\foj_m(\fou)$. A similar
discussion holds with $\fou'$ replacing $\fou$. This proves the
claim on existence of $\vartheta_m^\foj(\fob)$.

If $\foj$ is a consistent boundary joint and $\fou$ is a boundary
chamber we observed in the proof of Proposition~\ref{Prop: Boundary
consistency} that each broken line with endpoint in $\fou$
extends to $\partial B$. This implies that the tangent vector of
$m_\beta$ points from $\partial B$ into $B$. This shows that
$z^{m_\beta}$ in fact lies in the subring $R_\fou^\partial\subseteq
R_\fou$ defined in \eqref{R_fou^partial}.

Summarizing, if $\foj$ is a consistent joint then the
$\vartheta^\foj_m(\fou)$, $\vartheta^\foj_m(\fob)$ glue to a global
regular function on $\foX^\circ_\foj$.
\end{proof}

\begin{proposition}
\label{Prop: Properties of foX^foj}
In the situation of Proposition~\ref{Prop: vartheta^foj_m exist}
the $\vartheta^\foj_m$ freely generate the $A[Q]/I$-algebra
\[
R_\foj:= \Gamma(\foX_\foj^\circ, \O_{\foX_\foj^\circ})
\]
of global functions on $\foX_\foj^\circ$ as an $A[Q]/I$-module. In
particular, $R_\foj$ is flat over $A[Q]/I$. Moreover, the canonical
map
\[
\foX_\foj^\circ\lra \foX_\foj:=\Spec(R_\foj).
\]
is an open embedding.
\end{proposition}

\begin{proof}
Proposition~\ref{Prop: foX^o modulo I_0} and the description of
$X_0$ in Proposition~\ref{Prop: X_0} imply that the statements
are true modulo $I_0$. Denote by $X_{\foj,0}$ the flat
$A[Q]/I_0$-scheme obtained from $B_{\foj}$ via \eqref{Eqn: X_0}.
See Proposition~\ref{Prop: X_0} for an explicit
description. Following the proof of \cite{GHK1},
Theorem~2.28, consider the ringed space $\foX_\foj$ with underlying
topological space $|X_{\foj,0}|$ and sheaf of $A[Q]/I$-algebras
\[
\O_{\foX_\foj}:=  i_* \O_{\foX_\foj^\circ}
\]
where $i:|X_{\foj,0}^{\circ}|\rightarrow |X_{\foj,0}|$ is the inclusion. By the
existence of the functions $\vartheta^\foj_m$ the reduction morphism
$\O_{\foX_\foj}\to \O_{X_{\foj,0}}$ modulo $I_0$ is surjective. Thus by
\cite{GHK1}, Lemma~2.29, $\foX_\foj$ is flat over $A[Q]/I$. While \cite{GHK1}
only discusses the two-dimensional case, the proof of the cited lemma holds
literally in all dimensions provided Lemma~2.10 in \cite{GHK1} in the proof is
replaced by Lemma~\ref{Lem: Extension lemma} below. The $S_2$ condition
in this lemma is fulfilled by Proposition~\ref{prop: S2 condition}. Moreover,
$\foX_\foj$ as an infinitesimal extension of the affine scheme $X_{\foj,0}$ is
itself affine. Now \cite{GHK1}, Lemma~2.30, shows that the $\vartheta^\foj_m$
are an $A[Q]/I$-module basis of $\Gamma(\foX_\foj, \O_{\foX_\foj})$.

By the same token, $\Spec(R_\foj)$ is a flat deformation of $X_{\foj,0}$
with the same $A[Q]/I$-module basis for the ring of global regular
functions. Thus the embedding $\foX_\foj^\circ\to
\foX_\foj$ induces an isomorphism
\[
\Gamma(\foX_\foj,\O_{\foX_\foj})\lra R_\foj=
\Gamma(\foX_\foj^\circ, \O_{\foX_\foj^\circ}).
\]
In particular, $\foX_\foj=\Spec(R_\foj)$ and $R_\foj$ is freely
generated by the $\theta^\foj_m$.
\end{proof}
\medskip

In the proof we used the following technical lemma generalizing
\cite{hacking}, Lemma~A.3.

\begin{lemma}
\label{Lem: Extension lemma}
Let $X\rightarrow S$ be a flat family of schemes of pure dimension $n$
such that each fibre $X_s$ satisfies Serre's $S_2$ condition. Let $U\subseteq
X$ be an open subset such that $\dim\, (X\setminus U)\cap X_s\le n-2$ for
all $s\in S$. Then if $i:U\hookrightarrow X$ is the inclusion, the canonical
map $\O_X\rightarrow i_*i^*\O_X$ is an isomorphism.
\end{lemma}

\begin{proof}
By \cite{EGAIV.2}, Proposition~5.11.1, the sheaf $i_*i^*\O_X$ is
coherent. Let $K$ and $C$ be the kernel and cokernel of the
canonical map $\O_X\rightarrow i_*i^*\O_X$.  These are supported on
closed subsets $Z_K$ and $Z_C$ of $X\setminus U$ respectively. Let
$Z=Z_K$ or $Z_C$. We assume $Z$ is non-empty, so that we can choose 
a generic point $p$ of $Z_s=Z\cap X_s$ for some $s\in S$.
Necessarily the closure of $\{p\}$ is a closed subset of $X_s$ of
codimension $\ge 2$. So by the $S_2$ condition, there is a regular
sequence $x_s,y_s\in \fom_{X_s,p}$ for $\O_{X_s,p}$. By assumption
that $p$ is a generic point of $Z_s$,  we can replace $x_s,y_s$ with
$x_s^{\nu}, y_s^{\nu}$ for some $\nu\gg 0$ and assume $x_s, y_s$ lie
in the ideal of $Z_s$ in the local ring  $\O_{X_s,p}$. By
\cite{matsumura}, Theorem 16.1, $x_s, y_s$ is still a  regular
sequence for $\O_{X_s,p}$. We can lift $x_s, y_s$ to elements of the
ideal of $Z$ in $\O_{X,p}$, so that $x,y$ is a regular sequence for
$\O_{X,p}$ (see \cite{matsumura}, pg.\ 177, Cor.\ to Theorem 22.5).
Equivalently, we have an exact sequence
\[
0\rightarrow \O_{X,p} \mapright{(y,-x)} \O_{X,p}\oplus \O_{X,p}
\mapright{(x,y)} \O_{X,p}.
\]
Now consider first the case $Z=Z_K$. Since $Z_K$ is the support of $K$,
any given element of $K$ is annihilated by some power of the ideal
$\shI_Z\subseteq \O_X$. So since $K_p\not=0$, there exists a non-zero element
$g\in K_p$ such that $\shI_Z g=0$ locally at $p$. Then $xg=yg=0$,
contradicting the exactness of the above sequence. Thus $Z_K=\emptyset$.
Similarly, take $Z=Z_C$. Then there is a $g\in (i_*i^*\O_X)_p\setminus
\O_{X,p}$ such that $\shI_Zg\subseteq \O_{X,p}$. Again using the exact
sequence above, since $(yg,-xg)\mapsto 0$ under the second map, we
obtain $(yg,-xg)=(yg',-xg')$ for some $g'\in \O_{X,p}$.
But then $g=g'$, a contradiction. Thus $Z_C=\emptyset$.
\end{proof}

%===========================================================
\subsection{The canonical global functions $\vartheta_m$}
\label{Subsect: vartheta_m}

We now give the construction of the canonical global functions
$\vartheta_m$ in the general case. 

\begin{theorem}
\label{Thm: vartheta_m exist}
Let $\scrS$ be a consistent wall structure on the polyhedral pseudomanifold
$(B,\P)$, and let $\foX^\circ$ be the corresponding flat scheme over
$A[Q]/I$ (Proposition~\ref{Prop: foX^o exists}). 

Then for each asymptotic monomial $m$ (Definition~\ref{Def:
asymptotic monomial}) there exists a function
$\vartheta_m\in  \Gamma(\foX^\circ,\O_{\foX^\circ})$
restricting on $R_\fou$, $\fou$ a chamber of $\scrS$, to the sum
\begin{equation}
\label{Def: vartheta_m(p)}
\vartheta_m(p):=\sum_\beta a_\beta z^{m_\beta}.
\end{equation}
over normalized broken lines with asymptotic monomial $m$ and ending
at a general point $p\in\fou$. Moreover, the $\vartheta_m$ form an
$A[Q]/I$-module basis of $\Gamma(\foX^\circ, \O_{\foX^\circ})$:
\[
\Gamma(\foX^\circ, \O_{\foX^\circ}) =
\bigoplus_{m} \big(A[Q]/I\big)\cdot \vartheta_m.
\]
\end{theorem}

\begin{proof}
Without joints contained in the singular locus $\Delta$ of the
affine structure, compatibility of $\vartheta_m(p)$ with varying $p$
within a chamber is covered by \cite{CPS}, Lemma~4.7. This proof
works literally the same in the present case with the assumption of
consistency in codimension two. Here Proposition~\ref{Prop:
vartheta^foj_m exist} replaces \cite{CPS}, Proposition~3.2 at
codimension two joints. This latter proposition describes the result
of transporting a monomial across $B_\foj$ for a joint $\foj$, in
the context of \cite{affinecomplex} with $\Delta$ transverse to
joints, in particular defining the local canonical functions
$\vartheta^\foj_m$.

To see that the $\vartheta_m$ just defined generate
$\Gamma(\foX^\circ,\O_{\foX^\circ})$, denote by $\shI_0$
the pull-back of the ideal $I_0\subseteq A[Q]/I$ to
$\foX^{\circ}$, and let $X^{\circ}_k$ be the closed subscheme
of $\foX^\circ$ defined by the ideal $\shI_0^k$. We will show by induction
on $k$ that the $\vartheta_m$ form an $A[Q]/(I+I_0^{k+1})$-module
basis of $\Gamma(X^{\circ}_k,\O_{X^{\circ}_k})$. For sufficiently large
$k$, $I_0^{k+1}\subseteq I$ since $\sqrt{I}=I_0$, so we conclude the result
for $\foX^{\circ}$.

For the $k=0$ case, $X_0^{\circ}$ is the complement in $X_0$ of the union
of toric strata of codimension two, see Proposition 
\ref{Prop: foX^o modulo I_0}. In this case the statement follows from
standard toric geometry over $A[Q]/I_0$.

Suppose the result is true for $k-1$ with $k\ge 1$. By flatness, there is
a short exact sequence 
\[
0\lra (I+I_0^k)/(I+I_0^{k+1})\otimes \O_{X_0^{\circ}}
\lra \O_{X_k^\circ}\lra \O_{X_{k-1}^\circ}\lra 0
\]
of abelian sheaves on $X_0^\circ$, 
see \cite{matsumura}, Theorem~22.3.
Taking global sections gives the
following exact sequence of $A[Q]$-modules:
\begin{equation}
\label{Eqn: X_k versus X_(k-1)}
0\lra(I+I_0^k)/(I+I_0^{k+1})\otimes\Gamma(X_0^\circ,\O_{X_0^{\circ}})\lra 
\Gamma(X_0^\circ,\O_{X_k^\circ})
\lra \Gamma(X_0^\circ,\O_{X_{k-1}^\circ}).
\end{equation}
By the induction hypothesis, the $\vartheta_m$ form an $A[Q]/(I+I_0^k)$-basis
for $\Gamma(X_0^{\circ},\O_{X^{\circ}_{k-1}})$. Thus given any
$s \in \Gamma(X_0^{\circ},\O_{X_k^{\circ}})$, the image of 
$s$ in $\Gamma(X_0^{\circ},\O_{X_{k-1}^{\circ}})$ can be written as a finite
sum $\sum_i \bar c_i\vartheta_{m_i}$ with $\bar c_i\in A[Q]/(I+I_0^k)$.
Lifting each $\bar c_i$ to $c_i\in A[Q]/(I+I_0^{k+1})$, we have that
$s'=\sum_i c_i\vartheta_{m_i}\in \Gamma(X_0^{\circ},\O_{X_k^{\circ}})$
has the same image as $s$ in $\Gamma(X_0^{\circ},\O_{X_{k-1}^{\circ}})$.
Hence $s-s'\in (I+I_0^k)/(I+I_0^{k+1})\otimes \Gamma(X_0^{\circ},
\O_{X_0^{\circ}})$, which by the base case can be written as a sum
$\sum_j d_j \vartheta_{m'_j}$ with $d_j \in I+I_0^k$. Thus $s$ itself
can be written as a linear combination of theta functions.

Linear independence is shown similarly: if $\sum c_i\vartheta_{m_i}=0$
in $\Gamma(X_0^{\circ},\O_{X^{\circ}_k})$, then by the induction hypothesis
$c_i\in I+I_0^k$ for each $i$ and by the base case $c_i=0$.
\end{proof}

%===========================================================
\subsection{The conical case}
\label{Subsect: The conical case}

A particular case arises when all cells of $\P$ are cones
(``conical''). Then $\P$ has exactly one vertex, and this vertex is
the only bounded cell. On the scheme-theoretic side the condition
means that $X_0$ is affine. In the most general situation we will
want to construct theta functions as sections of a line bundle using
the cone over $B$, so conical pseudomanifolds play a crucial role in the
most general construction. It therefore seems appropriate to develop
the conical case here before treating the most general case.

\begin{definition}
\label{Def: Conical B}
A polyhedral pseudomanifold $(B,\P)$ is called \emph{conical} if each
element of $\P$ is a cone. A conical
polyhedral pseudomanifold has a single vertex $v$.
A wall structure $\scrS$ on a conical
polyhedral pseudomanifold is called \emph{conical} if each wall $\fop$ in
$\scrS$ is a cone with vertex $v$.
\end{definition}

Assume now that $\scrS$ is a conical wall structure on the conical polyhedral
pseudomanifold $(B,\P)$ that is consistent. Then by Theorem~\ref{Thm: vartheta_m
exist} for each asymptotic monomial $m$ we have one distinguished
global function $\vartheta_m$ on $\foX^\circ$. In the present conical case these
functions provide an embedding of $\foX^\circ$ into an affine scheme with the
complement of the image of codimension at least two.

\begin{proposition}
\label{Prop: foX in the conical case}
Let $\scrS$ be a consistent wall structure on the conical polyhedral
pseudomanifold $(B,\P)$ and let $\foX^\circ$ be the associated scheme over
$A[Q]/I$. Then the $\vartheta_m$ freely generate
$R:=\Gamma(\foX^\circ,\O_{\foX^\circ})$ as an $A[Q]/I$-module, and the
induced canonical morphism
\[
\foX^\circ\lra \foX:=\Spec R
\]
is an open embedding restricting to $X^\circ_0\to X_0$ modulo $I_0$.
In particular, $R$ is a flat $A[Q]/I$-module, so that $\foX$ is flat
over $\Spec A[Q]/I$.
\end{proposition}

\begin{proof}
The proof is completely analogous to the proof of Proposition~\ref{Prop:
Properties of foX^foj}. In fact, the only properties used are (1)~$X_0$ is
affine and $S_2$, (2)~the reductions modulo $I_0$ of the $\vartheta_m$
generate $\Gamma(X_0,\O_{X_0})$ and (3)~flatness of $\foX^\circ$ over $A[Q]/I$.
\end{proof}

\begin{example}
\label{Expl: Conical case with boundary revisited}
We are now in position to finish the discussion of
Example~\ref{Expl: Conical case with boundary}. In this example, the
asymptotic monomials of $(B,\P)$ are in bijection with integral
points of $B\cap\ZZ^2\setminus\{(0,0)\}$. If $m=(a,b)$ with $a\le 0$
then $\vartheta_m=z^{(a,b,0)}=x^{-a}w^b$ in $R^{\partial}_{\fou_1}$, 
while if $a\ge 0$ then
$\vartheta_m= z^{(a,b,a)}=y^aw^{b-a}$ in $R^{\partial}_{\fou_2}$. 
Taking into account the transport of monomials we see that
\[
X= \vartheta_{(-1,0)},\quad Y=\vartheta_{(1,1)},\quad
W=\vartheta_{(0,1)}.
\]
In fact, say for $X$, we find that an interior point of $\sigma_2$
is the endpoint of two broken lines with asymptotic monomial
$(-1,0)$. This yields the expression $(1+w)y^{-1}t$ that we gave for
the restriction of $X$ to $\Spec R^{\partial}_{\fou_2}$.

Moreover, by working in $R^{\partial}_{\fou_1}$ or in $R^{\partial}_{\fou_2}$, 
any other $\vartheta_m$ can be written as a polynomial in $X,W$ or in $Y,W$.
Thus by Proposition~\ref{Prop: foX in the conical case} $X,Y$ and $W$
generate $R=\Gamma(\foX^\circ,\O_{\foX^\circ})$, and they provide the
description of $\foX^\circ$ as the open subset of $\Spec R$ claimed in
Example~\ref{Expl: Conical case with boundary}.
\end{example}

\begin{example}
In the case of $(B,\Sigma)$ arising from a Looijenga pair $(Y,D)$ as
covered in Examples~\ref{Expl: standard examples},2, \ref{Expl:
MPA-functions},2, \ref{Expl: the n-vertex}, \ref{Expls: wall
structures},2 and \ref{Expl: Consistency at vertices} we note
$(B,\Sigma)$ is conical. In particular, since the canonical
scattering diagram provides a consistent wall structure,
Proposition~\ref{Prop: foX in the conical case} provides a flat
deformation of the $n$-vertex $\VV_n$. Note that
Proposition~\ref{Prop: foX in the conical case} is a
generalization of Theorem~2.26 of \cite{GHK1}.
\end{example}

%===========================================================
\subsection{The multiplicative structure}

In this section we give an a priori definition of the ring structure
on $\bigoplus_m \big(A[Q]/I\big)\cdot \vartheta_m$ turning the map
\[
\bigoplus_m \big(A[Q]/I\big)\vartheta_m\lra
\Gamma\big( \foX^\circ,\O_{\foX^\circ} \big)
\]
into an isomorphism of $A[Q]/I$-algebras. Our multiplication rule is
tropical in the sense that it is purely in terms of broken lines.

\begin{theorem}
\label{Thm: Multiplication}
Let $\scrS$ be a consistent wall structure on the polyhedral
pseudomanifold $(B,\P)$, and let $\foX^\circ$ be the corresponding flat scheme
over $A[Q]/I$ (Proposition~\ref{Prop: foX^o exists}).  For
asymptotic monomials $m_1$, $m_2$ let
\begin{equation}
\label{Eqn: product of vartheta's}
\vartheta_{m_1}\cdot\vartheta_{m_2}
=\sum_m \alpha_m(m_1,m_2)\cdot\vartheta_m
\end{equation}
be the expansion according to the direct sum decomposition of
Theorem~\ref{Thm: vartheta_m exist}. Thus the sum runs over the
asymptotic monomials of $(B,\P)$ and $\alpha_m(m_1,m_2)\in A[Q]/I$
is non-zero only for finitely many $m$.

For an asymptotic monomial $m$ let $\fou$ be an unbounded chamber of
$\scrS$ such that $m$ is an asymptotic monomial on $\fou$. Let $p\in
\fou$ be a point that is general for broken lines of asymptotics
$m_1$ and $m_2$. Then
\[
\alpha_m(m_1,m_2)= \sum_{(\beta_1,\beta_2)} a_{\beta_1}a_{\beta_2}, 
\]
where the sum is over all pairs $(\beta_1,\beta_2)$ of broken lines
with asymptotics $m_1, m_2$, with endpoint $p$ and
such that $\ol m_{\beta_1}+\ol m_{\beta_2}=m$, viewed as an equation
in $\Lambda_{\sigma_\fou}$.
\end{theorem}

\begin{proof}
This proof is a straightforward adaptation from \cite{GHK1}, \S2.4.
To find the coefficient $\alpha_m(m_1,m_2)$ in the stated expansion
we look at the coefficients of $z^m$ in $R_\fou=
(A[Q]/I)[\Lambda_\sigma]$ of both sides of \eqref{Eqn: product of
vartheta's}. Now the only broken line $\beta$ with endpoint $p\in
\Int\fou$ and with $\ol m_\beta=m$ lies entirely in $\fou$ and has
no bends. Thus in the local expression of the canonical functions in
$R_\fou$ only $\vartheta_m$ has a non-zero coefficient of $z^m$,
which is $1$. Thus $\alpha_m(m_1,m_2)$ agrees with the coefficient
of $z^m$ in the expansion of the left-hand side in $R_{\fou}$. The
statement now follows readily by plugging in the local definition of
$\vartheta_{m_1}$ and $\vartheta_{m_2}$ in terms of broken lines
with the respective asymptotics.
\end{proof}

%===========================================================
%===========================================================
\section{The projective case --- theta functions}
\label{Sect: Theta functions}

In the case that $X_0$ is not affine we are going to construct an
extension $\shL^\circ$ of the ample line bundle on $X_0$ to
$\foX^\circ$, and an $A[Q]/I$-module basis of global sections of
powers $(\shL^\circ)^{\otimes d}$ for $d\ge 0$. This is done by
constructing the total space $\foL^\circ$ of $(\shL^\circ)^{-1}$ as
an affine scheme over $\foX^\circ$. The canonical sections of
$(\shL^\circ)^{\otimes d}$ are then constructed as fibrewise
homogeneous canonical functions on $\foL^\circ$ of the kind
considered in Section~\ref{Sect: Global functions}. Eventually we
can then define the partial completion $\foX$ of $\foX^\circ$ as
$\Proj \left(\bigoplus_ d \Gamma(\foX^\circ, (\shL^\circ)^{\otimes
d})\right)$.

On the tropical side the transition from $\foX^\circ$ to
$\foL^\circ$ corresponds to taking a truncated cone over $(B,\P)$.
We begin with an investigation of the cone construction. 

%===========================================================
\subsection{Conical affine structures}
\label{Subsect: conical affine geom}

Let $B_0$ be an affine manifold (without singularities and not
necessarily integral for the moment, see \S\ref{Subsect: Polyhedral
affine manifolds}, and possibly with $\partial B\neq\emptyset$).
Thus $B_0$ is a real manifold of dimension~$n$ with an atlas such
that the transition functions are affine transformations $T\in
\Aff(\RR^n)= \GL(n,\RR) \ltimes \RR^n$. Our notation for affine
transformations of a real vector space $V$ is $T=A+b$ with
$A\in\GL(V)$, $b\in V$.

\begin{construction}
\label{Constr: cone B_0}
\emph{(The cone over an affine manifold)}
The \emph{cone over $B_0$} is the cone of $B_0$ as a topological
space
\[
\cone B_0:= \big(B_0\times\RR_{\ge 0}\big)\big/ \big(B_0\times \{0\}\big),
\]
endowed with the following affine structure with singularity at the
\emph{apex} $O\in \cone B_0$, the image of $B_0\times\{0\}$ in $\cone
B_0$. For $\psi:U\to \RR^n$ an affine chart for $B_0$, defined on an
open set $U\subseteq B_0$ we define the chart
\begin{equation}
\label{chart for CB}
\tilde\psi: \cone U\setminus \{O\}\lra \RR^{n+1},\quad
(x,h)\longmapsto (h\cdot \psi(x),h)
\end{equation}
for $\cone B_0$.

We remark that if $B_0$ is unbounded, it is not really appropriate for
the cone to have an apex, but rather the apex should be replaced by
an asymptotic version of $B_0$. This is easier to do when given a polyhedral
pseudomanifold, see Definition \ref{Def: cone(B,P)}. However, the precise nature of the
apex will not play a role in the discussion in this subsection.
\qed
\end{construction}
Thus if two charts $\psi_1$, $\psi_2$ are related by
$\psi_2=A\circ \psi_1+b$ for $A\in \GL(n,\RR)$, $b\in\RR^n$ then
\[
\tilde\psi_2=\tilde A \circ\tilde\psi_1
\]
with $\tilde A(x,h)= (Ax+hb,h)$. Intrinsically, if $\AA$ is an affine space with
underlying real vector space $V$, then the map associating to a pair $(A,b) \in
\GL(\RR^n)\times \RR^n$ the linear transformation $\tilde A\in \GL(\RR^{n+1})$
generalizes to
\[
\Aff(V)\lra \GL(V\oplus\RR),\quad
A+b\longmapsto \tilde A.
\]
We refer to this process as \emph{homogenization} of the affine
transformation $A+b$. Clearly, if $A+b\in \Aff(T_x B_0)$ is the
affine holonomy along a closed path $\gamma$ on $B_0$ starting and
ending at $x$, then $\tilde A$ is the affine monodromy of
$(\gamma,h)$ for any $h>0$. We think of the cone as standing on the
apex and call the second entry $h$ the \emph{height} of $(x,h)\in
\cone B_0$.

Note that all transition functions of $\cone B_0\setminus\{O\}$ are
linear. Hence $\cone B_0\setminus\{O\}$ is a radiant affine
manifold, that is, has vanishing radiance obstruction
(\cite{logmirror1}, Definition~1.6). Further features are that for
$h>0$ the rescaled affine manifold $h B_0$ with charts $h\psi$ is
embedded as the affine submanifold $B_0\times\{h\}$ of constant
height. Moreover, for any $x\in B_0$ the ray
\[
L_x:=\big\{(y,h)\in \cone B_0\setminus\{O\}\,\big|\, y=x\big\}
\]
is an affine line. However, if $x\neq y$ then $L_x$ and $L_y$ are
not parallel as they would be in the product affine manifold
$B_0\times\RR_{\ge0}$. To quantify this, we can consider the flat
\emph{affine} connection on $\cone B_0\setminus \{0\}$ induced by the
affine structure on $\cone B_0$.\footnote{There is
a confusion in the literature about the attributes ``linear'' versus
``affine'' for connections. Affine connections in the sense used here
take into account the moving of the base point also, see e.g.\
\cite{kobnim}, Chapter~III.}
Identifying the tangent space of $\cone B_0\setminus\{O\}$ at
$(x,h)$ with $T_x B_0\oplus \RR$ with the second factor the tangent
space to $L_x$, we  have the following description of parallel
transport with respect to this connection.

\begin{proposition}
\label{Prop: parallel transport cone}
Let
\[
T_\gamma=A+b: T_x B_0\lra T_y B_0,
\]
be the affine parallel transport for a path $\gamma$ in $B_0$ from
$x$ to $y$ and let $b\in\Lambda_y$ be the affine displacement vector
(in an affine chart, $b=x-y$). Then the linear part of the parallel
transport on $\cone B_0\setminus\{O\}$ from $(x,h_1)$ to $(y,h_2)$
along a path of the form $t\mapsto (\gamma(t),h(t))$ is given by
\[
T_x B_0\oplus\RR\lra T_y B_0\oplus\RR,\quad
(v,\eta) \longmapsto \big(h_2^{-1} (h_1 Av+\eta b),\eta\big).
\]
\end{proposition}

\begin{proof}
By a straightforward computation the claimed formula is compatible
with compositions of paths. Hence we can restrict to the domain of a
single chart, and in turn to $B_0$ an open subset of $\RR^n$. Let
$x_1,\ldots,x_n$ be the affine coordinates on $B_0$ thus defined and
consider $x_i$ as functions on $\cone B_0$ by pull-back via the
projection $\cone B_0\setminus\{O\}\to B_0$. Affine
parallel transport on $B_0$ in this chart gives $A=\id$ and
$b=\sum_i (x_i(x)-x_i(y))\partial_{x_i}$. The $x_i$ together with
the height function $h$ define a non-affine coordinate chart on
$\cone B_0\setminus \{O\}$. Let $w_1,\ldots, w_{n+1}$ be the affine
coordinate functions on $\cone B_0\setminus\{O\}$ defined
by~\eqref{chart for CB} for the given chart of $B_0$. In particular,
$\partial_{w_1},\ldots,\partial_{w_{n+1}}$ define a basis of flat
vector fields on $\cone B_0\setminus\{O\}$. Since $h=w_{n+1}$ and
$x_i= w_{n+1}^{-1} w_i = h^{-1}w_i$ we have
\[
\partial_{w_i}=h^{-1}\partial_{x_i},\quad
\partial_{w_{n+1}}=\partial_h+\sum_{i=1}^n(-h^{-2}w_i)\partial_{x_i}
= \partial_h- h^{-1}\sum_{i=1}^n x_i\partial_{x_i}.
\]
Thus $h^{-1}\partial_{x_i}$ and $\partial_h- h^{-1}\sum_i
x_i\partial_{x_i}$ are a basis of flat vector fields on
$\cone B_0\setminus\{O\}$. Evaluating at $(x,h_1)$ and on $(y,h_2)$
now establishes the claimed formula for the linear part of the parallel
transport on $\cone B_0\setminus\{O\}$. 
\end{proof}

\begin{remark}
1)\ The proposition shows that the parallel transport of the
\emph{linear} connection on $\cone B_0\setminus\{O\}$ contains all
the information of \emph{affine} parallel transport on $B_0$. Note
also that for a closed loop on $\cone B_0\setminus\{O\}$ affine
parallel transport is linear because $\cone B_0\setminus\{O\}$ is
radiant.\\[1ex]
2)\ A special case is that $h_1=h_2=h$, for example if $\gamma$ is
a closed loop. Then the map reads
\begin{equation}
\label{equal height parallel transport}
(v,\eta) \longmapsto \big( Av+h^{-1}\eta b,\eta\big).
\end{equation}
3)\ If $B_0$ is integral then also $\cone B_0$ is integral, and all
of the stated formulas respect the integral structure. But note that
the affine embedding $B_0\times\{h\} \hookrightarrow \cone B_0$ is
integral only for $h=1$.
\end{remark}

%===========================================================
\subsection{The cone over a polyhedral pseudomanifold}
\label{Subsect: cone of polyhedral}

Let us now assume that $B_0= B\setminus\Delta$ for a polyhedral
pseudomanifold $(B,\P)$. Recall from \eqref{Eqn: cone over unbounded cell}
the definition of $\cone\sigma$ for $\sigma$ a (possibly unbounded)
polyhedron. In particular, if $\sigma\subseteq \RR^n$, then the
intersection of $\cone\sigma$ with $\RR^n\times \{0\}$ is the
asymptotic cone of $\sigma$. If $(\tau_1\to \tau_2)\in \hom(\P)$
identifies $\tau_1$ with a face of $\tau_2$ then taking cones yields
an identification of $\cone\tau_1$ with a face of $\cone\tau_2$.

\begin{definition}
\label{Def: cone(B,P)}
The \emph{cone over the polyhedral pseudomanifold $(B,\P)$} is the
topological space
\[
\cone B=\varinjlim_{\sigma\in\P} \cone{\sigma}
\]
with polyhedral decomposition $\cone\P:=\big\{\cone\tau\,\big|\,
\tau\in\P\big\}$ and affine structure on $\cone B_0\subseteq \cone
B\setminus \cone\Delta$ defined in Construction~\ref{Constr: cone
B_0}.
\end{definition}

Note that the affine structure on $\cone B_0$ extends uniquely to
the closure in \eqref{Eqn: cone over unbounded cell} in a way
compatible with the inclusion of faces. Thus $(\cone B,\cone\P)$ is
a polyhedral affine pseudomanifold as defined in
Construction~\ref{Construction: B}. 

Clearly, $(\cone B,\cone \P)$ is conical (Definition~\ref{Def:
Conical B}). Note that the projection to the
second factor $\RR$ in \eqref{Eqn: cone over unbounded cell} defines
a global affine function $h:\cone B\to \RR$, the height, and
$h^{-1}(0)$ is the union of the asymptotic cones of $\sigma\in\P$.
Normalizing by the height defines a deformation retraction
\[
\cone B\setminus h^{-1}(0)\to B\times\{1\}
\]
with preimage of a subset $A\subseteq B=B\times\{1\}$ the punctured
cone $\cone A\setminus h^{-1}(0)$ over $A$.

Our next objective is to lift a wall structure $\scrS$ on $(B,\P)$ to
$(\cone B,\cone \P)$. Note first that a
$Q^\gp$-valued MPA function $\varphi$ on $B$ induces the MPA
function $\cone\varphi$ on $\cone B$ with kinks
\[
\kappa_{\cone\ul\rho}(\cone\varphi):=\kappa_{\ul\rho}(\varphi).
\]
This definition makes sense because the connected components of
$\cone \rho\setminus\cone\Delta$ are cones over the connected
components of $\rho\setminus\Delta$. The restriction of a local
representative of $\cone{\varphi}$ to $B=B\times\{1\}$ is a local
representative of $\varphi$. In fact, this is a non-trivial
statement only at general points of a codimension one cell
$\cone{\ul\rho}$, $\ul\rho\in\tilde \P^{[n-1]}$. By the definition
of $\ul\rho$ there is a vertex $v\in\ul\rho$. Then
$\Lambda_{\cone\rho}= (\Lambda_\rho\times\{0\})
\oplus(\ZZ\cdot(v,1))$. With this description of
$\Lambda_{\cone\rho}$, if $x\in\Int\ul\rho$ and $\xi\in
\Lambda_x$ generates $\Lambda_{B,x}/\Lambda_\rho$ then $(\xi,0)$
generates $\Lambda_{\cone{B},x}/ \Lambda_{\cone{\rho}}$. The
statement now follows from the definition of the kink of an MPA
function from a local representative (Definition~\ref{Def: kink}).

For the monomials on $\cone B_0$ use the integral affine embedding
$B\times\{1\}\to \cone B$ and parallel transport along rays
emanating from the apex $O\in\cone B$ to lift a monomial at $x\in B$
to a monomial at any point on $\cone x= \{x\}\times\RR_{\ge 0}
\subseteq \cone B$. By abuse of notation we interpret a monomial $m\in
\shP_x$ at a point $x\in B\setminus \Delta$ (a monomial on $B_0$)
also as a monomial on $\cone B$ at any point $(x,h)\in\cone{x}$.

The lifting of a wall $\fop$ of codimension zero shows a certain
subtlety that we now explain. Let $\sigma\in\P_\max$ be the maximal
cell containing $\fop$ and let $n\in\check\Lambda_\sigma$ generate
$\Lambda_\fop^\perp\subseteq \check\Lambda_\sigma$. Projection to the
last component (the height) induces the map of lattices
\[
\Lambda_{\cone{\fop}}\lra \ZZ.
\]
If this map is surjective then there exists $b\in\NN$ with $(n,-b)$
a generator of $\Lambda_{\cone{\fop}}^\perp\subseteq
\check\Lambda_{\cone{\sigma}}$. In fact, if
$(m,1)\in\Lambda_{\cone{\fop}}$ is a lift of $1\in\ZZ$, then 
$\Lambda_{\cone{\fop}}=\Lambda_\fop\times\{0\} \oplus\ZZ\cdot
(m,1)$; in this case $(n,-b)$ with $b:= \langle n,m\rangle$
generates $\Lambda_{\cone{\fop}}^\perp$. In general, the image of
$\Lambda_{\cone{\fop}}\to \ZZ$ is only a subgroup of
$\ZZ$, hence of the form $a\cdot\ZZ$ for some $a\in\NN$. Let
$(m,a)\in\Lambda_{\cone{\fop}}$ be a lift. Then
$\Lambda_{\cone{\fop}}= \Lambda_\fop\times\{0\}\oplus \ZZ\cdot(m,a)$
and
\[
\Lambda^\perp_{\cone{\fop}} =\ZZ\cdot(an,-b)
\]
with $b=\langle n,m\rangle$.

\begin{definition}
\label{Def: index cone fop}
For a rational polyhedral subset $\foa\subseteq B$ the index $a\in\NN$
of the image of the projection $\Lambda_{\cone{\foa}}\to \ZZ$ to the height is
called the \emph{index of $\cone{\foa}$}.
\end{definition}

Thus if we want to lift the wall in such a way that the attached
automorphism is compatible with the automorphism attached to $\fop$
we need to take an $a$-th root of $f_\fop$ for $a$ the index of
$\cone{\fop}$. Such a root exists uniquely by the following
elementary lemma whose proof is left to the reader.

\begin{lemma}
\label{Lem: roots exist}
Let $R$ be a ring containing $\QQ$ and $I_0\subseteq R$ a nilpotent
ideal. Then for any $f\in 1+I_0$ and $a\in\NN\setminus\{0\}$ there
exists a unique $g\in 1+I_0$ with $g^a=f$. 
\qed
\end{lemma}

\begin{definition}
\label{Def: cone of wall structure}
The \emph{cone of a  wall} $(\fop, f_\fop)$ on the polyhedral
pseudomanifold $(B,\P)$ is the wall on $(\cone B,\cone\P)$ with underlying
set $\cone \fop$ and function $f_{\cone \fop}:=f_\fop^{1/a}$, with
the monomials on $B$ canonically interpreted as monomials on $\cone
B$ as explained. Here $a$ is the index of $\cone{\fop}$
(Definition~\ref{Def: index cone fop}) and $f_\fop^{1/a}$ is the
$a$-th root of $f_\fop$ according to Lemma~\ref{Lem: roots exist}.

Taking cones of the elements of a wall structure $\scrS$ on $(B,\P)$ defines the
wall structure $\cone\scrS$ on $(\cone B,\cone\P)$. Technically, under certain
circumstances, this will not satisfy the definition of wall structure. Indeed,
if $\scrS$ has a chamber $\fou$ whose asymptotic cone $\fou_{\infty}$ is $n=\dim
B$-dimensional, and if in addition $\fou$ intersects $\partial B$ in a set of
dimension $n-1$, then $\cone\fou$ will intersect two different $n$-dimensional
cells of $\partial \cone B$ in $n$-dimensional sets. This violates the condition
on chambers intersecting $\partial B$ in Definition \ref{Def: Wall
structure}. As already remarked in Remark~\ref{Rem: wall structures},5,
this problem can be rectified by adding some walls to $\cone\scrS$ which have
attached function $1$. Since such walls do not affect anything, we will ignore
this technical issue.
\end{definition}

\begin{remark}
Note that there are no roots involved in codimension one walls since
they are contained in facets of the adjacent maximal cells, which
contain integral points, and hence they have index one. Slab functions
are not of the form covered by Lemma~\ref{Lem: roots exist} and
may not have roots.
\end{remark}

\begin{proposition}
\label{Prop: Consistency for cone scrS}
If the wall structure $\scrS$ on $(B,\P)$ is consistent (in
codimension $k$) then so is the lifted wall structure $\cone\scrS$
on $(\cone B,\cone \P)$.
\end{proposition}

\begin{proof}
\emph{(Consistency in codimension zero.)}
Let $\foj\subseteq B$ be a joint for $\scrS$ of codimension zero,
contained in some $\sigma\in\P_\max$. Label the adjacent walls
$\fop_1,\ldots,\fop_r$ cyclically and let
$\theta_{\fop_1},\ldots,\theta_{\fop_r}$ be the associated
automorphisms of $R_\sigma$. Then consistency of
$\fop_1,\ldots,\fop_r$ reads
\[
\theta_{\fop_r}\circ\cdots\circ \theta_{\fop_1}=\id. 
\]
With the identification of monomials at $x\in\Int\sigma$ with
monomials on $\cone{x}$ this equation readily implies the
claimed consistency
\begin{equation}
\label{Eqn: tilde theta consistency}
\big(\theta_{\cone{\fop_r}}\circ\cdots\circ
\theta_{\cone{\fop_1}}\big)(z^m)=z^m 
\end{equation}
for all monomials $m$ coming from $B$. Indeed, if $m$ is a monomial
defined on a wall $\fop$ of $\scrS$ with $\cone\fop$ of index $a$
and $\theta_\fop(z^m) = f_\fop^{\langle n_\fop, \ol m\rangle} \cdot
z^m$, then viewing $m$ as a monomial on $\cone B$ it holds
\[
\theta_{\cone\fop}(z^m) =f_{\cone\fop}^{a\langle n_\fop,\ol
m\rangle}\cdot z^m = f_\fop^{\langle n_\fop,\ol m\rangle}\cdot z^m.
\]
Since $\Lambda_{\cone{\sigma}}=  \Lambda_\sigma\times\{0\}\oplus
\ZZ\cdot(0,1)$ it remains to show \eqref{Eqn: tilde theta
consistency} for $m=(0,1)$. Here $(0,1)\in \Lambda_\sigma\oplus\ZZ$
is viewed as a monomial on $\Int \cone{\sigma}$ with vanishing
$Q$-component via \eqref{shP on Int(sigma)}. Since
$\big(\theta_{\cone{\fop_1}}\circ\ldots\circ
\theta_{\cone{\fop_r}}\big) (z^m)= (1+h)\cdot z^m$ with $h\in
I_0\cdot R_\sigma$ and in view of the uniqueness statement in
Lemma~\ref{Lem: roots exist}, it suffices to prove \eqref{Eqn: tilde
theta consistency} for any power of $z^{(0,1)}$. Let
$(m,a)\in\Lambda_{\cone{\foj}}$ be such that $a\in\NN$ is the index
of $\cone{\foj}$. Now $(0,a)=(m,a)-(m,0)$ with $m\in\Lambda_\sigma$,
and \eqref{Eqn: tilde theta consistency} already holds for
$z^{(m,0)}$, while $z^{(m,a)}$ is left invariant by any of the
$\theta_{\cone{\fop_i}}$. Hence \eqref{Eqn: tilde theta consistency}
holds for all monomials $m$ on $\cone{\sigma}$.
\bigskip

\noindent
\emph{(Consistency in codimension one.)} Let
$\foj$ be a codimension one joint and $\rho\in\P^{[n-1]}$ the
codimension one cell containing $\foj$. As in Definition~\ref{Def:
Consistency in codim one} let $\fob_1,\fob_2\in\scrS$ be the slabs
adjacent to $\foj$ and $\theta,\theta'$ the automorphisms of
$R_\sigma$, $R_{\sigma'}$ induced by passing through the walls
containing $\foj$ in the correct order. Here $\theta$ and $\theta'$
collect the walls in the two maximal cells $\sigma$, $\sigma'$
containing $\rho$, respectively. Let $\chi_{\fob_i,\sigma/\sigma'} :
R_{\fob_i} \to R_{\sigma/\sigma'}$ be the natural ring
homomorphisms. Consistency of $\scrS$ around $\foj$ says
\[
(\theta\times\theta')\big((\chi_{\fob_1,\sigma},\chi_{\fob_1,\sigma'})
(R_{\fob_1})\big)= (\chi_{\fob_2,\sigma},\chi_{\fob_2,\sigma'})
(R_{\fob_2}).
\]
The argument now runs analogously to the codimension zero case. 
Using a chart around $x\in\fob_1$ with $0$ in the affine span of the
image of $\fob_1$ shows that
\[
\Lambda_{\cone{\rho}}= \Lambda_{\rho}\oplus\ZZ\cdot(0,1).
\]
In particular, for $x\in \Int\fob$ the generator of $\Lambda_{\cone
B,x}/\Lambda_{\cone\rho}$ leading to the monomials $Z_+$,$Z_-$ can
be chosen to lie in $\Lambda_\sigma\oplus 0$. With this identification and
choice we have
\[
R_{\cone{\fob_i}} = \big(A[Q]/I\big)[\Lambda_\rho]
[z^{(0,1)},Z_+,Z_-] / (Z_+ Z_-- f_{\fob_i}z^{\kappa_{\ul\rho_i}}), 
\]
with $\ul\rho_i\supseteq\fob_i$ the $(n-1)$-cell of the barycentric
subdivision containing $\fob_i$. Let $\tilde\theta$, $\tilde\theta'$
be the automorphisms of $R_{\cone{\sigma}}$, $R_{\cone{\sigma'}}$
induced by crossing the walls containing $\cone{\foj}$ on
$\cone{B}$. Writing $\tilde\chi_{\fob_i,\sigma/\sigma'}:
R_{\cone{\fob_i}}\to R_{\cone{\sigma}/\cone{\sigma'}}$ for the
natural localization homomorphisms, the equation
\begin{equation}
\label{Eqn: tilde consistency in codim 1}
(\tilde\theta\times\tilde\theta')
\big((\tilde\chi_{\fob_1,\sigma}, \tilde\chi_{\fob_1,\sigma'})
(R_{\cone{\fob_1}})\big)=
(\tilde\chi_{\fob_2,\sigma}, \tilde\chi_{\fob_2,\sigma'})
(R_{\cone{\fob_2}})
\end{equation}
for consistency around $\cone{\foj}$ already holds for monomials
lifted from $B$. In fact, for $m\in \Lambda_\sigma$ we have seen in
the treatment of consistency in codimension zero that
$\tilde\theta(z^m)=\theta(z^m)$ (with the usual abuse of notation of
interpreting monomials on $B$ as monomials on $\cone{B}$), and
similarly for $\sigma'$ and $\tilde\theta'$. Since $Z_+, Z_-$ are
monomials lifted from $B$, both for $R_{\cone{\fob_1}}$ and
$R_{\cone{\fob_2}}$, the equality~\eqref{Eqn: tilde consistency in
codim 1} holds for any monomial lifted from $B$.

It remains to treat $z^{(0,1)}\in R_{\cone{\fob_1}}$. Let
$a\in\NN\setminus\{0\}$ be the index of $\cone{\foj}$ and let
$m\in\Lambda_\rho$ be such that $(m,a)\in\Lambda_{\cone{\foj}}$.
Then $(m,a)$ is tangent to each wall containing $\cone{\foj}$ and
hence
\begin{equation}
\label{Eqn: z^(m,a) consistency}
\begin{aligned}
(\tilde\theta\times\tilde\theta')
\big((\tilde\chi_{\fob_1,\sigma},
\tilde\chi_{\fob_1,\sigma'})(z^{(m,a)})\big)
&=\big(\tilde\theta,\tilde\theta'\big)(z^{(m,a)})\\
&=\big( z^{(m,a)}, z^{(m,a)}\big) =
(\tilde\chi_{\fob_2,\sigma}, \tilde\chi_{\fob_2,\sigma'})(z^{(m,a)}).
\end{aligned}
\end{equation}
Moreover, $(m,0)$ is a monomial lifted from $B$, and $m$ is
invariant under monodromy around $\foj$ for $m\in\Lambda_\rho$. Thus
by consistency on $B$ there exists $h\in R_{\fob_2}$ with
\begin{equation}
\label{Eqn: (theta,theta')(z^m)}
(\theta,\theta')(z^m) =
(\chi_{\fob_2,\sigma}, \chi_{\fob_2,\sigma'})(h)
\end{equation}
and $h$ is congruent to $z^m$ modulo $I_0$. Since 
$\chi_{\fob_2,\sigma/\sigma'}(h)$ is thus obtained from $z^m$ by
wall crossing there exists $f\in 1+ I_0\cdot R_{\fob_2}$ with
$h=f\cdot z^m$. Hence it holds $\big(\tilde\theta,
\tilde\theta'\big)(z^{(m,0)}) = (\tilde\chi_{\fob_2,\sigma},
\tilde\chi_{\fob_2,\sigma'})(f\cdot z^{(m,0)})$. Together with
\eqref{Eqn: z^(m,a) consistency} this shows $\big(\tilde\theta,
\tilde\theta'\big)(z^{(0,a)}) = (\tilde\chi_{\fob_2,\sigma},
\tilde\chi_{\fob_2,\sigma'}) (f^{-1}\cdot z^{(0,a)})$. Taking roots
according to Lemma~\ref{Lem: roots exist} then yields
\[
\big(\tilde\theta,\tilde\theta'\big)(z^{(0,1)})
= (\tilde\chi_{\fob_2,\sigma}, \tilde\chi_{\fob_2,\sigma'})(f^{-1/a}\cdot
z^{(0,1)}),
\]
establishing~\eqref{Eqn: tilde consistency in codim 1} for the
remaining generator of $R_{\fob_1}$.

\bigskip
\noindent
\emph{(Consistency in codimension two.)} Let $\foj$ be a codimension
two joint of $B$, and let $\tau\in\P^{[n-2]}$ be the minimal cell containing
$\foj$. In contrast to the previous cases of codimension zero and one, the index
of $\cone{\foj}$ is always one. Indeed, since $\tau$ has integral points,
$\cone{\tau}$ has index one, and
\[
\Lambda_{\cone{\foj}}=\Lambda_{\cone{\tau}},
\]
because $\Int\foj$ is an open subset of $\tau$. Thus in a chart for
any $\sigma\in\P_\max$ containing $\foj$ and centered at an integral
point of $\tau$ we have the decomposition
\[
\Lambda_{\cone{\sigma}}= \big(\Lambda_\sigma\times\{0\}\big)
\oplus \ZZ\cdot (0,1).
\]
Now consistency around $\foj$ means that the functions $\vartheta^\foj_m(p)$
are independent of the choice of general point $p\in B_\foj$ inside one
chamber and are related by chamber morphisms on adjacent chambers
(Definition~\ref{Def: Consistency in codim two}). As in codimension zero and one
the analogous statements then hold for $\cone\foj$. Indeed, these
properties are immediate for $\vartheta^{\cone\foj}_{(m,0)}(p)$ with
$(m,0)\in\Lambda_\sigma\times\{0\}$, that is, for monomials lifted from $B$. On
the other hand, a monomial tangent to $\foj$ is left invariant by any of the
ring homomorphisms changing chambers. In particular,
$\vartheta^{\cone\foj}_{(m,a)}= z^{(0,a)}\cdot
\vartheta^{\cone\foj}_{(m,0)}$. This proves consistency around $\cone\foj$.
\end{proof}
\medskip

For later use we also express here the asymptotic monomials
(Definition~\ref{Def: asymptotic monomial}) of $\cone{B}$ in terms
of the geometry of $B$. First note that the projection to the height
maps any tangent vector $\ol m$ of a monomial $m$ on $\cone{B}$ to
an integral tangent vector on $\RR$. We call this integer the
\emph{degree of $m$}, written $\deg m$. If $m$ is an asymptotic
monomial of $\cone{B}$ then $\deg m\in\NN$.

\begin{proposition}
\label{Prop: Asymptotic monomials cone(B)}
The set of asymptotic monomials on $\cone{B}$ of degree $d>0$ are in
canonical bijection with the set $B\big(\frac{1}{d}\ZZ\big)$ of
$1/d$-integral points of $B$. The set of asymptotic monomials on
$\cone{B}$ of degree $d=0$ are in canonical bijection with the set of
asymptotic monomials of $B$.
\end{proposition}

\begin{proof}
For an integral polyhedron $\sigma\subseteq \Lambda_\RR$ an asymptotic
monomial on $\cone{\sigma}$ is just an element of $\cone{\sigma}\cap
(\Lambda\times\ZZ)$, that is, an integral point in the cone. If
$(m,d)$ is such a point then $d$ is the degree of the asymptotic
monomial and $m\in d\cdot \sigma\cap \Lambda$. For $d>0$ this means
$\frac{1}{d}m\in \sigma\cap\big( \frac{1}{d}\Lambda\big)$, giving a
$1/d$-integral point of $\sigma$; for $d=0$ we have an asymptotic
monomial of $\cone{\sigma}\cap \big(\Lambda_\RR\times\{0\}\big)$,
that is, an asymptotic monomial of the asymptotic cone
$\sigma_\infty$ of $\sigma$.

The general statement follows from the statement for an individual
cell since the identification of asymptotic monomials on faces is
compatible with the stated identification of asymptotic monomials on
$\cone{\sigma}$.
\end{proof}

%===========================================================
\subsection{Theta functions and the Main Theorem}
\label{Subsect: Theta functions}

Starting from a consistent wall structure $\scrS$ on the polyhedral
pseudomanifold $(B,\P)$, we have now arrived at a consistent wall
structure $\cone{\scrS}$ on the cone $(\cone{B},\cone{\P})$ of
$(B,\P)$, see Definition~\ref{Def: cone of wall structure} and
Proposition~\ref{Prop: Consistency for cone scrS}. Then $\scrS$ and
$\cone{\scrS}$ lead to the schemes $\foX^\circ$ and
$\foY^\circ:=\foX^\circ_{\cone{\scrS}}$, respectively. We can then
construct $\foW:=\Spec\Gamma(\foX^{\circ}, \O_{\foX^{\circ}})$ and
$\foY:=\Spec\Gamma(\foY^{\circ},\O_{\foY^{\circ}})$. Each will be
flat over $A[Q]/I$, with $\foY^{\circ}$ an open subset of the affine
scheme $\foY$. The object of the present subsection is the
construction of a similarly canonical open embedding
$\foX^\circ\hookrightarrow \foX$, now with $\foX$ projective over
$\foW$. This will be done by relating $\foY$ to the total space of
$\O_\foX(-1)$, the dual of an ample invertible sheaf on $\foX$
coming naturally with the construction.

The first step in establishing this picture is the construction of
the total space $\foL^\circ$ of a line bundle over $\foX^\circ$. The sheaf
of sections of $\foL^\circ$ will be identified with the restriction to
$\foX^\circ$ of $\O_\foX(-1)$.

\begin{construction}
\label{Constr: Truncated cone}
\emph{(The truncated cone $\ol\cone{B}$ and the associated schemes
$\foL^{o,\times}\subseteq \foY^\circ\subseteq \foL^\circ$.) }
Let $(B,\P)$ be a polyhedral pseudomanifold. The \emph{truncated
cone} $(\ol\cone{B},\ol\cone{\P})$ over $(B,\P)$ is the polyhedral
pseudomanifold with underlying topological space
\[
\ol\cone{B}:= \big\{(x,h)\in\cone{B}\,\big|\, h\ge 1\big\},
\]
endowed with the induced affine structure and induced polyhedral
decomposition with cells $\ol\cone{\sigma}:=
\big\{(x,h)\in\cone{\sigma}\,\big|\, h\ge 1\big\}$, $\sigma\in\P$.
Clearly, the boundary of $\ol\cone{B}$ decomposes into two parts,
one coming from $\partial B$, one from the truncation:
\[
\partial\big(\ol\cone{B}\big)= \ol\cone(\partial B)\cup (B\times\{1\}).
\]

If $\scrS$ is a (consistent) wall structure on $(B,\P)$ the wall
structure $\cone{\scrS}$ restricts to a (consistent) wall structure
$\ol\cone{\scrS}$ on the truncated cone $(\ol\cone{B},\ol\cone{\P})$
(subject to the same caveat of  Definition \ref{Def: cone of wall
structure} of perhaps needing to add trivial walls). Indeed, the
only thing to check is consistency of joints introduced by the
truncation. These are either of the form  $\rho\times \{1\}$ where $\rho$
is an $(n-1)$-dimensional cell in $\partial B$ or $\fop\times\{1\}$
where $\fop$ is a wall in $\scrS$. However, there are
no walls of $\ol\cone{\scrS}$ containing joints of the first sort, and hence
consistency follows trivially from Proposition  \ref{Prop: Boundary
consistency}. For joints of the second sort, there is no consistency condition
if $\fop$ is a codimension zero wall. If $\fop$ is a slab, consistency
again follows from Proposition \ref{Prop: Boundary consistency}, this time
using the fact that all exponents $m$ appearing in $f_{\cone{\fop}}$ have
$\ol{m}$ tangent to $B\times\{1\}$. Thus in the consistent case we obtain from
$\cone{\scrS}$ and $\ol\cone{\scrS}$ two flat $A[Q]/I$-schemes
$\foY^\circ$ and $\foL^\circ$. Both schemes are covered by spectra
of rings with a $\ZZ$-grading defined by the degree of monomials,
introduced in the text before Proposition~\ref{Prop: Asymptotic
monomials cone(B)}, and the gluings respect the grading. In
particular, $\foL^{\circ}$ and $\foY^{\circ}$ come with a
$\Gm$-action.

Note that $\foL^{\circ}$ contains one stratum for every maximal cell
$\sigma\in\P$ induced by the cell of the lower boundary
$\sigma\times\{1\} \subseteq B\times\{1\}$. On the other hand, if
the asymptotic cone  $\sigma_{\infty}$ of $\sigma$ has dimension
$n$, then $\sigma_{\infty}$ is an $n$-cell of the lower
boundary of $\cone B$, and hence there is a stratum of
$\foY^{\circ}$ indexed by $\sigma_{\infty} \times\{0\}$.
Furthermore, if $\fou$ is a chamber of $\scrS$ contained in $\sigma$
with $\fou_{\infty}$ $n$-dimensional, then the rings 
$R^{\partial}_{\ol\cone \fou}$ contributing to $\foL^{\circ}$ and
$R^{\partial}_{\cone\fou}$ contributing to $\foY^{\circ}$ coincide.
In particular, $\foY^{\circ}$ is thus a subscheme of $\foL^{\circ}$.

Let $\foL^{\circ,\times}\subseteq \foL^{\circ}$ be the open
subscheme obtained by deleting the codimension one strata of
$\foL^{\circ}$ corresponding to the lower boundary cells $B \times
\{1\} \subseteq \partial(\ol\cone B)$. This is obtained by gluing
together only those charts of the form $\Spec R_{\fou}$ for any
$\fou$,  $\Spec R_{\fou}^{\partial}$ for  those $\fou$ intersecting
$\partial(\ol\cone B)\setminus B\times\{1\}$ in a codimension one
set, and $R_{\fob}$ for $\fob$ a slab. Note that the same set of
rings appears in the description of $\foY^{\circ}$, and hence
$\foL^{\circ,\times}\subseteq \foY^{\circ}$ also (and in fact we
have equality provided that all cells of $\P$ have asymptotic cone
of dimension less than $n$).

For the rings used for constructing $\foL^{\circ}$,
$\foL^{\circ,\times}$,  or $\foY^{\circ}$, each subring of elements
of degree zero can be identified with one of the rings in the
construction of $\foX^\circ$, with each ring for $\foX^\circ$
occurring. Hence $\foL^\circ$, $\foL^{\circ,\times}$ and
$\foY^{\circ}$ come with a $\GG_m$-invariant surjection to
$\foX^\circ$.

We claim that $\foL^{\circ,\times}$ has naturally the structure of
the total  space of a $\GG_m$-torsor over $\foX^\circ$, that is, a
line bundle minus the zero section. Moreover, we have
$\foL^{\circ,\times}\subseteq \foY^{\circ} \subseteq \foL^{\circ}$,
with the inclusion of $\foL^{\circ,\times} \subseteq \foL^{\circ}$
partially compactifying this $\GG_m$-torsor by filling in the zero
section over the complement of the codimension one strata in
$\foX^\circ$.

Local trivializations of $\foL^{\circ,\times}$ are given as
follows.  For $\rho\in \P^{[n-1]}$, $\rho\not\subseteq\partial B$, any
choice of integral point $v\in\rho$ induces an isomorphism
$\Lambda_{\cone{\rho}} = \big(\Lambda_\rho\times\{0\}\big) \oplus
\ZZ\cdot(v,1)$. Hence in view of \eqref{Def: R_fob}, for any slab
$\fob\in\scrS$ contained in $\rho$ the choice of an integral point
$v\in\rho$ induces an isomorphism of $R_\fob$-algebras
\[
R_{\cone{\fob}} \stackrel{\simeq}{\lra} R_{\fob}[u, u^{-1}],
\]
identifying $z^{(v,1)}$ with $u$. This induces a local
trivialization
\[
\Spec (R_{\cone{\fob}}) \simeq
\Spec(R_\fob)\times \GG_m.
\]
Here $\GG_m=\Spec\big( \ZZ[u,u^{-1}] \big)$ and the product is taken
over $\ZZ$. A different choice of integral point leads to the
multiplication of $u$ by some $z^m$ with $m\in \Lambda_\rho$, an
invertible homomorphism of $R_\fob$-algebras. Moreover, any crossing
of codimension one joint from $\fob$ to $\fob'$ leads to the
multiplication of $u$ by an invertible element in $R_{\fob'}$ and is
otherwise compatible with the isomorphism of rings $R_\fob\to
R_{\fob'}$.  Similarly, any integral point on a maximal cell
$\sigma$ induces a local trivialization $\Spec (R_{\cone{\fou}})
\simeq \Spec(R_\fou)\times\GG_m$ for chambers $\fou\subseteq\sigma$,
and wall crossings are again homogeneous of degree zero. This shows
that $\foL^{\circ,\times}$ comes with the structure of a
$\GG_m$-torsor over $\foX^\circ$.

The construction of $\foL^\circ$ only adds
$\Spec(R^{\partial}_{\tilde\fou})$  for $\tilde\fou$ a chamber of
$\ol\cone{\scrS}$ that intersects the lower boundary
$B\times\{1\}\subseteq \ol\cone{B}$ in $\fou\times\{1\}$, where $\fou$
is a chamber of $\scrS$. Then $R_{\tilde\fou}\subseteq
R^{\partial}_{\tilde\fou}$ leads to the partial $\GG_m$-equivariant
compactification
\[
\Spec(R_\fou)\times\GG_m\subseteq \Spec(R_\fou)\times\AA^1.
\]
This process adds the zero-section of a line bundle over the
complement of the codimension one strata in $\foX^\circ$, as claimed.
\end{construction}

We are now in the position to prove one of the main results of this
paper.

\begin{theorem}
\label{Thm: Main}
Let $\scrS$ be a consistent wall structure on the polyhedral
pseudomanifold $(B,\P)$. Denote by $\cone{\scrS}$ the induced consistent
wall structure\footnote{We assume here $\QQ\subseteq A$ to assure the
existence of the roots of the wall functions $f_\fop$ required in
Definition~\ref{Def: cone of wall structure}. In some other cases
one can derive the existence of $\cone{\scrS}$ by a priori methods
independently of this assumption. For example, for locally rigid
singularities one may run the inductive construction from
\cite{affinecomplex} directly on $\cone{B}$.} on
$(\cone{B},\cone{\P})$. Let $\foX^\circ$, $\foY^\circ$ be the
associated flat $A[Q]/I$-schemes according to Proposition~\ref{Prop:
foX^o exists}\, for $\scrS$ and $\cone{\scrS}$, respectively. Let
\[
R_\infty:=\Gamma\big( \foX^\circ,\O_{\foX^\circ}\big),\quad
S:=\Gamma\big( \foY^\circ,\O_{\foY^\circ}\big)
\]
be the $A[Q]/I$-algebras with canonical $A[Q]/I$-module basis of
sections $\vartheta_m$ constructed in Theorem~\ref{Thm: vartheta_m
exist}. Here $m$ runs through the set of asymptotic monomials on $B$
for $R_\infty$ and on $\cone{B}$ (cf.\ Proposition~\ref{Prop:
Asymptotic monomials cone(B)}) for $S$, respectively.

Then the following holds.
\begin{enumerate}
\item[(a)]
The affine schemes $\foW:=\Spec R_\infty$ and $\foY:=\Spec S$
are flat over $A[Q]/I$.
\item[(b)]
The ring $S$ is a $\ZZ$-graded $R_\infty$-algebra, with the grading
given by $\deg\vartheta_m:=\deg m$. Also, $S_0=R_{\infty}$, where $S_0$
denotes the degree $0$ part of $S$.
\item[(c)]
The scheme $\foX^\circ$ embeds canonically as an open dense subscheme
into $\foX:= \Proj (S)$, and $\foX$ is flat over $A[Q]/I$. Moreover,
$\foX$ is the unique flat extension of $X_0$ from $A[Q]/I_0$ to
$A[Q]/I$ containing $\foX^\circ$ as an open subscheme and proper over
$\foW$.
\item[(d)]
Denote by $\foL\to\foX$ the line bundle with sheaf of sections $\O_{\Proj(S)}
(-1)$. Then there is a canonical isomorphism $\Gamma(\foL,\O_{\foL}) \simeq S$
that induces a morphism $\foL\rightarrow \foY$ contracting the zero-section of
$\foL$ to the fixed locus of the $\GG_m$-action on $\foY$ defined by the grading
of $S$. In particular, $S=\bigoplus_{d\in\NN} \Gamma\big(\foX,\O_\foX(d)\big)$
is the homogeneous coordinate ring of $(\foX,\O_\foX(1))$.
\end{enumerate}
\end{theorem}

\begin{proof}
(a)\ \ 
In the case of $S$, this follows from Proposition 
\ref{Prop: foX in the conical case} because
$\cone\scrS$ is a conical wall structure on the conical polyhedral
pseudomanifold $\cone B$. For $R_{\infty}$, we apply Theorem
\ref{Thm: vartheta_m exist}.\\[1ex]
(b)\ \ In Construction~\ref{Constr: Truncated cone} we saw that
$\foY^\circ$ is a partial compactification of a $\GG_m$-torsor over 
$\foX^\circ$ inside its corresponding line bundle. The weight with respect
to the induced $\GG_m$-action on $S=\Gamma(\foY^\circ,\O_{\foY^\circ})$
defines the $\ZZ$-grading on $S$. Of course, an element of $S$ is
homogeneous of degree $d$ if and only if
its local representatives in the rings $R_{\cone{\fob}}$ and
$R_{\cone{\fou}}$ are homogeneous of degree $d$ as defined
in Construction~\ref{Constr: Truncated cone}.

The degree zero part of $S$ has an $A[Q]/I$-basis $\vartheta_m$ with
$m$ an asymptotic monomial on $\cone{B}$ of degree zero, which is
hence an asymptotic monomial on $B$ (Proposition~\ref{Prop:
Asymptotic monomials cone(B)}). Embedding $B$ as $B\times\{1\}$ into
$\cone{B}$ this shows that we can identify the degree zero part of
$S$ with the ring of global functions on $\foX^\circ$. In fact, the
latter has an $A[Q]/I$-basis of canonical global functions with the
same index set, and the multiplication rule only depends on broken
lines for monomials of degree zero. These broken lines are parallel
to $B\times\{1\}$, hence are in bijection with broken lines on
$B$.\\[1ex]
(c)\ \ Denote by $\shL^\circ$ the invertible sheaf on $\foX^\circ$
associated to the dual of the $\GG_m$-torsor 
$\foL^{\circ,\times}\to\foX^\circ$. By Construction~\ref{Constr:
Truncated cone} the linear space associated to $\shL^\circ$ extends
$\foL^\circ$ over the codimension one strata of $\foX^\circ$. Thus
the sheaf of sections of the dual of $\foL^\circ$ agrees with the
restriction of $\shL^\circ$ to the complement of the codimension one
strata of $\foX^\circ$. Denote by $\foL\to \foX= \Proj(S)$ the line
bundle with sheaf of sections $\O_\foX(-1)$. Defining a morphism
\[
\Phi: \foX^\circ\lra \foX
\]
together with an isomorphism $\shL^\circ\simeq \Phi^*
\big(\O_\foX(1)\big)$ amounts to writing down a homomorphism
\[
\phi: S\lra\bigoplus_{d\in\NN} \Gamma \big(\foX^\circ,(\shL^\circ)^{\otimes
d}\big)
\]
with the image of the first graded piece $S_1\subseteq S$ generating
$\shL^\circ$ (\cite{EGAII}, \S3.7). For the definition of $\phi$ note
that $S$ has an $A[Q]/I$-module basis of theta functions
$\vartheta_m$ labelled by asymptotic monomials on $\cone{B}$. Now
the asymptotic monomials on $\cone{B}$ and on $\ol\cone{B}$ agree, and
hence for any such $m$ there is also a theta function $\vartheta'_m$
on $\foL^\circ$. By the definition of the local trivializations of
$\foL^\circ$ (Construction~\ref{Constr: Truncated cone}), $\vartheta'_m$
with $\deg m=d$ is homogeneous of degree $d$ in the fibre
coordinate, and hence it defines a section of $(\shL^\circ)^{\otimes
d}$. Define $\phi$ by mapping $\vartheta_m$ to $\vartheta'_m$. Note
that $\phi$ is compatible with the multiplicative structures by
comparison on $\foL^{\circ,\times}\subseteq\foY^\circ$. To see that the
image of $S_1$ generates $\shL^\circ$, choose an interior codimension one
cell $\rho\in \P$, $\fob$ a slab in $\rho$, let $v\in \rho$ be an integral point
and $m$ the associated asymptotic monomial on $\cone{B}$ of
degree~$1$ (Proposition~\ref{Prop: Asymptotic monomials cone(B)}).
Then in the isomorphism $R_{\ol\cone{\fob}} \simeq R_\fob[u,u^{-1}]$
induced by the choice of $v$ (see Construction~\ref{Constr:
Truncated cone}),
\[
\vartheta_m= u+\cdots
\]
with the dots standing for elements obtained by wall crossing. In
particular, $\vartheta_m\equiv u$ modulo $I_0$, and hence
$\vartheta_m$ generates $\shL^\circ$ on the whole chart. A similar
argument applies for the charts with ring $R^{\partial}_{\ol\cone\fou}$.

It remains to show that $\Phi:\foX^\circ\to \foX$ is an open embedding.
The following argument is analogous to the affine case of
Propositions~\ref{Prop: Properties of foX^foj} and \ref{Prop: foX in
the conical case}. By Proposition~\ref{Prop: foX^o modulo I_0} the
statement is true modulo $I_0$. There are two flat deformations of
$X_0$, one given by $i_*\O_{\foX^\circ}$, the other by $\foX=\Proj S$.
In both cases flatness follows by the criterion of \cite{GHK1},
Lemma~2.29. In fact, if $v\in \P$ is a vertex and $x\in X_0$ the
corresponding zero-dimensional toric stratum, let $U_v\subseteq X_0$
be the affine open subset defined as the complement of toric strata
disjoint from $x$. Denote by $m_0$ the asymptotic monomial of degree
one defined by $v$. Then on $U_v$ there is an $A[Q]/I_0$-module
basis of regular functions of the form $\vartheta_m/
\vartheta_{m_0}^d$, $d=\deg m$. Any of these lift to both
deformations, as a quotient of theta functions. This proves flatness
of both deformations. Moreover, by \cite{GHK1}, Lemma~2.30, the
stated liftings are $A[Q]/I$-bases of the rings of regular
functions. Since $\Phi|_{U_v}$ maps these liftings onto each other,
we also obtain an isomorphism $(X_0,i_*\O_{\foX^\circ})\simeq \foX$ by
Lemma \ref{Lem: Extension lemma}.
In particular, $\foX^\circ\to \foX$ is an
open embedding and $\foX$ has the stated uniqueness property.\\[1ex]
(d)\ \ By (c) we can now identify $\foX^\circ$ with the complement of
the codimension two strata in $\foX$. With this identification we
have seen that $\foL^\circ$ is the restriction of the total space $\foL$
of $\O_\foX(-1)$ to $\foX^\circ$, with the codimension one strata of the
zero section removed. Since, for a vertex $v$ with corresponding
asymptotic monomial $m$, $\vartheta_m$ yields a trivialization of
$\shL^{\circ}$ on $\foX^\circ\cap U_v$,
we also see that $i_*\shL^\circ=
\O_\foX(1)$. For the statement on global functions
on $\foL$ note the following sequence of inclusions and
identifications
\begin{equation}
\label{Eqn: Global sections of foL}
\Gamma(\foL,\O_\foL)\subseteq \Gamma(\foL^{\circ},\O_{\foL^\circ})
\subseteq \Gamma(\foY^\circ,\O_{\foY^\circ})= \Gamma(\foY,\O_\foY)=S.
\end{equation}
Conversely, the $A[Q]/I$-module basis elements $\vartheta_m$ of $S$
lift to global sections of $\O_\foX(d)$, $d=\deg m$, hence to an
element of $\Gamma(\foL,\O_\foL)$, of fibrewise degree $d$. Hence
all inclusion in \eqref{Eqn: Global sections of foL} are indeed equalities.

The remaining statements follow from the usual correspondence
between the projective variety associated to a $\ZZ$-graded ring
generated in degree~$1$ and its affine cone.
\end{proof}

\begin{remark}
\label{Rem: foD, projective case}
Following up on Remark~\ref{Rem: Boundary divisor} we now also
obtain a partial completion $\foD\subseteq\foX$ of the divisor
$\foD^\circ\subseteq\foX^\circ$. Indeed, each $\rho\subseteq\partial B$ defines
a restriction map
\[
\Gamma(\foX^\circ, \O_\foX(d))\lra
\Gamma \big(\foD_\rho^\circ, \O_{\foX}(d)|_{\foD_\rho^\circ}\big).
\]
Taking the direct sum over $d$ of the kernel of these maps defines a
graded ideal $K_\rho\subseteq S$. It is then easy to see that $K_\rho$
is a free $A[Q]/I$-module with generators defined by the theta
functions $\vartheta_m$ with $m$ an asymptotic monomial of
$\cone{B}$ but not of $\cone{\rho}$. In particular, the quotient
$S_\rho:= S/K_\rho$ is a free $A[Q]/I$-module with basis (the
restrictions to $D_\rho^\circ$ of) the theta functions $\vartheta_m$
with $m$ running over the asymptotic monomials of $\cone{\rho}$.
Moreover, this construction of $S_\rho$ is obviously compatible with
the construction of the homogeneous coordinate ring of $\foD_\rho^\circ$
via the wall structure $\scrS_\rho$ on $\rho$ in Remark~\ref{Rem:
Boundary divisor}. In particular, $\foD_\rho:=\Proj(S_\rho)$ defines
a partial completion of $\foD^\circ_\rho$. Note however that by our
definition of $\Delta\subseteq B$ the complement of $\foD_\rho^\circ$ in
$\foD_\rho$ is a union of divisors rather than of codimension two
subsets as for $\foX^\circ\subseteq\foX$.

The divisor $\foD\subseteq\foX$ is then defined as the scheme
theoretic union of the $\foD_\rho$, that is, it is the closed
subscheme of $\foX$ given by the homogeneous ideal $\bigcap_\rho
K_\rho\subseteq S$.
\end{remark}

\begin{remark}
\label{Rem: Doubly truncated cone}
It is also easy to treat the completion $\PP(\O_\foX(1)\oplus
\O_\foX)= \PP(\O_\foX\oplus\shL)$ of $\foL$ to a $\PP^1$-bundle in
the current framework. For an integer $a>1$ consider the polyhedral
pseudomanifold
\[
\cone_{[1,a]}{B}:=\big\{ (x,h)\in \cone{B}\,\big|\, h\in[1,a]
\big\},
\]
and write $\foP^\circ$ for the associated flat scheme over $A[Q]/I$.
Then for any chamber $\fou$ for $\scrS$ and an integral point
$v\in \fou$ there are two non-interior slabs $\fou\times\{1\}$ and
$\fou\times\{a\}$ for the induced wall structure on
$\cone_{[1,a]}{B}$. These give rise to two charts for $\foP^\circ$,
\[
R_{\fou\times\{1\}} \stackrel{\simeq}{\lra} R_{\fou}[u],
\quad
R_{\fou\times\{a\}} \stackrel{\simeq}{\lra} R_{\fou}[v].
\]
Clearly, $\foP^\circ$ contains the $\GG_m$-torsor $\foL^{\circ,\times}$ as an
open dense subscheme and $u|_{\foY^\circ}=(v|_{\foY^\circ})^{-1}$
generate this $\GG_m$-torsor locally. Thus $\foP^\circ$ is an open
subscheme of the $\PP^1$-bundle
$\PP(\O_{\foX^\circ}\oplus\shL^\circ)$ with complement two copies of
the codimension one strata of $\foX^\circ$. It extends to the flat
deformation $\PP(\O_{\foX}\oplus\O_\foX(1))$ of $\PP(\O_{X_0}\oplus
\O_{X_0}(1))$ by pushing forward the sheaf of regular functions.

Note that the integral length $a-1$ of the interval $[1,a]$ agrees
with the degree (as a line bundle over $\PP^1$) of
$\O_{\PP(\O_{\foX} \oplus\O_\foX(1))}(1)$ on a fibre of
$\PP(\O_\foX\oplus\O_\foX(1))$. The lower and upper boundaries of
$\cone_{[1,a]}B$ represent the two distinguished sections of
$\PP(\O_{\foX}\oplus\O_\foX(1))$, with the lower boundary giving the
contractible (``negative'') section. Interestingly, the restriction
of the polarization to the other section does not produce the
polarization on $\foX$ but its $a$-fold multiple.
\end{remark}

%===========================================================
\subsection{The action of the relative torus}
\label{Subsect: torus action}

Another feature of the construction in many cases is the existence of a
canonical action of a large algebraic torus on $\foX$ as a projective scheme.
Our theta functions generate isotypical components for the induced action on
$\Gamma(\foX,\O_\foX(d))$.

We begin by identifying the group of automorphisms of $X_0$ over $A[Q]/I_0$
preserving the toric strata. In practice, the preservation of strata is either
automatic or desired, for example if $B$ is closed or if one also wants to
include the divisor in $\foX$ defined by $\partial B$. We view the decomposition
of $X_0$ into toric strata as being given by the construction of $X_0$ and
therefore write $\Aut_{A[Q]/I_0} (X_0)$ for the strata preserving closed
subgroup of the automorphism group of $X_0$ as a scheme over $A[Q]/I_0$. This
result is of motivational character for explaining the role of $\PL(B)^*$ both
in the present subsection as well as in \S\ref{Subsect: GS torus action}. Here
we write $\PL(B)= \PL(B,\ZZ)$ and $\PL(B)^*=\Hom (\PL(B),\ZZ)$ for brevity. Note
that $\PL(B)$ depends only on the affine structure on the interiors of the
maximal cells, just as the central fibre $X_0$.

\begin{proposition}
\label{Prop: Aut(X_0)}
The connected component of the identity of $\Aut_{A[Q]/I_0} (X_0)$ is the torus
over $\Spec(A[Q]/I_0)$ with character lattice $\PL(B)^*$.
\end{proposition}

\begin{proof}
Write $S=A[Q]/I_0$ for brevity. Since the irreducible components of $X_0$ are
toric varieties over $S$ labelled by $\P_\max$, the connected component of the
identity of $\Aut_{S}(X_0)$ is a closed subgroup of a product of tori $\GG_m^n$,
one for each maximal cell of $\P$. Intrinsically, the torus for the component
labelled by $\sigma\in\P_\max$ is $\bT_\sigma:=\Spec \big( S [\Lambda_\sigma]
\big)$. For any facet $\rho\subseteq\sigma$ the inclusion
$\Lambda_\rho\subseteq\Lambda_\sigma$ defines an epimorphism $\bT_\sigma\to
\bT_\rho$, with $\bT_\rho= \Spec \big( S [\Lambda_\rho] \big)$ the torus for the
codimension one stratum labelled by $\rho$. Compatibility of the actions for
$\sigma,\sigma'\in\P_\max$ adjacent to $\rho\in\P^{[n-1]}$ restricts the
automorphism group of the union of the corresponding components to the fibre
product
\[
\bT_\sigma\times_{\bT_\rho} \bT_{\sigma'}
=\Spec\big(S[\Lambda_\sigma\oplus_{\Lambda_\rho}
\Lambda_{\sigma'} ] \big).
\]
By trivially writing $\Lambda_\sigma= \PL(\sigma,\ZZ)^*$, the fibred
sum can conveniently be interpreted as the dual of the group of
$\ZZ$-valued piecewise linear functions on $\sigma\cup\sigma'$.

Globally we need to take the limit of the category with morphisms
$T_\sigma\to T_\rho$ for $\rho\in\P^{[n-1]}$, $\sigma\in\P_\max$,
$\rho\subseteq \sigma$. In fact, since $X_0$ is $S_2$ it suffices to
check compatibility of the actions on the irreducible components in
codimension one. Dually this leads to the colimit of
$\Lambda_\rho\to\Lambda_\sigma$, which is $\PL(B)^*$.
\end{proof}
\medskip

Let us now discuss the general procedure for obtaining a torus
action. For the character lattice of the acting torus take a finitely
generated free abelian group $\Gamma$. For the various
$\Gamma$-graded rings, $\deg_\Gamma$ always denotes the degree, as a
map from the set of homogeneous elements to $\Gamma$. Assume we have
a $\Gamma$-grading on our ground ring $A[Q]$ induced by gradings on
$A$ and $Q$ and that $I\subseteq A[Q]$ is a homogeneous ideal.  Denote
by $\delta_Q: Q\to\Gamma$ the homomorphism defining the grading on
the monomials of $A[Q]$ and by $A^Q_0\subseteq A[Q]$ the degree zero
subring. Our torus is
\[
\TT:= \Spec\big(A^Q_0[\Gamma]\big).
\]
Then $\TT$ acts naturally on $\Spec A[Q]$, defined on the ring level by
\[
A[Q]\lra A^Q_0[\Gamma]\otimes_{A^Q_0} A[Q],\quad
a\longmapsto z^{\deg_\Gamma(a)}\otimes a,
\]
the map written on a homogeneous element $a$.
The fixed locus of the action is $\Spec A^Q_0$.

For lifting this action to $\foX$ recall from Proposition~\ref{Prop:
Universal MPA function} the universal monoid $Q_0=\MPA(B,\NN)^\vee$
and the homomorphism $h:Q_0\to Q$ defining the given $Q$-valued
MPA-function $\varphi$. By the explicit description in 
Proposition~\ref{Prop: MPA is defined by kinks}, one has $Q_0^\gp
=\MPA(B,\ZZ)^*$. Hence the dual of the map $\PL(B)\to \MPA(B,\ZZ)$
defines a homomorphism
\[
g:Q_0\hooklongrightarrow Q_0^\gp=\MPA(B,\ZZ)^*\lra \PL(B)^*.
\]
We now assume given a further homomorphism $\delta_B: \PL(B)^*\to
\Gamma$ fitting into the following commutative diagram.
\begin{equation}
\label{Eqn: delta_b and delta_Q}
\begin{CD}
Q_0 @>{g}>> \PL(B)^*\\
@V{h}VV @V{\delta_B}VV\\
Q@>{\delta_Q}>>\Gamma.
\end{CD}
\end{equation}
This data provides a grading of our monomials as follows. Recall
that a monomial $m$ is an integral tangent vector on $\BB_\varphi$
at a point $x$ of $\varphi(B_0)\subseteq \BB_\varphi$ and that
$\pi:\BB_\varphi\to B$ denoted the projection. Assuming that $\ol
m=\pi_*(m)$ points from $x$ into the tangent wedge of a cell $\tau$
at $x$, the directional derivative in the direction of $\ol m$
defines an element $\nabla_{\ol m}\in \PL(B)^*$. Said differently,
for $\psi\in \PL(B)$ the restriction to $\tau$ defines an element of
$\check\Lambda_\tau$, and we define
\begin{equation}
\label{Eqn: nabla(PL)}
\nabla_{\ol m}(\psi):= (\psi|_\tau)(\ol m).
\end{equation}
A monomial $m$ also yields an element $m_Q\in Q$, by subtracting the
lift of $\ol m$ to $\BB_\varphi$ via the piecewise affine section
$\varphi: B\to\BB_\varphi$ of $\pi$. Note that in
the canonical identification \eqref{shP on Int(sigma)} in the
interior of a maximal cell we have $m=(\ol m, m_Q)$.
The \emph{$\Gamma$-degree} of $m$ or of $z^m$ is now defined as
\[
\deg_\Gamma(m):= \delta_Q(m_Q)+\delta_B(\nabla_{\ol m}).
\]
We have thus made the rings $R_\sigma$ from \eqref{Eq: R_sigma can
iso} into $\Gamma$-graded rings. The basic result of this subsection
is that the whole construction is $\Gamma$-graded provided all
functions $f_\fop$ given by the walls are homogeneous of degree
zero.

\begin{definition}
\label{Def: Homog wall structure}
Assume $A[Q]/I$ and the monomials on $B_0$ are graded by a finitely
generated free abelian group $\Gamma$ via a homomorphism $\delta_B:
\PL(B)^*\to\Gamma$ making \eqref{Eqn: delta_b and delta_Q}
commutative, as just described. Let $\scrS$ be a wall structure on
$(B,\P)$. We say that $\scrS$ is a \emph{homogeneous wall structure}
if  all functions $f_\fop$ defining walls are homogeneous of degree
$0$. 
\end{definition}

\begin{theorem}
\label{Thm: Torus action}
Let $\scrS$ be a consistent homogeneous wall structure on $(B,\P)$. Then the
action of the algebraic torus $\TT= \Spec\big(A^Q_0[\Gamma]\big)$ on $X_0$ from
Proposition~\ref{Prop: Aut(X_0)} and the homomorphism
$\delta_B:\PL(B)^*\to\Gamma$ extends to an equivariant action on the flat family
$\foX\to \Spec (A[Q]/I)$ from Theorem~\ref{Thm: Main}. Moreover, for an
asymptotic monomial $\ol m$ of $B$, the degree zero theta function
$\vartheta_{\ol m}$ is homogeneous of degree $\delta_B(\nabla_{\ol
m})\in\Gamma$.
\end{theorem}

\begin{proof}
As all fibres of $\foX\to \Spec (A[Q]/I)$ satisfy Serre's condition $S_2$ 
by Proposition~\ref{prop: S2 condition}, Lemma~\ref{Lem: Extension lemma}
implies it is enough to
prove the statement after restricting to the complement
$\foX^\circ\subseteq \foX$ of codimension two strata. Recall that
$\foX^\circ$ is covered by rings of the form $\Spec R_\fou$, $\Spec
R_\fou^\partial$ and $\Spec R_\fob$ for chambers $\fou$ and slabs
$\fob$ for $\scrS$, with the gluing coming from canonical embeddings
and automorphisms governed by wall crossing. For a chamber $\fou$
contained in a maximal $\sigma$ we have already seen that
$R_\fou=R_\sigma$ is naturally $\Gamma$-graded by the grading of
monomials.

For the $\Gamma$-grading of the rings $R_\fob$ \eqref{Def:
R_fob} we have to check that $Z_+Z_- -f_\fob
z^{\kappa_{\ul\rho}}$ is homogeneous. Again, since $f_\fob$ is
homogeneous of degree zero this statement is equivalent to
\[
\deg_\Gamma(Z_+)+ \deg_\Gamma(Z_-)=
\deg_\Gamma\big( z^{\kappa_{\ul\rho}}\big).
\]
To prove this equality recall that if $m_+$, $m_-$ are the tangent
vectors on $\BB_\varphi$ with $Z_+= z^{m_+}$, $Z_-= z^{m_-}$, then
$\deg_{\Gamma}(m_{\pm})
=\delta_B(\nabla_{\ol m_{\pm}})$, while $\deg_{\Gamma}(z^{\kappa_{\ul\rho}})
=\delta_Q(\kappa_{\ul\rho})$. But $\nabla_{\ol m_{+}}+
\nabla_{\ol m_-}$ is the linear functional on $\PL(B)$ given by
$\psi\mapsto \kappa_{\ul\rho}(\psi)$. This is precisely the
image of $e_{\ul\rho}\in Q_0$ under the map $Q_0\rightarrow \PL(B)^*$ 
(see Proposition \ref{Prop: Universal MPA function}). But
$\kappa_{\ul\rho}\in Q$ is also the image of $e_{\ul\rho}$. Thus
the claimed equality now follows from
the commutativity of \eqref{Eqn: delta_b and delta_Q}.

The same degree computation shows homogeneity of the localization
maps $R_\fob\to R_\fou$ of \eqref{Eqn: Localization for slabs} for
a slab $\fob$ and an adjacent chamber $\fou$. If $\fou$ is
a boundary chamber the localization $R_\fou^\partial\to R_\fou$ of
\eqref{Eqn: Localization for slabs} is localization at a
homogeneous monomial, so respects the grading. The automorphisms
$\theta_\fop: R_\fou\to R_{\fou'}$ of \eqref{Eqn: theta_fop}
associated to crossing a wall $\fop$ separating chambers
$\fou,\fou'$ manifestly respects the grading since by hypothesis
$f_\fop$ is homogeneous of degree zero. Finally, for slabs
$\fob,\fob'$ separated by a joint, the isomorphism $R_\fob\to
R_{\fob'}$ from \eqref{Eqn: theta_foj} is induced by a composition
of wall crossing homomorphisms and localizations, hence also
respects the grading.
\end{proof}

\begin{remark}
\label{Rem: grading on cone(N)}
For lifting the statement of Theorem~\ref{Thm: Torus action} to the total space
$\foL\to\foX$ of the line bundle in Theorem~\ref{Thm: Main}, note that there is
a bijection between the group $\PA(B,\ZZ)$ of piecewise affine functions on $B$
and $\PL(\cone{B},\ZZ)$. Moreover, since there is a bijection between the
interior codimension one cells of $B$ and the interior codimension one cells of
$\cone{B}$, we can identify the universal monoids $Q_0$ of $B$ and of
$\cone{B}$. Thus taking for $\delta_{\cone B}$ a homomorphism
$\PA(B,\ZZ)^*\to\Gamma$ both in \eqref{Eqn: delta_b and delta_Q} and in the
statement of Theorem~\ref{Thm: Torus action}, the action of $\TT$ lifts to
$\foL$. Note that the condition on homogeneity of the wall functions $f_\fop$
can nevertheless be checked on $B$ by the definition of lifted wall structures
(Definition~\ref{Def: cone of wall structure}). A theta function
$\vartheta_m$ with $m\in B\big(\frac{1}{d}\ZZ\big)$ is homogeneous of degree
$\delta_{\cone B}(\tilde m)\in \Gamma$, with $\tilde m$ the asymptotic monomial
on $\cone B$ corresponding to $m$ via Proposition~\ref{Prop: Asymptotic monomials
cone(B)}.
\end{remark}

\begin{remark}
\label{Rem: torus action with gluing data}
In \S\ref{Subsect: Twisting} the construction of $\foX$ will be
modified by the introduction of gluing data. They are given by
homomorphisms $s_{\sigma\ul\rho}:\Lambda_\sigma\to A^*$. For
projective open gluing data (Definition~\ref{Def: Projective gluing
data}) the analogue of Theorem~\ref{Thm: Main} holds. In this
modified setup Theorem~\ref{Thm: Torus action} holds true provided
the $s_{\sigma\ul\rho}$ take values in degree zero, for then the
localization homomorphisms $\chi_{\fob,\fou}$ remain homogeneous.
The rest of the construction is untouched.

Without a projectivity assumption one only obtains $\foX^\circ\to
\Spec(A[Q]/I)$ and, again assuming the gluing data to be homogeneous
of degree zero, a torus action on $\foX^\circ$. 
\end{remark}

\begin{example}
In \cite{GHKS} (see Examples \ref{Expl: standard examples},2 and
\ref{Expl: MPA-functions},2)  one takes $A=\kk$ and for $Q$
(denoted $P$ in loc.cit.) a toric submonoid of $N_1(\shY/T)=Q^\gp$
containing $\NE(\shY/T)$, the group of effective $1$-cycles of the
mirror family $\shY\to T$ of K3 surfaces. The character lattice
$\Gamma$ of $\TT$ is $\ZZ^{B(\ZZ)}$, the free abelian group
generated by the integral points of $B$. In this example every
integral point $v$ of $B$ is a vertex, hence labels an irreducible
component $Y_v\subseteq \shY_0$. The $\Gamma$-grading on $\kk[Q]$ is
defined by intersection theory on $\shY$:
\[
\delta_Q: N_1 (\shY/T)\lra \Gamma,\quad
C\longmapsto \sum_{v\in B(\ZZ)} (C\cdot Y_v)\cdot e_v,
\]
$e_v\in\Gamma$ the canonical basis vector labelled by $v$.

As for $\delta_{\cone{B}}: \PA(B,\ZZ)^*\to\Gamma$ (Remark~\ref{Rem:
grading on cone(N)}) note that in this case all maximal cells are
standard simplices. Hence $\PA(B,\ZZ)= \ZZ^{B(\ZZ)}$. We take
$\delta_B=\id_{\ZZ^{B(\ZZ)}}$.

To check commutativity of \eqref{Eqn: delta_b and delta_Q} it
suffices to trace the generators of $Q_0$,
\[
\kappa_\rho: \MPA(B,\NN)\lra \NN,\quad
\psi\longmapsto \kappa_\rho(\psi),
\]
measuring the kink of an MPA function along an edge $\rho\in\P$,
through the diagram. The image of $\kappa_\rho$ in $\PA(B,\ZZ)^*=
\PL(\cone{B},\ZZ)^*$ measures the kink of a piecewise affine function
$\psi$ along $\rho$, still denoted $\kappa_\rho(\psi)$. Let
$v_0,v_1$ be the vertices of $\rho$ and let $v_2,v_3$ be the
remaining vertices of the two triangles containing $\rho$. Denote by
$D^2_{vw}$ the self-intersection number of the double curve $Y_v\cap
Y_w$ inside $Y_w$. A straightforward computation in the affine chart
\eqref{Eq: GHKS chart} shows
\[
\kappa_{\rho}(\psi)= \psi(v_2)+\psi(v_3)-(D^2_{v_1v_0}+2)\psi(v_0)
+D^2_{v_1v_0}\psi(v_1).
\]
Noting that $D^2_{v_0 v_1}+D^2_{v_1v_0}=-2$ we see that
$\kappa_\rho\in Q_0$ maps to the symmetric expression
\begin{equation}
\label{Eq: commuting in GHKS}
e_{v_2}+e_{v_3}-(D^2_{v_1v_0}+2) e_{v_0} -(D^2_{v_0 v_1}+2) e_{v_1}
\in \PA(B,\ZZ)^*.
\end{equation}
On the other hand, going via $Q\subseteq A_1(\shY/T)$ maps
$\kappa_\rho$ first to $C=Y_{v_0}\cap Y_{v_1}\subseteq\shY$ and then
on to $\sum_v (C\cdot Y_v)\cdot e_v$. Now $C$ intersects $Y_{v_2}$
and $Y_{v_3}$ transversely, while
\[
C\cdot Y_{v_0}= \deg_C \O_{\shY}(Y_{v_0}) = -\deg_C \O_{\shY} (
Y_{v_1}+Y_{v_2}+Y_{v_3}) = -2- \deg_C \O_{Y_{v_0}}(C)
= - 2 - D_{v_1v_0}^2,
\]
and similarly for $C\cdot Y_{v_1}$. Here we used that
$\O_\shY(\sum_v Y_v)= \O_\shY$. Since $C$ is disjoint from all other
$Y_v$ we obtain the same expression as in \eqref{Eq: commuting in
GHKS}.

In \cite{GHKS}, homogeneity of the wall functions for the walls
emanating from vertices will follow by an a priori argument. The
remaining walls will be seen to be homogeneous because the scattering procedure
via the Kontsevich-Soibelman lemma manifestly respects the grading.
\end{example}

The somewhat complementary case of GS-type singularities will be
treated in \S\ref{Subsect: GS torus action}.

%===========================================================
\subsection{Jagged paths}
\label{Subsect: Jagged paths}

An alternative point of view on the construction of our theta
functions from \S\ref{Subsect: Theta functions} works directly on
$B$ rather than on $\cone{B}$. Recall that
$\cone{B}=(B\times\RR_{\ge0})/(B\times\{0\})$ topologically. The
projection to the second factor induces an affine map
\[
h:\cone B\longrightarrow \RR_{\ge 0},
\]
the height functions, while the projection to the first factor (the
radial directions) defines a non-affine retraction
\[
\kappa:\cone{B}\setminus\{O\}= h^{-1}(\RR_{> 0})\longrightarrow B.
\]
The projection of broken lines via $\kappa$ leads to the notion of
\emph{jagged paths}. The image of a broken line still consists of a
union of straight line segments in $B$, but the slopes need not be
rational since the projection $\kappa$ is not affine linear. The
notion of jagged paths predates the notion of broken lines. It had
been discussed early in 2007 in a project on tropical Morse theory
of the first and third authors of this paper jointly with Mohammed
Abouzaid. Tropical Morse theory is a tropical version of Floer
theory for Lagrangian sections of the SYZ fibration, see \cite{DBr},
\S8.4 and \cite{thetasurvey}.

We begin by reexamining the affine geometry of $\cone{B}$ from
\S\ref{Subsect: conical affine geom} and \S\ref{Subsect: cone of
polyhedral}. Denote by
\[
j: B\lra \cone{B} 
\]
the identification of $B$ with $B\times\{1\}\subseteq \cone{B}$.

First we want to interpret the tangent vectors on $\cone{B_0}$
purely in terms of the affine geometry of $B_0$.

\begin{lemma}
\label{Lem: shAff(B_0,ZZ)^*}
There is a canonical isomorphism $\shAff(B_0,\ZZ)^*\simeq
j^*(\Lambda_{\cone{B_0}})$.
\end{lemma}

\begin{proof}
It suffices to establish this isomorphism for an $n$-dimensional
lattice polyhedron $\sigma\subseteq\RR^n$. For $x\in \sigma$ an
integral point there is a canonical identification
\[
\shAff(\sigma,\ZZ)_x = \Aff(\Lambda_x,\ZZ)
\stackrel{\simeq}{\lra} \check\Lambda_x\oplus\ZZ,
\]
mapping $0\oplus \ZZ$ to the constant functions and
$\check\Lambda_x\oplus 0$ to the affine functions vanishing at $x$.
Dualizing gives
\begin{equation}
\label{shAff_p^*}
\shAff(\sigma,\ZZ)_x^*= \Lambda_x\oplus\ZZ.
\end{equation}
The latter is canonically isomorphic to
$\Lambda_{\cone{\sigma},(x,1)}$ by mapping $(0,1)\in
\Lambda_x\oplus\ZZ$ to $\partial_r$, the tangent vector in the
radial direction, which is integral at the integral point $x$, while
$\Lambda_x$ is canonically embedded into
$\Lambda_{\cone{\sigma},(x,1)}$ via $j_*$.

Changing coordinates clearly respects this identification of
$\Lambda_x$ with those homomorphisms $\shAff(\sigma,\ZZ)_x\to \ZZ$ that
vanish on constant functions, that is, which factor over
$\check\Lambda_x$. To generate $\shAff(\sigma,\ZZ)^*_x$ it suffices to
take in addition the evaluation homomorphism $\operatorname{ev}_p:
\shAff(\sigma,\ZZ)_x\to\ZZ$ at $x$. Under the isomorphism
\eqref{shAff_p^*} this element of $\shAff(\sigma,\ZZ)_x^*$ corresponds to
$(0,1)$, hence it maps to the primitive radial tangent vector
$\partial_r\in \Lambda_{\cone{\sigma}, (x,1)}$. Parallel transport to a
nearby integral point $y=x+v \in \sigma$, $v\in\Lambda_x$, takes
$(\alpha,c)\in \check\Lambda_{y}\oplus\ZZ= \Aff(\Lambda_{y},\ZZ)$ to
$(\alpha,c- \langle \alpha,v\rangle)\in \Aff(\Lambda_x,\ZZ)$. Thus
$\operatorname{ev}_x= (-v,1)\in \Lambda_{y}\oplus\ZZ$.
This result agrees with the parallel transport of $\partial_r$ at
$(x,1)$ to $(y,1)$ in $\cone{\sigma}$ (Proposition~\ref{Prop: parallel
transport cone}).
\end{proof}
\medskip

Lemma~\ref{Lem: shAff(B_0,ZZ)^*} demonstrates that
the tangent sequence for $B_0$ in $\cone{B_0}$
\begin{equation}
\label{Eq: tangent sequence on CB}
0\lra \Lambda \lra j^*\Lambda_{\cone{B_0}}\stackrel{h_*}{\lra}
\ul\ZZ\lra 0
\end{equation}
agrees with the dual of \eqref{Eqn: Aff},
\begin{equation}
\label{Eqn: dual of Eqn Aff}
0\lra \Lambda \lra \shAff(B_0,\ZZ)^*\stackrel{\deg}{\lra}\ul\ZZ\lra
0.
\end{equation}
The homomorphism $\deg$ can be characterized by the property that
$\deg(\tilde m)$ for $\tilde m\in \shAff(\sigma,\ZZ)^*_x$ is the
integer $d$ with $\tilde m- d\cdot\operatorname{ev}_x\in \Lambda_x$,
again for $\sigma\in \P$ a maximal cell. Said differently, $\deg(\tilde
m)=\tilde m(1)$, the value of $\tilde m$ at the constant affine
function $1$. Note that the sequence also shows that
$\shAff_d(B_0,\ZZ):= \deg^{-1}(d)$ is a $\Lambda$-torsor.

Given a $Q$-valued MPA-function $\varphi$ the correspondence can be
applied to the cone over $\BB_\varphi$
(Construction~\ref{Construction: B_varphi}). Note that
$\cone\BB_\varphi= \BB_{\cone{\varphi}}$. In turn, we have a
definition of monomials at a point $(x,h)\in\cone{B_0}$ as elements
in $\shAff(\BB_\varphi,\ZZ)_x^*$. Here we use parallel translation
in the radial direction to reduce to the case $h=1$ treated in
Lemma~\ref{Lem: shAff(B_0,ZZ)^*} and the discussion following it. In
other words, $\shAff(\BB_\varphi,\ZZ)^*= j^*\shP_{\cone{B_0}}$ for
$\shP_{\cone{B_0}}$ the sheaf on $\cone{B_0}$ according to
Definition~\ref{Def: shP}. Denote the pullback via the secton
$\varphi: B_0\to\BB_\varphi$ of either of these sheaves by
$\tilde\shP$, and the corresponding subsheaf of monomials by
$\tilde\shP^+$ (Definition~\ref{Def: Toric strata}). The
homomorphism $\deg$ agrees with the grading of the monomials on
$\cone{B_0}$ defined before Proposition~\ref{Prop: Asymptotic
monomials cone(B)}. The following diagram encapsulates the above
discussion:
\[
\xymatrix@C=30pt
{&0\ar[d]&0\ar[d]&&\\
&\ul{Q}^{\gp}\ar[d]\ar[r]^=&\ul{Q}^{\gp}\ar[d]&&\\
0\ar[r]&\shP\ar[d]^{\pi_*}\ar[r]&\tilde\shP\ar[d]^{\pi_*}
\ar[r]^{\deg}&\ul{\ZZ}\ar[d]^{=}\ar[r]&0\\
0\ar[r]&\Lambda\ar[r]\ar[d]&\shAff(B,\ZZ)^*\ar[r]\ar[d]&\ul{\ZZ}\ar[r]
&0\\
&0&0&&
}
\]
We now have a generalization of the notion of
monomials on $B$ to higher degree.

\begin{definition}
\label{Def: deg=d monomials}
Denote by $\tilde\shP_d:=\deg^{-1}(d)\subseteq \tilde\shP$. A
\emph{monomial of degree $d$} at $x\in B_0$ is a formal expression
$a z^m$ with $a\in A$ and $m\in (\tilde\shP_d)_x$. If $m$
is a monomial at $x\in B_0$ we still denote by $\ol m\in
\shAff(B_0,\ZZ)^*= \Lambda_{\cone{B_0},(x,1)}$ the image induced by
the affine projection $\BB_{\cone{\varphi}}\to \cone{B_0}$. 
\end{definition}

Monomials of degree zero are the monomials from Definition~\ref{Def:
monomials} that we worked with so far. To construct a section of
$\shL^d$ one restricts to monomials of degree $d$. Note also that
the notion of transport of monomials (Definition~\ref{Def: transport
of monomials}) readily generalizes to monomials of higher
degree, simply by working on $\cone{B_0}$ locally.

To translate the condition of constant velocity on the domains of
affine linearity of a broken line (Definition~\ref{Def: broken
lines},(1)) on $\cone{B_0}$ to $B_0$ we need to compose the map
$m\mapsto \ol m$ with the differential $\kappa_*:
\shT_{\cone{B_0},(x,1)} \to \shT_{B_0,x}$. Here $\shT_{B_0}$ is the
sheaf of \emph{differentiable} vector fields on $B_0$, and similarly
on $\cone{B_0}$. The resulting homomorphism is denoted
\[
\vect: \tilde\shP \lra \shT_{B_0},\quad
m\longmapsto \kappa_*(\ol m).
\]
The image of a local section of $\tilde\shP$ under $\vect$ provides
a flat section of $\shT_{B_0}$ with respect to the \emph{affine}
connection on $T_{B_0}$, see the discussion before 
Proposition~\ref{Prop: parallel transport cone}.

An alternative description of $\vect$ is by noting that it factors
via the map $m\to\ol m$ with target in $\shAff(B_0,\ZZ)^* \subseteq
\shAff(B_0,\RR)^* $ and
\[
\operatorname{vect}: \shAff(B_0,\RR)^*\lra \shT_{B_0}.
\]
This latter map sends a linear functional on $\shAff(B,\RR)_x$
for $x\in B_0$ to its restriction to the subspace of germs of affine 
linear functions vanishing at $x$. In fact, the dual of the space of germs of
affine linear functions vanishing at $x$ is the tangent space to $B$ at $x$.

\begin{definition}
\label{Def: Jagged paths}
A (normalized) \emph{jagged path} of \emph{degree $d$} for a wall
structure $\scrS$ on $(B,\P)$ is a proper continuous map
\[
\gamma: [0,t_r]\to B
\]
with $\gamma\big((0,t_r)\big)\subseteq B_0$ and disjoint from any
joints of $\scrS$, along with a sequence $0=t_0< t_1<\ldots< t_r$
for some $r\ge 1$ with $\gamma(t_i)\in |\scrS|$ for
$i=1,\ldots,r-1$, and for $i=1,\ldots,r$ monomials $a_i z^{m_i}$ of
degree $d$ defined at all points of $\gamma([t_{i-1},t_i])$, subject
to the following conditions.
\begin{enumerate}
\item
$\gamma|_{(t_{i-1},t_i)}$ is a map with image
disjoint from $|\scrS|$, hence contained in the interior of a
unique chamber $\fou_i$ of $\scrS$, and $\gamma'(t)=-\vect(m_i)$
for all $t\in (t_{i-1},t_i)$.
\item
For each $i=1,\ldots,r-1$ the monomial $a_{i+1} z^{m_{i+1}}$ is a
result of transport of $a_i z^{m_i}$ from $\fou_i$ to $\fou_{i+1}$.
\item $a_1=1$, $m_1= d\cdot\varphi_*
(\ev_{\gamma(0)})$, $\gamma(0)\in
B\big(\frac{1}{d}\ZZ\big)$.
\end{enumerate}
The \emph{type} of $\gamma$ is the tuple of all $\fou_i$ and $m_i$.
As for broken lines we suppress the data $t_i,a_i, m_i$ when
talking about jagged paths, but introduce the notation
\[
a_\gamma:= a_r,\quad m_\gamma:=m_r.
\]
\end{definition}

In (3) the push-forward $\varphi_*$ is understood by first
restricting $\varphi$ to a maximal cell $\sigma\in\tilde\P$
containing $\gamma\big( (0,t_1)\big)$ to obtain an affine map
$\Int\sigma\to \BB_{\varphi}$.

Comparing to the notion of broken line the one point to emphasize is
that while a broken line has an asymptotic vector (Remark~\ref{Rem:
broken lines},1), a jagged path has an initial point $\gamma(0)$.

\begin{proposition}
Let $\scrS$ be a wall structure on the polyhedral pseudomanifold $(B,\P)$.
Then the projection $\kappa: \cone{B}\to B$ induces a bijection
between the set of broken lines on $\cone{B}$ for $\cone{\scrS}$
with endpoint $p$ and the set of jagged paths on $B$ for $\scrS$
with endpoint $\kappa(p)$. If $\beta$ is a broken line on $\cone{B}$
with asymptotic monomial $\ol m$ of degree $d$, the initial point of
the associated jagged path is the point $x\in B(\frac{1}{d}\ZZ)$
corresponding to $\ol m$ according to Proposition~\ref{Prop:
Asymptotic monomials cone(B)}.
\end{proposition}

\begin{proof}
This follows directly from the definitions.
\end{proof}
\medskip

Having related the notion of broken line on $\cone{B}$ to the
notion of jagged path on $B$ it is now immediate to express all
results in \S\ref{Subsect: Theta functions} in terms of
jagged paths.

%===========================================================
%===========================================================

\section{Additional parameters}
\label{Sect: Parameters}

So far $X_0$ is the pull-back of a scheme over $\Spec \ZZ$ to $\Spec\big
(A[Q]/I_0 \big)$. Moreover, by the definition of the rings $R_{\ul\rho}$ in
\eqref{Eqn: R_ul rho} the closed subscheme of $\Spec\big( A[Q]/I\big)$ defined
by $I_0$ describes a trivial deformation. This is enough for certain cases, for
example to describe projective deformations of certain degenerate K3 surfaces
with all irreducible components copies of $\PP^2$ \cite{GHKS}, but in general it
is important to include also non-trivial locally trivial\footnote{Recall that a
deformation is called \emph{locally trivial} if the total space has an \'etale
covering by open subsets of trivial families.} deformations. For example, in
\cite{logmirror2}, \S5.2 we describe a locally trivial family $X_0$ parametrized
by the algebraic torus $\Spec\big( \kk[\Gamma]\big)$ with $\Gamma$ the quotient
by the torsion subgroups of the abelian group $H^1(B, i_*\Lambda)$,
$i:B_0\hookrightarrow B$. This family comes with a log structure and is versal
as a family of log schemes keeping the singularity structure, and it usually is
non-trivial as a family of schemes. Assuming projectivity, \cite{affinecomplex}
yields a deformation $\foX$ of $X_0$ much of the same form as the construction
presented here, but involving parameters in the localization morphisms. To keep
the presentation simple we chose not to include these in the discussion up to
this point. The purpose of this section is to finally include these
additional parameters.

Another motivation comes from the case of abelian varieties discussed
in Example~\ref{Expl: Riemann theta functions}. To reproduce all Riemannian
theta functions from our theta functions requires the use of gluing data.

%===========================================================
\subsection{Twisting the construction}

We begin by a general consideration on including additional
parameters abstractly. In this general framework $A[Q]$ also
includes these parameters. The reader is advised to think of $\Spec
A$ as the space of such gluing parameters, although this may not
be strictly true in practice.

In the new setup the definition of the sheaf of rings $A[\shP]$, the
notion of wall structure and the rings $R_{\fou}$ and $R_\fob$ are
as before. The only data that has to be changed in our construction
is the localization morphism from the ring for a slab $\fob$ to an
adjacent chamber $\fou$, which previously was defined canonically in
terms of the affine geometry of $B_0$. For each such pair
$(\fob,\fou)$ we now have as additional datum a homomorphism of
$A[Q]/I$-algebras
\[
\chi_{\fob,\fou}: R_\fob\lra R_{\fou}.
\]
At this level of generality there are no restrictions on
$\chi_{\fob,\fou}$. This new definition of the transition between
$R_\fob$ and $R_\fou$ changes also the notion of consistency in
codimension one (Definition~\ref{Def: Consistency in codim one}) and
the definition of the isomorphism $\theta_\foj$ between rings
$R_\fob$ associated to crossing a codimension one joint \eqref{Eqn:
theta_foj}. Under the assumption of consistency of the wall
structure $\scrS$ in codimension zero and one in this modified sense,
Proposition~\ref{Prop: foX^o exists} on the construction and
properties of $\foX^\circ$ hold true. Moreover, there is still
a notion of consistency in codimension two which ensures the
existence of enough local functions. One can then proceed to
construct the canonical basis $\vartheta_m$ of the ring of global
functions of $\foX^\circ$ via broken lines as in Section~\ref{Sect:
Global functions}. Note however, that now there possibly is an
additional dependence of the initial coefficient $a_1$ of a broken
line on the initial maximal cell.

The construction of the partial compactification $\foX$ of $\foX^\circ$
in Section~\ref{Sect: Theta functions} depended on the fact that we
can lift the construction to the cone $\cone B$ over $B$. In the
present situation this means we need a lift
\[
\tilde\chi_{\fob,\fou}= \chi_{\cone\fob,\cone\fou}:
R_{\cone\fob}\lra R_{\cone\fou}
\]
of $\chi_{\fob,\fou}$. Unlike in the untwisted situation this does
not follow canonically and is an additional datum to be
provided along with $\scrS$.

To go any further we need to make closer contact with the affine
geometry. This is the content of the next subsection.

%===========================================================
\subsection{Twisting by gluing data}
\label{Subsect: Twisting}

We now restrict to the following class of transition maps
$\chi_{\fob,\fou}$ that covers all cases which have occured in practice
so far.  Denote by $\chi^\can_{\fob,\fou}$ the canonical
localization homomorphisms of \eqref{Eqn: Localization for slabs}.
\smallskip

\ul{Open gluing data.}
For any $\ul\rho\in\tilde\P^{[n-1]}_\inte$ and adjacent maximal cell
$\sigma$ choose a homomorphism of abelian groups
\[
s_{\sigma\ul\rho}:\Lambda_\sigma\to A^\times,
\]
subject to the constraint
\begin{equation}
\label{Eqn: X_0 consistency after twisting}
s_{\sigma\ul\rho}|_{\Lambda_\rho} \cdot
\big(s_{\sigma\ul\rho'}|_{\Lambda_\rho}\big)^{-1}
=s_{\sigma'\ul\rho}|_{\Lambda_\rho} \cdot
\big(s_{\sigma'\ul\rho'}|_{\Lambda_\rho}\big)^{-1},
\end{equation}
as homomorphisms $\Lambda_{\rho}\rightarrow A^\times$
%\begin{equation}
%\label{openconstraint}
%s_{\sigma\ul{\rho}}|_{\Lambda_{\rho}}=s_{\sigma\ul{\rho}'}|_{\Lambda_{\rho}}
%\end{equation}
whenever $\ul{\rho},\ul{\rho}'$ are contained in the same
codimension one cell $\rho\in\P$ with adjacent maximal cells
$\sigma$ and $\sigma'$. Following
\cite{logmirror1}\cite{affinecomplex} we call the collection
$\mathbf s=(s_{\sigma\ul\rho})$ \emph{(open) gluing data}. ``Open''
refers to the fact that these gluing data modify open embeddings,
while ``closed gluing data'' to be considered below can be
interpreted as changing the closed embeddings defined by the
inclusion of toric strata on $X_0$. Condition  \eqref{Eqn: X_0
consistency after twisting}  is a necessary and sufficient condition
to guarantee that the analogue  of $X_0^{\circ}$ exists.
\smallskip

\ul{Changing $\chi_{\fob,\fou}$.} 
Define the localization homomorphism $\chi_{\fob,\fou}$ modified by
open gluing data by composing $\chi_{\fob,\fou}^{\can}:R_{\fob}\rightarrow
R_{\fou}$ with the map\footnote{Our sign convention for gluing data
in this paper is opposite to the conventions in the previous work
of the first and last authors \cite{logmirror1},
\cite{affinecomplex}. The signs in these works were initially chosen
from the point of view of gluing toric strata, but it now is clear
that the point of view of gluing open sets is more important.}
\begin{equation}
\label{Eqn: Modified localization hom}
\begin{array}{rcl}
s_{\sigma\ul\rho}: R_\fou = (A[Q]/I)[\Lambda_\sigma]
&\lra& R_\fou\\
z^m&\longmapsto&
s_{\sigma\ul\rho}(m)z^m.
\end{array}
\end{equation}
Since now consistency of a wall structure $\scrS$ in
codimension one and two depends on the choice of gluing data we
speak of \emph{consistency for the gluing data $\mathbf s$}.

\begin{remark}
With trivial gluing data, \eqref{f_rho' versus f_rho} assured that
for any $\rho\in \P^{[n-1]}$ all the rings $R_{\ul\rho}$ for
$\ul\rho\subseteq \rho$ are canonically isomorphic. While this
statement is superseded by consistency in codimension one
(Remark~\ref{Rem: wall structures},5) and hence ultimately is
redundant, the local models determine a log structure on $X_0^\circ$
that sometimes is important information. In fact, by
\cite{logmirror1}, Theorem~3.27, the log structure on $X_0^\circ$ is
equivalent to giving functions $f_{\ul\rho} \in
(A[Q]/I_0)[\Lambda_\rho]$ fulfilling an equation of the form
\eqref{f_rho' versus f_rho}. The generalization to non-trivial
gluing data can be derived from consistency in codimension one
and no walls of codimension zero present. To this end consider
$\ul\rho$, $\ul\rho'\subseteq\rho$ with slab functions $f_{\ul\rho}$,
$f_{\ul\rho'}$. Then we have the two models $\Spec R_{\ul\rho}$ and
$\Spec R_{\ul\rho'}$ from \eqref{Eqn: R_ul rho} for the affine
neighbourhood in $\foX^\circ$ of the $(n-1)$-stratum
$\Spec(A[Q]/I_0)[\Lambda_\rho]$ of $X_0^\circ$. Requiring these models to
be compatible with respect to the localization morphisms twisted by
gluing data leads to the following conditions. 

First, consistency of gluing of monomials with exponents in
$\Lambda_\rho$ is equivalent to \eqref{Eqn: X_0 consistency after
twisting}. Second, for consistency of gluings of $Z_\pm$, let
$\xi=\xi(\rho)\in \Lambda_\sigma$ and denote by $Z_\pm$ and $Z'_\pm$
the associated generators of the rings $R_{\ul\rho}$ and
$R_{\ul\rho'}$, respectively. Write also $\xi'\in\Lambda_{\sigma'}$
for $-\xi$ via parallel transport through  $\ul{\rho}$ into
$\sigma'$, so that \eqref{xiprimexieq} holds along $\ul{\rho}'$.
Then going from $\ul\rho$ to $\ul\rho'$ via $\sigma$ maps $Z_+$ to
$s_{\sigma\ul\rho}(\xi) s_{\sigma \ul\rho'}^{-1}(\xi)\cdot Z'_+$.
Similarly, going via $\sigma'$ maps $Z_-$ to
$s_{\sigma'\ul\rho}(\xi') s_{\sigma'\ul\rho'}^{-1} (\xi')\cdot
Z'_-z^{m_{\ul{\rho}\ul{\rho'}}}$. Note the additional term
$m_{\ul\rho\,\ul\rho'}$ coming from monodromy. Comparing with the
respective equations $Z_+Z_-= f_{\ul\rho}z^{\kappa_{\ul\rho}}$,
$Z'_+Z'_-= f_{\ul\rho'}z^{\kappa_{\ul\rho'}}$,  and assuming
\eqref{Eqn: X_0 consistency after twisting} now leads to the
following analogue
of \eqref{f_rho' versus f_rho}:
\begin{equation}
\label{f_rho' versus f_rho with gluing}
f_{\ul{\rho}'}z^{\kappa_{\ul{\rho}'}}=
\frac{s_{\sigma\ul\rho'}(\xi) s_{\sigma'\ul\rho'}(\xi')}{
s_{\sigma\ul\rho}(\xi) s_{\sigma'\ul\rho}(\xi')}
s_{\sigma\ul{\rho}'}^{-1}(s_{\sigma\ul\rho}(f_{\ul\rho})) z^{\kappa_{\ul\rho}}
z^{m_{\ul\rho'\ul\rho}}
\end{equation}
This equation holding modulo $I_0$ is necessary and
sufficient for the induced log structure on $X_0$ to glue
consistently locally. In particular, it is a necessary condition for the
existence of $\foX^\circ$ also under the presence of additional
walls and refinements of slabs.
\end{remark}

The choice of gluing data can be formulated cohomologically as follows. Consider
the open cover $\ul\W$ of $B$ consisting of the open stars of the barycentric
subdivision. The notation is $W_{\ul\tau}$ for the open star of
$\ul\tau\in\tilde\P$. We also use the notation $W_\tau$ for $\tau\in\P$ to
denote the open star of $\tau$ with respect to $\tilde\P$. Denote by
$\ul\W_0\subseteq\ul\W$ the subset consisting of interiors of maximal cells
$\sigma$ (the open star of the barycenter of $\sigma$) and of the open stars of
$\ul\rho\in \tilde\P^{[n-1]}_\inte$ not contained in $\partial B$. Thus the
elements of this covering are $W_\sigma=\Int\sigma$ and pairwise disjoint open
neighbourhoods $W_{\ul\rho}$, one for each $\ul\rho$ not contained in $\partial
B$. Since the elements of $\ul\W_0$ and their non-trivial intersections
$W_{\ul\rho \sigma}:= W_{\ul\rho}\cap W_\sigma$ are contractible, $\ul\W_0$ is a
Leray covering for the locally constant sheaf $\check\Lambda\otimes_\ZZ \ul
A^\times$ on $B_0\setminus\partial B$. Moreover, one has
$\Gamma(W_{\ul\rho\sigma}, \check\Lambda\otimes\ul A^\times)=
\Hom(\Lambda_\sigma,A^\times)$. Thus $(s_{\sigma\ul\rho} )_{\ul\rho,\sigma}$
defines a \v Cech $1$-cocycle with values in $\check\Lambda\otimes_\ZZ\ul
A^\times$ for the covering $\ul\W_0$, but not all \v Cech $1$-cocycles satisfy
\eqref{Eqn: X_0 consistency after twisting}. Cohomologous cocycles lead to
isomorphisms between constructions of $\foX^\circ$. To state this, note that
each pair $(\scrS,\mathbf s)$ consisting of a wall structure and gluing data
consistent in codimensions zero and one gives rise to a directed system of rings
$(R_\fou, R_\fob)$. We are interested in isomorphisms of such associated
directed systems of rings acting trivially on the labelling set $\{\fou,\fob\}$.
We call such isomorphisms
\emph{special}.

\begin{proposition}
\label{Prop: Cohomologous twists}
Let $(B,\P)$ be a polyhedral pseudomanifold.  There is an action of the group
$C^0(\ul\W_0, \check\Lambda\otimes_\ZZ \ul A^\times)$ on the
set of pairs $(\mathscr S,\mathbf s)$ consisting of a wall structure
and open gluing data, with $\mathbf t=(t_\sigma,t_{\ul\rho})\in C^0(\ul\W_0,
\check\Lambda\otimes_\ZZ \ul A^\times)$ acting on $\mathbf s= (s_{\sigma\ul\rho}
)_{\ul\rho,\sigma}$ by
\[
s_{\sigma\ul\rho}\longmapsto t_\sigma\cdot
s_{\sigma\ul\rho}\cdot t_{\ul\rho}^{-1}.
\]
This action preserves the set of structures with open gluing data that are
consistent in codimension zero and one, and for these consistent structures, the
associated directed systems of rings $(R_{\fou},R_{\fob})$ are related by
special isomorphisms.
\end{proposition}

\begin{proof}
We let $\mathbf t=(t_\sigma,t_{\ul\rho})\in C^0(\ul\W_0,
\check\Lambda\otimes_\ZZ \ul A^\times)$ act on the rings $R_\sigma$ via
$z^m\mapsto t_\sigma(m)\cdot z^m$, and take the induced action on the
rings $R_\fou$ and on the functions $f_\fop$ carried by walls. Similarly,
$t_{\ul\rho}$ acts on the rings $(A[Q]/I)[\Lambda_{\rho}]$. For a slab $\fob$
define $\fob(\mathbf t)$ by applying this action to $f_{\fob}$. Then $t_{\ul
\rho}$ induces an isomorphism $R_\fob\to R_{\fob(\mathbf{t})}$ for any slab
$\fob\subseteq \rho$, taking $Z_{\pm}$ to $t_{\ul\rho}(\pm \xi)Z_{\pm}$, where
as usual $\xi$ is the chosen element of $\Lambda_x$ for $x\in \fob$ determining
$Z_+$. The data $\mathbf t$ then modifies $\scrS$ by replacing each slab $\fob$
with the slab $\fob(\mathbf t)$ and applying $t_{\sigma}$ to each wall function
$f_{\fop}$ contained in $\sigma$. It is then easy to see that this new structure
$\scrS(\mathbf t)$ is consistent in codimension zero and one with respect to the
twisted gluing data, assuming $\scrS$ was consistent in these codimensions with
respect to the original gluing data. Indeed, codimension zero follows trivially.
Codimension one consistency follows easily from the definition and the fact that
if $\theta$, $\theta'$ are the wall-crossing automorphisms occuring in
Definition \ref{Def: Consistency in codim one} for the structure $\scrS$ and
$\theta_{\mathbf t}$, $\theta'_{\mathbf t}$ the corresponding automorphisms for
$\scrS(\mathbf t)$, one has $t_{\sigma}\circ \theta =\theta_{\mathbf t}\circ
t_{\sigma}$ and $t_{\sigma'}\circ \theta' =\theta'_{\mathbf t}\circ
t_{\sigma'}$. It is then straightforward to check that the action of
$t_{\sigma}$ on the rings $R_{\fou}$ with $\fou\subseteq\sigma$ and the action
$t_{\ul\rho}:R_{\fob} \rightarrow R_{\fob(\ul t)}$ defines a special isomorphism
of directed systems of rings.
\end{proof}
\medskip

In particular, it makes sense to call two sets of open gluing data
\emph{equivalent} if they are cohomologous as \v Cech $1$-cocycles.
\smallskip

\ul{Closed gluing data.}
Since our gluing data already changes the gluing modulo $I_0$,
consistency in codimension one and two may fail modulo $I_0$. Thus
we may not even obtain a scheme $X_0^{\circ}$ over $\Spec
(A[Q]/I_0)$. If consistency holds in codimension zero and one, we do obtain
$\foX^{\circ}$, but have no guarantee that there is a scheme $X_0$
analogous to that of \S\ref{Subsect: X_0} containing the reduction
$X_0^\circ$ of $\foX^{\circ}$ modulo $I_0$ as a dense open
subscheme. In general, arbitrary choices of $X_0$ can be described
by \emph{closed gluing data}, which explains how to  assemble $X_0$
by gluing along closed strata. This was carried out in
\cite{logmirror1}, \S2.1. Furthermore, without access to local
models in codimension $\ge 2$, we will rely on projectivity
to compactify $\foX^{\circ}$, already modulo $I_0$. We will now explore what
additional conditions must be imposed on open gluing data to
guarantee the existence of $X_0$.

There are some obvious obstructions to the existence of $X_0$ in codimensions
one and two associated with interior joints, as follows. Let $\foj$ be an
interior joint of codimension one or two for the wall structure
$\scrS$ and $\tau=\sigma_\foj\in\P$ the minimal cell containing $\foj$. Then for
any $m\in\Lambda_\tau$ we have a monomial $z^m$ in the rings $R_\fob$ and
$R_\fou$ for slabs $\fob\supseteq\foj$ and chambers $\fou\supseteq\foj$, but
passing from $\fob\subseteq\ul\rho$ to an adjacent $\fou\subseteq\sigma$
introduces the factor $s_{\sigma\ul\rho}(m)$. Thus $z^m\in R_\fou$ extends to a
function on the scheme $X_{\foj,0}$ of \S\ref{Subsect: X_0} constructed from
$B_{\foj}$ only if
\begin{equation}
\label{Eqn: Gluing data along joint}
\prod_{i=1}^r s_{\sigma_{i+1}\ul\rho_i}(m)
\cdot s_{\sigma_i\ul\rho_i}^{-1}(m)=1.
\end{equation}
Here the $\sigma_i$ and $\ul\rho_i$ are the maximal and codimension one cells
containing $\foj$ ordered in such a way that
$\ul\rho_i\subseteq\sigma_i\cap\sigma_{i+1}$ and with $i$ taken modulo $r$. We
note that for $\foj$ a joint intersecting the interior of a codimension one
cell, the above equation is precisely \eqref{Eqn: X_0 consistency after
twisting} which is assumed to hold for open gluing data. The condition for
joints contained in codimension two cells is more subtle. As we will see in
Theorem \ref{Prop: ob and consistency in codim 2}, in the case that $B$
is a manifold possibly with boundary, \eqref{Eqn: Gluing data along joint} is
equivalent to the existence of a version $(\bar s_{\tau\omega})_{\omega,\tau}$
of the closed gluing data, labelled by any inclusion of cells
$\omega\subseteq\tau$, which twists the construction of $X_0$ and acts on the
starting monomials of broken lines. More generally, there is a local cohomology
obstruction to the existence of such closed gluing data, see
Proposition~\ref{Prop: Cohomology sheaf for CGD obstruction} below.

The collection of $\bar s_{\tau\omega}$ is a one-cocycle on $B$ for a sheaf
$\shQ$ which is constructible with respect to a decomposition $\check\P$ of $B$
that is dual to $\P$. This dual decomposition $\check\P$ is
canonically defined by taking the cell $\check\tau\in\check\P$ dual to
$\tau\in\P$ as the union of all cells of the barycentric subdivision
$\tilde\P$ labelled by $\tau=\tau_0\subseteq\tau_1
\subseteq\ldots\subseteq\tau_k$. Recall that by our conventions for the
barycentric subdivision in the unbounded case given in
Construction~\ref{Construction: B}, the smallest cell $\tau_0$ must be bounded
and the barycenters of unbounded cells are replaced by asymptotic directions.

We will need the following facts about $\check\P$ and the open cover $\ul{\W}$,
which require a little bit of care because of unbounded cells:

\begin{lemma}
\label{lem: dual P facts}
Suppose given a polyhedral pseudomanifold $B$ of dimension
$n$ with decomposition $\P$. We have:
\begin{enumerate}
\item
If $\tau\in\P$ is unbounded, then $\check\tau$ is empty; otherwise $\check\tau$
is non-empty, $\dim \check\tau=n-\dim\tau$ and $\tau \cap\check\tau$ consists of
a single point, the barycenter of $\tau$.
\item
If $\tau\cap\check\omega$ is non-empty, then $\omega\subseteq\tau$ and
$\dim\tau\cap\check\omega =\dim\tau-\dim\omega$.
\item
Let $p\in B$ and let $\ul{\tau}$ be the minimal cell of the barycentric
subdivision $\tilde\P$ of $\P$ containing $p$. Suppose that $\ul{\tau}$
corresponds to a sequence of cells $\tau_0\subseteq\cdots\subseteq\tau_k$. Then
there is a sequence of continuous maps $\kappa_i:W_{\ul\tau}\rightarrow
W_{\ul\tau}$ compatible with $\tilde\P$, with $\kappa_i(p)=p$, $\kappa_i$ a
homeomorphism onto its image, and $\bigcap_i \im(\kappa_i)=\{p\}$. Similarly,
there exists a sequence of maps $\kappa_i:W_{\ul\tau}\rightarrow W_{\ul\tau}$
compatible with $\tilde\P$, with $\kappa_i|_{W_{\ul\tau}\cap
\check\tau_0}$ the identity, $\kappa_i$ a homeomorphism onto its image, and
$\bigcap_i \im(\kappa_i)=W_{\ul\tau}\cap\check\tau_0$.
\end{enumerate}
\end{lemma}

\begin{proof}
(1) In the definition of $\tilde\P$ in Construction~\ref{Construction:
B}, there is no cell of $\tilde\P$ corresponding to a chain $\tau_0\subseteq
\cdots\subseteq\tau_k$ with $\tau_0$ (and hence all $\tau_i$) unbounded, hence
the first statement. If $\tau$ is bounded, it is immediate from the definition
that $\tau\cap\check\tau$ consists just of the barycenter of $\tau$. To see the
dimension statement, choose a chain $\tau=\tau_0\subseteq \cdots\subseteq\tau_k$
which is maximal, and such that $\tau_0,\ldots,\tau_{\ell}$ are bounded and
$\tau_{\ell+1},\ldots,\tau_k$ are unbounded. Furthermore, choose these so that
$\ell$ is as large as possible. One can then check that
$u_{\tau_{\ell+1}}, \ldots,u_{\tau_k}$ are linearly independent and thus from the
definition the corresponding cell of $\tilde\P$ is of dimension $n-\dim\tau$.
Further, it is clear from the definition that every cell of $\tilde\P$ contained
in $\check\tau$ is dimension at most $n-\dim\tau$, hence the claim.

(2) The first statement follows immediately from the definition of
$\check\omega$. For the dimension statement, note that the
intersection is a union of cells of $\tilde\P$ corresponding to
chains $\omega=\omega_0\subseteq\cdots\subseteq\omega_k=\tau$. Such a
cell is always of dimension at most $\dim\tau-\dim\omega$, and a
similar argument as in (1) shows that there is at least one such
cell achieving this bound.

(3) In each of the two cases, it is sufficient to construct maps
$\kappa_i$ defined on each cell $\ul{\omega}$ of $\tilde\P$
containing  $\ul{\tau}$ which are compatible with inclusions of
faces. To this end,  suppose $\ul{\omega}$ corresponds to a sequence
of cells $\omega_0\subseteq \cdots\subseteq\omega_m$ of $\P$, with
$\omega_0,\ldots,\omega_{\ell}$ bounded and
$\omega_{\ell+1},\ldots,\omega_m$ unbounded. The condition
$\ul{\tau} \subseteq \ul{\omega}$ is equivalent to the sequence
$\tau_0,\ldots,\tau_k$ being a subsequence of
$\omega_0,\ldots,\omega_m$. Recall that
\[
\ul{\omega}=\conv\{a_{\omega_0},\ldots,a_{\omega_{\ell}}\}
+\sum_{\ell+1\le j \le p} \RR_{\ge 0} u_{\omega_j}.
\]
By passing to a subsequence of cells, we can assume that the vectors
$u_{\ell+1},\ldots, u_m$ are linearly independent without changing
$\ul{\omega}$. Then every element of $\ul{\omega}$ has a unique representative
as $\sum \beta_j a_{\omega_j} + \sum \beta_j u_{\omega_j}$, with
$\sum_{j=0}^{\ell} \beta_j=1$. In particular, for the first case, we write
$p=\sum \alpha_j a_{\omega_j} + \sum \alpha_j u_{\omega_j}$ with
$\alpha_j=0$ if $\omega_j$ does not appear in the sequence $\{\tau_j\}$.

To define the sequence of retractions $\kappa_i$, choose
once and for all a sequence of maps $\phi_i:\RR_{\ge 0}\rightarrow \RR_{\ge 0}$
such that $\phi_i$ is a homeomorphism onto its image, $\phi_i(0)=0$, and
$\bigcap_i \im(\phi_i)=\{0\}$. Also fix a sequence of real numbers $\lambda_i\in
(0,1]$ with $\lambda_i\rightarrow 0$. Define $\psi_{ij}:[0,1]\rightarrow [0,1]$
for $0\le j \le \ell$ by
\[
\psi_{ij}(\beta)=\lambda_i(\beta-\alpha_j)+\alpha_j
\]
and $\psi_{ij}:\RR_{\ge 0}\rightarrow \RR_{\ge 0}$ for $\ell+1\le j\le m$
by
\[
\psi_{ij}(\beta)=\begin{cases}
\phi_i(\beta-\alpha_j)+\alpha_j&\beta\ge\alpha_j\\
\lambda_i(\beta-\alpha_j)+\alpha_j&\beta\le \alpha_j.
\end{cases}
\]
Then define $\kappa_i:\ul{\omega}\rightarrow\ul{\omega}$ by
\[
\kappa_i\left(\sum_{j=0}^{\ell} \beta_j a_{\omega_j}+\sum_{j=\ell+1}^m 
\beta_j u_{\omega_j}\right)=
\sum_{j=0}^{\ell} \psi_{ij}(\beta_j) a_{\omega_j}+\sum_{j=\ell+1}^m 
\psi_{ij}(\beta_j) u_{\omega_j}.
\]

For the sequence of maps with image converging to
$W_{\ul\tau}\cap\check\tau_0$, suppose $\tau_0=\omega_q$. Necessarily
$\tau_0$ is bounded, so $q\le \ell$. We instead define
\begin{align*}
&\kappa_i\left(\sum_{j=0}^{\ell} \beta_j a_{\omega_j}+\sum_{j=\ell+1}^m 
\beta_j u_{\omega_j}\right)\\
= {} &
\sum_{j=0}^{q-1} \lambda_i\beta_j a_{\omega_j}
+\sum_{j=q}^{\ell} \left({1-\lambda_i\sum_{h=0}^{q-1}\beta_h\over\sum_{h=q}^{\ell}
\beta_h}\right)\beta_ja_{\omega_j}
+\sum_{j=\ell+1}^m 
\lambda_i\beta_j u_{\omega_j}.
\end{align*}
One checks easily that these maps are homeomorphism onto their
images and $\bigcap_i\im(\kappa_i)=\check\tau_0$.
\end{proof}

The sheaf $\shQ$ mentioned before Lemma~\ref{lem: dual P facts} is the
sheaf constructible with respect to $\check\P$ with constant stalks
\[
\shQ_{\check\tau}:=\check\Lambda_\tau=\Hom(\Lambda_{\tau},\ZZ)
\]
along $\Int\check\tau$. For $\check\tau\subseteq\check\omega$ the generization
map $\shQ_{\check\tau}\to \shQ_{\check\omega}$ is defined as the dual of the
inclusion $\Lambda_\omega\to \Lambda_\tau$.

For the cohomological treatment of closed gluing data we use the covering
$\W=\{W_\tau\,|\, \tau\in\P\}$ of $B$ introduced before Proposition~\ref{Prop:
Cohomologous twists}. For $\tau\in\P$ the open set $W_\tau$ is the union of the
interiors of all simplices of the barycentric subdivision $\tilde\P$ of $\P$
intersecting $\Int(\tau)$, that is, having the barycenter $a_\tau\in\tau$ as a
vertex.\footnote{\cite{logmirror1} assumes $B$ bounded, but with our
generalization of the barycentric subdivision $\tilde\P$ for unbounded $B$ the
results generalize.

\begin{lemma}
\label{Lem: Gamma(W_omega,shQ)}
For $\omega\in\P$ the stalk $\shQ_{\check\omega}$ surjects onto each stalk $\shQ_x$
for any $x\in W_\omega$. In particular, $\Gamma(W_\omega,\shQ)=\check\Lambda_\omega$,
and if $\omega\subseteq\tau$ then
\[
\Gamma(W_\omega\cap W_\tau,\shQ) =\check\Lambda_\omega.
\]
\end{lemma}
\begin{proof}
We only need to prove the first statement. For a strictly increasing sequence
$\tau_0\subset \tau_1 \subset\ldots\subset\tau_k$ in $\P$ with $\tau_0$ bounded,
by abuse of notation we write $\ul\tau$ both for the labelling set
$\{\tau_0,\ldots,\tau_k\}$ and the corresponding cell of the barycentric
subdivision $\tilde\P$. By definition, it holds
\begin{equation}
\label{Eqn: W_omega as union}
W_\omega = \bigcup_{\omega\in \ul\tau} \Int \ul\tau.
\end{equation}
Note that the right-hand side in this description is a union of the interiors of
cells in a simplicial complex and is hence a disjoint union. Thus given $x\in
W_\omega$ there is a unique $\ul\tau$ with $\omega\in\ul\tau$ such that $x\in
\Int \ul\tau $. There is also a unique $\tau\in\P$ with $x\in\Int\check\tau$. By
the definition of $\check\tau$ we have a similar description of $\Int\check\tau$
as a disjoint union
\[
\Int\check\tau= \bigcup_{\tau\in \ul\tau' \text{\ minimal}} \Int \ul\tau',
\]
with ``minimality'' referring to the ordering of the elements of
$\ul\tau=\{\tau_0,\ldots,\tau_k\}$ by inclusion as subsets of $B$. Thus if $x\in
\check\tau$ then the unique $\ul\tau'$ with $x\in \Int \ul\tau'$ agrees with
$\ul\tau$ that we obtained from \eqref{Eqn: W_omega as union}. This shows
\[
\shQ_x=\shQ_{\check\tau}.
\]
Moreover, since $\omega,\tau\in\ul\tau$ and $\tau$ is minimal, we also
conclude $\tau\subseteq\omega$, and conversely, for any $\tau$ with $\tau\subseteq
\omega$ we have $\Int\check\tau\cap W_\omega\neq\emptyset$. Thus for any $x\in
W_\omega$ we have a surjective generization map
\[
\shQ_{\check\omega} \lra \shQ_x,
\]
which is an isomorphism for $x=a_\omega$. The statement now follows by compatibility of generization maps under further generizations.
\end{proof}
}
By Lemma~\ref{Lem: Gamma(W_omega,shQ)} a one-cocycle for $\shQ\otimes_\ZZ
A^\times$ with respect to $\W=\{W_{\tau}\,|\, \tau\in\P\}$ is a collection of
homomorphisms $\bar s_{\tau\omega}: \Lambda_\omega\to A^\times$, one for each
inclusion of cells $\omega\subseteq\tau$, fulfilling the cocycle
condition $\bar s_{\tau''\tau'}|_{\Lambda_{\tau}} \cdot \bar s_{\tau'\tau} =
\bar s_{\tau''\tau}$. As in the case with GS-type singularities
\cite{logmirror1}, Definition~2.10, we refer to these one-cocycles as
\emph{closed gluing data}, with a notion of equivalence defined by coboundaries.
Note also that the same argument as in \cite{logmirror1}, Lemma~5.5 shows that
$\mathscr W$ is an acyclic cover for $\shQ\otimes \ul A^\times$ and hence
\[
H^k(B,\shQ\otimes \ul A^\times)= 
H^k(\mathscr W,\shQ\otimes \ul A^\times).
\]

In what follows, we denote by $\Delta_k\subseteq \Delta$ for $k\ge 2$
the codimension $k$ skeleton of $\P$. Note that $\Delta\setminus \Delta_2$ is
covered by the interiors of those $(n-2)$-cells of $\tilde\P$ intersecting the
interiors of $(n-1)$-cells of $\P$. We also make use of the open cover of
$B\setminus\Delta_2$ given by
\[
\W_1=\{W_{\tau}\,|\, \hbox{$\tau\in \P$, $\dim \tau$ is $n$ or $n-1$}\}.
\]
We then have
\[
H^k(B\setminus\Delta_2,\shQ\otimes \ul A^\times)= 
H^k(\mathscr W_1,\shQ\otimes \ul A^\times).
\]

\begin{lemma}
\label{Lem: open codim one}
Let $\mathbf s$ be open gluing data for $(B,\P,\varphi)$. 
Then $\mathbf s$ uniquely determines an element of $H^1(\mathscr
W_1,\shQ\otimes \ul A^\times)$.
\end{lemma}

\begin{proof}
For each codimension one $\rho\in\P$, $\rho\not\subseteq \partial B$,
choose an ordering $\sigma, \sigma'$ of the two maximal cells
containing $\rho$. Replace $\mathbf s$ with the cohomologous cycle
using the action of Proposition \ref{Prop: Cohomologous twists}
with  $t_{\ul\rho_i}=s_{\sigma'\ul\rho_i}$ and
$t_\sigma=t_{\sigma'}=1$. Thus we can assume that
$s_{\sigma'\ul\rho_i}(m)=1$ for all $m$. Now let  $\ul\rho_1,
\ul\rho_2$ be two codimension one cells of the barycentric
subdivision contained in a common codimension one cell $\rho$ of
$\P$, and contained in two codimension zero cells $\sigma,\sigma'$.
Then \eqref{Eqn: X_0 consistency after twisting} now simply states
that 
\[
s_{\sigma\ul\rho_1}|_{\Lambda_{\rho}}=s_{\sigma\ul\rho_2}|_{\Lambda_{\rho}},
\]
while the same statement for $\sigma'$ is trivially true since each side of the
equality is $1$ by construction. Thus defining $\bar
s_{\sigma\rho}=s_{\sigma\ul\rho}|_{\Lambda_{\rho}}$ for any
$\ul\rho\subseteq\rho$ and $\bar s_{\sigma'\rho}=1$,
we obtain well-defined sections of $\shQ\otimes \ul A^\times$ over $W_{\sigma}
\cap W_{\rho}$ and $W_{\sigma'}\cap W_{\rho}$ respectively. So we obtain a \v
Cech one-cocycle $\bar{\mathbf s}=(\bar s_{\sigma\rho})$ for the sheaf
$\shQ\otimes\ul A^\times$ on the cover $\W_1$. One checks
easily that changing the ordering of maximal cells or replacing
$\mathbf s$ by a cohomologous cycle, changes $\bar{\mathbf s}$ by a
cohomologous cycle. Hence we obtain a well-defined element of
$H^1(\W_1,\shQ\otimes \ul A^\times)$ as desired.
\end{proof}

\begin{definition}
If $\mathbf s$ is open gluing data for $(B,\P,\varphi)$
we write $\bar{\mathbf
s}$ for the induced element of $H^1(B\setminus\Delta_2,\shQ\otimes
\ul A^\times)$.
\end{definition}

We now define an obstruction class $\ob_{\Delta_2}(\bar {\mathbf s})$ which
is the obstruction to extending $\bar {\mathbf s}$ from
$B\setminus\Delta_2$ to $B$, defined via the connecting
homomorphism in the long exact sequence for local cohomology:
\begin{equation}
\label{Eqn: Local cohomology sequence}
H^1_{\Delta_2}(B,\shQ\otimes \ul A^\times)\lra H^1(B,\shQ\otimes \ul A^\times)\lra
H^1(B\setminus\Delta_2, \shQ\otimes \ul A^\times)\stackrel{\ob_{\Delta_2}}{\lra}
H^2_{\Delta_2}(B,\shQ\otimes \ul A^\times).
\end{equation}

\begin{lemma}
\label{Lem: local cohomology of shQ}
The local cohomology sheaves $\shH^k_{\Delta_2}(\shQ)$ vanish for
$k=0,1$. In particular, $H^k_{\Delta_2}(\shQ)=0$ for
$k=0,1$,
\[
\shH^k_{\Delta_2}\big(\shQ)=
R^{k-1}j_*(\shQ|_{B\setminus\Delta_2}), \quad k\ge 2,
\]
with $j: B\setminus\Delta_2\to B$ the inclusion, and
$H^2_{\Delta_2}(B,\shQ) = H^0\big(B,\shH^2_{\Delta_2}(\shQ)\big)$.
Analogous statements hold for $\shQ\otimes \ul A^\times$.
\end{lemma}

\begin{proof}
The first two cohomology sheaves with closed support fit into
the exact sequence
\[
0\lra \shH^0_{\Delta_2} (\shQ)\lra \shQ\lra j_*(\shQ|_{B\setminus
\Delta_2})\lra \shH^1_{\Delta_2} (\shQ)\lra 0,
\]
while $\shH^k_{\Delta_2} (\shQ) = R^{k-1}j_*
(\shQ|_{B\setminus\Delta_2})$ for $k\ge 2$, see \cite{hartshorneLC},
Corollary~1.9. For the vanishing of the first two local cohomology sheaves we
thus have to show that $\shQ\to j_*(\shQ|_{B\setminus \Delta_2})$ is an
isomorphism. Let $p\in \Delta_2$ and let $\ul{\tau}$ be the minimal cell of the
barycentric decomposition $\tilde\P$ containing $p$. Recall that $\shQ$ is a
constructible sheaf for $\check\P$, hence also for the refinement $\tilde\P$. By
Lemma \ref{lem: dual P facts},(3), there is a sequence of retractions
$\kappa_i:W_{\ul\tau}\rightarrow W_{\ul{\tau}}$ compatible with $\tilde\P$ and
with $\kappa_i^*\shQ \simeq \shQ$ and $\bigcap_i\im(\kappa_i)=\{p\}$. This shows
that $(j_*\shQ)_p=H^0(W_{\ul{\tau}}\setminus\Delta_2,\shQ)$.

On the other hand, if $\ul{\tau}$ corresponds to a sequence
$\tau_0\subseteq\cdots\subseteq\tau_k$ of cells of $\P$, then
$\dim\tau_0\le n-2$ since $\ul\tau\subset\Delta_2$. In particular, for each
maximal simplex $\ul\omega$ contained in $\check\tau_0$ the two maximal elements
in $\ul\omega$ are of dimensions $n-1$ and $n$, respectively, and hence
$\omega\not\subseteq\Delta_2$. In particular, $\Delta_2\cap\check\tau_0$ is
nowhere dense in $\check\tau_0$. Now using Lemma \ref{lem: dual
P facts},(3), a similar retraction argument shows that $H^0(W_{\ul{\tau}}
\setminus\Delta_2, \shQ)\cong H^0((W_{\ul{\tau}}\cap \check\tau_0)
\setminus\Delta_2, \shQ|_{(W_{\ul\tau}\cap\check\tau_0) \setminus\Delta_2})$.
Note $\shQ|_{W_{\ul\tau}\cap\check\tau_0}$ is a constant sheaf, and thus
$H^0(W_{\ul{\tau}},\shQ) =H^0(W_{\ul{\tau}} \cap\check\tau_0,
\shQ|_{\check\tau_0}) =\check\Lambda_{\tau_0}$. It remains to show that
the restriction map $H^0(W_{\ul{\tau}} \cap\check\tau_0, \shQ|_{\check\tau_0})
\to H^0((W_{\ul{\tau}} \cap\check\tau_0)\setminus \Delta_2,
\shQ|_{\check\tau_0})$ is surjective.

It follows from the $S_2$ condition on $B$ that $W_{\ul{\tau}}
\setminus\Delta_2$ is connected. Indeed, this can be shown inductively by
computing $H^0(W_{\ul{\tau}}\setminus \Delta_l,\ZZ)$ by downward induction
on $l$, with the base case $l=\codim \tau_k+1$ trivial since $W_{\ul\tau}\cap
\Delta_{\codim \tau_k+1}=\emptyset$. For the induction step, note that the $S_2$
condition implies that $\nu_*\ul\ZZ=\ul\ZZ$ for $\nu:W_{\ul\tau}\setminus
\Delta_l\to W_{\ul\tau}\setminus \Delta_{l+1}$ the inclusion. The local
cohomology sheaf sequence (\cite{hartshorneLC}, Corollary~1.9) together with the
local to global spectral sequence for cohomology with supports
(\cite{hartshorneLC}, Proposition~1.4) thus shows
\[
H^0_{\Delta_l\setminus\Delta_{l+1}}(W_{\ul{\tau}}\setminus
\Delta_{l+1},\ul{\ZZ})= 
H^1_{\Delta_l\setminus\Delta_{l+1}}(W_{\ul{\tau}}\setminus
\Delta_{l+1},\ul{\ZZ})=0.
\] 
Since $W_{\ul{\tau}}\setminus\Delta_2$ retracts onto
$(W_{\ul{\tau}}\cap\check\tau_0)\setminus\Delta_2$, it follows that the latter
is connected and thus $H^0((W_{\ul{\tau}}\cap \check\tau_0) \setminus\Delta_2,
\shQ|_{\tau_0})= \check\Lambda_{\tau_0}$, establishing the claimed
surjectivity of the restriction map. We conclude that the map $\shQ\rightarrow
j_*(\shQ|_{B\setminus\Delta_2})$ is an isomorphism, as desired.

The claims on $H^k_{\Delta_2}(B,\shQ)$, $k\le 2$, now follow from
the local to global spectral sequence for cohomology with supports
\cite{hartshorneLC}, Proposition~1.4.
\end{proof}

\begin{proposition}
\label{Prop: Cohomology sheaf for CGD obstruction}
A one-cocycle $\bar{\mathbf{s}}= (\bar s_{\sigma\rho})$ for
$\shQ\otimes \ul A^\times$ on $B\setminus\Delta_2$ extends to a
one-cocycle on $B$ if and only if the local obstruction
$\ob_{\Delta_2}(\bar{\mathbf{s}})\in \Gamma \big(B, 
\shH^2_{\Delta_2}(\shQ\otimes \ul A^\times)\big)$ for doing so
vanishes. An extension is unique up to equivalence.
\end{proposition}

\begin{proof}
The existence statement is immediate from \eqref{Eqn: Local cohomology sequence}
and Lemma~\ref{Lem: local cohomology of shQ}. The same sequence shows that the
equivalence class of the extension is unique up to the action of
$H^1_{\Delta_2}(\shQ\otimes\ul A^\times)$, which vanishes by Lemma~\ref{Lem:
local cohomology of shQ}.
\end{proof}
\medskip

We will now connect the vanishing of the local obstruction class
$\ob_{\Delta_2}(\bar{\mathbf s})$ with \eqref{Eqn: Gluing data along joint},
obtaining the strongest results in the case that $B$ is topological
manifold with boundary.

\begin{proposition}
\label{Prop: ob and consistency in codim 2}
Let $\mathbf s$ be open gluing data for $(B,\P,\varphi)$. If the
obstruction $\ob_{\Delta_2}( \bar{\mathbf s}) \in
H^2_{\Delta_2}(B,\shQ\otimes \ul A^\times)$ for extending $\mathbf
s$ to closed gluing data on all of $B$ vanishes then the consistency
condition \eqref{Eqn: Gluing data along joint} holds for interior
joints $\foj$ of the form $\ul\tau\in\tilde\P$ contained in
$\Delta_2$ and for all $m\in\Lambda_\tau$, $\tau\in\P$ the minimal
cell containing $\ul\tau$. Furthermore, this implication is an
equivalence if $B$ is a topological manifold with boundary.
\end{proposition}

\begin{proof}
By Proposition~\ref{Prop: Cohomology sheaf for CGD obstruction} it
suffices to consider the vanishing statement locally. We first
consider the case that $B$ is a topological manifold with
boundary, and in this case show that for the vanishing
of the obstruction $\ob_{\Delta_2}( \bar{\mathbf s})
\in H^2_{\Delta_2}(B,\shQ\otimes \ul A^\times)$, it is sufficient
to consider the vanishing at general points of the
codimension two cells covering $\Delta_2$. We will then show that this
latter vanishing is in any case equivalent to 
\eqref{Eqn: Gluing data along joint}.

To this end, suppose $B$ is a manifold with boundary, and consider part
of the long exact sequence of cohomology with supports for $\Delta_3$
(\cite{hartshorneLC}, Proposition~1.9):
\[
\shH^2_{\Delta_3} (\shQ\otimes \ul A^\times)\lra
\shH^2_{\Delta_2} (\shQ\otimes \ul A^\times)\lra
\shH^2_{\Delta_2\setminus \Delta_3} (\shQ\otimes \ul A^\times).
\]
We claim that $\shH^2_{\Delta_3} (\shQ\otimes \ul A^\times)=0$. Then
in view of the excision formula (\cite{hartshorneLC}, Proposition~1.3)
vanishing of the local obstruction class can be tested on
$B\setminus \Delta_3$. To prove the claim denote by $j_3: B\setminus
\Delta_3\to B$ the inclusion. Then $\shH^2_{\Delta_3} (\shQ\otimes
\ul A^\times)= R^1 {j_3}_*(\shQ\otimes \ul A^\times)$
(\cite{hartshorneLC}, Corollary~1.9). Let $p\in\Delta_3$ and $\ul\tau$
the minimal cell of the barycentric decomposition $\tilde\P$
containing $p$. By the same retraction argument of $W_{\ul{\tau}}$ to
$p$ as in the proof of Lemma~\ref{Lem: local cohomology of shQ}, we obtain
\[
\big(R^1{j_3}_*(\shQ\otimes\ul A^\times)\big)_p=
H^1(W_{\ul\tau}\setminus \Delta_3, \shQ\otimes\ul A^\times).
\]
If $\ul\tau$ corresponds to a sequence $\tau_0\subseteq\cdots\subseteq\tau_k$
of cells of $\P$, then similarly to the proof of 
Lemma~\ref{Lem: local cohomology of shQ}, we have
\[
H^1(W_{\ul\tau}\setminus \Delta_3, \shQ\otimes\ul A^\times)\cong
H^1((W_{\ul\tau}\cap\check\tau_0)\setminus \Delta_3, 
(\shQ\otimes \ul A^\times)|_{W_{\ul\tau}\cap\check\tau_0}).
\]
Furthermore, by Lemma~\ref{lem: dual P facts}, (2),
it follows that $\Delta_3\cap \check\tau_0$ is codimension three in 
$\check\tau_0$. In particular, since $\shQ|_{W_{\ul\tau}\cap\check\tau_0}$
is a constant sheaf with stalk $\check\Lambda_{\tau_0}$ and, if we assume
$B$ is a manifold with boundary,
$W_{\ul\tau}\cap\check\tau_0$ is an open ball in a manifold with boundary, 
we see that 
\begin{equation}
\label{eq: H1vanishing}
H^1((W_{\ul\tau}\cap\check\tau_0)\setminus \Delta_3, 
(\shQ\otimes \ul A^\times)|_{W_{\ul\tau}\cap\check\tau_0})=0.
\end{equation}
This finishes the proof of the claim if $B$ is a manifold with boundary.

Now consider $B$ arbitrary, and consider the map
$\ob_{\Delta_2\setminus \Delta_3}: H^2(B\setminus \Delta_2,
\shQ\otimes\ul{A}^\times) \rightarrow H^2_{\Delta_2 \setminus
\Delta_3}(B\setminus \Delta_3,\shQ\otimes \ul{A}^\times)$. We have
just shown that if $B$ is a manifold with boundary, then
$\ob_{\Delta_2\setminus\Delta_3}(\bar{\mathbf s})=0$ if and only if
$\ob_{\Delta_2}(\bar{\mathbf s})=0$. We now show in any case that
vanishing of $\ob_{\Delta_2\setminus\Delta_3}(\bar s)$ is equivalent
to the stated consistency condition. We have $H^2_{\Delta_2
\setminus\Delta_3} (B\setminus\Delta_3, \shQ\otimes \ul{A}^\times)
\cong H^0(\shH_{\Delta_2\setminus \Delta_3}^2 (B\setminus\Delta_3,
\shQ\otimes\ul{A}^\times))$. By constructibility of $\shQ$ it
suffices to test the vanishing of a section of
$\shH_{\Delta_2\setminus\Delta_3}^2 (B\setminus\Delta_3,
\shQ\otimes\ul{A}^\times)$ at $p$ the barycenter
$\tau\cap\check\tau\in \tilde\P^{[0]}$ of a cell $\tau\in\P$ of
codimension two. By the same argument of constructibility as in the
discussion of the codimension three locus, there is an isomorphism
\[
\left(\shH^2_{\Delta_2\setminus\Delta_3}(\shQ\otimes \ul A^\times)\right)_p
= H^1(W_\tau\setminus \Delta_2,\shQ\otimes\ul A^\times)
= H^1(W_\tau\setminus \tau,\shQ\otimes\ul A^\times).
\]
In the present case there is a sequence of retractions
$\kappa_k:W_\tau\setminus\tau\to W_\tau\setminus\tau$ with
$\bigcap_k \im(\kappa_k)= (\Int\check\tau)\setminus \{p\}$, and
hence
\[
H^1(W_\tau\setminus \tau,\shQ\otimes\ul A^\times)
= H^1\big((\Int\check\tau)\setminus\{p\},\shQ\otimes \ul A^\times\big).
\]
As before, the restriction of $\shQ\otimes\ul A^\times$ to
$\Int \check \tau$ is a constant sheaf with stalks $\check\Lambda_\tau
\otimes\ul A^\times$. Hence we can compute
\[
\left(\shH^2_{\Delta_2\setminus\Delta_3}(\shQ\otimes \ul A^\times)\right)_p
= H^1\big((\Int\check\tau)\setminus\{p\},
\check\Lambda_\tau\otimes
A^\times).
\]
If $\tau\subseteq\partial B$, then by the $S_2$ condition
$(\Int\check\tau)\setminus\{p\}$ is homotopic to an interval, and this
group is zero. Otherwise, $(\Int\check\tau)\setminus\{p\}$ is homotopic to $S^1$
and we obtain
\[
\left(\shH^2_{\Delta_2\setminus\Delta_3}(\shQ\otimes \ul A^\times)\right)_p
=H^1(S^1,\check\Lambda_\tau\otimes\ul{A}^\times)=
\check\Lambda_\tau\otimes\ul{A}^\times.
\]
Under this sequence of identifications the restriction of the
obstruction class $\ob_{\Delta_2\setminus\Delta_3}
(\bar{\mathbf{s}})$ is mapped to
$\prod_{i=1}^r \bar s_{\sigma_{i+1}\rho_i}(m) \cdot \bar
s_{\sigma_i\rho_i}^{-1}(m)$. Thus the local obstruction vanishes
along $\Int\tau$ if and only if the consistency
condition~\eqref{Eqn: Gluing data along joint} holds for all 
$m\in\Lambda_{\tau}$.
\end{proof}
\medskip

\begin{definition}
We say open gluing data $\mathbf s$ for $(B,\P,\varphi)$ is 
\emph{consistent} if $\ob_{\Delta_2}(\bar{\mathbf s})=0$.
In this case, we obtain uniquely induced closed gluing data $\bar{\mathbf
s}\in H^1(B,\shQ\otimes\ul A^\times)$, by 
Proposition \ref{Prop: Cohomology sheaf for CGD obstruction}. 
\end{definition}

\begin{remark}
\label{Rem: precise relation open/closed}
In fact if $\bar{\mathbf s}$ exists for a given ${\mathbf s}$, then we
can assume that we have specific representatives of both, such that 
for $\rho\subseteq \sigma\in\P_{\max}$ with $\rho$ codimension one, and
any $\ul\rho\subseteq \rho$, $m\in\Lambda_{\rho}$, one has 
\begin{equation}
\label{Eqn: s sbar equality}
\bar s_{\sigma\rho}(m)=s_{\sigma\ul\rho}(m).
\end{equation}
Indeed, the argument of the proof of Lemma \ref{Lem: open codim one}
implies we can replace ${\mathbf s}$ with equivalent open gluing
data so that we can define a cocycle $\bar{\mathbf s}$ for
$\shQ\otimes \ul A^\times$ over $B\setminus\Delta_2$ by  \eqref{Eqn: s
sbar equality}. The vanishing of $\ob_{\Delta_2}(\bar{\mathbf s})$
then implies $\bar{\mathbf s}$ lifts as a cohomology class
$\bar{\mathbf s}'$ to $H^1(B,\shQ\otimes \ul A^\times)$. Thus the
restriction of  $\bar{\mathbf s}'$ to $B\setminus \Delta_2$ is
cohomologous to $\bar{\mathbf s}$, that is,  there exists a
collection of data $t_{\sigma}\in \check\Lambda_{\sigma}\otimes
\ul{A}^\times$, $t_{\rho}\in \check \Lambda_{\tau}\otimes
\ul{A}^\times$ such that
\[
\bar s_{\sigma\rho}'(m)=t_{\sigma}(m)\bar s_{\sigma\rho}(m) t_{\rho}(m)^{-1}
\]
for all $\rho\subseteq\sigma$, $m\in\Lambda_{\rho}$. For each
$\ul\rho\subseteq \rho$, choose a lift $t_{\ul\rho}$ of $t_{\rho}$ to
$\Gamma(W_{\ul\rho}, \check\Lambda\otimes_{\ZZ} \ul A^\times)$.
Replacing ${\mathbf s}$ by the equivalent open gluing data induced
by the $t_{\sigma}$ and $t_{\ul\rho}$ then yields open gluing data
${\mathbf s}$ satisfying  \eqref{Eqn: s sbar equality}.
\end{remark}

Let us now assume given consistent open gluing data ${\mathbf{s}}$.
Then we have a unique choice of induced
closed gluing data $\bar{\mathbf{s}}= (\bar
s_{\tau\omega})$. Unfortunately, having induced closed gluing data 
is insufficient to construct $X_0$ as a scheme; at best one can hope
only to construct $X_0$ as an algebraic space as a direct limit of
closed immersions of toric varieties twisted by closed gluing data. However,
we shall take an easier route in the case that $X_0$ still carries an
ample line bundle. This case is detected by another obstruction, which we
turn to now.

\smallskip 

\ul{Projectivity.}
As we want to
follow the strategy from Section~\ref{Sect: Theta functions} for the
partial completion of $\foX^\circ$, we need to go over to the cone
$\cone{B}$ and construct global functions on the corresponding
affine scheme. This process is obstructed in general already on
$X_0$ for there exist non-projective locally trivial deformations of
such schemes. An example is provided by certain regluings of the
degenerate quartic surface $X_0 X_1 X_2 X_3=0$ in $\PP^3$, see
\cite{Friedman}, Remark~2.12.

The first problem is the lifting of gluing data ${\mathbf s}$ and
$\bar{\mathbf s}$ to $\cone{B}$. Let $h:\cone{B}\to\RR_{\ge 0}$ be
the height function and identify
$B$ with the slice $h^{-1}(1)\subseteq \cone{B}$. Denote
by $r: \cone{B}\setminus h^{-1}(0)\to B$ the retraction along rays 
emanating from the
apex. Recall that $r$ does not respect the affine structure, but
$h$ does. 
If we denote by $\tilde\shQ$ the sheaf analogous to $\shQ$ on $\cone B$, 
we have the exact sequence on $\cone B\setminus h^{-1}(0)$
\begin{equation}
\label{Eqn: tilde Q}
0\lra\ul A^\times\stackrel{h^*}{\lra} \tilde\shQ\otimes\ul{A}^\times
\lra r^*\shQ\otimes\ul{A}^\times\lra 0.
\end{equation}
In this sequence the morphism to $r^*\shQ\otimes\ul{A}^\times$ is
induced by identifying $r^*\Lambda$ with $\ker h_*\subseteq
\tilde\Lambda$. Then we can view $\bar {\mathbf s}$ as an element in
$H^1(B,\shQ\otimes\ul{A}^\times)\cong  H^1(\cone B\setminus h^{-1}(0),
r^*\shQ\otimes \ul{A}^\times)$, and hence we have an element
\begin{equation}
\label{Eqn: ob_PP}
\ob_{\PP}(\bar {\mathbf s})\in H^2(\cone B\setminus h^{-1}(0),A^\times)
\cong H^2(B,A^\times)
\end{equation}
via the connecting homomorphism in the long exact cohomology sequence
of \eqref{Eqn: tilde Q}. 

\begin{definition}
\label{Def: equiv lifts}
Fix a \v Cech representative $(\bar
s_{\tau\omega})_{\omega\subseteq\tau}$ for closed gluing data
$\bar{\mathbf s}$. Suppose  $\tilde{\bar {\mathbf s}}$, $\tilde{\bar
{\mathbf s}}'$ are two lifts of $\bar{\mathbf s}$ to $H^1(\cone
B\setminus h^{-1}(0), \tilde\shQ\otimes A^\times) \cong
H^1(B,(\tilde\shQ\otimes A^\times)|_B)$, given by representatives
$(\tilde{\bar s}_{\tau\omega})$,  $(\tilde{\bar s}'_{\tau\omega})$
with the image of both  $\tilde{\bar s}_{\tau\omega}$ and
$\tilde{\bar s}'_{\tau\omega}$ in $\check\Lambda_{\omega}
\otimes\ul{A}^\times$ coinciding with  $\bar s_{\tau\omega}$. Then
we say $\tilde{\bar{\mathbf s}}$ and $\tilde{\bar{\mathbf s}}'$ are
\emph{equivalent} if there exists for all $\omega\in\P$ a choice of
$t_{\omega}\in A^\times$ such that,  viewing $t_{\omega}$ as a
section of $\tilde\shQ\otimes A^\times$ via $h^*$,
\[
\tilde{\bar s}'_{\tau\omega}= t_{\tau} \tilde{\bar s}_{\tau\omega}
t_{\omega}^{-1}
\]
for all $\omega\subseteq\tau$.
\end{definition}

We then have

\begin{proposition}
\label{Prop: CGD obstruction on C{B}}
For consistent open gluing data ${\mathbf s}$ with associated closed
gluing data $\bar{\mathbf s}$, 
$\ob_{\PP}(\bar{\mathbf s})$ vanishes if and only if $\bar {\mathbf s}$ lifts
to closed gluing data for $\cone B$. Moreover, if 
$\ob_{\PP}(\bar{\mathbf s})$ vanishes,
then the set of lifts $\tilde{\bar{\mathbf s}}$ up to equivalence is
a torsor for $H^1(B,\ul A^\times)$. Finally, for each such lift 
$\tilde{\bar{\mathbf s}}$,
there is a choice of open gluing data $\tilde{\mathbf s}$ for $\cone B$
inducing closed gluing data $\tilde{\bar{\mathbf s}}$.
\end{proposition}

\begin{proof}
The first two statements follow from examining explicit \v Cech
representatives with respect to the open cover $\mathscr W$ in the long exact
cohomology sequence for \eqref{Eqn: tilde Q}. For the last statement, we can
assume ${\mathbf s}$ and $\bar{\mathbf s}$ are related as in Remark \ref{Rem:
precise relation open/closed}. Then given the lift $\tilde{\bar{\mathbf s}}$ of
$\bar{\mathbf s}$, we construct a lift of ${\mathbf s}$ to open gluing data
$\tilde{\mathbf s}$ for $\cone B$ simply by defining, for
$m\in\tilde\Lambda_{\cone\rho}$, $\tilde s_{\sigma\ul\rho}(m)=
\tilde{\bar{s}}_{\sigma\rho}(m)$, and for $m\in\Lambda_{\sigma}\subseteq
\tilde\Lambda_{\cone\sigma}$, $\tilde
s_{\sigma\ul\rho}(m)=s_{\sigma\ul\rho}(m)$.
\end{proof}

Proposition~\ref{Prop: CGD obstruction on C{B}} prompts us
to make the following definition.

\begin{definition}
\label{Def: Projective gluing data}
We call consistent open gluing data ${\mathbf s}$ \emph{projective} if
the induced closed gluing data
$\bar{\mathbf{s}}\in H^1(B,\shQ\otimes\ul{A}^\times)$
satisfies $\ob_{\PP}(\bar{\mathbf s})=0$.
\end{definition}

\begin{remark}
\label{Rem: ob_PP in GS}
If the singularities of $B$ are of the type considered in
\cite{logmirror1} and the closed gluing data $\mathbf s$ are the
restriction of open gluing data $\mathbf s'$ on all of $B$, then
$\ob_{\PP}(\mathbf{s})\in H^2(B,\kk^\times)$ agrees with the image of
$\mathbf s'$ under the homomorphism $o$ in \cite{logmirror1},
Theorem~2.34.
\end{remark}

We are then in
position to modify the constructions from Sections~\ref{Sect: Wall
structures} and~\ref{Sect: Global functions} consistently as
follows.
\smallskip

\ul{Construction of $X_0$.} 
Assume now given consistent open gluing data ${\mathbf s}$ with induced
closed gluing data $\bar{\mathbf s}$. Suppose further that $\ob_{\PP}(\bar
{\mathbf s})$ vanishes, and choose a lift $\tilde{\bar{\mathbf s}}$ of
$\bar{\mathbf s}$.
In the notation of \S\ref{Subsect:
X_0} define the ring $S[B](\tilde{\bar{\mathbf s}})$ with the same elements
as $S[B]$ but with the multiplication of monomials $z^m\cdot z^{m'}$
modified as follows. Let $\omega,\omega'$ be the minimal cells with $m\in
\cone\omega$, $m'\in\cone\omega'$. Assume that there is a cell
$\tau$ containing $\omega\cup\omega'$. Taking $\tau$ minimal with
this property we define
\begin{equation}
\label{Eq: twisted multiplication}
z^m\cdot z^{m'}:= \tilde{\bar s}_{\tau \omega}(m) 
\tilde{\bar s}_{\tau\omega'}(m')z^{m+m'}.
\end{equation}
If no such $\tau$ exists the product is zero as before.
As for associativity let $m,m',m''$ be contained in the cones for
$\omega,\omega',\omega''$ and assume $\sigma$ is the minimal cell
containing $\omega\cup\omega'\cup\omega''$. Denote by
$\tau,\tau',\tau''$ the minimal cells containing
$\omega'\cup\omega'', \omega''\cup\omega$ and $\omega\cup\omega'$,
respectively. Then
\begin{eqnarray*}
(z^m\cdot z^{m'})\cdot z^{m''} &=&
\tilde{\bar s}_{\tau''\omega}(m) \tilde{\bar s}_{\tau''\omega'}(m') 
\tilde{\bar s}_{\sigma\tau''}(m+m')
\tilde{\bar s}_{\sigma\omega''}(m'') z^{m+m'+m''}\\
&=& \tilde{\bar s}_{\sigma\omega}(m) \tilde{\bar s}_{\sigma\omega'}(m')
\tilde{\bar s}_{\sigma\omega''}(m'') z^{m+m'+m''}.
\end{eqnarray*} 
For the second equality we used multiplicativity of
$\tilde{\bar s}_{\sigma\tau''}$ and the cocycle conditions for
$\omega\subseteq\tau''\subseteq\sigma$ and
$\omega'\subseteq\tau''\subseteq\sigma$.

We now take $X_0=\Proj S[B](\tilde{\bar{\mathbf s}})$. To characterize the
irreducible components of the modified $X_0$ (Proposition~\ref{Prop: X_0}) note
that the monomials for a fixed $\sigma\in \P$ generate a quotient ring of
$S[B](\tilde{\bar{\mathbf s}})$ that is not obviously isomorphic to the standard
toric ring $S[\cone\sigma\cap (\Lambda_\sigma\oplus \ZZ)]$. An isomorphism can
however be easily defined by mapping $z^m$ to $\tilde{\bar s}_{\sigma\tau}(m)
z^m$ for $\tau$ the minimal cell with $m\in\cone\tau$. In fact, under
this map, the left-hand side of \eqref{Eq: twisted multiplication} maps to
$\tilde{\bar s}_{\sigma\omega}(m)\cdot \tilde{\bar s}_{\sigma\omega'}(m') z^{m+m'}$,
which agrees with the image
\[
\tilde{\bar s}_{\tau \omega}(m) \tilde{\bar s}_{\tau\omega'}(m')
\cdot\tilde{\bar s}_{\sigma\tau}(m+m')z^{m+m'}
= \tilde{\bar s}_{\sigma\tau}(m)\tilde{\bar s}_{\tau \omega}(m)
\tilde{\bar s}_{\sigma\tau}(m')\tilde{\bar s}_{\tau\omega'}(m') z^{m+m'}
\]
of the right-hand side.

\ul{Construction of $\foX^{\circ}$.} 
Assuming the wall structure $\scrS$ is consistent for the gluing data $\mathbf
s$ in codimension zero and one, the construction of $\foX^\circ$ in
Proposition~\ref{Prop: foX^o exists} is unchanged with the new definition of the
localization homomorphism $\chi_{\fob,\fou}= \chi_{\fob,\fou}(\mathbf s)$ in
\eqref{Eqn: Modified localization hom}. Note that the change of
$\chi_{\fob,\fou}$ implicitly also changes the isomorphism $\theta_\foj$ for
crossing a codimension one joint \eqref{Eqn: theta_foj}. Proposition~\ref{Prop:
Cohomologous twists} implies that changing open gluing data by a cocycle leads
to isomorphic directed systems of rings, hence to isomorphic schemes
$\foX^\circ$.

We then have the analogue of Proposition~\ref{Prop: foX^o modulo I_0}:

\begin{proposition}
Suppose $\ob_{\PP}(\bar{\mathbf s})=0$ and $\tilde{\bar{\mathbf s}}$ is
a lift of $\bar{\mathbf s}$, yielding $X_0$. Then
the reduction of $\foX^{\circ}$ modulo $I_0$ is canonically isomorphic
to the complement of the union of codimension two strata in $X_0$.
In particular, $\foX^{\circ}$ is separated as a scheme over $A[Q]/I$.
\end{proposition}

\begin{proof}
Taking $I=I_0$, it is enough to construct suitable
maps $\psi_{\fob}:\Spec R_{\fob}
\rightarrow X_0$ and $\psi_{\fou}:\Spec R_{\fou}\rightarrow X_0$, 
for all slabs $\fob$ and chambers $\fou$, such that whenever
$\fob\subseteq\fou$, we have a commutative diagram
\begin{equation}
\label{Eqn: compatibility of open/closed}
\xymatrix@C=30pt
{
\Spec R_{\fou} \ar[rr]^{\psi_{\fou}} \ar[d]_{\chi_{\fob,\fou}}&& X_0\\
\Spec R_{\fob} \ar[urr]_{\psi_{\fob}}&&
}
\end{equation}
with $\psi_{\fou}$, $\psi_{\fob}$ being open immersions. To describe
these maps, suppose $\fob\subseteq \ul\rho\subseteq\rho$, with
$\rho\subseteq\sigma,\sigma'\in\P_{\max}$. Choose a point $v\in
B({1\over d}\ZZ)\cap \Int(\cone\rho)$ for some $d>0$. By using the
affine chart defined in a neighbourhood of $\Int \ul\rho$ to embed
$\sigma\cup\sigma'$ into an affine space, for any $d'>0$, any point
$m\in B({1\over dd'}\ZZ)\cap (\cone\sigma\cup\cone\sigma')$  yields a
tangent vector $m-d'v\in \Lambda_x$, for $x\in\Int \ul\rho$. Now
the image of $\psi_{\fob}$ will be the open affine subset 
$U_{\rho}:=\Spec (S[B](\tilde{\bar {\mathbf{s}}}))_{(z^v)}$ of
$X_0$, where the  localization is in degree $0$.  The map
$\psi_{\fou}$ can be defined  as follows. The localized ring
$(S[B](\tilde{\bar{\mathbf{s}}}))_{(z^v)}$  is generated as an
$A[Q]/I_0$-module by elements $z^m/(z^v)^{d'}$ for $m$ ranging over
elements of $B({1\over dd'}\ZZ)\cap \Int(\cone\sigma
\cup\cone\sigma')$ for any $d'>0$. For any such $m$,  one can write 
$m-d'v=m'+a\xi$ where $m'\in\Lambda_\rho$, $a\in\ZZ$ and
$\xi=\xi(\rho)\in\Lambda_x$ is the chosen tangent vector pointing
into $\sigma$. Then
\[
\psi_{\fob}^*(z^m/(z^{v})^{d'})=
\begin{cases}
z^{m'} & a=0\\
s_{\sigma\ul{\rho}}(m'+a\xi)^{-1}z^{m'} Z_+^a & a>0\\
s_{\sigma'\ul{\rho}}(m'+a\xi)^{-1}z^{m'} Z_-^{-a} & a<0.
\end{cases}
\]

Similarly, if $\fou\subseteq\sigma$, we define $\psi_{\fou}$ by
\[
\psi_{\fou}^*(z^m/(z^{v})^{d'})=
\begin{cases}
s_{\sigma\ul{\rho}}(m')z^{m'} & a=0\\
z^{m'+a\xi} & a>0\\
0 & a<0,
\end{cases}
\]
with a similar definition reversing the roles of $a>0$ and $a<0$ if
$\fou\subseteq\sigma'$.  One checks easily that these two maps are
ring homomorphisms and that \eqref{Eqn: compatibility of open/closed} commutes.
\end{proof}
\smallskip

\ul{Modification of broken lines.} 
If sums over broken lines are to extend the definition of (certain)
monomials $z^m$ in the construction of $X_0$ to $\foX^\circ$ they
need to be modified by closed gluing data analogously. Recall from
Remark~\ref{Rem: broken lines},1 that each broken line $\beta$
defines an asymptotic monomial $m$. Denote by $\P_m\subseteq \P$ the
polyhedral subcomplex consisting of all cells $\tau$ and their faces
having $m$ as an asymptotic monomial. Each maximal cell in $\P_m$
could be the starting cell of a broken line contributing to
$\vartheta_m$. Passing between neighbouring maximal cells
$\sigma,\sigma'\in\P_m$ through $\ul\rho\subseteq \sigma\cap\sigma'$,
the initial coefficient has to change by
$s_{\sigma\ul\rho}^{-1}(\ol m)\cdot s_{\sigma'\ul\rho}(\ol m)\in
A^\times$. The evaluations $s_{\sigma\ul\rho}(\ol m)$ for
$\sigma,\ul\rho\subseteq |\P_m|$ define a one-cocycle on $\P_m$ with
values in $A^\times$, whose class in $H^1(\P_m, A^\times)$ is an
obstruction for a consistent choice of starting data of broken lines
with asymptotic monomial $m$. Once this obstruction vanishes the
choice of a maximal cell $\sigma\in\P_m$ gives finitely many
distinguished normalizations of $\vartheta_m$ by only requiring the
starting coefficient $a_1$ for broken lines asymptotically contained
in $\sigma$ to be $1$. The starting coefficient on a different maximal
cell $\sigma'\in \P_m$ is then uniquely determined from the open
gluing data by consistency. Thus in this case, the definition of
normalized in Definition~\ref{Def: broken lines} is replaced by
$a_1$ being a product of $s_{\sigma''\ul\rho}(\ol m)$ and their inverses for
$\sigma'',\ul\rho\subseteq|\P_m|$.

Fortunately, $\P_m$ is usually contractible and hence the
obstruction vanishes. This is for example the case under the natural
assumption that $B$ is asymptotically convex in the sense that for
each asymptotic monomial $m$ there is a unique minimal cell $\tau$
carrying it. Thus $\tau$ is contained in any other cell on which $m$
is an asymptotic monomial and hence $\P_m$ retracts to $\tau$. For
example, convexity is trivially true in the two-dimensional conical
case of \cite{GHK1}. Contractibility of $\P_m$ also holds for
asymptotic monomials of degree $d>0$ on $\cone{B}$ because in this
case $\P_m$ retracts to an open set of the form $W_{\tau}$ of $B$,
and such a set is contractible.

With this modification all arguments in Section~\ref{Sect: Global
functions} go through. 

\begin{remark}
\label{Prop: alternative view}
An alternative view on twisting the construction via gluing data runs as
follows. Recall that $\shP$ defines an extension of $\Lambda$ by the
constant sheaf $\ul Q^\gp$. Now $\shExt^1(\Lambda, \ul A^\times)=0$ since
$\Lambda$ is locally free, and hence $H^1(B_0,\check\Lambda\otimes_\ZZ
\ul A^\times) =\Ext^1_{B_0}(\Lambda, \ul A^\times)$ by the local to global
spectral sequence. Therefore an equivalence class of open gluing data
$\mathbf s$ yields an equivalence class of extensions
\[
0\lra \ul A^\times\lra \shP'\lra \Lambda\lra 0
\]
of abelian sheaves on $B_0$. Then $\shP_{\mathbf s}:=\shP\times_\Lambda
\shP'$
is an extension of $\Lambda$ by $\ul Q^\gp\oplus\ul A^\times$:
\begin{equation}
\label{Eqn: tilde shP}
0\lra \ul Q^\gp\oplus \ul A^\times\lra \shP_{\mathbf s}\lra \Lambda\lra 0.
\end{equation}
Since $\Lambda$ is locally free this sequence splits locally and we have
(non-canonical) isomorphisms $\shP_{\mathbf s,x}\simeq Q^\gp\oplus
A^\times\oplus \Lambda_x$ for $x\in B_0$. Define $\shP_{\mathbf s}^+\subseteq
\shP_{\mathbf s}$ as the preimage of $\shP^+\to \shP$ under the projection
$\shP_{\mathbf s}\to\shP$. From $\shP_{\mathbf s}^+$ we can first define a
sheaf of rings $\shR$ with fibre over an interior point $x$ of a
maximal cell isomorphic to $(A[Q]/I)[\Lambda_x]$ as follows. Denote by
$A[\shP_{\mathbf s}^+]$ the sheaf of monoid rings with exponents in the stalks
$\shP^+_{{\mathbf s},x}$ and coefficients in $A$. Clearly, $A[\shP_{\mathbf
s}^+]$ is a sheaf of $A[Q]$-algebras, and hence $I\subseteq A[Q]$ defines a
sheaf of ideals $\shI\subseteq A[\shP_{\mathbf s}^+]$. Moreover, there is an
embedding of $\ul A^\times$ into $A[\shP_{\mathbf s}^+]$ by mapping $a\in
A^\times$ to $a^{-1}\cdot z^{(0,a)}$ with $(0,a)$ viewed as a section of
$\shP_{\mathbf s}^+$ via \eqref{Eqn: tilde shP}. The induced action leaves
$\shI$ invariant and hence descends to the quotient. We may therefore define
\[
\shR:= \big( A[\shP_{\mathbf s}^+]/\shI\big)/\ul A^\times.
\]
Note that by the local description of $\shP_{\mathbf s}^+$ this
sheaf of rings has the predicted stalks.

This sheaf of rings is related to the construction otherwise used in
the paper as follows. For a maximal cell $\sigma$, we have a canonical isomorphism
$R_\sigma=\Gamma(\Int\sigma, \shR)$. For codimension one, the analogue of
$(A[Q]/I)[\Lambda_\rho]$ in \eqref{Def: R_fob}, which hosts $f_\fob$, is the
$A[Q]$-subalgebra $R_{\rho}$ of $\shR_x$ generated by $\Lambda_\rho$, for some
$x\in \Int\ul\rho$. For $Z_{\pm}\in \shR_x$ take some lifts of complementary
vectors $\pm\xi\in\Lambda_x$ to $(\shP_{\mathbf{s}})_x$ with $Z_+ Z_-=
z^{\kappa_{\ul\rho}}$ in $\shR_x$. Now $R_\fob$ can be defined in analogy with
\eqref{Def: R_fob} by
\[
R_\fob= R_\rho[Z_+,Z_-]/
\big(Z_+Z_--f_\fob z^{\kappa_{\ul\rho}}\big).
\]
Note that while $Z_\pm$ depends on choices, $R_\fob$ is defined
invariantly as a subring of a localization of $\shR_y$ for $y$ close
to $x\in \Int\ul\rho$ in a maximal cell. From  this point of view the
localization morphism $\chi_{\fob,\fou}$ is again defined
canonically, this time by parallel transport inside the locally
constant sheaf $\shR$. The twist in comparison with trivial gluing
data comes from global non-triviality of the extension \eqref{Eqn:
tilde shP}.
\end{remark}

The next topic concerns consistency of the lifting of a wall
structure from $B$ to $\cone{B}$, generalizing
Proposition~\ref{Prop: Consistency for cone scrS}.

\begin{proposition}
\label{Prop: Consistency for cone scrS, with
twisting}
Let $\scrS$ be a wall structure on $(B,\P)$ that is consistent for
gluing data $\mathbf s$ on $B$, with induced closed gluing data
$\bar{\mathbf s}$. If $\tilde{\mathbf s}$ and $\tilde{\bar{\mathbf s}}$ are
as given by Proposition \ref{Prop: CGD obstruction on C{B}}, then
the lifted wall structure
$\cone{\scrS}$ on $(\cone{B},\cone{\P})$ is consistent as well.
\end{proposition}

\begin{proof}
We reexamine the proof of Proposition~\ref{Prop: Consistency for
cone scrS}. For consistency in codimension zero the gluing data play
no role and the proof works as before. In codimension one we
distinguished two kinds of monomials, those lifted from $B$ (of
degree zero) and one of the form $z^{(m,a)}$ with $a>0$ and with
$(m,a)$ tangent to the joint. Consistency for monomials of degree
zero follows as before by observing that $\tilde{\mathbf s}$
restricts to $\mathbf s$ on $\Lambda$.

For $z^{(m,a)}$ equation~\eqref{Eqn: z^(m,a) consistency} still holds,
while \eqref{Eqn: (theta,theta')(z^m)} now reads
\[
(\theta\circ\chi_{\fob_1,\sigma},
\theta'\circ\chi_{\fob_1,\sigma'})(z^m)=
(\chi_{\fob_2,\sigma}, \chi_{\fob_2,\sigma'})(h),
\]
still with $h=f\cdot z^m$ and $f\in 1+I_0\cdot R_{\fob_2}$. The rest
of the argument remains unchanged.

For the argument in codimension two consistency of $z^{(0,1)}$
follows from Proposition~\ref{Prop: ob and consistency in codim 2}
characterizing the vanishing of the obstruction class
$\ob_{\Delta}(\tilde{\mathbf s})$ in terms of consistency for
monomials tangent to codimension two cells. Note that the associated
local regular function $\vartheta^\foj_{(0,1)}(p)$ for $p$ in a chamber
$\fou$ now restricts to $a_\fou  z^{(0,1)}$ in $R_\fou$ for some
$a_\fou\in A^\times$. Non-trivial closed gluing data around $\foj$
are reflected by $a_\fou\neq 1$.
\end{proof}
\medskip

The main result Theorem~\ref{Thm: Main} now generalizes. Suppose given ${\mathbf
s}$ projective consistent open gluing data on a polyhedral
pseudomanifold $(B,\P)$, $\tilde{\bar{\mathbf s}}$ a choice of lift of the
induced closed gluing data $\bar{\mathbf s}$ with $\tilde{\mathbf s}$ the
corresponding lift of open gluing data to $\cone B$, as provided by
Proposition~\ref{Prop: CGD obstruction on C{B}}. If $\scrS$ is a consistent
wall structure for ${\mathbf s}$, we obtain schemes $\foX^{\circ}$,
$\foY^{\circ}$ from the structures $\scrS$ and $\cone\scrS$ respectively. This
gives rise to rings $R_{\infty}$ and $S$ as in Theorem~\ref{Thm: Main}, and
$\foW:=\Spec R_{\infty}$, $\foY:=\Spec S$, and $\foX:=\Proj S$.

\begin{theorem}
The conclusions of Theorem~\ref{Thm: Main} continue to hold in this generalized setup
incorporating gluing data.
\end{theorem}

\begin{remark}
\label{Rem: Independence of construction for cohomologous data}
Changing both the lifted open and closed gluing data by compatible cocycles and
the wall structures accordingly leads to an isomorphism of polarized
families taking theta functions to theta functions. Without a classification of
consistent wall structures this statement is not particularly useful and we
refrain from making a formal statement. Precise results in the setup of
\cite{logmirror1}, \cite{affinecomplex} are given in the appendix.
\end{remark}

\begin{remark}
The line bundle $\foL\rightarrow\foX$ depends on the choice of lift
$\tilde{\bar{\mathbf s}}$ of $\bar{\mathbf s}$. In fact, it is is not difficult
to see that two choices of lift $\tilde{\bar{\mathbf s}}$, $\tilde{\bar{\mathbf
s}}'$ define isomorphic line bundles over $\foX$ if and only if the two lifts
are equivalent in the sense of Definition~\ref{Def: equiv lifts}.
\end{remark}

\begin{remark}
It is worthwhile emphasizing that the projectivity of the
construction only depends on projectivity of the central fibre $X_0$
over $W_0$, where in the notation of Theorem~\ref{Thm: Main} the
affine scheme $W_0\subseteq \foW$ is the fibre over
$\Spec\big(A[Q]/I_0\big)\subseteq \Spec (A[Q]/I)$. Projectivity then
automatically continues to hold for $\foX\to \foW$.
\end{remark}

In concluding this section let us emphasize that with trivial gluing
data, $X_0$ is the pull-back of a scheme over $\Spec\ZZ$ to $\Spec
A$. Without gluing data it is therefore impossible to produce
non-trivial, locally trivial deformations. Only if all locally
trivial deformations are trivial can one hope to retrieve all
deformations already with trivial gluing data. This is the case for
example in the projective smoothing of $X_0$ with all
irreducible components $\PP^2$, see \cite{GHKS}.

%===========================================================
%===========================================================

\section{Abelian varieties and other examples}
\label{Sect: Abelian varieties}

We will discuss several extended examples. The longest is a discussion
of abelian varieties. The main point is to compare our construction with
classical theta functions: this motivates the use of the term 
``theta function.'' In particular, we will show that in this case our theta 
functions equal the classical theta functions up to some explicit rescaling.
We then look at some examples with very complex wall structures, but
for which we can nevertheless say something non-trivial.

\begin{example}
\label{Expl: Riemann theta functions}
Continuing with Example~\ref{Expl: torus}, taking $B=M_{\RR}/\Gamma$, $P$, $Q$
and $\varphi_0$ as given there, let $I_0=Q\setminus \{0\}$. In this case the
empty wall structure is consistent, so we obtain for any monomial ideal $I$ of
$Q$ with radical $I_0$ a projective family $\foX\rightarrow \Spec \kk[Q]/I$ with
$\foX=\Proj S$ and with theta functions $\vartheta_m\in \Gamma(\foX,
\O_{\foX}(d))$ for $m\in B({1\over d}\ZZ)$. Taking the inverse limit of the
rings $S$ over all ideals with radical $I_0$ gives a graded ring $\widehat{S}$
over the completion $\widehat{\kk[Q]}$ of $\kk[Q]$ with respect to the ideal
$I_0$, and hence a projective family $\shX=\Spec\widehat S\rightarrow \Spec
\widehat{\kk[Q]}$. This is of course a degenerating family of abelian varieties,
a variant of constructions of Mumford \cite{mumford} and Alexeev \cite{alexeev}. 

Before examining this in more detail, let us first obtain a better  
understanding of the function $\varphi_0$. 
The periodicity relation \eqref{periodicity} generalizes
to a periodicity relation for $\varphi_0$ given by
\begin{equation}
\label{periodicity2}
\varphi_0(x+\gamma)=\varphi_0(x)+\alpha_{\gamma}(x)
\end{equation}
for $\alpha_{\gamma}:M\rightarrow Q^{\gp}$ an affine linear
function. We can write $\alpha_{\gamma}$ as a sum of a linear
function and a constant,
$\alpha_\gamma=d\alpha_{\gamma}+c_{\gamma}$. Applying
\eqref{periodicity2} for $\gamma=\gamma_1, \gamma_2$ and
$\gamma_1+\gamma_2$ gives
\[
\alpha_{\gamma_1+\gamma_2}=d\alpha_{\gamma_1}+d\alpha_{\gamma_2}+
(d\alpha_{\gamma_1}(\gamma_2)+c_{\gamma_1}+c_{\gamma_2}),
\]
so in particular $\overline{Z}:\Gamma\times\Gamma\rightarrow
Q^{\gp}$ given by
$\overline{Z}(\gamma_1,\gamma_2)=d\alpha_{\gamma_1}(\gamma_2)$ is a
symmetric form. This gives rise to a quadratic function
\[
\bar\varphi_0:M_{\RR}\rightarrow Q^{\gp}_{\RR}, \quad
\bar\varphi_0(x)={1\over 2} \overline{Z}(x,x),
\]
satisfying the periodicity
condition
\[
\bar\varphi_0(x+\gamma)=\bar\varphi_0(x)+d\alpha_{\gamma}(x)+
{1\over 2}\overline{Z}(\gamma,\gamma).
\]

Choose a basis $e_1,\ldots,e_n$ of $M$ such that $\Gamma$ is
generated by $\{f_i=d_ie_i\}$, $d_1|\cdots|d_n$ positive integers.
Denote by $f_1^*,\ldots,f_n^*$ the dual basis. As $\varphi_0$ can be
changed by an affine linear function without affecting the
construction, we can replace $\varphi_0$ with 
\begin{equation}
\label{newvarphi0 eq}
\varphi_0+\sum_{i=1}^n ({1\over
2}\overline{Z}(f_i,f_i)-c_{f_i})f_i^*-\varphi_0(0).
\end{equation}
This has the effect\footnote{This claim comes down to showing that
$\frac{1}{2}\overline{Z}(\gamma,\gamma) = c_\gamma+
\sum_{i=1}^n(\frac{1}{2}\overline{Z}(f_i,f_i)-c_{f_i}) f_i^*(\gamma)$.
This can be proved by induction, showing that the equation holds for
$\gamma\pm f_i$ if it holds for $\gamma$. The computation is
straightforward, but is omitted for its length.}
of replacing $c_{\gamma}$ with ${1\over
2}\overline{Z}(\gamma,\gamma)$ as the constant part of
$\alpha_{\gamma}$. As a consequence, $\bar\varphi_0 -\varphi_0$ is a
single-valued function $\psi$ on $B$. Note that $\varphi_0$ may no
longer have integral slopes (that is, slopes in $N\otimes Q^{\gp}$)
but after rescaling $\varphi_0$ (which has the effect of
base-changing the construction), we may assume it continues to have
integral slope. This allows us to assume a standard form for the
$c_{\gamma}$.

In any event, regardless of the choice of $\varphi_0$, there is a
standard description of the family $\widehat{\foX}\rightarrow \Spf
\widehat{\kk[Q]}$ as the quotient of a (non-finite type) fan, as
follows. Consider in $M_{\RR}\times Q^{\gp}_{\RR}$ the polytope
\[
\Xi_{\varphi_0}:=\{ (m, \varphi_0(m)+q) \,|\, q\in Q_\RR\},
\]
where $Q_\RR=\Hom(P,\RR_{\ge 0})$ and $Q=Q_\RR\cap Q^{\gp}$.
There is a lift of the $\Gamma$-action on $M_\RR$ to $M_\RR\times
Q_\RR^\gp$ leaving $\Xi_{\varphi_0}$ invariant by letting
$\gamma\in\Gamma$ act by
\[
(m,q)\longmapsto (m+\gamma, q+\alpha_\gamma(m)).
\]
Let $\Sigma$ be the normal fan to $\Xi_{\varphi_0}$ in 
$N_{\RR} \times P^{\gp}_{\RR}$ (with $N=\Hom(M,\ZZ)$).  The 
one-dimensional rays of $\Sigma$ are dual to maximal faces of
$\Xi_{\varphi_0}$. If $\sigma\in\bar\P_{\max}$ and $\tau$ is a
codimension one face of $Q_\RR$, then 
\[
\{(m, \varphi_0(m)+q)\,|\, m\in \sigma, q\in \tau\}
\]
is a maximal face of $\Xi_{\varphi_0}$ and all maximal faces are of
this form. The primitive normal vector to this face is
$(-d(\varphi_0|_{\sigma})^t (p_{\tau}), p_{\tau})$, where 
$d(\varphi_0|_{\sigma})$ is viewed as an element of
$\Hom(M,Q^{\gp})$ and its transpose as an element of
$\Hom(P^{\gp},N)$, and  $p_{\tau}\in P$ is the primitive generator
of the edge of $P_\RR=\Hom(Q,\RR_{\ge0})$ corresponding to the face
$\tau$ of $Q_\RR$.

In particular, the projection $N_{\RR}\times
P^{\gp}_{\RR}\rightarrow P^{\gp}_{\RR}$  defines a map of fans from
$\Sigma$ to the fan of faces of $P_\RR$. If $X_{\Sigma}$ denotes the
toric variety (not of finite type) defined by $\Sigma$, we obtain a
flat morphism  $f:X_{\Sigma}\rightarrow \Spec \kk[Q]$. The action of
$\Gamma$ on $\Xi_{\varphi_0}$ induces an action on $\Sigma$ by
taking the transpose of its linear part. Thus $\gamma\in\Gamma$ acts
by $(n,p)\mapsto (n+(d\alpha_{\gamma})^t(p),p)$. Here we view
$d\alpha_\gamma$ as a homomorphism $M\to Q^\gp$, and its transpose
$(d\alpha_\gamma)^t$ accordingly as a homomorphism $P^\gp\to N$.
In turn, $\Gamma$ acts on $X_{\Sigma}\times_{\Spec\kk[Q]} \Spec
\kk[Q]/I$ for any monomial ideal $I$ with radical $I_0$, and the
quotient
\begin{equation}
\label{eq: Xsigmaquotient}
(X_{\Sigma}\times_{\Spec\kk[Q]} \Spec \kk[Q]/I)/\Gamma\rightarrow
\Spec \kk[Q]/I
\end{equation}
can be seen to coincide with $\foX\rightarrow\Spec \kk[Q]/I$.

Using this description, theta functions on $\foX$ are traditionally
seen by extending the action of $\Gamma$ on $M \times Q^{\gp}$ to an
action of $\Gamma$ on $M \times Q^{\gp}\times \ZZ$ which preserves
the cone $C(\Xi_{\varphi_0})$ and the last factor $\ZZ$.  This lifts
the $\Gamma$-action on $X_{\Sigma}$ to the total space of the line
bundle $\tilde\shL$ induced by the polytope $\Xi_{\varphi_0}$.  This
lifting is given as follows. The action on $M\times Q^{\gp}$ is given 
by $\gamma$
acting via $(m,q)\mapsto (m+\gamma, q+d\alpha_{\gamma}(m)+c_\gamma)$ with
$c_{\gamma}\in Q^{\gp}$ the constant part of $\alpha_{\gamma}$ as
before. The extension is then given by
\[
T_{\gamma}:(m,q,r)
\mapsto (m+r\gamma, q+d\alpha_{\gamma}(m) + r c_{\gamma},r).
\]
Now for $r\in \ZZ_{>0}$, the integral points of $r\Xi_{\varphi_0}$
correspond to sections of the line bundle $\tilde\shL^{\otimes r}$
on $X_{\Sigma}$. Using the action of $\Gamma$ on the total space of
$\tilde\shL$, compatible with the action of $\Gamma$ on
$X_{\Sigma}$, the line bundle $\tilde\shL$ descends to a line bundle
$\shL$ on $\foX$. Similarly, $\Gamma$-invariant sections of
$\tilde\shL$ descend to sections of $\shL$.  We then obtain theta
functions on $\foX\rightarrow \Spec\kk[Q]/I$ by constructing
$\Gamma$-invariant sections on $X_{\Sigma}\times_{\Spec\kk[Q]}
\Spec\kk[Q]/I$:  for $m\in B({1\over r} \ZZ)$,  we have
\begin{equation}
\label{thetamequation}
\vartheta_m=\sum_{\gamma\in\Gamma}
z^{T_{\gamma}(rm, r\varphi_0(m), r)}
= \sum_{\gamma\in\Gamma} z^{(r(m+\gamma), r\varphi_0(m+\gamma),r)}.
\end{equation}

Before comparing this formula with the one given by jagged paths,
let us compare the above formula with the classical notion of theta
function; indeed, it is this comparison which justifies the use of
the term ``theta function.'' To do so, we work
complex analytically, with $\kk=\CC$. There is some analytic open
neighbourhood $S\subseteq\Spec\kk[Q]$ of the  zero-dimensional
stratum such that the action of $\Gamma$ on $f^{-1}(S)$ is free,
giving $\bar f:\shX:=f^{-1}(S)/\Gamma\rightarrow S$ an analytic
degeneration of abelian varieties. For $p\in S \cap
\Spec\kk[Q^{\gp}]\subseteq P^{\gp}\otimes \Gm$, $f^{-1}(p)$ is
canonically the algebraic torus $N\otimes\Gm$ and $\bar f^{-1}(p)$
is the abelian variety $(N\otimes\Gm)/\Gamma$ where
$\Gamma\hookrightarrow N\otimes\Gm$ via $\gamma\mapsto
(d\alpha_{\gamma})^t(p)$. 

In terms of a period matrix, note the choice of basis $\{f_i\}$
define coordinates $u_1,\ldots,u_n$ on $N\otimes \CC$ via pairing
with $f_i$. Said differently, $u_i$ is the pull-back of $z^{f_i}\in
\CC[M]$ via the exponential map
\[
N\otimes \CC\lra N\otimes\GG_m,\quad
\sum_i e^*_i\otimes \lambda_i\longmapsto
\sum_i e^*_i\otimes e^{2\pi\sqrt{-1} \lambda_i}.
\]
Note that the kernel of this map is generated by $e_1^*=
d_1f_1^*,\ldots, e_n^*=d_nf_n^*$. Then using the coordinates
$\{u_i\}$, one sees that $\bar f^{-1}(p)$ has period matrix $(D,
Z(p))$, where $D=\mathrm{Diag}(d_1,\ldots,d_n)$ and
\[
Z(p)_{ij}=\langle (d\alpha_{f_j})(f_i), 
{\log p\over 2\pi\sqrt{-1}}\rangle.
\]
Here $\log$ is a local choice of inverse to $\exp: P^{\gp} \otimes
\CC \rightarrow P^{\gp}\otimes \Gm$. This is the matrix in the
basis $\{f_i^*\}$ for the $\CC$-valued bilinear form
$Z(p)(\gamma_1,\gamma_2)=\langle \bar Z(\gamma_1,\gamma_2), {\log
p\over 2\pi\sqrt{-1}}\rangle$. Note that $Z(p)$ defines the period
point in the Siegel upper half space $\mathfrak{H}_n =\{Z\in
M(n,\CC)\, |\, Z=Z^t,\, \operatorname{Im} Z>0\}$.

In what follows, we will assume that $c_{\gamma}=
{1\over 2}\bar Z(\gamma,\gamma)$ and $\varphi_0(0)=0$, as is achieved by
\eqref{newvarphi0 eq}.
Now the line bundle $\tilde\shL$ on $X_{\Sigma}$ is trivialized when
restricted to $f^{-1}(p)$, and we can choose, say, 
$z^{(0,\varphi_0(0),1)}=z^{(0,0,1)}$ as a trivializing section. Then as a
regular function on $f^{-1}(p)$, for $m\in B(\ZZ)$, the theta
function $\vartheta_m$ takes the form $\sum_{\gamma\in\Gamma} 
z^{m+\gamma}z^{\varphi_0(m+\gamma)}(p)$. Writing this as a function 
on $N\otimes\CC$ and using $\zeta=\sum_i u_if_i^*$, gives an 
expression as a function of the $u_i$ 
\[
\vartheta_m=\sum_{\gamma\in\Gamma}  z^{\varphi_0(m+\gamma)}(p)
\exp(2\pi \sqrt{-1} \langle \zeta, m+\gamma\rangle).
\]
Finally, writing
$\varphi_0(m+\gamma)=\psi(m)+\bar\varphi_0(m+\gamma)$ (as $\psi$ is
single-valued), this becomes
\[
\vartheta_m=
z^{\psi(m)}(p)\sum_{\gamma\in\Gamma}
\exp(\pi\sqrt{-1}Z(p)(m+\gamma,m+\gamma)
+2\pi\sqrt{-1}\langle \zeta, m+\gamma\rangle).
\]
Except for the scale factor $z^{\psi(m)}(p)$, after 
writing the exponent in terms of the basis $f_i$, this gives the
standard form for the classical theta function
$\vartheta\left[\begin{matrix}c^1\\ 0\end{matrix}\right](\zeta,Z)$ where
$c^1=(m_1,\ldots,m_n)$, $m=\sum m_if_i$, see
e.g.~\cite{birkenhakelange}, p.223).

We now compare the formula \eqref{thetamequation} for $\vartheta_m$ 
with the description given by jagged paths. Fix any maximal cell
$\sigma\in \P$ and $p\in\sigma$ a chosen basepoint. For $m\in
B({1\over r}\ZZ)$, the set of all jagged paths from $m$ to $p$ is
easily described via a factorization through the universal cover
$\pi:M_{\RR}\rightarrow M_{\RR}/\Gamma$ of $B$. Fixing one lift
$\tilde m$ of $m$ to $M_{\RR}$, any lift $\tilde p$ of $p$ to
$M_{\RR}$ yields a jagged path $\tilde\gamma$ whose image is just a
straight line joining $\tilde m$ to $\tilde p$. The composition with
$\pi$ gives a jagged path in $B$.

The resulting theta function, described as a sum of monomials
indexed by  jagged paths, is easily compared with
\eqref{thetamequation}. Indeed, the sheaf $\Lambda$ on $B$ is just
the constant sheaf with stalk $M$, the sheaf $\shP$ has stalk
$M\times Q^{\gp}$ with monodromy the linear part of the action of
$\Gamma$ on $M\times Q^{\gp}$ by affine transformations described
above, and the sheaf  $\widetilde\shP$ has stalk $M\times
Q^{\gp}\times\ZZ$ with monodromy given by the action of $\Gamma$ on
this latter group as described above. Thus, after choosing a local
representative for $\varphi_0$ near $m$, say the representative
given by $\varphi_0$ in a neighbourhood of $\tilde m$, we see
$(r\tilde m, r\varphi_0(\tilde m), r)$ represents $r\cdot
\varphi_{0*}(\ev_m)$, and $T_{\gamma}(r\tilde m, r\varphi_0(\tilde
m),r)$ represents parallel transport of $r\cdot\varphi_{0*}(\ev_m)$
around a loop  corresponding to $\gamma\in\Gamma$. In particular,
after parallel transport to $p$, we see that $\vartheta_m$ coincides
with $\vartheta_m$ as defined using jagged paths, see
\S\ref{Subsect: Jagged paths}.

\medskip

The theta functions constructed above are not all theta functions on
an abelian variety. Indeed, given an ample line bundle $\shL$ on an
abelian variety $A$, and $t_x:A\rightarrow A$ denoting translation
by an element $x\in A$, $t_x^*\shL$ is isomorphic to $\shL$ only for
finitely many values of $x$. To see this write $K(\shL)$ for the
kernel of the map $A\rightarrow A^{\vee}$ given by $x\mapsto
(t_x^*\shL)\otimes \shL^{-1}$. By ampleness of $\shL$ this kernel is
a finite group. 

To identify such line bundles in our construction, we need to use gluing
data as in \S \ref{Sect: Parameters}. Indeed, we continue to use trivial
gluing data to construct our family $\shX\rightarrow \Spec
\widehat{\kk[Q]}$, but there is a choice of lifting trivial gluing data to
$\cone{B}$ as described in Proposition~\ref{Prop: CGD obstruction on C{B}}.
This tells us that the set of liftings
$\tilde{{\mathbf s}}$ of the trivial gluing data $\mathbf s$ up to equivalence
is canonically in bijection with $H^1(B,\kk^\times)$.

The technically easiest way to think about such a choice of lifts is
to use the description of $\widehat{\foX}\rightarrow \Spf \widehat{\kk[Q]}$
as a quotient, and to any finite order, we can use
\eqref{eq: Xsigmaquotient}. In particular, the cohomology group $H^1(B,
\kk^\times)$ can be represented by the group cohomology 
$H^1(\Gamma,\kk^\times)$, with trivial action of $\Gamma$ on $\kk^\times$.
Then $H^1(\Gamma,\kk^\times)\cong \Hom(\Gamma,\kk^\times)$, so
we view gluing data $\tilde{\mathbf s}$ lifting the trivial open gluing
data as a map $\tilde{\mathbf s}:\Gamma\rightarrow\kk^\times$. We can
then use this to twist the action of $\Gamma$ on the line bundle $\tilde
\shL$ on $X_{\Sigma}$, by $\gamma$ acting on a monomial section $z^{(m,q,r)}$ of
$\tilde\shL^{\otimes r}$ by taking it to 
$\tilde{\mathbf s}(\gamma)z^{T_{\gamma}(m,q,r)}$. Thus again we can look
at $\Gamma$-invariant sections under this new action, getting for
$m\in B({1\over r}\ZZ)$,
\[
\vartheta_m=\sum_{\gamma\in\Gamma} \tilde{\mathbf s}(\gamma)z^{(r(m+\gamma),
r\varphi_0(m+\gamma),r)}.
\]
Again, before comparing this with what we get from broken lines, 
let us compare this expression with classical theta functions. Restricting
to a point $p \in S\cap\Spec\kk[Q^{\gp}]$ as before, this becomes
\begin{equation}
\label{generalthetaeq}
\vartheta_m=
z^{\psi(m)}(p)\sum_{\gamma\in\Gamma}
\tilde{\mathbf s}(\gamma)\exp(\pi\sqrt{-1}Z(p)(m+\gamma,m+\gamma)
+2\pi\sqrt{-1}\langle \zeta, m+\gamma\rangle).
\end{equation}

This can be interpreted as a classical theta function as follows.
Because the period matrix of $f^{-1}(p)$ is $(D, Z(p))$, viewing
$Z(p)$ as giving a map $Z(p):M_{\RR}\rightarrow N_{\CC}$ via
$Z(p)(m)=Z(p)(m,\cdot)$, any element of $N_{\CC}$ can be written
uniquely as $Z(p)c^1+c^2$ for some vectors $c^1\in M_{\RR}$,
$c^2\in N_{\RR}$. In particular, there are such vectors $c^1, c^2$
such that $\tilde{\mathbf s}(\gamma)=\exp(2\pi\sqrt{-1}\langle Z(p)c^1+c^2,
\gamma\rangle)$ for all $\gamma\in\Gamma$. Thus we can write
\[
\vartheta_m=z^{\psi(m)}(p) \sum_{\gamma\in\Gamma}
\textstyle\frac{\exp\big(\pi\sqrt{-1} Z(p)(m+\gamma+c^1,m+\gamma+c^1)
+2\pi\sqrt{-1}\langle \zeta+c^2,m+\gamma+c^1\rangle\big)}{
\exp\big( \pi\sqrt{-1}Z(p)(c^1,c^1)+ 2\pi\sqrt{-1}(Z(p)(m,c^1)
+\langle \zeta+c^2,c^1\rangle+\langle c^2,m\rangle)
\big)}.
\]
We notice the denominator is nowhere zero and independent of
$\gamma$. Recall $\vartheta_m$ defines a section of a line bundle on
$f^{-1}(p)$
by trivializing the pull-back of the line bundle to the universal cover
$N\otimes\CC$. A different choice of trivialization is determined
by an entire invertible function. In particular, after changing this
trivialization (thereby changing the factor of automorphy determining
the line bundle, see \cite{birkenhakelange}, \S 2.1), the above 
function describes the same section of a line bundle as
\begin{align*}
& \quad\quad\quad
z^{\psi(m)}(p)
\vartheta\left[ \begin{matrix} m+c^1 \\ c^2\end{matrix}\right]\\
{} = &
z^{\psi(m)}(p)\sum_{\gamma\in\Gamma}
\exp\big(\pi\sqrt{-1}Z(p)(m+\gamma+c^1,m+\gamma+c^1)+2\pi\sqrt{-1}\langle
\zeta+c^2,m+\gamma+c^1\rangle\big).
\end{align*}

To see that the expression \eqref{generalthetaeq} agrees with that
given by jagged paths, it is easiest to use the description of
parallel transport of monomials given by Remark \ref{Prop:
alternative view}. Indeed, the element $\tilde{\mathbf
s}:\Gamma\rightarrow\kk^\times$  defines an extension $\shP'$ of
$\Lambda$ by $\ul{\kk}^\times$. This extension is trivial when
$\shP'$ is pulled back to $M_{\RR}$, with $\gamma\in \Gamma$ acting
on $\kk^\times\times M$ by  $(s,m)\mapsto (s \tilde{\mathbf
s}(\gamma),m+\gamma)$. Then parallel transport of monomials along
jagged paths as already described in the case of trivial gluing data
will provide the formula for $\vartheta_m$ in 
\eqref{generalthetaeq}.

We note that this description of these more general theta functions is
not really canonical either from the point of view of classical theta
functions (as we need to change the trivialization on $N\otimes\CC$,
implying a change in factor of automorphy) or from the point of view of
homological mirror symmetry. From the latter point of view, $B({1\over d}\ZZ)$
is not the natural parameterizing set for theta functions, but rather
this set translated by the vector $c^1$. Indeed, the Lagrangian mirror
to $t_x^*\shL$ should be a translate of the Lagrangian mirror to $\shL$.
The expectation from \cite{PZ}
is that the image under the SYZ fibration of the intersection
points between this translated Lagrangian and the zero section of the SYZ
fibration is this translated set. It is possible there is a more natural
way to represent these theta functions corresponding to translated line
bundles than done here.
\end{example}

\begin{example}
\label{P2example}
(Cf. \cite{CPS}, Example~2.4)
Let $B$ be the triangle in $\RR^2$ with vertices $v_1=(-1,-1)$,
$v_2=(-1,2)$ and $v_3=(2,-1)$. Let $\P$ be the star decomposition
of  $B$, that is, each two-dimensional cell of $\P$ is the convex
hull of $0$ and an edge of $B$.

We will take the base monoid to be $Q=\NN$, writing $\kk[Q]=\kk[t]$,
and the PL function $\varphi$ to be single-valued with
\[
\varphi(0)=0, \quad \varphi(-1,-1)=\varphi(-1,2)=\varphi(2,-1)=1.
\]
Note the sheaf $\shP$ is the
constant sheaf with coefficients $\ZZ^2\oplus Q^{\gp}$, and we write the
monomials $x:=z^{(1,0,0)}$, $y:=z^{(0,1,0)}$.

We construct a structure $\scrS$ for this data as follows. First consider
the structure
\[
\scrS_{\inc}:=\{(\rho_1,1+tx^{-1}y^{-1}), 
(\rho_2, 1+tx^{-1}y^{2}), (\rho_3,1+tx^{2}y^{-1})\},
\]
where $\rho_i$ is the edge of $\P$ connecting $0$ to $v_i$. Note
that we do not bother to subdivide the $\rho_i$ (hence abandoning
the notation $\ul{\rho}_i$) as the affine structure does not
have any singularities. By applying the Kontsevich-Soibelman lemma
(see e.g., \cite{TGMS}, Theorem 6.38 for the simplest statement that
incorporates the case needed here) we obtain a structure
$\scrS\supseteq \scrS_{\inc}$ which is consistent, by \cite{CPS},
Lemma~4.7 and Lemma~4.9. Consistency at the boundary follows by the
convexity criterion Proposition~\ref{Prop: Boundary consistency}.
All walls added to obtain $\scrS$ are of the form
\[
\big ((-\RR_{\ge 0} m_0)\cap B, 1+\sum_{\ol m\in \RR_{>0} m_0}
c_mz^m\big).
\]
Technically, this $\scrS$ is not quite a structure according to our
definition because there might be distinct walls with the same
support. However a standard structure can be obtained by replacing
all walls with the same support with a single wall whose attached
function is a product. Using this structure, we can build a family
$\foX_k=\Proj S_k$ over the ring $\kk[Q]/I_k$ with $I_k=(t^{k+1})$
for each $k$. Taking the inverse limit of the $S_k$ gives a graded
ring $\widehat{S}$, getting a projective family $\shX\rightarrow T$
with
\[
T=\Spec \lim_{\longleftarrow} \kk[Q]/I_k=\Spec \kk\lfor t\rfor.
\]
We sketch the properties of this family.

For $k=0$, $\foX_0$ is a union of three toric varieties, each a weighted
projective space.
To see what the generic fibre of $\shX\rightarrow T$ is,
let us analyze local models near the vertices
of $B$. Without loss of generality, consider the vertex $v_1=(-1,-1)$. Our
construction involves gluing the spectra of various rings. Explicitly,
the ring $R_{\rho_1}$ given by \eqref{Def: R_fob} is
\[
R_{\rho_1}=(\kk[Q]/I_k)[\Lambda_{\rho_1}][Z_+,Z_-]/(Z_+Z_--
(1+tx^{-1}y^{-1})f_{\rho_1}t),
\]
where 
\[
f_{\rho_1}=\prod_{(\fod,f_{\fod})\in \scrS\setminus\scrS_{\inc}
\atop \fod=\rho_1} f_{\fod}.
\]
We also have two rings $R_{\sigma_{\pm},\rho_{\pm}}$ where
$\sigma_{\pm}$ are the two two-cells containing $\rho_1$ and
$\rho_{\pm}$ are the corresponding edges of $B$ containing $v_1$. Taking
$\rho_+$ to have vertices $v_1$ and $v_2$, we have by 
\eqref{Eqn: R_sigmarho} 
\[
R_{\sigma_+,\rho_+}=(\kk[Q]/I_k)[x,y^{\pm 1}],
\quad
R_{\sigma_-,\rho_-}=(\kk[Q]/I_k)[x^{\pm 1},y].
\]
We glue $\Spec R_{\rho_1}$ and $\Spec R_{\sigma_+,\rho_+}$ by localizing
at $Z_-$ and $x$ respectively, and then identifying the generator $(1,1)$ of
$\Lambda_{\rho_1}$ with $xy$ and $Z_+$ with $y$. We glue $\Spec R_{\rho_1}$
and $\Spec R_{\sigma_-,\rho_-}$ by localizing at $Z_+$ and $y$ respectively,
and then identifying the generator $(1,1)$ of $\Lambda_{\rho_1}$ with
$xy$ and $Z_-$ with $y^{-1}$.

It is easy to check that the ring of regular functions on this glued scheme
can then be written as
\[
R^k_{v_1}:=(\kk[Q]/I_k)[X,Y,W]/(XY+tW(1+tW^{-1})f_{\rho_1}(t,W)).
\]
The inclusion of this ring in $R_{\rho_1}$ is given by
\begin{align*}
X \mapsto {} & z^{(1,1)} Z_- \quad (z^{(1,1)}\in\kk[\Lambda_{\rho}])\\
Y \mapsto {} & Z_+\\
W \mapsto {} & z^{(1,1)}.
\end{align*}
Note that the ideal is generated by $XY+t(W+t)f_{\rho_1}(t,W)$ and
$f_{\rho_1}$ is congruent to $1$ modulo $t$.

The scheme $\Spec R^k_{v_1}$ is the affine completion of this glued
scheme.  There are similar descriptions of rings $R^k_{v_2}$,
$R^k_{v_3}$, and the three schemes $\Spec R^k_{v_i}$ cover
$\foX_k\setminus \{0\}$, where $0$ is the point in $\foX_0$ where
the three irreducible components meet.

The boundary $\foD_k$ of $\Spec R_{v_1}^k\subseteq\foX_k$ is given by
$W=0$, see Remark \ref{Rem: Boundary divisor}. Thus we see that our
construction gives a family of pairs $(\shX,\shD)\rightarrow S$, and
locally the equation for $\shD$ to order $k$ is
$XY+t^2(1+t(\cdots))=0$, clearly a smoothing of $XY=0$. Thus
$\shD\rightarrow T$ is a smoothing of a triangle of $\PP^1$'s. 

We next claim that with $\eta=\Spec\kk((t))$, the generic fibre
$\shX_{\eta}$ of $f:\shX\rightarrow T$ is smooth. We sketch the
argument. For sufficiently large $k$, it is easy to check that the
singular locus of the map $f_k:\Spec R^k_{v_i}\rightarrow \Spec
\kk[Q]/I_k$ is not scheme-theoretically surjective in the sense of
\cite{GHK1}, Definition-Lemma 4.1 and following.  Furthermore, the
argument of \S 4 of \cite{GHK1} shows a similar statement for a
neighbourhood of $0$ in $\foX_k$. Thus the singular locus of
$\foX_k\rightarrow\Spec \kk[Q]/I_k$ is not  scheme-theoretically
surjective for large $k$.  Now consider the singular locus of $f$. 
The formation of singular locus commutes with base change (see
\cite{GHK1}, Definition-Lemma 4.1 again). As $\shX\times_T \Spec
\kk[Q]/I_k =\foX_k$ for any $k$, we must not have $\Sing(f)$
surjecting onto $T$. Since $f$ is proper, this means $\Sing(f)$ is
disjoint from $\shX_{\eta}$, and the latter scheme is smooth over
$\eta$.

To identify $\shX_{\eta}$, we proceed as follows. Note that the
relatively ample line bundle $\shL$ on $\shX$ restricts to a very
ample line bundle on $\foX_0$ by inspection, embedding $\foX_0$ as a
surface of degree $9$ in $\PP^9$. This implies $\shL|_{\shX_{\eta}}$
is also very ample. Note also that $\shL|_{\foX_0}\cong
\omega_{\foX_0}^{-1}$, again by inspection. As
$H^1(\foX_0,\O_{\foX_0})$ can be calculated to be $0$, we also have
$\shL|_{\shX_{\eta}}\cong \omega_{\shX_{\eta}}^{-1}$. Thus, passing
to $\bar\eta=\Spec \overline{\kk((t))}$, we see $\shX_{\bar\eta}$ is
a del Pezzo surface of degree $9$. Here $\overline{\kk((t))}$
denotes the algebraic closure of $\kk((t))$.  From the
classification of del Pezzo surfaces, we have
$\shX_{\bar\eta}\cong\PP^2_{\bar\eta}$. So $\shX_{\eta}$ is a
Brauer-Severi scheme over $\eta$.

A result of Witt (see e.g., \cite{GiSz}, Corollary 6.3.7) implies
that if  $\kk$ is an algebraically closed field of characteristic
zero, then the Brauer group of  $\kk((t))$ is trivial. Thus
$\shX_{\eta}\cong\PP^2_{\eta}$.

There remains the question of describing the theta functions we have
constructed on $\PP^2_{\eta}$. We do not see at this point how to
describe these functions completely. The structure $\scrS$ is
expected to be very complicated, containing non-trivial rays of
every rational slope. Hence it is likely to be very difficult to
control jagged paths. Nevertheless, there is a certain amount of
symmetry which gives us some information.

There is an action of the group $H=\ZZ_3^2$, generated by $\alpha$
and $\beta$, on the data of our construction. The generator $\alpha$
acts on $B\subseteq\RR^2$ as the linear transformation
$\begin{pmatrix}  0&1 \\ -1&-1\end{pmatrix}$. This action preserves
$\varphi$, and we then get an action on monomials lifting the action
on $\RR^2$. In particular, this gives an action on walls taking
$(\fop, 1+\sum c_m z^m)$ to $(\alpha(\fop), 1+\sum c_m
z^{\alpha(m)})$. One sees that $\scrS_{\inc}$ is preserved by this
action, and hence so is $\scrS$. Further,  for $\tau\in\P$ of
codimension $0$ or $1$, the ring $R_{\tau}$ is canonically
identified using $\alpha$ with $R_{\alpha(\tau)}$. Thus $\alpha$
acts as an automorphism of $\foX_k/\Spec(\kk[Q]/I_k)$.

The generator $\beta$ leaves $B$ fixed, but acts on monomials via
\[
\beta(x)=\zeta x, \quad \beta(y)= \zeta^2 y, \quad \beta(t)=t,
\]
where $\zeta$ is a primitive third root of unity. Again $\scrS_{\inc}$ is
left invariant under this action, so the same is true of $\scrS$. Then $\beta$
also acts as automorphisms of the rings $R_{\tau}$, hence again $\beta$
induces an automorphism of $\foX_k/\Spec(\kk[Q]/I_k)$.

In conclusion, the group $H$ acts on $\foX_k/\Spec(\kk[Q]/I_k)$ for all $k$ 
(with the trivial action on $\Spec\kk[Q]/I_k$) and hence acts on $\shX/T$.
This action preserves $\shD$, and is clearly non-trivial on $\shD$ since
it permutes the components of $\shD_0$. 

Note furthermore that $\shD_{\eta}$ has a point over $\eta$: certainly
$\shD_0$ has many $\kk$-valued points which are non-singular points of
$\shD_0$, so by Hensel's lemma $\shD\rightarrow\Spec\kk\lfor t\rfor$ has
a section. Thus $\shD_{\eta}$ has a $\kk((t))$-valued point, and hence has
the structure of an abelian variety.

In particular, if $\phi\in H$ induces an automorphism
$\phi:\shD_{\eta}\rightarrow\shD_{\eta}$, then
$\phi^*\O_{\PP^2}(1)|_{\shD_{\eta}} \cong
\O_{\PP^2}(1)|_{\shD_{\eta}}$. Since the $j$-invariant of
$\shD_{\eta}$ is non-constant, the only choice for such an
automorphism is translation by an element in the kernel of the
polarization $\O_{\PP^2}(1)|_{\shD_{\eta}}$. It is then standard
(see e.g., \cite{Mu1}) that the group action of $H$ lifts to an
action on  $\O_{\PP^2}(1)|_{\shD_{\eta}}$ after passing to a central
extension \[ 1\rightarrow \Gm \rightarrow G \rightarrow H\rightarrow
0. \] Here $G$ is the Heisenberg group. The representation of $H$ on
$H^0(\shD_{\eta}, \O_{\PP^2}(1)|_{\shD_{\eta}}) =\kk((t))^{\oplus
3}$ is then isomorphic to the Schr\"odinger representation, that is,
there is a basis $x_0,x_1,x_2$ (with indices taken modulo $3$) of 
$H^0(\shD_{\eta}, \O_{\PP^2}(1)|_{\shD_{\eta}})$ with (lifts of)
$\alpha$ and $\beta$ acting by \[ \alpha(x_i)=x_{i+1}, \quad
\beta(x_i)=\zeta^ix_i \] for $\zeta\in\kk$ a primitive third root of
unity.

Now consider the theta function $\vartheta_0$ corresponding to $0\in
B(\ZZ)$. Because the monomial corresponding to $0$ is left invariant
by $G$ and the scattering diagram itself is invariant under the
action of $G$, it follows that  $\vartheta_0$ is invariant. Since
$\vartheta_0|_{\shX_{\eta}}$ is a section of
$\omega_{\shX_{\eta}}^{-1}\cong \O_{\PP^2}(3)$, $\vartheta_0$ must
be an invariant cubic. In the coordinates $x_0,x_1,x_2$ in which the
Schr\"odinger representation is described above, the general such
invariant cubic is $\lambda_1(x_0^3+x_1^3+x_2^3)
+\lambda_2x_0x_1x_2$. Here $\lambda_1,\lambda_2 \in \kk((t))$. It is
also easy to see that $\vartheta_0$ vanishes on $\shD$. Thus the
equation of $\shD_{\eta}$ is given by the above cubic, for some
choice of $\lambda_1,\lambda_2$, and $\vartheta_0$ takes the same
form. Moreover, $\shD$ is the result of applying our construction to
$\partial B$ with its decomposition into nine unit intervals. In
dimension one our construction is purely toric and, for $B=S^1$,
produces a Tate curve, with known $j$-invariant. Hence the quotient
$\lambda_1/\lambda_2$ can be computed from this $j$-invariant of
$\shD$.

If $p\in B(\ZZ)\setminus \{0\}$, then one can classify jagged paths
for $p$ which are contained entirely in $\partial B$. Indeed,
because of the form of $\scrS$, a jagged path which starts at
$p\in\partial B$ can only bend at a ray of
$\scrS\setminus\scrS_{\inc}$ if it bends outwards. Thus jagged paths
contained in $\partial B$ can only bend at the rays of
$\scrS_{\inc}$, and a simple calculation shows that such jagged
paths must bend as much as possible whenever a ray of $\scrS_{\inc}$
is crossed. Using this, one can compare $\vartheta_p|_{\shD}$ with
theta functions on $\shD$. One finds the description as given in
the earlier part of this section. We omit the details. 
\end{example}

\begin{example}
Consider the family $\shX\rightarrow S$ 
of quartic K3 surfaces in $\PP^3\times S$ given by
the equation
\begin{equation}
\label{quarticfamily}
s(x_0^4+x_1^4+x_2^4+x_3^4)+x_0x_1x_2x_3=0,
\end{equation}
where $S$ is the spectrum of a discrete valuation ring over a field
$\kk$ with uniformizing parameter $s$. This is a toric degeneration
(see Example 4.2 in \cite{logmirror1}). If we use the polarization
given by $\O_{\PP^3}(1)|_{\shX}$,  we obtain an intersection complex
$(B,\P)$ which is a union of four standard simplices, forming a
tetrahedron. There is in fact an affine structure with singularities on
$B$ which extends across the vertices. As in \cite{logmirror1},
this is specified by defining a \emph{fan structure} at each vertex,
i.e., an identification of a neighbourhood of each vertex with the neighbourhood
of $0$ of a fan. (See \cite{gokova}, Example 2.10 for the dual
intersetion complex version of this example.)
The fan structure at each vertex is given by the fan
for $\PP^2$, as the degeneration is normal crossings at the
zero-dimensional strata of $\shX_0$. Using $(B,\P)$, we can work
backwards and construct a smoothing of $\shX_0$ using the algorithm
of  \cite{affinecomplex}  to construct a consistent structure. The
initial data used to construct this structure is induced by the log
structure on $\shX_0$ coming from the inclusion $\shX_0\subseteq
\shX$. Note that $B$ has six singularities, one each at the
barycenter of each edge. There are then initial walls emanating from
each singular point, with attached function of the form $1+z^{4m}$,
where $m$ is primitive with $\ol m$ tangent to the edge. We omit the
details.

From this initial data, \cite{affinecomplex} gives
a consistent structure, giving a polarized
deformation $\foX\rightarrow \Spec \kk\lfor t \rfor$ as usual, with
$\shL$ the line bundle on $\foX$. For a simpler exposition of this
result in two dimensions, see \cite{TGMS}, Chapter 6.
Then $B(\ZZ)$ consists of the vertices
of $\P$, giving four theta functions $\vartheta_0,\ldots,\vartheta_3$ which
are sections of $\shL$. By construction, these can be chosen so their
restriction to the central fibre gives $x_0,\ldots,x_3$. Since $\shL$ is
very ample when restricted to the central fibre, it is also very ample when
restricted to the generic fibre $\foX_{\eta}$. It is then clear that
$\shL|_{\foX_{\eta}}$ embeds $\foX_{\eta}$ as a quartic surface in 
$\PP^3_{\eta}$, so one can ask which quartic equation is satisfied by
the $\vartheta_i$'s. 

To see this, one can observe as in Example \ref{P2example} that
$\foX$ has a large symmetry group. First observe that $B$ has an
action of a  permutation of the vertices, and this action preserves
the initial structure, and hence preserves the structure defining
$\foX$. This action acts on the theta functions $\vartheta_i$ by
permutation also.

We can also find an action of multiplication by fourth roots of unity.
Indeed, one can easily check that the monodromy of $\Lambda$ 
around each singular
point of $B$ takes the form $\begin{pmatrix}1&4\\ 0&1\end{pmatrix}$
in a suitable basis. Thus the local system $\Lambda\otimes_{\ZZ} \ZZ/4\ZZ$
has no monodromy, and hence is trivial. Fix an isomorphism 
$\Lambda\otimes_{\ZZ} \ZZ/4\ZZ$ with the constant sheaf with stalk 
$(\ZZ/4\ZZ)^2$; this can be done by fixing an isomorphism $\Lambda_x
\cong \ZZ^2$ at some point $x\in B_0$. In particular, given any character
$\chi:(\ZZ/4\ZZ)^2\rightarrow \kk^\times$, we obtain a map $\chi:\shP
\rightarrow \kk^\times$ via the factorization $\shP\rightarrow\Lambda
\rightarrow \Lambda\otimes_{\ZZ}(\ZZ/4\ZZ)\cong (\ZZ/4\ZZ)^2\rightarrow 
\kk^\times$. Thus such a $\chi$ gives a well-defined action on monomials.
Because of the form of the initial walls of the structure, the structure
is left invariant under this action, as are all relevant rings and gluing
maps. Hence $\chi$ acts on $\foX$. 

If we take $H$ to be the group $S_4\times \Hom((\ZZ/4\ZZ)^2,\kk^\times)$,
then there is a central extension 
\[
1\rightarrow \Gm\rightarrow G\rightarrow H\rightarrow 1
\]
of $H$ which acts on the line bundle $\shL$ and hence on its space of
sections. It is easy to see that there is a lift of 
an element $\alpha\in S_4$ to $G$ acting by the
corresponding permutation on the $\vartheta_i$. 

Given a character $\chi$, it is clear that the theta functions are also 
eigensections of the action of a lift of $\chi$ to $G$. Continuing to write
such a lift as $\chi$, since we can modify a lift by an element of 
$\Gm$, we can always assume $\chi(\vartheta_0)=\vartheta_0$. 
Once this is fixed, the action of $\chi$ on the other $\vartheta_i$
is determined. To see this explicitly, let $z^{m_i}$ be a monomial appearing
in $\vartheta_i$ as expanded at the point $v_0$, the point of $B(\ZZ)$ 
corresponding to $\vartheta_0$.
In particular, each $m_i$ lives in the
stalk of $\tilde\shP_1$ at $v_0$, and the difference $m_i-m_0$
lives in the stalk of $\shP$ at $v_0$. After taking the image of the difference
in $\Lambda\otimes_{\ZZ}(\ZZ/4\ZZ)$, we get a well-defined element 
of $\Lambda_{v_0}\otimes_{\ZZ}(\ZZ/4\ZZ)$ only depending on $i$ and not
on the particular terms taken. Using the fan structure defining the
affine structure in a neighbourhood of $v_0$, we can choose coordinates so
that $v_0,\ldots,v_3$ have coordinates $(0,0)$, $(1,0)$, $(0,1)$ and $(-1,-1)$
respectively. From this one sees using these coordinates that the lifted
action of $\chi$ is now
\[
\chi(\vartheta_0)=\vartheta_0, \quad
\chi(\vartheta_1)=\chi(1,0)\cdot\vartheta_1, \quad
\chi(\vartheta_2)=\chi(0,1)\cdot\vartheta_2, \quad
\chi(\vartheta_3)=\chi(-1,-1)\cdot\vartheta_3.
\]

As the sections $\vartheta_i$ embed $\foX$ into $\PP^3_{\kk\lfor t\rfor}$,
and at $t=0$ they satisfy the quartic equation 
$\vartheta_0\vartheta_1\vartheta_2\vartheta_3=0$, we obtain
a family of quartics invariant under the action of $G$ described above, and
hence is necessarily of the form
\begin{equation}
\label{Eqn: Modified Dwork family}
\lambda(t)(\vartheta_0^4+\vartheta_1^4+\vartheta_2^4+\vartheta_3^4)
+\vartheta_0\vartheta_1\vartheta_2\vartheta_3=0,
\end{equation}
with $\lambda(t)$ a formal power series in $t$ vanishing at $t=0$.
In particular, the family is the base-change of an algebraic family,
even if $\lambda(t)$ is not algebraic, and the theta functions are
just the coordinates. That theta functions should have such a simple
expression in general is not clear at all.

For $\kk=\CC$ one can also show that $\lambda(t)$ is analytic in $t$ as follows.
In the analytic version of the Dwork family \eqref{quarticfamily} there are
families of $2$-cycles $\alpha=\alpha(s)$ and $\beta=\beta(s)$ with $g(s)=
\exp(\int_\beta \Omega/\int_\alpha \Omega)$ an analytic coordinate on the
parametrizing disc. Here $\Omega$ is a choice of holomorphic $2$-form. These
period integrals are of the form treated in \cite{RS}, hence can be computed
also on $\foX$ to give a monomial in $t$. Comparing with \eqref{Eqn: Modified
Dwork family} shows that $g(\lambda(t))= c\cdot t$ for some $c\in\CC^\times$.
Then $\lambda$ is obtained by inverting $g$.
\end{example}

%===========================================================
%===========================================================
\begin{appendix}
\section{The GS case}

The purpose of this section is to discuss previous work of the first
and last authors within the framework established in this paper. The
main references are \cite{logmirror1} and \cite{affinecomplex}.

%===========================================================
\subsection{One-parameter families}
\label{Subsect: GS one-parameter}

The setup in \cite{logmirror1} and \cite{affinecomplex} is more
restrictive than here in that the affine structure extends over a
neighbourhood of each vertex. In fact, the discriminant locus
consists only of those codimension two cells of the barycentric
subdivision neither containing vertices nor intersecting the
interiors of maximal cells. For clarity we denote this smaller
discriminant locus by $\breve\Delta$, and by $\iota:\breve \Delta\to
B$ the inclusion. The strongest results in \cite{logmirror1} and
\cite{affinecomplex} are proved under the assumption that the affine
singularities of $B$ are \emph{positive} and \emph{simple}.
Positivity is a necessary condition for the affine structure to come
from a degeneration. Simplicity is a strong local primitivity
condition expressed by requiring that certain integral polytopes
spanned by the local monodromy vectors are elementary
simplices\footnote{An elementary simplex is a lattice simplex whose
only integral points are its vertices.}. In the positive, simple
case over an algebraically closed field $\kk$, one of the main
results of \cite{logmirror1} shows that the set of isomorphism
classes of possible $X_0$ as log spaces over the standard log point
is canonically in bijection with $H^1(B, \iota_*
\check\Lambda\otimes\ul{\kk}^\times)$ (\cite{logmirror1},
Theorem~5.4)\footnote{In \cite{logmirror1} we mostly work with the
\emph{dual} intersection complex or fan picture, hence the
dualization in the sheaf compared to the original statement, see
\cite{logmirror1}, Proposition~1.50. A summary of the setup in the
intersection picture of the present paper is contained in
\cite{affinecomplex},~\S1.}. This cohomology group
provides so-called \emph{lifted gluing data} $\mathbf
s=(s_{\omega\tau})$ (\cite{logmirror1}, Definition~5.1), which
induces both open gluing data and a log structure on
$X_0=X_0(\mathbf s)$ over the standard log point $(\Spec\kk,
\kk^\times\oplus\NN)$. Unlike in Subsection~\ref{Subsect: Twisting},
the open gluing data obtained in this way consists of
homomorphisms\footnote{The notation in \cite{affinecomplex} is $s_e$
for $e:\omega\to\tau$} $s_{\omega\tau}: \Gamma(W_{\omega\tau}, 
\iota_*\Lambda)\to \kk^\times$ for \emph{all} inclusions
$\omega\to\tau$, regardless of the dimensions. Here $W_{\omega\tau}$
is the open star of the edge $\omega\tau$ in the barycentric
subdivision as introduced in Subsection~\ref{Subsect: Twisting}.
Lifted gluing data also induces closed gluing data in the sense
considered here. The log structure is equivalent to providing slab
functions $f_{\rho,v}$ for any pair $\rho\in\P^{[n-1]}$ and
$v\in\rho$ a vertex labelling a connected component of
$\rho\setminus\breve\Delta$. There is an additional multiplicative
compatibility condition for each $\tau\in \P^{[n-2]}$ that involves
all $f_{\rho,v}$ with $\rho\supset\tau$, see also the discussion in
\cite{affinecomplex}, \S1.2.

Assuming $B$ bounded and with positive and simple singularities, the algorithm
in \cite{affinecomplex} then readily produces a mutually compatible series of
consistent wall structures $\scrS^{\mathrm{GS}}= \scrS_k$ on $(B,\P)$ for
$Q=\NN$, $A=\kk$ and $I=(t^{k+1})\subseteq \kk[t]=A[Q]$, for any $k\in \NN$, see
\cite{affinecomplex}, Theorem~3.1. Here $X_0$ has an implicit dependence on the
gluing data $\mathbf s$. The notion of wall structure is almost identical, there
being two differences. The first is that $\Delta$ in \cite{affinecomplex} was
chosen transverse to any rational polyhedral subset. In particular, our present
requirement $\Int\fob\cap\Delta=\emptyset$ for all slabs $\fob$ can only be
fulfilled for $\breve\Delta=\emptyset$. However, this requirement of
transversality with $\Delta$ was purely technical and can be removed as follws.
Modifying the argument in \cite{affinecomplex}, Remark~5.3, consider $B\times
[0,1]$ as an affine manifold with a discriminant locus restricting on
$B\times\{0\}$ to the barycentrically centered one from the present setup and
fulfilling the conditions required in \cite{affinecomplex} on $B\times (0,1]$.
Because there is never any scattering on the boundary, the algorithm still works
in this setup. Once the discriminant locus is barycentric, we can then decompose
each slab $\fob$ into the closures of connected components of
$\fob\setminus\Delta$. We then have the polyhedra of a wall structure $\scrS$ in
the sense of Definition~\ref{Def: Wall structure},2. For clarity we write
$\ul\fob$ for the slabs in the sense of this paper obtained by decomposition.
For a wall $\fop\in\scrS$ of codimension zero take the attached function
$f_\fop$ identical to the one in $\scrS^{\mathrm GS}$.

The different treatment of gluing data in the present work compared to
\cite{affinecomplex} requires a modification of the slab functions as follows.
Given the choice of open gluing data ${\mathbf s}=(s_{\omega\tau})$, one obtains
open gluing data in the sense of \S\ref{Subsect: Twisting} by
taking\footnote{The inverse arises from the different sign convention taken in
\cite{logmirror1}, \cite{affinecomplex}.}
$s_{\sigma\ul{\rho}}=s^{-1}_{\rho\sigma}$. Then given a slab $\ul\fob$ in
$\scrS$, the function $f_{\ul\fob}$ is obtained by choosing any point $x\in\Int
\ul\fob$ and considering the function $f_{\fob,x}$ attached to the point $x\in
\fob$, where $\fob$ is the slab of $\scrS^{\mathrm{GS}}$ containing $\ul\fob$.
Let $v$ be the vertex of $\rho$ contained in the same connected component of
$\rho\setminus\breve\Delta$ as $x$. We then take
\begin{equation}
\label{Eqn: GS slab functions}
f_{\ul\fob}=s_{v\rho}^{-1}(f_{\fob,x}).
\end{equation}
This formula for $f_{\ul\fob}$ arises from the change of chambers homomorphisms $\theta$
of log rings in the case $\sigma_\fou\neq\sigma_{\fou'}$, see
\cite{affinecomplex}, lower half of p.1349.\footnote{The factor
$D(s_{v\rho},\rho,v)$ from \cite{affinecomplex}, Definition~1.20 does not
appear here since we work with lifted gluing data.} Then \cite{affinecomplex},
(1.11) implies \eqref{f_rho' versus f_rho with gluing}.

For later use we observe that the order zero reduction $f_{\ul\rho}$ of
$f_{\ul\fob}$ for any slab $\ul\fob\subseteq\ul\rho$ can be written down
explicitly and turns out to have constant coefficients, not depending on gluing
data. To state this result recall that in the present case with positive and
simple singularities, the set
\[
\Delta(\rho,v)= \big\{ m_{\ul\rho\ul\rho'}\,\big|\, \ul\rho'\subseteq\rho\big\}
\]
of monodromy vectors of closed paths in $W_\rho\setminus\Delta$ starting and
ending at $v$, are the vertices of an elementary simplex. Here $\ul\rho\subseteq
\rho$ is the cell of $\tilde\P$ containing $v$.

\begin{lemma}
\label{Lem: Order zero slab functions}
Let $(B,\P)$ be bounded and with positive and simple singularities and $\ul\fob$
a decomposed slab. Then the order zero reduction of $f_{\ul\fob}$ is
\[
f_{\ul\rho}= \sum_{m\in\Delta(\rho,v)} z^m.
\]
\end{lemma}
\begin{proof}
Let $v\in\ul\fob$ be the unique vertex of $\P$ contained in $\ul\fob$.
Theorem~5.2,2 from \cite{logmirror1} says that in the positive, simple case, the
order zero slab functions $f_{v\to\rho}$ are determined uniquely by the gluing
data. An explicit formula is given in the proof of Theorem~5.2,(2) on p.304 of
\cite{logmirror1}:
\[
f_{v\to\rho}= \sum_{m\in\Delta(\rho,v)} s_{v\rho}(m)z^m.
\]
Here $z^m$ denotes the unique monomial of order zero with tangent vector
$m$. The claimed formula is now immediate by reducing \eqref{Eqn: GS slab
functions} modulo $t$.
\end{proof}
\medskip

We continue with fitting the construction of \cite{affinecomplex} into the
present framework.

\begin{lemma}
\label{Lem: Consistency in GS}
The wall structure $\scrS$ coming from \cite{affinecomplex},
Theorem~3.1 is consistent in the sense of Definition~\ref{Def:
Consistency in codim two}.
\end{lemma}

\begin{proof}
The notion of consistency agrees in codimension zero, but differs in
codimensions one and two. In codimension one, we first observe that
our ring $R_{\ul\fob}$ arises as a fibre product of the rings
$R_{\sigma_+}$ and $R_{\sigma_-}$ over $R_\rho$. This is discussed
in the proof of \cite{affinecomplex}, Lemma~2.34. Expressed in terms
of this fibre product, the notion of consistency along a codimension
one joint of \cite{affinecomplex}, Definition~2.28, yields the
notion in Definition~\ref{Def: Consistency in codim one}.

Consistency in codimension two for $\scrS$ in the sense of
Definition~\ref{Def: Consistency in codim two} is the content of
\cite{CPS}, Proposition~3.2.
\end{proof}
\medskip

\begin{proposition}
\label{Prop: Equivalence with GS}
Let $\foX^{\rm GS}\to \Spec \big(\kk[t]/(t^{k+1})\big)$ be the flat
deformation constructed from $\scrS$ in \cite{affinecomplex}, \S2.6.
Then the complement of the codimension two strata of $X_0\subseteq
\foX^{\rm GS}$ is canonically isomorphic to $\foX^\circ$ constructed in \S2.4.

Moreover, if the lifted gluing data $\mathbf s$ is projective
(Definition~\ref{Def: Projective gluing data}) then $\foX^{\rm GS}$
agrees with the projective scheme denoted by $\foX$ in
Theorem~\ref{Thm: Main}.
\end{proposition}

\begin{proof}
The constructions agree away from the codimension
two locus, observing the partial gluing in codimension one discussed in
the proof of Lemma~\ref{Lem: Consistency in GS}.

In the projective case, both $\O_{\foX^{\rm GS}}$ and $\O_\foX$ are
sheaves on $X_0$ fulfilling Serre's condition $S_2$ and which are
canonically isomorphic on $X_0^\circ$, an open dense subset with
complement of codimension two. Denote by $i: X_0^\circ\to X_0$ the
inclusion. Then also $\foX=\foX^{\rm GS}$ canonically since
\[
\O_{\foX}= i_* \O_{\foX^\circ}=
i_* \O_{(\foX^{\rm GS})^\circ} =\O_{\foX^{\rm GS}},
\]
by the $S_2$ condition.
\end{proof}

\begin{remark}
\label{Rem: Avoiding Grothendieck existence theorem}
In the projective setup Proposition~\ref{Prop: Equivalence with GS} provides a
homogeneous coordinate ring $S_k$ as a flat $\kk[t]/(t^{k+1})$-algebra. Varying
$k$, the $S_k$ form an inverse system of $\kk\lfor t\rfor$-algebras. Thus taking
the limit $S:=\liminv S_k$ shows that in \cite{affinecomplex}, Theorem~1.30, we
do not only get a flat formal scheme over $\Spf(\kk\lfor t\rfor)$, but a flat
scheme $\foX:=\Proj(S)\to \Spec(\kk\lfor t\rfor)$, without imposing further
cohomological assumptions and invoking Grothendieck's existence theorem as in
\cite{affinecomplex}, Corollary~1.31.

A similar remark holds in the case of higher dimensional bases
discussed in \S\ref{Subsect: GS universal}.
\end{remark}

%===========================================================
\subsection{The universal formulation}
\label{Subsect: GS universal}

In \cite{logmirror2}, \S5.2, it is discussed how to build a family
$(X_0,\M_{X_0})\to (\Spec A,\M_A)$ of toric log Calabi-Yau spaces
(\cite{logmirror1}, Definition~4.3) parametrized by variations of lifted gluing
data. The base is an algebraic torus with $A=\kk\big[ H^1(B,\iota_*
\check\Lambda)^*\big] = \kk\big[ H^1(B,\iota_* \check\Lambda)_f^*\big]$ with
$\kk$ an algebraically closed field. Here the subscript $f$ denotes the free
part of a finitely generated abelian group. The set of closed points of this
torus is $H^1(B,\iota_*\check\Lambda) \otimes\kk^\times = H^1(B,\iota_*
\check\Lambda)_f \otimes\kk^\times$. The last equality is due to the fact that
$\kk^\times$ is a divisible group thanks to $\kk$ being algebraically closed.
The family depends on the choice of a right-inverse $\sigma_0: H^1(B, \iota_*
\check\Lambda)_f \to H^1(B, \iota_*\check\Lambda)$ of the quotient by the
torsion subgroup\footnote{The torsion subgroup of $H^1(B,\iota_*\check\Lambda)$
is related to isotrivial families. In fact, $H^1(B, \iota_*\check\Lambda \otimes
\kk[s^{\pm1}]^\times)$ is isomorphic to $H^1(B,\iota_*\check\Lambda
\otimes\kk^\times) \oplus H^1(B,\iota_*\check\Lambda)$. Hence a class
$\sigma_T\in H^1(B,\iota_*\check\Lambda)$ with $b\cdot\sigma_T=0$ for some $b>0$
defines gluing data $(1,\sigma_T)$ over $\GG_m=\Spec\kk[s^{\pm1}]$ that becomes
trivial after the base change $s\to s^b$. Similarly, our universal families for
different choices of $\sigma_0$ become isomorphic after a finite \'etale cover.
} and on an element $s_0\in H^1(B, \iota_* \check\Lambda \otimes \kk^\times)$.
The choice of $s_0$ selects one of the pairwise disjoint $H^1(B,\iota_*
\check\Lambda) \otimes\kk^\times$-torsors that cover $H^1(B,\iota_*\check\Lambda
\otimes\kk^\times)$ and which are parametrized by
$H^2(B,\iota_*\check\Lambda)_{\mathrm{tors}}$, see \cite{logmirror2}, first displayed
formula in \S5.2. As a log scheme the base is the product of $\Spec A$ with
trivial log structure and the standard log point. In particular,
$\ol\M_A=\ul\NN$ is constant and there is a global chart $\NN\to
\Gamma(\Spec A, \M_A)$. As a family of toric log Calabi-Yau spaces this family
is defined by the pair $(s_0,\sigma_0)$ viewed as an element in $H^1 (B,
\iota_*\check\Lambda\otimes\ul{A}^\times)$ as follows. Since $A$ is a Laurent
polynomial ring,
\[
A^\times= \kk^\times\oplus H^1(B,\iota_*\check\Lambda)_f^*.
\]
Thus
\[
H^1(B,\iota_*\check\Lambda\otimes\ul{A}^\times) 
= H^1(B,\iota_*\check\Lambda\otimes\ul{\kk}^\times)\oplus
\big( H^1(B,\iota_*\check\Lambda)\otimes
H^1(B,\iota_*\check\Lambda)_f^*\big),
\]
and the second summand equals $\Hom(H^1(B,\iota_*\check\Lambda)_f,
H^1(B,\iota_*\check\Lambda))$. Thus $(s_0,\sigma_0)$ can be viewed as
an element of $H^1 (B, \iota_*\check\Lambda\otimes\ul{A}^\times)$. Note
that by compatibility with base change, the fibre of this family
$(X_0,\M_{X_0})\to (\Spec A,\M_A)$ over the closed point in $\Spec
A$ defined by an element $\xi\in H^1(B, \iota_*\check\Lambda)_f
\otimes \kk^\times$ is classified by $\sigma_0(\xi)\cdot s_0 \in
H^1(B, \iota_*\check\Lambda\otimes \ul{\kk}^\times)$.

Let us now refine the discussion in \S\ref{Subsect: GS
one-parameter} to this situation. We will do this by first enlarging
the log structure on $\Spec A$ to the universal one and then arguing
that \cite{affinecomplex} also produces a consistent wall structure
for this case.

In \cite{logmirror1} and \cite{affinecomplex} the base monoid $Q$ is always
$\NN$, but we have the freedom of choosing a strictly convex MPA-function
$\varphi$ with values in $\NN$, which we fixed so far. We now want to replace
$\NN$ by the universal monoid analogous to $Q_0$ in Proposition~\ref{Prop:
Universal MPA function}. To work this out recall that the current choice of
$\varphi$ is more restrictive than in Definition~\ref{Def: convex MPA-function}
in that we require $\kappa_{\ul\rho}(\varphi)=\kappa_{\ul\rho'}(\varphi)$ for
any $\ul\rho,\ul\rho'$ contained in the same $\rho\in\P^{[n-1]}$ and we impose
additive conditions along codimension two cells to assure the existence of local piecewise linear representatives with the given set of kinks, see Example~\ref{Expl:
MPA-functions},1. For a toric monoid $Q'$ denote the subgroup of ${Q'}^\gp$-valued
MPA-functions in this restricted sense by ${\breve\MPA}(B,{Q'}^\gp)\subseteq
\MPA(B,{Q'}^\gp)$ and let ${\breve\MPA}(B, Q')$ be the corresponding monoid of
convex functions. Refining Proposition~\ref{Prop: Universal MPA function}, the
universal point of view runs as follows. Since the additive condition is
provided by a homomorphism into a torsion-free group, $\MPA(B,\ZZ)/{\breve\MPA}
(B,\ZZ)$ is torsion-free. As a consequence, the restriction map
\[
r: \Hom\big({\MPA}(B,\NN), \ZZ\big) \lra 
\Hom\big({\breve\MPA}(B,\NN), \ZZ\big)
\] 
is a surjection. Denote by $e_{\ul\rho_i}$ the generators of the monoid
$Q_0=\Hom\big(\MPA(B,\NN),\NN\big)$ as defined in Proposition~\ref{Prop:
Universal MPA function}. Then the kernel of $r$ is the saturation of the
subgroup generated by elements of the form (1) $e_{\ul\rho}-e_{\ul \rho'}$ for
$\ul\rho,\ul\rho'$ both contained in the same codimension $1$ cell $\rho$, and
(2) elements of the form $\sum_i \langle m,n_i\rangle e_{\ul\rho_i}$. Here the
$\ul\rho_i$, $n_i$ are as in \eqref{Eqn: Balancing condition}, and there is one
element of the latter kind for each codimension two $\tau\in\P$,
$\tau\not\subseteq\partial B$ and $m\in\Lambda_v$, $v\in\tau$ a vertex. Define
$Q$ to be the saturation of the submonoid of $\Hom\big({\breve\MPA}
(B,\NN),\ZZ\big)$ generated by $r(Q_0)$, so in particular we have a map
$r:Q_0\rightarrow Q$. Define also
\begin{equation}
\label{Eqn: Universal restricted MPA fct}
\breve\varphi:= r\circ \varphi_0 \in \breve\MPA(B, Q)
\end{equation}
for $\varphi_0\in \MPA(B, Q_0)$ the universal MPA-function from
Proposition~\ref{Prop: Universal MPA function}. In fact, by the definition of
$Q$, the MPA-function $\breve\varphi$ fulfills the properties defining
$\breve\MPA(B,Q)$ as a subspace of $\MPA(B,Q)$. Note also that by construction,
$\breve\varphi$ is convex. In \cite{logmirror1}, \cite{affinecomplex} we assume
the existence of a strictly convex MPA-function with values in $\NN$, and hence
$\breve\varphi$ is even strictly convex.

Analogously to $\varphi_0$ the MPA-function $\breve\varphi$ has the universal
property for restricted MPA-functions with values in any fine saturated monoid
$Q'$. Indeed, given any $Q'$-valued restricted MPA-function $\psi$ on $B$, we
obtain a function $h:Q_0\rightarrow Q'$ so that $\psi=h\circ\varphi_0$, by
Proposition~\ref{Prop: Universal MPA function}. On the other hand, because
$\psi$ is restricted, $h$ must vanish on the elements of $Q_0^{\gp}$ generating
the kernel of $r:Q_0^{\gp}\rightarrow Q^{\gp}$ described above, and hence $h$
descends to a well-defined map $\breve h:r(Q_0)\rightarrow Q'$. Since $Q'$ is
fine and saturated, this map extends to $\breve h:Q\rightarrow Q'$. Then
$\psi=h\circ \varphi_0=\breve h\circ r \circ\varphi_0=\breve h \circ
\breve\varphi$ by construction. It is also not hard to see that if
$\psi\in\breve\MPA(B,\ZZ)$ takes values in integers, then the corresponding
classifying map $\breve h:Q\to \ZZ$ is given by the homomorphism
$\breve\MPA(B,\ZZ)^*\to \ZZ$ that evaluates at $\psi$. Said differently,
$\breve\varphi$ is the tautological restricted MPA-function with kinks
$\kappa_\rho\in Q\subseteq \breve\MPA(B,\ZZ)^*$ such that for any $\psi\in
\breve\MPA(B,\ZZ)$ we have
\begin{equation}
\label{Eqn: Kinks of universal restricted MPA-fct}
\kappa_\rho(\psi)= \langle \kappa_\rho(\breve\varphi),\psi\rangle\in\ZZ.
\end{equation}

\begin{construction}
\label{Constr: GS ghost sheaf}
(\emph{Construction of $\ol\M_{X_0}^{\breve \varphi}$}.) Analogous
to \cite{logmirror1}, Example~3.17, for the case of $\NN$, there is
a fine sheaf of monoids $\ol\M_{X_0}^{\breve\varphi}$ in the Zariski
topology on $X_0$, with constant stalks along toric strata, along
with a homomorphism $ Q\to\Gamma (X_0,\ol\M_{X_0}^{
\breve\varphi})$. For the construction observe that the affine
structure on the $Q_\RR^\gp$-torsor $\BB_{\breve\varphi}\to B$ from
Construction~\ref{Construction: B_varphi} now extends over the
preimage of $\Delta\setminus\breve\Delta$. In particular, the sheaf
of monomials $\shP^+\subseteq \shP= {\breve\varphi}^*
\Lambda_{\BB_{\breve\varphi}}$ from \S\ref{Subsect: rings} is
defined over $B\setminus \breve\Delta$. Then the restriction of
$\ol\M_{X_0}^{\breve\varphi}$ to the algebraic torus in the toric
stratum $X_\tau\subseteq X_0$ is constant with stalks
$\shP^+_x/\shP_x^\times$ for any $x\in \Int\tau\setminus\Delta$. By
the definition of ${\breve\MPA}( B, Q^\gp) \subseteq \MPA(B, Q^\gp)$,
local parallel transport yields canonical isomorphisms between the
monoids $\shP^+_x/\shP_x^\times$ for different choices of $x$, even
between different connected components of $\tau\setminus\breve
\Delta$. If $\omega\subseteq\tau$ and $\eta_\omega$, $\eta_\tau$ are
the generic points of the corresponding toric strata, the
generization map $\ol\M_{X_0,\eta_\omega} \to \ol\M_{X_0,\eta_\tau}$
is defined by generization $\shP^+_y/\shP_y^\times\to
\shP^+_x/\shP_x^\times$ for $x\in \Int\tau \setminus\breve\Delta$,
$y\in \Int\omega\setminus \breve\Delta$. Since these generization maps
are compatible with the map $Q\to \shP^+_x/\shP_x^\times$, the sheaf
$\ol\M_{X_0}^{\breve\varphi}$ comes with a homomorphism $Q\to
\Gamma(X_0,\ol\M_{X_0}^{\breve\varphi})$.
\end{construction}

\begin{construction}
\label{Constr: GS log structure}
(\emph{Construction of the log structure $\M_{X_0}^{\breve\varphi}$}
on $X_0$.)
Recall that we now have two MPA-functions on $B$ in the restricted
sense, the universal one $\breve\varphi$ with values in $Q$ and the
chosen one $\varphi$ with values in $\NN$. By the universal property
there is a unique homomorphism $\breve h:Q\to\NN$ with $\varphi=
\breve h\circ\breve\varphi$ inducing a homomorphism
$\ol\M^{\breve\varphi}_{X_0}\to \ol\M_{X_0}$ of ghost sheaves on our
family over the algebraic torus $\Spec A$. Since $\varphi$ is
strictly convex, $\breve h$ is a local homomorphism of monoids, that is,
$\breve h^{-1}(0)=\{0\}$. Then
\begin{equation}
\label{M^brevevarphi}
\M^{\breve\varphi}_{X_0}:= \ol\M_{X_0}^{\breve\varphi}
\times_{\ol\M_{X_0}} \M_{X_0}
\end{equation}
is a sheaf of monoids, and the composition
\[
\M^{\breve\varphi}_{X_0}\lra \M_{X_0}\lra \O_X 
\]
of the projection with the structure homomorphism for $\M_{X_0}$ defines a log
structure on $X_0$ with ghost sheaf $\ol\M_{X_0}^{\breve\varphi}$. In fact, the
preimage of $\O_{X_0}^\times$ in $\M^{\breve\varphi}_{X_0}$ is readily seen to
be $\{0\}\times \O_{X_0}^\times\simeq \O_{X_0}^\times$.

Moreover, denote by $\M_A^Q$ the log structure on $\Spec A$
associated to the chart $Q\to A$ mapping $Q\setminus\{0\}$ to $0$.
Then the map from $Q$ into sections of the first factor in
\eqref{M^brevevarphi} induces a morphism of log schemes
\begin{equation}
\label{Eqn: Log universal family}
(X_0,\M_{X_0}^{\breve\varphi})\lra (\Spec A,\M_A^Q).
\end{equation}
This morphism has a universal property for families of log schemes
with closed fibres isomorphic to fibres of $X_0\to \Spec A$ and
arbitrary log structures on the base. Since this is not important
for the present discussion we omit the details.
\end{construction}

Now we are in position to run the smoothing algorithm of \cite{affinecomplex}
with the following modification. As ground field ($\kk$ in \cite{affinecomplex})
take the quotient field $A_{(0)}$ of $A$. Denote by $I_0\subseteq A_{(0)}[Q]$
the ideal generated by $Q\setminus\{0\} =\breve h^{-1} \big(\NN\setminus
\{0\}\big)$. Then in the algorithm replace $\NN$ by $Q$, but define the notion
of order of exponents (\cite{affinecomplex}, Definition~2.3) as before by first
composing with $\breve h: Q\to \NN$. Geometrically this corresponds to base
changing from $A_{(0)}[Q]$ to $A_{(0)}[t]$ by means of $\breve h$. The change
from $\NN$ to $Q$ enters in the propagation of exponents on $B$ in the smoothing
algorithm in that elements of $Q$ are being added when changing cells. On a
formal level this just means interpreting $t^l$ as a monomial in $A_{(0)}[Q]$
rather than in $A_{(0)}[t]$. Moreover, by \cite{affinecomplex}, Theorem~5.2,
since the slab functions are defined over $A$, the smoothing algorithm
nevertheless produces a wall structure defined over $A[Q]$. The result is a
compatible system $\big(\scrS_k\big)_{k\in \NN}$ of consistent wall structures,
producing a compatible system of flat morphisms $\foX_k\to \Spec\big(A[Q]/I_0^k
\big)$, or a flat formal scheme
\begin{equation}
\label{Eqn: Universal formal family}
\foX\lra \Spf\big( A\lfor Q\rfor\big),
\end{equation}
extending $X_0\to\Spec A$. Here $A\lfor Q\rfor$ denotes the
completion of $A[Q]$ with respect to $\breve h:Q\to\NN$. In particular, the
base of this family is the completion of the affine toric variety
$\Spec \big( A[Q] \big)$ along its minimal toric stratum $\Spec A$.

To obtain a projective family, hence to make contact with
Theorem~\ref{Thm: Main}, restrict to any closed subspace of
$\Spec A$ with vanishing obstruction class $\ob_{\PP}$ from
Proposition~\ref{Prop: CGD obstruction on C{B}} and Remark~\ref{Rem:
ob_PP in GS}. To do this universally define an obstruction map
$\ob_\PP$ on lifted gluing data by composing the general obstruction
map from \eqref{Eqn: ob_PP}, denoted $\ob_\PP$ there, with the map
turning lifted gluing data to closed gluing data:
\[
\ob_\PP: H^1(B,\iota_*\check\Lambda\otimes\ul{A}^\times)\lra
H^1(B,\shQ\otimes\ul{A}^\times)\lra
H^2(B,A^\times).
\]
With the previous noted equalities $A^\times= \kk^\times\oplus
H^1(B,\iota_*\check\Lambda)_f^*$ and $H^1(B, \iota_*\check\Lambda)^*
=H^1(B, \iota_*\check\Lambda)_f^*$, and since the construction of
$\ob_\PP$ as a connecting homomorphism is functorial, this map
decomposes as a direct sum of a map $\ob_{\PP}^{\kk^\times}:
H^1(B,\iota_*\check\Lambda\otimes \ul{\kk}^\times)\to H^2(B,
\kk^\times)$ and of
\begin{equation}
\label{Eqn: obPPZZ}
\ob_\PP^\ZZ\otimes \id: H^1(B,\iota_*\check\Lambda)\otimes
H^1(B, \iota_*\check\Lambda)^*_f\lra
H^2(B,\ZZ)\otimes H^1(B, \iota_*\check\Lambda)^*_f.
\end{equation}
Now take $s_0\in H^1(B, \iota_*\check\Lambda\otimes\ul{\kk}^\times)$ with
$\ob_{\PP}^{\kk^\times} (s_0)=0$ and let $K=\ker(\ob_\PP^\ZZ)\subseteq H^1(B,
\iota_*\check\Lambda)$. If the free part $K_f$ of $K$ is saturated, for example
if $H^2(B,\ZZ)$ is torsion-free, the splitting $\sigma_0$ of $H^1(B,
\iota_*\check\Lambda)\to H^1(B,\iota_*\check\Lambda)_f$ implicit in the above
construction (denoted $s_{\id}$ in \cite{logmirror2}) can be chosen in such a
way that it maps $K_f\subseteq H^1(B,\iota_*\check\Lambda)_f$ into $K$, with
$K_f$ the free part of $K$. By the following lemma, $\Spec A_\PP$ with $A_\PP:=
\kk[K_f^*]$ is the maximal closed subspace of $\Spec A$ with vanishing
$\ob_\PP$.

\begin{lemma}
Assume that $K_f\subset H^1(B,\iota_*\check\Lambda)_f$ is saturated and
$\sigma_0(K_f)\subseteq K$. Let $R$ be a $\kk$-algebra and $\lambda:
H^1(B,\iota_*\check\Lambda)_f^*\to R^\times$ a homomorphism such that the
induced map
\[
H^1(B,\iota_*\check\Lambda)\otimes H^1(B,\iota_*\check\Lambda)_f^*
\stackrel{\ob_\PP^\ZZ\otimes\lambda}{\lra}  H^2(B,\ZZ)\otimes R^\times
\]
maps $\sigma_0$ to $0\otimes 1$. Then $\lambda$ factors over the quotient map
$H^1(B,\iota_*\check\Lambda)^*\to K_f^*$.
\end{lemma}

\begin{proof}
Since $K_f\subseteq H^1(B,\iota_*\check\Lambda)_f$ is saturated, there exists a
basis
\[
n_1,\ldots, n_s,n_{s+1},\ldots,n_r\in H^1(B,\iota_*\check\Lambda)_f
\]
with
$n_{s+1},\ldots,n_r$ generating $K_f$. Denoting the dual basis by
$m_1,\ldots,m_r\in H^1(B,\iota_*\check\Lambda)^*$, it holds
\[
\sigma_0=\sum_{i=1}^s \tilde n_i\otimes m_i,
\]
with $\tilde n_i=\sigma_0(n_i)$. Thus putting $a_i=\ob_\PP^\ZZ(\tilde n_i)\in
H^2(B,\ZZ)$, $i=1,\ldots,r$, and identifying $H^1(B,\iota_*\check\Lambda)
\otimes H^1(B,\iota_*\check\Lambda)_f^*$ with $\Hom(H^1(B,\iota_*\check\Lambda)_f,
H^1(B,\iota_*\check\Lambda))$, we have
\[
\sigma_0^*\ob_\PP^\ZZ= \sum_{i=1}^r a_i\otimes m_i\in  H^2(B,\ZZ)\otimes H^1(B,\iota_*\check\Lambda)^*.
\]
Thus $(\ob_\PP^\ZZ\otimes\lambda)(\sigma_0)= \sum_{i=1}^s a_i\otimes
\lambda(m_i)$. Now since the composition of $\ob_\PP^\ZZ$ with $H^2(B,\ZZ)\to
H^2(B,\ZZ)_f$ maps $H^1(B,\iota_*\check\Lambda)_f/K_f$ injectively into
$H^2(B,\ZZ)_f$, it follows that the coefficients of $a_1,\ldots,a_s$ have to be
trivial for $(\ob_\PP^\ZZ\otimes\lambda)(\sigma_0)$ to be trivial:
$\lambda(m_1)=\ldots=\lambda(m_s)=1\in R^\times$. It follows that
$\lambda:H^1(B,\iota_*\check\Lambda)^*\to R^\times$ factors over $K_f^*$ as
claimed.
\end{proof}

If $K\subset H^1(B,\iota_*\check\Lambda)$ is not saturated, one can still
proceed with the construction of a projective family from a splitting of $K\to
K_f$, but the resulting family may not be isomorphic to the pull-back of a
family over $\Spf\big( A\lfor Q\rfor\big)$, not even when restricted to the central fibre.

Assuming now that we have such a splitting of $K\to K_f$, the composition
\[
K_f\lra K\lra H^1(B,\iota_*\check\Lambda)
\]
defines an element $\sigma_\PP$ of $H^1(B,\iota_*\check\Lambda)\otimes_\ZZ
K_f^*$, hence gluing data $(s_0,\sigma_\PP)$ over $A_\PP=\kk[K_f^*]$ by the
isomorphism
\[
H^1(B,\iota_*\check\Lambda\otimes\ul \kk^\times)\oplus
\big(H^1(B,\iota_*\check\Lambda)\otimes_\ZZ K_f^* \big)
= H^1(B,\iota_*\check\Lambda\otimes \ul A_\PP^\times).
\]
The obstruction to projectivity of $(s_0,\sigma_\PP)$, as an element of
\[
H^2(B,\kk^\times)\oplus \big(H^2(B,\ZZ)\otimes K_f^*\big)
= H^2(B,\kk^\times)\oplus \Hom\big(K_f,H^2(B,\ZZ)\big)
\]
has first component $\ob_\PP^{\kk^\times}(s_0)=1$ and second component the composition
\[
\ob_\PP^\ZZ|_{K_f}:K_f\lra K\lra H^1(B,\iota_*\check\Lambda)\lra H^2(B,\ZZ),
\]
which vanishes by the definition of $K$.

According to Proposition~\ref{Prop: CGD obstruction on C{B}} and
Proposition~\ref{Prop: Consistency for cone scrS, with twisting} with
ground ring $A_\PP$, we can now lift the wall structure to a consistent
wall structure on $(\cone B,\cone\P)$. In particular,
Theorem~\ref{Thm: Main} holds with $A_\PP$ instead of $A$ and
with $I=I_0^k$ for any $k$. Taking the limit $k\lra \infty$ yields a
family over $\Spec \big(A_\PP\lfor Q\rfor\big)$. 

\begin{theorem}
\label{Thm: Main theorem, simple case}
Let $(B,\P)$ be a bounded polyhedral manifold with positive and simple
singularities admitting a strictly convex MPA function
$\varphi\in\breve\MPA(B,\NN)$. Denote by $Q\subseteq \breve\MPA(B,\NN)^*$ the
universal base monoid (see \S\ref{Subsect: GS universal}) and
$A_\PP=\kk[K_f^*]$ with $K=\ker(\ob_\PP^\ZZ)\subseteq
H^1(B,\iota_*\check\Lambda)$ as above.

Then there exists a canonical polarized family
\[
\foX_\PP\lra\Spec( A_\PP\lfor Q\rfor)
\]
constructed from a consistent wall structure on $(\cone{B},\cone{\P})$ and
inducing the universal family of polarized log Calabi-Yau spaces over $\kk$ with
intersection complex $(B,\P)$ (the polarized analogue of \eqref{Eqn: Log
universal family}). Moreover, the statements of Theorem~\ref{Thm: Main} apply
verbatim.
\end{theorem}

\begin{proof}
It only remains to check the statement on the family over $\Spec\kk$, that is,
after restriction to $\Spec A_\PP\subseteq \Spec \big(A_\PP\lfor Q\rfor\big)$.
First, working over $\kk\lfor Q\rfor$ rather than $\kk\lfor t\rfor$ only has the
effect of incorporating a universal choice of ghost sheaf in the way discussed
in Constructions~\ref{Constr: GS ghost sheaf} and \ref{Constr: GS log
structure}. Thus it suffices to check the statement after pulling back to
$\Spec\big( A_\PP\lfor t\rfor\big)$ by a homomorphism $h:Q\to\NN$.

It is clear from comparison of the construction given here with the construction
in \cite{affinecomplex} that $\foX_\PP$ reduced modulo $I_0$ agrees with the
base change $X_0^\PP\to \Spec A_\PP$ of $X_0\to \Spec A$ from \eqref{Eqn: Log
universal family} outside codimension two, hence everywhere. Moreover, the log
structure $\shM'_{X_0}$ on $X_0$ coming with the construction in
\cite{affinecomplex} agrees with the log structure $\shM_{X_0}$ constructed in
\cite{logmirror2}, \S5.2. Finally, the algorithm in \cite{affinecomplex} also
provides a compatible system of charts for the reduction of $\foX_\PP$ modulo
$I^k$ and hence, by taking the limit $k\to\infty$, a system of charts for $\foX_\PP$
compatible with charts for $(X_0,\shM'_{X_0})$. Hence $\shM'_{X_0}$ also agrees
with the log structure on $X_0^\PP$ defined by the closed embedding
$X_0^\PP\subseteq \foX_\PP$. Thus we obtain the statement on the induced family over $\Spec A_\PP$.
\end{proof}

As a final remark we give an interpretation of the obstruction map
$\ob_\PP^\ZZ$ in terms of the radiance obstruction.

\begin{proposition}
The obstruction map $\ob_\PP^\ZZ: H^1(B,\iota_*\check\Lambda) \to
H^2(B,\ZZ)$ equals the cup product with the radiance obstruction
$c_B\in H^1(B,\iota_*\Lambda)$ followed by the trace homomorphism.
\end{proposition}

\begin{proof}
Since by \cite{logmirror1}, Proposi\-tion~1.29, the radiance
obstruction $c_B$ vanishes locally, it can be defined as the
extension class $\tilde c_B\in \Ext^1(\iota_*\check\Lambda,\ul\ZZ)$
of
\[
0\lra \ul\ZZ \lra \shAff(B,\ZZ) \lra \iota_*\check\Lambda \lra 0,
\]
the push-forward to $B$ of \eqref{Eqn: Aff}. The obstruction map $\ob_\PP^\ZZ$
from \eqref{Eqn: obPPZZ} is a connecting homomorphism of the associated long
exact cohomology sequence. The result now follows from the standard fact that
connecting homomorphisms are given by cup product with the extension
class. Identifying $H^1(B,\iota_*\check\Lambda)$ with
$\Ext^1_B(\ul\ZZ,\iota_*\check\Lambda)$ this cup product is the Yoneda
composition product,
\[
\Ext^1_B(\iota_*\check\Lambda,\ul\ZZ)
\otimes\Ext^1_B(\ul\ZZ,\iota_*\check\Lambda)
\lra \Ext^2_B(\ul\ZZ,\ul\ZZ)=H^2(B,\ZZ).
\]
In terms of sheaf cohomology this means taking the cup product with
$c_B$ to arrive at a class in $H^2(B,\iota_*\check\Lambda
\otimes\iota_*\Lambda)$, followed by $H^2$ of the trace homomorphism
$\iota_*\check\Lambda \otimes\iota_*\Lambda\to \ul\ZZ$.
\end{proof}

%===========================================================
\subsection{Equivariance}
\label{Subsect: GS torus action}

There are two tori acting compatibly on the universal families constructed in
\S\ref{Subsect: GS universal}. To define these actions, denote by $\breve\PL(B)$
the subgroup of $\PL(B)=\PL(B,\ZZ)$ consisting of \emph{restricted
PL-functions}, that is, having kinks satisfying the conditions in codimensions one
and two stated in Example \ref{Expl: MPA-functions},1. The space $\breve\PL(B)$
agrees with the space of piecewise linear function in \cite{logmirror1},
Definition~1.43. Writing $\breve\MPA(B)= \breve\MPA(B,\ZZ)$, the map
$\kappa:\PL(B,\ZZ)\to \MPA(B,\ZZ)$ associating to a PL-function $\varphi$ the
associated MPA-function with kinks $\kappa_\rho(\varphi)$, then descends to a
map $\breve\PL(B)\rightarrow \breve\MPA(B)$ that we also denote $\kappa$.

Now the first action is by the relative torus from \S\ref{Subsect: torus action}
with character lattice $\breve\PL(B)^*$, treated abstractly in Theorem~\ref{Thm:
Torus action}. This action is trivial on the coefficient ring $A$. The second
torus has character lattice the dual of $\breve\MPA(B)/\kappa
\big(\breve\PL(B)\big)$, up to finite index. This latter action is typically
non-trivial on $A=\kk[H^1(B,\iota_*\check\Lambda)^*]$, which is graded by the
dual of the connecting homomorphism $c_1: \breve \MPA(B)\to
H^1(B,\iota_*\check\Lambda)$. 
\bigskip

Let us first treat the relative torus action. Our universal family is a flat
formal family $\foX\to \Spf(A\lfor Q\rfor)$ or, restricting to projective
gluing data, a flat projective family $\foX_\PP\to \Spec (A_\PP\lfor Q\rfor)$.
Here $Q\subseteq {\breve\MPA}(B,\ZZ)^*$, $A= \kk[H^1(B,
\iota_*\check\Lambda)_f^*]$ and $A_\PP=\kk[K_f^*]$ with $K=\ker(\ob_\PP^\ZZ)$,
and $K_f$ its free part, as discussed in \S\ref{Subsect: GS universal}. We are
now in the situation with non-trivial gluing data commented on in
Remark~\ref{Rem: torus action with gluing data}. Recall from
Proposition~\ref{Prop: Aut(X_0)} that the identity component of $\Aut_A(X_0)$ is
the torus over $A$ with character lattice $\PL(B)^*$. We are going to extend the
restriction of this action to the subtorus with character lattice
$\Gamma:=\breve\PL(B)^*$ to $\foX$.\footnote{Elements of $\Aut_A(X_0)$ not
in this subtorus move the log-singular locus on the central fibre; hence their
action on $X_0$ can not extend to an action on $\foX$ with trivial action
on $\Spec A$.}

In the notation of \S\ref{Subsect: torus action} we take the $\Gamma$-grading on
$A$ to be trivial, the map $\delta_B:\PL(B)^* \rightarrow
\breve\PL(B)^*$ dual to the inclusion, and $\delta_Q$ the dual of the map
$\breve\PL(B)\rightarrow \breve\MPA(B)$:
\[
\delta_Q:Q\hookrightarrow \breve\MPA(B)^*
\lra \breve\PL(B)^*=\Gamma.
\]
Commutativity in the compatibility diagram~\eqref{Eqn: delta_b and delta_Q} is
trivially true. The degree zero part of $A[Q]$ is $A_0^Q=A[Q']$ with
$Q'=\delta_Q^{-1}(0)$. In particular, $A\subseteq A_0^Q$, and hence the gluing
data have degree zero as required in Remark~\ref{Rem: torus action with gluing
data}.

\begin{proposition}
\label{Prop: Relative torus action in GS}
The action of the torus $\Spec(A[\Gamma])\subseteq\Aut_A(X_0)$ on $X_0$ extends
canonically to actions on $\foX\to \Spf(A\lfor Q\rfor)$ and on
$\foX_\PP\to \Spec(A_\PP\lfor Q\rfor)$.
\end{proposition}

\begin{proof}
To apply Theorem~\ref{Thm: Torus action}, modified for non-trivial gluing data
according to Remark~\ref{Rem: torus action with gluing data}, it remains to
check that the wall structure $\scrS$ is homogeneous in the sense of
Definition~\ref{Def: Homog wall structure}. We argue inductively and first
observe that the initial wall structure $\scrS_0$ is homogeneous. Indeed,
Lemma~\ref{Lem: Consistency in GS}, which continues to hold with universal
gluing data with the same proof, shows that an initial slab function $f_{v\to
\rho}$ is a sum of monomials $z^{m_{\ul\rho\ul\rho'}}$ with
$m_{\ul\rho\ul\rho'}$ the local monodromy vectors from \eqref{Eqn: monodromy
vectors} and $\ul\rho\subseteq\rho$ the cell of $\tilde\P$ with $v\in\ul\rho$.
But $\varphi\in \PL(B)$ lies in $\breve \PL(B)$ only if it is invariant under
monodromy along $\rho$, which is equivalent to $\nabla_{m_{\ul\rho\ul\rho'}}
(\varphi)=0$ for all $\ul\rho'\subseteq \rho$. Here
$\nabla_{m_{\ul\rho\ul\rho'}}$ is the directional derivative on $\PL$-functions
defined in \eqref{Eqn: nabla(PL)}. This shows $\deg_\Gamma
(z^{m_{\ul\rho\ul\rho'}})=0$. Hence any order zero slab function $f_{v\to\rho}$
is homogeneous of degree~$0$.

Next observe that the wall crossing isomorphisms $\theta_\fop$ are of the form
\[
z^m\longmapsto \tilde f_\fop^{\langle n_\fop,m\rangle} z^m,
\]
with $\tilde f_\fop$ differing from $f_\fop$ by changing the constants via an
application of gluing data, see the two displayed formulas on p.1349 of
\cite{affinecomplex}, Construction~2.24.

Now assume inductively that the wall structure $\scrS_k$ at order $k$ is
homogeneous. Then up to subdivision of walls and slabs, the wall structure
$\scrS_{k+1}$ at order $k+1$ is obtained from $\scrS_k$ by adding walls carrying
monomials of order $k+1$ and changing the functions $f_{\ul\fob}$ carried by slabs
$\ul\fob$ by adding monomials of order $k+1$. Consistency is checked by
composing automorphisms associated to walls containing a local affine
submanifold $\foj$ of codimension two (a \emph{joint}) in a cyclic order. New
walls are always created by such a consistency check of $\scrS_k$ in a version
of the tropical vertex group \cite{GPS} at order $k+1$, as in Lemma~3.7 in
\cite{affinecomplex}\footnote{If $\foj$ is contained in a cell of $\P$ of
codimension one or two, the coefficient ring has to be enlarged from
$A[\Lambda_\foj]$, the Laurent polynomial ring with exponents the integral
tangent vectors to $\foj$, to a localization at the slab functions $f_{\ul\rho}$
with $\ul\rho\supseteq\foj$. The careful treatment of this situation in
\cite{affinecomplex}, Chapter~4, assures that no denominators appear in the
coefficients of newly inserted walls.}. This shows that if $\theta$ is the
composition of automorphisms for walls in $\scrS_k$ containing $\foj$, then to
order $k+1$ there is a unique expansion
\begin{equation}
\label{Eqn: TVG computation}
\theta=\exp \big(\sum_i c_i z^{m_i}\partial_{n_i} \big),
\end{equation}
with $m_i\in\Lambda_x$, $n_i\in\check\Lambda_x$ for some $x\in\foj$ and $c_i$
contained in (a localization of) $A[\Lambda_\foj]$ and the order of each summand along $\foj$ equal to $k+1$. Each summand on the
right-hand side produces one new wall emanating from $\foj$. Now
all the rings involved are graded, and by the induction hypothesis $\theta$ is
homogeneous of degree zero. Moreover, in the computation of the exponential on
the right-hand side any cross terms have order strictly larger than $k+1$ and
therefore vanish. Hence omitting any summands $c_i z^{m_i}\partial_{n_i}$ with
$c_i z^{m_i}$ of non-zero degree from the right-hand side of \eqref{Eqn: TVG
computation} leads to another expression of $\theta$ as the exponential of a
vector field of order $k+1$. Thus by uniqueness of the expansion in \eqref{Eqn:
TVG computation}, any of the monomials $c_i z^{m_i}$ of newly inserted walls are
homogeneous of degree zero.

For joints contained in cells of codimensions one and two, we modify the slab
function $f_{\ul\fob}$ with $\ul\fob$ contained in a codimension one cell
$\rho\supseteq\foj$ by the addition of those terms $c_i z^{m_i}$ with
$m_i\in\Lambda_\rho$. Thus the slab functions stay homogeneous of degree zero in
the algorithm as well.

Apart from some subdivision of slabs or walls, which obviously does not spoil
homogeneity, the insertion of new walls and the change of slab functions, there
are two more changes to obtain $\scrS_{k+1}$. The first of these (Step~II in the
algorithm of \cite{affinecomplex}) is the propagation of changes made to one
slab function $f_{\ul\fob}$ to other slabs $\ul\fob'$ contained in the same
codimension one cell $\rho$, see \cite{affinecomplex}, \S3.5. This step only
adds homogeneous terms $c_i z^{m_i}$ of degree zero from $f_{\ul\fob}$ to
$f_{\ul\fob'}$. Hence this step also preserves homogeneity.

The last modification (Step~III in the algorithm of \cite{affinecomplex})
concerns the normalization condition by adding terms of the form $c_i
z^{m_{\ul\rho\ul\rho'}}$ with $c_i\in A[Q]$ to the slab functions
(\cite{affinecomplex}, \S3.6). The coefficient $c_i$ is again obtained from a
unique expansion of a homogeneous expression by a process that is linear at the
given order. Hence by the same argument as before each $c_i$ is homogeneous of
degree zero, as is $z^{m_{\ul\rho\ul\rho'}}$. Thus this step also does not spoil
homogeneity of the wall structure $\scrS_{k+1}$.

Note that homogeneity of $\scrS_{k+1}$ for all $k$ establishes the torus action
readily on $\foX$, not only on the complement of the codimension two locus
$\foX^\circ$.

In the projective setup the homogeneous coordinate ring becomes naturally
$\Gamma$-graded, hence the statement in this case.
\end{proof}

\begin{remark}
The action of the relative torus on $A$ being trivial implies that we have an
induced action on the fibre over any closed point in $\Spec A$, viewed as a log
space over the log point $(\Spec\kk, Q\oplus\kk^\times)$. This action keeps the
isomorphism class of this log space over $(\Spec\kk, Q\oplus\kk^\times)$, in
agreement with the interpretation of $\Spec A$ as a moduli space of log spaces
over a log point.

This is in contrast to the second action below, which acts effectively on $\Spec
A$.
\end{remark}
\medskip

The second torus, referred to as the \emph{regluing torus}, has character
lattice the dual of a finite index sublattice $H$ of $c_1\big (\breve\MPA(B)\big)
\subseteq H^1(B,\iota_*\check\Lambda)$. Here $c_1: \breve\MPA(B)\to
H^1(B,\iota_*\check\Lambda)$ is the connecting homomorphism coming from the
analogue for restricted PL-functions of the short exact sequence \eqref{Eqn:
MPA} (see \cite{logmirror1}, Definition~1.45). The regluing torus acts by
changing how $\foX$ is glued from open subsets. To define this action, we start
with the exact sequence
\[
\breve\PL(B)\stackrel{\kappa}{\lra} \breve\MPA(B)\stackrel{c_1}{\lra} H^1(B,\iota_*\check\Lambda).
\]
Neither $\kappa$ nor $c_1$ need to have saturated images. With the superscripts
``$\sat$'' and ``$\tor$'' denoting the saturation of a subgroup and the torsion
part of an abelian group, define
\[
\tilde H=c_1(\breve\MPA(B)),\quad
K:=\kappa\big(\breve\PL(B)\big)^\sat= c_1^{-1}\big(H^1(B,\iota_*\check\Lambda)^\tor\big),
\]
and choose complements $\tilde L$ to $K\subseteq\breve\MPA(B)$ and $F$ to
$\tilde H^\sat\subseteq H^1(B,\iota_*\check\Lambda)$. By taking $F$
the image under $\sigma_0$ of a complement to $q_f(\tilde H^\sat)=
\big(q_f(\tilde H)^\sat\big)$ in $H^1(B,\iota_*\Lambda)_f$, we may also assume
$F\subseteq \im(\sigma_0)$. We have direct sum decompositions
\[
\breve\MPA(B)= K\oplus \tilde L,\quad
H^1(B,\iota_*\check\Lambda)= \tilde H^\sat\oplus F,
\]
and $c_1$ composed with the quotient by the torsion subgroup of
$H^1(B,\iota_*\check\Lambda)$ maps $\tilde L$ isomorphically to $\tilde
H_f=\tilde H/\tilde H^\tor$. The dual of the inclusion $\tilde H\subseteq \tilde
H^\sat$ defines a finite index sublattice $(\tilde H^\sat)^*\subseteq \tilde
H^*$. The torus with character lattice $\tilde H^*$ acts canonically on the
finite \'etale extension $\kk[\tilde H^*\oplus F^*]$ of
$A=\kk[H^1(B,\iota_*\check\Lambda)^*]= \kk[(\tilde H^\sat)^*\oplus F^*]$. If
$H^1(B,\iota_*\check\Lambda)$ is torsion-free this action lifts to the pull-back
of the universal family.

In general, the action of the regluing torus depends on a good representative of
the universal gluing data $(s_0,\sigma_0)$ from \S\ref{Subsect: GS universal}
compatible with $c_1: \breve\MPA(B)\to H^1(B,\iota_*\check\Lambda)$. To find
this representative we may need another finite \'etale extension of base rings and go
over to an isogeneous torus. Denote by $q_f: H^1(B,\iota_*\check\Lambda)\to
H^1(B,\iota_*\check\Lambda)_f$ the quotient by the torsion subgroup.

\begin{lemma}
There is a finite index sublattice $L\subseteq \tilde L$ with the property
\[
\sigma_0\circ q_f|_{c_1(L)} =\id_{c_1(L)}.
\]
\end{lemma}
\begin{proof}
This follows since $\sigma_0$ being a right-inverse to $q_f$, the image of
$\sigma_0\circ q_f-\id$ lies in the torsion subgroup of
$H^1(B,\iota_*\check\Lambda)$.
\end{proof}

Choose a finite index sublattice $L\subseteq\tilde L$ as given by the lemma and
define $H=c_1(L)\subseteq\tilde H$. Then $c_1|_L$ factors over $\sigma_0|_{H_f}$
and $H\simeq H_f$ is torsion-free and isomorphic to $L$ via $c_1$. Note that
since $H^\sat=\tilde H^\sat$ we still have the direct sum decomposition
\[
H^1(B,\iota_*\check\Lambda)= H^\sat\oplus F.
\]
Since $H$ and $F$ are free we identify them with their respective images in
$H^1(B,\iota_*\check\Lambda)_f$. We take $L^*\simeq H^*$ as the character
lattice of the regluing torus. Now $H\oplus F$ defines a sublattice of finite
index in $H^1(B,\iota_* \check\Lambda)$. Dualizing we obtain a finite \'etale
extension of rings
\[
A=\kk[H^1(B,\iota_*\check\Lambda)^*]= \kk[(H^\sat)^*\oplus F^*]\lra
\tilde A:=\kk[H^*\oplus F^*].
\]
The action of the regluing torus is only defined after pulling back the universal
family $\foX\to \Spf \big(A\lfor Q\rfor\big)$ from \S\ref{Subsect: GS
universal} by the finite and \'etale morphism $\Spf\big(\tilde A\lfor
Q\rfor\big)\to \Spf \big(A\lfor Q\rfor\big)$.

For the projective case the alternative definition
\[
\breve\shMPA(B,\ZZ)= \breve\shPA(B,\ZZ)/\shAff(B,\ZZ)
\]
shows that $c_1: \breve\MPA(B)\to H^1(B,\iota_*\check\Lambda)$ factors over
$H^1(B,\shAff(B,\ZZ))$. Thus since $\ob_\PP^\ZZ$ is the connecting homomorphism
of the long exact sequence for \eqref{Eqn: Aff}, pushed-forward to $B$ by
$\iota_*$, we get $\ob_\PP^\ZZ\circ c_1=0$. Hence $\ob_\PP^\ZZ$ vanishes on
$H$ and we can choose $F\subseteq\ker \ob_\PP^\ZZ$ as a complement to
$H^\sat\subseteq \ker\ob_\PP^\ZZ$ to obtain an analogous ring extension
$A_\PP\subseteq \tilde A_\PP$. The resulting families are
\begin{equation}
\label{Eqn: Universal families pulled back}
\tilde\foX\lra \Spf\big(\tilde A\lfor Q\rfor\big),\qquad
\tilde\foX_\PP\lra \Spec\big(\tilde A_\PP\lfor Q\rfor\big).
\end{equation}
The action of the regluing torus on the base space of these families is defined
by $L^*$-gradings on $\tilde A$ and on $Q\subseteq \breve\MPA(B)^*$.
We define the grading on $Q$ by the transpose of the inclusion $L\subseteq
\breve\MPA(B,\ZZ)$, that is, by the composition
\begin{equation}
\label{Eqn: regluing torus grading I}
\delta_Q:Q\lra \breve\MPA(B)^*= K^*\oplus \tilde L^*\lra \tilde L^* \lra L^*.
\end{equation}
Here the first arrow is the inclusion, the second arrow the projection, the
third the transpose of $L\to\tilde L$. The grading on $\tilde A$ or $\tilde
A_\PP$ is defined by projection, followed by the isomorphism $H^*\simeq L^*$:
\begin{equation}
\label{Eqn: regluing torus grading II}
H^*\oplus F^*\lra H^* \stackrel{\simeq}{\lra} L^*.
\end{equation}

Next, recall the universal restricted MPA-function $\breve\varphi
\in\breve\MPA(B,Q)$ from Equation~\eqref{Eqn: Universal restricted MPA fct}.
Composing with $\delta_Q$, that is, restricting each kink
$\kappa_\rho(\breve\varphi)\in Q\subseteq\breve\MPA(B,\ZZ)^*$ to $L\subset
\breve\MPA(B,\ZZ)$, yields an element
\[
\breve\varphi|_L\in \breve\MPA(B,L^*) = \Hom \big(L,\breve\MPA(B,\ZZ)\big).
\]
Recall from \eqref{Eqn: Kinks of universal restricted MPA-fct} that as
a map $\breve\MPA(B,\ZZ)\to \ZZ$, the kink $\kappa_\rho(\breve\varphi)$ maps
$\psi$ to $\kappa_\rho(\psi)$. In other words, $\breve\varphi$ is the
tautological restricted MPA-function with values in $\breve\MPA(B,\ZZ)^*$ that
evaluated on $\psi\in\breve\MPA(B,\ZZ)$ retrieves $\psi$. In particular, the
restriction of $\breve\varphi$ to $L\subset\breve\MPA(B,\ZZ)$, viewed as a map
\[
\breve\varphi|_L:L\lra \breve\MPA(B,\ZZ),
\]
is just the inclusion map of $L$ as a subset of $\breve\MPA(B,\ZZ)$.
\footnote{Another way to put this is to observe that as an element of
$\breve\MPA(B,\MPA(B,\ZZ)^*)= \Hom(\breve\MPA(B,\ZZ),\breve\MPA(B,\ZZ))$, the
universal MPA-function $\breve\varphi$ is the identity. Restricting to $L$ then
gives the inclusion of $L$ into $\breve\MPA(B,\ZZ)$.}. Now let
$\big(\breve\varphi_\tau\big)_{\tau\in\P}$ be a piecewise linear
representative of $\breve\varphi$ and define
\[
\psi_\tau=\breve\varphi_\tau|_L =\delta_Q\circ\breve\varphi_\tau.
\]
Thus for any $\tau\in\P$ we now have a choice of $PL$-function $\psi_\tau:
|\Sigma_\tau|\to \Hom(L,\RR)$ on the fan $\Sigma_\tau$ defined by the tangent
wedges to $\tau$, with kinks
\[
\kappa_\rho(\psi_\tau)= \kappa_\rho(\breve\varphi)|_L\in L^*,
\]
for any $\rho\in\P^{[n-1]}$ containing $\tau$. For convenience of the later
discussion we take $\breve\varphi_\tau|_{\Lambda_\tau}=0$.\footnote{This choice
has the effect that the action is trivial on order zero monomials with exponents
tangent to the cell considered.} These choices give the desired good
representative of the gluing data $\sigma_0$ on $H=c_1(L)$. To make
a precise statement, recall that $B$ has an open cover $\W$ by open stars
$W_\tau$ of the barycentric subdivision of $\P$ containing $\Int\tau$
(\cite{logmirror1}, Definition~1.25). This open cover is acyclic for
$\iota_*\check\Lambda$ (\cite{logmirror1}, Lemma~5.5). Elements of
$H^1(B,\iota_*\check\Lambda)$ therefore can be represented by \v Cech 1-cocycles
$\big(s_{\omega\tau}\big)_{\omega,\tau}$ with labelling by pairs
$\omega,\tau\in\P$ with $\omega\subsetneq\tau$ and $s_{\omega\tau}\in
\Gamma(W_{\omega}\cap W_{\tau}, \iota_*\check\Lambda)$.

\begin{lemma}
\label{Lem: sigma_0 on L}
The restriction of $c_1$ to $L\subseteq \breve\MPA(B,\ZZ)$,
\[
c_1|_L: L\lra H^1(B,\iota_*\check\Lambda) 
\]
can be represented by the \v Cech 1-cocycle
$\big(\sigma_{\omega\tau}\big)_{\omega,\tau} \in
C^1(\W,\iota_*\check\Lambda)\otimes L^*$ with
\[
\sigma_{\omega\tau} = \psi_\tau|_{W_\omega\cap W_\tau} - \psi_\omega|_{W_\omega\cap W_\tau}.
\]
\end{lemma}

\begin{proof}
This is immediate from the fact that $c_1$ arises as a connecting homomorphism
in the long exact cohomology sequence for the analogue of \eqref{Eqn: MPA} (see
\cite{logmirror1}, Definition~1.45).
\end{proof}

We are now ready to construct the action of the regluing torus.

\begin{proposition}
\label{Prop: Regluing torus acting}
The actions of the torus $\Spec \big(\kk[L^*]\big)$ on $\Spf \big(\tilde A\lfor
Q\rfor\big)$ and on $\Spec \big(\tilde A_\PP\lfor Q\rfor\big)$ defined by the
gradings \eqref{Eqn: regluing torus grading I},\eqref{Eqn: regluing torus
grading II} lift to actions on the finite pull-backs $\tilde\foX\to
\Spf\big(\tilde A\lfor Q\rfor\big)$ and $\tilde\foX_\PP\to \Spec\big(\tilde
A_\PP\lfor Q\rfor\big)$ in \eqref{Eqn: Universal families pulled back} of the
universal families.
\end{proposition}

\begin{proof}
We only discuss the case of $\tilde\foX\to \Spf\big(\tilde A\lfor Q\rfor\big)$.
The statement for the projective family then follows as in the proof of
Proposition~\ref{Prop: Relative torus action in GS} by grading the homogeneous
coordinate ring.
\smallskip

\noindent
\ul{\emph{Step~1: Choice of gluing data.}}\ \ 
By functoriality, the gluing data describing the central fibres of $\tilde\foX\to
\Spf\big(\tilde A\lfor Q\rfor\big)$ are given by the image of
\[
(s_0,\sigma_0)\in H^1(B,\iota_*\check\Lambda\otimes \ul A^\times)=
H^1(B,\iota_*\check\Lambda\otimes\kk^\times)
\oplus\Hom\big(H^1(B,\iota_*\check\Lambda)_f,H^1(B,\iota_*\check\Lambda)\big)
\]
under
\[
A^\times= \kk^\times\oplus H^1(B,\iota_*\check\Lambda)_f^*
= \kk^\times\oplus (H^\sat)^*\oplus F^* \lra \tilde A^\times=
\kk^\times\oplus H^*\oplus F^*.
\]
Thus this base change simply leaves $s_0$ unchanged and restricts $\sigma_0$ to
$q_f(H\oplus F)\subseteq H^1(B,\iota_*\check\Lambda)_f = q_f(H^\sat\oplus F)$.
Denote this restriction of $\sigma_0$ by $\tilde\sigma_0: q_f(H\oplus F)\to
H^1(B,\iota_*\check\Lambda)$. If we abuse notation and identify
$H\oplus F$ with its image in $H^1(B,\iota_*\check\Lambda)_f$, then
$\tilde\sigma_0$ is just the inclusion $H\oplus F\to
H^1(B,\iota_*\check\Lambda)$. As a \v Cech 1-cocycle, $\tilde\sigma_0$ is
represented by $(\tilde\sigma_{\omega\tau})_{\omega,\tau}$ with
$\tilde\sigma_{\omega\tau}\in \Gamma(W_\omega\cap W_\tau,\iota_*\check\Lambda)
\otimes (H\oplus F)^*$.

Now Lemma~\ref{Lem: sigma_0 on L} gives particular \v Cech $1$-cocycles
for elements of $H=c_1(L)$. Choosing \v Cech representatives for the elements of
$F$ arbitrarily, we arrive at a \v Cech $1$-cocyle
$(\tilde\sigma_{\omega\tau})_{\omega,\tau}$ of $\tilde\sigma_0$ with
the property
\begin{equation}
\label{Eqn: representing tilde sigma}
\tilde\sigma_{\omega\tau}\circ q_f\circ c_1
= \sigma_{\omega\tau}=\psi_\tau|_{W_\omega\cap W_\tau} -
\psi_\omega|_{W_\omega\cap W_\tau} \in \Gamma(W_\omega\cap
W_\tau,\iota_*\check\Lambda)\otimes L^*.
\end{equation}
We write $\big(\tilde s_{\omega\tau}\big)_{\omega,\tau}$ for the corresponding
representative of the base-changed universal gluing data $(s_0,\tilde \sigma_0)$.
\smallskip

\noindent
\ul{\emph{Step~2: Definition of the grading.}}\ \ 
We are now in position to define the grading on our rings. As in the proof of
Proposition~\ref{Prop: Relative torus action in GS} we prove inductively that the
wall structure is homogeneous of degree $0$. The grading on the coefficent ring
$\tilde A\lfor Q\rfor$ has already been defined in \eqref{Eqn: regluing torus
grading I} and \eqref{Eqn: regluing torus grading II}. In \cite{affinecomplex}
the $k$-th order approximation to $\foX$ is glued from rings denoted
$R^k_{\omega\to\tau,\sigma}$ indexed by cells $\omega\subseteq\tau$ in $\P$ and
a reference maximal cell $\sigma$ containing $\tau$ (\cite{affinecomplex},
Construction~2.7). These $\tilde A\lfor Q\rfor$-algebras are given by the
localization at order zero slab functions $f_{v\to\rho}$ with
$\rho\supseteq\omega$, of a quotient by a monomial ideal of a monomial ring. The
monomials $z^m$ have exponents $m\in P_{\omega,\sigma}$ with
\begin{equation}
\label{Eqn: monoid P_omega,sigma}
P_{\omega,\sigma}=\big\{ m=(\ol m,h)\in \Lambda_\sigma\oplus Q^\gp\,\big|\,
h\in \breve\varphi_\omega(\ol m)+Q\big\}.
\end{equation}
Here $\breve\varphi_\omega$ is the representative of $\breve\varphi$ on $W_\omega$ chosen
above, now viewed as a piecewise linear map $\Lambda_\sigma\to Q^\gp$ by
means of a chart for the affine structure at any point of $\Int\omega$. This
description is independent of the choice of point by the local monodromy
invariance of our notion of piecewise linear function. The definition of
$P_{\omega,\sigma}$ in \cite{affinecomplex}, Construction~2.7 is more intrinsic,
but the equivalence with the one given here is not hard to show.

For $m=(\ol m,h)$ we now define
\[
\deg_{L^*} z^m= \delta_Q (h),
\]
with $\delta_Q$ the grading on $Q$ from \eqref{Eqn: regluing torus grading I}.
In particular, the degree of a monomial of order zero, that is $m=(\ol m,h)$
with $h=\breve\varphi_\omega(\ol m)$, equals $\psi_\omega(\ol
m)=\delta_Q\big(\breve\varphi_\omega(\ol m)\big)$.

It is instructive to write down explicitly the action on the rings in
codimension zero and one used elsewhere in this paper. On the rings for
chambers $R_\fou=R_\sigma$ the grading is trivial on the monomials by our choice
$\breve\varphi_\sigma=0$ and hence just comes from the grading on the
coefficients in $\tilde A[ Q]\subseteq \tilde A\lfor Q\rfor$. For a slab
$\ul\fob$ let $\rho$ be the codimension one cell containing $\ul\fob$ and
$\sigma,\sigma'$ the two adjacent maximal cells, respectively. Then our ring
$R_{\ul\fob}$ arises from the rings $R^k_{\omega\to\tau,\sigma}$ as a fibre
product (see the proof of \cite{affinecomplex}, Lemma~2.34):
\begin{equation}
\label{Eqn: R_fob as fibre product}
R_{\ul\fob}= R^k_{\rho\to\sigma,\sigma}\times_{R^k_{\rho\to\rho,\sigma}} R^k_{\rho\to\sigma',\sigma'}.
\end{equation}
The homomorphisms in this fibre product are defined by the relevant changes of
strata and change of chambers (\cite{affinecomplex}, Construction~2.24). Since
$\breve\varphi_\rho|_{\Lambda_\rho}=0$ the induced grading of the ring
$R_{\ul\fob}$ is trivial on the monomials $z^m$ with $m\in\Lambda_\rho$. For the
two remaining generators we have
\[
\deg_{L^*} Z_+= \psi_\rho(\xi)\in L^*,\quad
\deg_{L^*} Z_-= \psi_\rho(-\xi)\in L^*.
\]
Recall that our sign conventions say that $Z_+$ maps to a monomial on the maximal
cell $\sigma=\sigma(\rho)$, see \eqref{Eq: Localization morphism rho -> sigma}.
Provided the slab function $f_{\ul\fob}$ is homogeneous of degree zero, this
definition of the grading of monomials turns $R_{\ul\fob}$ into a graded $\tilde A\lfor
Q\rfor$-algebra. In fact, the only relation $Z_+ Z_- - f_{\ul\fob} t^{\kappa_\rho}$
is homogeneous of degree $\delta_Q(\kappa_\rho) =
\psi_\rho(\xi)+\psi_\rho(-\xi)\in L^*$.
\smallskip

\noindent
\ul{\emph{Step~3: Homogeneity of order zero slab functions.}}\ \ 
Lemma~\ref{Lem: Order zero slab functions} continues to hold for universal
gluing data with the same proof. Thus the order zero reduction of one of our
slab functions $f_{\ul\fob}$ equals $f_{\ul\rho}=\sum_{m\in\Delta(\rho,v)}
z^m$. The degree $\deg_{L^*}(z^m)$ of any monomial occuring in $f_{\ul\rho}$
equals $\psi_\rho(\ol m)$ with $m=m_{\ul\rho\ul\rho'}$. Now as a restricted
PL-function, $\psi_\rho$ is invariant under monodromy in $W_\rho\setminus\Delta$
and hence $\psi_\rho(\ol m)=0$ for any $m\in\Delta(\rho,v)$. Note this argument
holds regardless of our special choice of PL-functions $\psi_\rho$. This shows
that the order zero reduction of a slab function $f_{\ul\fob}$ is homogeneous of
degree zero, as an element of the ring $R^0_{\rho\to\rho,\sigma}$ for
$\rho\supset\ul\fob$. Homogeneity of the slab functions in the other rings
$R^0_{\omega\to\tau,\sigma}$ then follows by homogeneity of the gluing morphisms
of type~(I) in the following Step~4.
\smallskip

\noindent
\ul{\emph{Step~4: Homogeneity of gluing.}}\ \ 
We have to check homogeneity of the gluing morphisms in \cite{affinecomplex},
Construction~2.24. We assume inductively that the functions associated to slabs
and walls are homogeneous of degree zero. There are three cases.

(I) (Change of strata). For cells $\omega\subseteq\omega'
\subseteq\tau'\subseteq \tau\subseteq \sigma$ with $\sigma$ maximal, we have a
homomorphism\footnote{The formulas in \cite{affinecomplex} define homomorphisms
of log rings; here we only need and state the induced homorphisms of ordinary rings.
Moreover, since we work with lifted gluing data, the restriction of gluing data
to a maximal cell necessary in \cite{affinecomplex} is not needed here.}
\begin{equation}
\label{Eqn: Change of strata}
R^k_{\omega\to\tau,\sigma}\lra R^k_{\omega'\to\tau',\sigma},\quad
z^m\longmapsto \tilde s_{\omega\omega'}(m)^{-1}\cdot z^m.
\end{equation}
Assuming without loss of generality that $z^m$ has order zero, as an element of
$R^k_{\omega\to\tau,\sigma}$ we have $\deg_{L^*} z^m= \psi_\omega(\ol m)$. Hence
homogeneity follows from computing the degree of the right-hand side of \eqref{Eqn: Change of strata} in $R^k_{\omega'\to\tau',\sigma}$:
\[
\deg_{L^*} \big(\tilde s_{\omega\omega'}(m)^{-1}\cdot z^m\big)
= \deg_{L^*}\big (-\tilde\sigma_{\omega\omega'}(\ol m)\big) +\psi_{\omega'}(\ol m)
\ =\ -\sigma_{\omega\omega'}(\ol m)+\psi_{\omega'}(\ol m) = \psi_\omega(\ol m).
\]
For the second equality note from \eqref{Eqn: representing tilde sigma}
that $\tilde\sigma_{\omega\omega'}\circ q_f\circ c_1=\sigma_{\omega\omega'}$.
The last equality then follows from the definition of $\sigma_{\omega\omega'}$ in
Lemma~\ref{Lem: sigma_0 on L}.
\smallskip

(II.1) (Change of chambers $\fou\to\fou'$ with $\fou,\fou'\subseteq\sigma$).
This is a sequence of wall crossing homomorphisms, each defined by $\tilde
s_{\omega\sigma}(f_\fop)$ with $f_\fop$ homogeneous of degree zero. Thus
$f_\fop$ can be written as a sum of expressions $c_m z^m$ with $c_m\in\tilde
A[Q]$ and $z^m$ a monomial of order zero on $\sigma$. Now it holds $\deg_{L^*}
z^m=0$ by our choice of $\breve\phi_\sigma$, and then necessarily also
$\deg_{L^*} c_m=0$. As an element of $R^k_{\omega\to\tau,\sigma}$, the term $c_m
\tilde s_{\omega\sigma}(m) z^m$ then has degree
\[
\deg_{L^*}\big(c_m \tilde s_{\omega\sigma}(m) z^m\big) = \sigma_{\omega\sigma}(\ol m) +\psi_\omega(\ol m)=\psi_\sigma(\ol m)=0.
\]
\smallskip

(II.2) (Change of chambers $\fou\to\fou'$ with $\fou\subseteq\sigma$,
$\fou'\subseteq\sigma'$ and $\sigma\neq\sigma'$). This is the most interesting
case $\theta:R^k_{\omega\to\tau,\sigma}\to R^k_{\omega\to\tau,\sigma'}$ of two
chambers $\fou\subseteq\sigma$, $\fou'\subseteq\sigma'$ separated by a slab
$\ul\fob\subseteq\rho$. Now the homomorphism $R^k_{\omega\to\rho,\sigma}\to
R^k_{\omega\to\tau,\sigma}$ from (I) is a composition of a homogeneous
epimorphism with a localization at a product of order zero slab functions, see
\cite{affinecomplex}, Remark~2.9. Moreover, the homomorphism
$R^k_{\omega\to\rho,\sigma}\to R^k_{\rho\to\rho,\sigma}$, also from (I), is
induced by a localization homomorphism $P_{\omega,\sigma}\to P_{\rho,\sigma}$ of
the toric monoids from \eqref{Eqn: monoid P_omega,sigma} and is hence injective.
The analogous statements holds for $\sigma$ exchanged by $\sigma'$. Thus it is
enough to treat the case $\omega=\tau=\rho$ and prove homogeneity of the change
of chambers homomorphism $R^k_{\rho\to\rho,\sigma}\to
R^k_{\rho\to\rho,\sigma'}$.

All order zero monomials $z^m$ of the two rings with $\ol m\in\Lambda_\rho$ have
degree zero and are mapped to each other by identifying $\Lambda_\rho$ as
sublattices of $\Lambda_\sigma$ and $\Lambda_{\sigma'}$ via parallel transport
through $\rho\setminus\breve\Delta$. In particular, $\theta$ maps the localizing
elements to each other and hence is also homogeneous on the localized subrings
generated by $\Lambda_\rho$. Each of the two rings has two more generators
$z_+,z_-\in R^k_{\rho\to\rho,\sigma}$ and $z'_+,z'_-\in
R^k_{\rho\to\rho,\sigma'}$ with homogeneous relations
\[
z_+z_-= t^{\kappa_\rho},\quad z'_+z'_-= t^{\kappa_\rho}.
\]
The change of chamber homomorphism also depends on the choice of a vertex
$v\in\omega\subseteq\rho$, which selects one of the codimension one cells
$\ul\rho\subseteq\rho$ of the barycentric subdivision. Recall also the choice of
primitive normal vector $\xi=\xi(\rho)\in\Lambda_\sigma$ for
$\sigma=\sigma(\rho)$ and denote by $\xi'\in\Lambda_{\sigma'}$ the parallel
transport of $\xi$ to $\sigma'$ through $\Int\ul\rho$. For easy comparison with
the rings $R_{\ul\fob}$ with $\ul\fob=\fob\cap\ul\rho$ the decomposed slab from
\S\ref{Subsect: GS one-parameter}, we choose $z_+=
z^{(\xi,\breve\varphi_\rho(\xi))}$ and $z'_+=z^{(\xi',
\breve\varphi_\rho(\xi))}$. According to \cite{affinecomplex}, p.1349, and
\eqref{Eqn: GS slab functions} we have
\[
\begin{array}{lclcl}
\theta(z_+)&=& f_{\ul\fob}\cdot (z'_-)^{-1}\cdot t^{\kappa_\rho}  &=& f_{\ul\fob}\cdot z'_+\\
\theta(z_-)&=& f_{\ul\fob}^{-1}\cdot (z'_+)^{-1}\cdot t^{\kappa_\rho} & =& f_{\ul\fob}^{-1}\cdot z'_-.
\end{array}
\]
Now $\deg_{L^*} z_+=\deg_{L^*}{z'_+}= \psi_\rho(\xi)$ and $f_{\ul\fob}$ is homogeneous of degree zero, and similarly for $z_-$ This shows homogeneity of $\theta$.

This wall crossing homomorphism induces the grading on the ring $R_{\ul\fob}$ already
discussed in Step~2.
\smallskip

\noindent
\ul{\emph{Step~5: Homogeneity of wall structure.}}\ \ 
Assuming inductively that the wall structure $\scrS_k$ at order $k$ is
homogeneous of degree zero, following the steps in the algorithm of
\cite{affinecomplex} shows homogeneity of the wall structure $\scrS_{k+1}$ at
the next order, just as in the proof of Proposition~\ref{Prop: Relative torus
action in GS}.
\end{proof}

\begin{remark}
It is obvious from the constructions that our two group actions commute. Thus we
have an action of a product torus with character lattice $\breve\PL(B)^*\oplus
L^*$ on the finite \'etale pull-backs $\tilde\foX\to \Spf\big(\tilde A\lfor
Q\rfor\big)$ and $\tilde\foX_\PP\to \Spec\big(\tilde A_\PP\lfor Q\rfor\big)$ of
the universal families. Here $L\subseteq\breve\MPA(B)$ together with
$\kappa\big(\breve\PL(B)\big)$ spans a finite index subgroup of $\breve\MPA(B)$.
The first torus factor contains the subtorus with character lattice the
dual of the saturated sublattice $H^0(B,\iota_*\check\Lambda)\subseteq
\breve\PL(B)$ dealing with automorphisms of the family relative the base. The
quotient of the product torus $\Spec\big(\kk[\breve\PL(B)^*\oplus L^*]
\big)$ by this subtorus has character lattice $\breve \MPA(B)^*=
Q^\gp$, up to finite index. Hence up to isogeny it can be identified with the
torus for the affine toric variety $\Spec \kk[Q]$. Only the second torus,
with character lattice $L^*$, acts non-trivially on the fibre
$\Spec(\tilde A)$ of $\Spf \big(\tilde A\lfor Q\rfor\big)\to \Spf \big(\kk\lfor
Q\rfor\big)$ over $0$. This fibre parametrizes log structures relative to the
standard log point and hence only this part induces a non-trivial action on
$H^1(B,\iota_*\check\Lambda\otimes\kk^*)$.

Note that the projection $\Spf \big(\tilde A\lfor Q\rfor\big)\to \Spf \big(\kk\lfor
Q\rfor\big)$ is equivariant for the isogeny of tori
\begin{equation}
\label{Eqn: isogeny of acting tori}
\Spec\big( \kk[\kappa(\breve\PL)^*\oplus L^*]\big) \lra \Spec\big(
\kk[\breve\MPA(B,\ZZ)^*]\big),
\end{equation}
and the action of $\Spec\big( \kk[\kappa(\breve\PL)^*\oplus L^*]\big)$ lifts to
$\tilde\X$. Restricting to the fibre over $0$, we obtain a family $\tilde
X_0\to\Spec(\tilde A)$ of log schemes over the log-point
$(\Spec\kk,Q\oplus\kk^\times)$ parametrized by $\Spec(\tilde A)$. Any two closed
fibers of this family in the same orbit of the group action by $\Spec\big(
\kk[\kappa(\breve\PL)^*\oplus L^*]\big)$ are isomorphic as log-schemes over a
log point with monoid $Q$. But note that since the group action is non-trivial
on $(\Spec \kk,Q\oplus \kk^\times)$, the isomorphism is not an isomorphism
relative to this log point. Indeed, the fibres of $\tilde X_0\to \Spec(\tilde
A)$ classify isomorphism classes of certain log schemes over $(\Spec \kk,Q\oplus
\kk^\times)$, up to the finite base change from $A$ to $\tilde A$. From this
point of view, the action on the fibre over $0$ can be understood by pulling
back the given chart for the log point $(\Spec\kk,Q\oplus\kk^\times)$ by the
group action via \eqref{Eqn: isogeny of acting tori}.
\end{remark}

%===========================================================
\subsection{The non-simple case in two dimensions}
\label{Subsect: K3}
In two dimensions the local rigidity assumption in
\cite{affinecomplex} is empty. We can therefore also treat
non-simple singularities. The singularities are then at the
barycenters of edges with local affine monodromy conjugate to
$\left(\begin{smallmatrix}1&0\\r&1\end{smallmatrix}\right)$. Here
the edge is parallel to the first coordinate axis and
$r\ge 1$. We call such a singularity an \emph{$r$-fold singularity},
so $r=1$ is a simple singularity.

The following proposition generalizes \cite{logmirror1}, Theorem~5.4
to the non-simple case in two dimensions, formulated in the Legendre
dual fan picture as in loc.cit., thus not requiring any projectivity
assumption.

\begin{proposition}
\label{Prop: Log strs in 2d}
Let $(\check B,\check \P)$ be a closed two-dimensional polyhedral
affine manifold with positive singularities (at the barycenters of
the edges). Denote by $n_r$ the number of $r$-fold singularities and
let $K=\sum_r n_r(r-1)$.

Then the set of isomorphism classes of positive log Calabi-Yau
spaces over the standard log point $(\Spec\kk,\NN)$ with dual
intersection complex $(\check B,\check \P)$ modulo isomorphism
preserving $\check B$ is $H^1(\check B,\iota_*\Lambda \otimes\kk^\times)\times
\kk^K$.
\end{proposition}

\begin{proof}
Let $(X_0,\M_{X_0})$ be a log Calabi-Yau space as in the statement.
Reexamination of the proof of \cite{logmirror1}, Theorem~5.4, shows that there
exists a unique isomorphism class of lifted gluing data $\check s\in H^1(\check
B,\iota_*\Lambda \otimes\kk^\times)$ normalizing $(X_0,\M_{X_0})$
(\cite{logmirror1}, Definition~4.24). Unlike the simple case (\cite{logmirror1},
Theorem~5.2,2), now the normalization condition does not fix the log structure
completely. Let $\rho\in\check \P$ be an edge containing an $r$-fold singularity
and $f_{\rho,v}\in\CC[w]$ the slab function determining the log structure on the
toric stratum $X_\rho\subseteq X_0$ associated to $\rho$. Here $v\in\rho$ is a
choice of vertex and $w$ is the toric coordinate for $X_\rho\simeq\PP^1$ at the
corresponding zero-dimensional stratum. Then $\deg f_{\rho,v}=r$ and as in the
proof of \cite{logmirror1}, Theorem~5.2,2, the normalization condition
determines the constant and highest coefficients of $f_{\rho,v}$. Thus
$f_{\rho,v}=1+a_1 w+\ldots+ a_{r-1} w^{r-1} + c(\check s) w^r$ with $c(\check
s)\in\kk^\times$ determined by the gluing data. The other coefficients
$a_2,\ldots,a_{r-1}$ are completely free to vary, contributing a factor
$\kk^{r-1}$.

Once a choice of vertex on each edge with a singularity has been
made, this description is unique up to changing $\check s$ by a
coboundary.
\end{proof}

With the generalization to the non-simple case from Proposition~\ref{Prop: Log
strs in 2d}, the previous discussion of the case of simple singularities
generalizes with the only change of replacing $A_\PP=\kk[K_f^*]$ by $\tilde
A_\PP=A_\PP[\NN^K]= \kk[K_f^*\oplus\NN^K]$. In particular, the additional factor
$\NN^K$ does not affect projectivity, the obstruction $\ob_\PP^\ZZ$ still only
takes the first factor $H^1(B,\iota_*\check\Lambda)$ as an input.

An interesting additional feature is the preservation of singularities in the
family. To state this fact, denote by $\scrS_k$ the inductively obtained wall
structure on $(B,\P)$ that is consistent modulo $I_0^k$. Then for
$\rho\in\P^{[1]}$ let $f_\rho\in \tilde A_\PP[w^{\pm1}]$ be the order zero slab
function in $\scrS_k$ for the stratum $\rho$ as given by the codimension one
case of Definition~\ref{Def: Wall structure},1. Here $w$ is the toric coordinate
along the stratum $X_\rho\subseteq X_0$. Then $f_\rho$ equals the reduction
modulo $I_0$ of any slab function $f_{\ul\fob}$ with $\ul\fob\subseteq \rho$ and it is
related to $f_{\rho,v}$ in the proof of Proposition~\ref{Prop: Log strs in 2d}
via formula~\eqref{Eqn: GS slab functions}.

\begin{theorem}
Let $(B,\P)$ be a closed two-dimensional polyhedral affine manifold
with positive singularities and assume there exists a strictly
convex $\varphi\in\breve \MPA(B,\NN)$. Then the conclusions of
Theorem~\ref{Thm: Main theorem, simple case} hold with $A_\PP$
replaced by $\tilde A_\PP=A_\PP[\NN^K]$, yielding a family
\[
\foX_\PP\lra \Spec(A_\PP\lfor Q\rfor)\times_\kk \AA_\kk^K.
\]
Moreover, the completion of $\foX_\PP$ along the big cell of the toric
stratum for an edge $\rho\in\P^{[1]}$ is isomorphic relative $\tilde
A_\PP\lfor Q\rfor$ to
\[
\Spf\big( \tilde A_\PP\lfor Q\rfor\lfor x,y\rfor[w^{\pm 1}]/(xy -
f_\rho z^{\kappa_\rho}) \big).
\]
In particular, for a closed point $(s,a)\in \Spec(A_\PP) \times\AA_\kk^K$ with
$f_\rho(s,a)$ having $n^\rho_r$ zeros of order $r$, the generic fibre of the
restriction of $\foX_\PP$ to $\{(s,a)\}\times \Spec \kk\lfor Q\rfor$ has
$n^\rho_r$ many $A_{r-1}$ singularities with closure intersecting
$X_\rho\subseteq X_0$.
\end{theorem}

\begin{proof}
Recall that Legendre duality (\cite{logmirror1}, \S1.4) swaps the
fan and cone pictures (\cite{logmirror1}, Theorem~2.34) and 
transforms an $r$-fold singularity along an edge $\rho$ into an
$r$-fold singularity along the Legendre-dual edge
(\cite{logmirror1}, Proposition~1.50,1). Thus Proposition~\ref{Prop:
Log strs in 2d} applies to the Legendre dual of $(B,\P,\varphi)$ for
some choice of strictly convex $\varphi\in \breve\MPA(B,\NN)$.

The local model for the order $k$ deformation $\foX_\PP^\circ$ along the
codimension one stratum $X_\rho\subseteq X_0$ is $Z_+Z_-=f_{\ul\fob}
z^{\kappa_\rho}$ in $\tilde A_\PP[Q]/(I_0^{k+1})[Z_+, Z_-,
w^{\pm1}]$. Now by the inductive construction, at any finite order, a
slab function $f_{\ul\fob}$ for $\ul\fob\subseteq\rho$ is of the form $f_{\ul\fob}=
f_\rho\cdot\prod_\mu(1+a_\mu w^{l_\mu})$ with $a_\mu\in I_0$. At
order $k$ the slab $\ul\fob$ only factors with $a_\mu\in I_0^k$ are
being added. Thus taking the limit $k\to\infty$ for slabs $\ul\fob_k$
containing a general point of $\rho$, this product converges to some
$f_{\ul\fob}=f_\rho\cdot h \in\tilde A_\PP\lfor Q\rfor [w^{\pm1}]$. Now
take $x=Z_+$, $y=h^{-1}Z_-$ to arrive at the stated equation for the
formal completion along $X_\rho^\circ\subseteq \foX_\PP^\circ$.
\end{proof}

\end{appendix}

%===========================================================
%===========================================================


\begin{thebibliography}{cccccccc}
\bibitem[ACGS]{ACGS} D.~Abramovich, Q.~Chen, M.~Gross, B.~Siebert: 
	\emph{Punctured logarithmic curves},
	in preparation.
\bibitem[Al]{alexeev} V.~Alexeev:
	\emph{Complete moduli in the presence of semiabelian group
	action},
	Ann.\ of Math.~\textbf{155} (2002), 611--708.
\bibitem[An]{andersen} J.E.~Andersen:
	\emph{Hitchin's connection, Toeplitz operators, and symmetry
	invariant deformation quantization},
	Quantum Topol.~\textbf{3} (2012), 293--325. 
\bibitem[APW]{axelrodetal} S.~Axelrod, S.~Della Pietra, E.~Witten:
	\emph{Geometric quantization of Chern-Simons gauge theory},
	J.\ Differential Geom.~\textbf{33} (1991), 787--902. 
\bibitem[BMN]{baieretal} T.~Baier, J.M.~Mour\~ao, J.P.~Nunes:
	\emph{Quantization of abelian varieties: distributional sections
	and the transition from K\"ahler to real polarizations},
	J.\ Funct.\ Anal.~\textbf{258} (2010), 3388--3412. 
\bibitem[BiLa]{birkenhakelange} C.\ Birkenhake, H.\ Lange:
	\emph{Complex abelian varieties},
	second edition, Springer~2004.
\bibitem[BBR]{BBR} M.~Brun, W.~Bruns, T.~R\"omer:
	\emph{Cohomology of partially ordered sets and local cohomology
	of section rings},
	Adv.\ Math.\ {\bf 208} (2007), 210--235. 
\bibitem[BH]{BH} W.~Bruns, J.~Herzog: \emph{Cohen-Macaulay rings},
	Cambridge Studies in Advanced Mathematics, {\bf 39},
	Cambridge University Press, Cambridge, 1993. xii+403 pp.
\bibitem[CPS]{CPS} M.\ Carl, M.\ Pumperla, B.\ Siebert:
	\emph{A tropical view on Landau-Ginzburg models},
	preprint.
\bibitem[CZZ]{CGMMRSW} M.~Cheung, M.~Gross, G.~Muller, G.~Musiker,
	D.~Rupel, S.~Stella, and H.~Williams:
	\emph{The greedy basis equals the theta basis},
	J.\ Combin.\ Theory Ser.~A~\textbf{145} (2017), 150--171.
\bibitem[DBr]{DBr} P.~Aspinwall, T.~Bridgeland, A.~Craw, M.~Douglas,
	M.~Gross, A.~Kapustin, G.~Moore, G.~Segal, B.~Szendr\H{o}i,
	P.~Wilson:
	\emph{Dirichlet branes and mirror symmetry},
	Clay Mathematics Monographs, \textbf{4}, AMS 2009.
\bibitem[Fr]{Friedman} R.~Friedman,
	\emph{Global smoothings of varieties with normal crossings,}
	Ann.\ Math.\ {\bf 118}, 75--114 (1983).
\bibitem[Fu]{Fukaya} K.~Fukaya:
	\emph{Morse homotopy, $A^\infty$-category, and Floer
	homologies},
	Proceedings of GARC Workshop on Geometry and Topology~'93
	(Seoul, 1993), 1--102,
	Lecture Notes Ser., 18, Seoul Nat. Univ., Seoul, 1993. 
\bibitem[GiSz]{GiSz} P.~Gille, T.~Szamuely:
	\emph{Central simple algebras and Galois cohomology},
	Cambridge Studies in Advanced Mathematics, 
	{\bf 101}, Cambridge University Press, Cambridge, 2006, xii+343pp.
\bibitem[GH]{GH} W.~Goldman, and M.~Hirsch:
	\emph{The radiance obstruction  and parallel forms on affine
	manifolds},
	Trans.\ Amer.\ Math.\ Soc.\ {\bf 286} (1984), 629--649.
\bibitem[Gr1]{GrBB} M.~Gross:
	\emph{Toric Degenerations and Batyrev-Borisov Duality},
	Math.\ Ann.\ \textbf{333} (2005), 645--688.
\bibitem[Gr2]{PP2mirror} M.\ Gross:
	\emph{Mirror symmetry for $\PP^2$ and tropical geometry},
	Adv.~Math.\ {\bf 224} (2010), 169--245.
\bibitem[Gr3]{TGMS} M.\ Gross:
	\emph{Tropical geometry and mirror symmetry},
	CBMS Regional Conf. Ser. in Math.~114, A.M.S., 2011.
\bibitem[GHK1]{GHK1} M.\ Gross, P.\ Hacking, S.\ Keel:
	\emph{Mirror symmetry for log Calabi-Yau surfaces I},
	 Publ.\ Math.\ Inst.\ Hautes \'Etudes Sci.~\textbf{122} (2015),
	 65--168. 
\bibitem[GHK2]{GHK2} M.\ Gross, P.\ Hacking, S.\ Keel:
	\emph{Mirror symmetry for log Calabi-Yau surfaces II},
	in preparation.
\bibitem[GHKK]{GHKK} M.~Gross, P.~Hacking, S.~Keel, M.~Kontsevich:
	\emph{Canonical bases for cluster algebras},
	J.\ Amer.\ Math.\ Soc.~\textbf{31} (2018), 497--608. 
\bibitem[GHKS]{GHKS} M.\ Gross, P.\ Hacking, S.\ Keel, B.\ Siebert:
	\emph{Theta functions and K3 surfaces},
	in preparation.
\bibitem[GPS]{GPS} M.\ Gross, R.\ Pandharipande, B.\ Siebert,
	\emph{The tropical vertex},
	Duke Math.\ J.\ {\bf 153}, (2010) 297--362.
\bibitem[GrSi1]{gokova} M.\ Gross, B.\ Siebert:
	\emph{Affine manifolds, log structures, and mirror symmetry},
	Turkish J.\ Math.\ {\bf 27} (2003), 33--60.
\bibitem[GrSi2]{logmirror1} M.\ Gross, B.\ Siebert:
	\emph{Mirror symmetry via logarithmic degeneration data I},
	J.\ Differential Geom.~\textbf{72} (2006), 169--338.
\bibitem[GrSi3]{logmirror2} M.\ Gross, B.\ Siebert:
	\emph{Mirror symmetry via logarithmic degeneration data~II},
	J.\ Algebraic Geom.~\textbf{19} (2010), 679--780.
\bibitem[GrSi4]{affinecomplex} M.\ Gross, B.\ Siebert:
	\emph{From real affine to complex geometry},
	Ann.\ of Math.~\textbf{174} (2011), 1301--1428.
\bibitem[GrSi5]{invitation} M.\ Gross, B.\ Siebert:
	\emph{An invitation to toric degenerations},
	Surv.\ Differ.\ Geom.~\textbf{16}, 43--78, Int.\ Press~2011.
\bibitem[GrSi6]{thetasurvey} M.\ Gross, B.\ Siebert:
	\emph{Theta functions and mirror symmetry},
	Surv.\ Differ.\ Geom.~\textbf{21}, International Press 2016.
\bibitem[GrSi7]{UtahSurvey} M.\ Gross, B.\ Siebert:
	\emph{Intrinsic mirror symmetry and punctured Gromov-Witten invariants},
	in: \textsl{Algebraic geometry: Salt Lake City~2015}, 199--230,
	Proc.\ Sympos.\ Pure Math.~{97.2}, AMS~2018.
\bibitem[Gt1]{EGAII} A.~Grothendieck:
	\textsl{\'El\'ements de g\'eom\'etrie alg\'ebrique II:
	\'Etude globale \'el\'ementarire de quelques classes de
	morphismes},
	Publ.\ Math.\ Inst.\ Hautes \'Etud.\ Sci.\ \textbf{17} (1961).
\bibitem[Gt2]{EGAIV.2} A.~Grothendieck:
	\textsl{\'El\'ements de g\'eom\'etrie alg\'ebrique IV:
	\'Etude locale des sch\'emas et des morphismes de sch\'emas},
	Publ.\ Math.\ Inst.\ Hautes \'Etud.\ Sci.\ \textbf{24} (1965).
%\bibitem{EGA IV.3} A.~Grothendieck:
%	\emph{\'El\'ements de g\'eom\'etrie alg\'ebrique IV: \'Etude
%	locale des sch\'emas et de morphismes de sch\'emas},
%	Publ.\ Math., Inst.\ Hautes \'Etud.\ Sci.\ \textbf{28} (1966).
\bibitem[Hk]{hacking} P.~Hacking:
	\emph{Compact moduli of plane curves},
	Duke Math.\ J.~\textbf{124} (2004), 213--257.
\bibitem[Hr1]{hartshorneLC} R.~Hartshorne:
	\emph{Local cohomology},
	Lecture Notes in Mathematics~\textbf{41}, Springer~1967. 
\bibitem[Hr2]{hartshorneAG} R.~Hartshorne:
	\emph{Algebraic Geometry},
	Springer~1977.	
\bibitem[Ht]{hitchin} N.~Hitchin:
	\emph{Flat connections and geometric quantization},
	Comm.\ Math.\ Phys.~\textbf{131} (1990), 347--380. 
\bibitem[KoNi]{kobnim} S.\ Kobayashi, K.\ Nomizu:
	\emph{Foundations of differential geometry}, Vol I,
	Interscience Publishers~1963.
\bibitem[KoXu]{KollarXu} J.~Koll\'ar, C.~Xu:
	\emph{The dual complex of Calabi-Yau pairs},
	Invent.\ Math.~\textbf{205} (2016), 527-–557. 
\bibitem[Ma]{matsumura}  H.\ Matsumura:
	\emph{Commutative ring theory},
	 Cambridge University Press 1989.
\bibitem[Mu1]{Mu1} D.~Mumford: \emph{On the equations defining Abelian 
         varieties II}, Inv.\ Math.\ {\bf 3}, (1967) 75--135.
\bibitem[Mu2]{mumford} D.~Mumford:
	\emph{An analytic construction of degenerating abelian varieties
	over complete rings},
	Compositio Math.~\textbf{24} (1972), 239--272.
\bibitem[NiSi]{nisi} T.\ Nishinou, B.\ Siebert:
	\emph{Toric degenerations of toric varieties and tropical curves},
	Duke Math.\ J.~\textbf{135} (2006), 1--51.
\bibitem[NX]{NX} J.~Nicaise, C.~Xu:
	\emph{The essential skeleton of a degeneration of algebraic
	varieties},
	Amer.\ J.\ Math.~\textbf{138} (2016), 1645--1667.
\bibitem[PZ]{PZ} A.\ Polishchuk, E.\ Zaslow:
	\emph{Categorical mirror symmetry: the elliptic curve},
	 Adv.\ Theor.\ Phys.~\textbf{2} (1998), 443--470.
\bibitem[RS]{RS} H.~Ruddat, B.~Siebert:
	\emph{Canonical coordinates in toric degenerations},
	preprint~\texttt{arXiv:1409.4750 [math.AG]}, 39pp.
\bibitem[Ty]{tyurin} A.~Tyurin:
	\emph{Geometric quantization and mirror symmetry},
	preprint~\texttt{arXiv:math/9902027 [math.AG]}, 53pp.
\bibitem[Yo]{yoneda}  N.\ Yoneda:
	\emph{Note in products in $\Ext$},
	Proc.\ Amer.\ Math.\ Soc.~\textbf{9} (1958), 873--875. 
\bibitem[Yu]{Yu} S.~Yuzvinsky: \emph{Cohen-Macaulay rings of sections},
        Adv. Math.~\textbf{63} (1987), 172--195.
\end{thebibliography}
\end{document}